\documentclass[fleqn,12pt]{article}
\usepackage{amsmath,amssymb,amsfonts,mathrsfs,color,enumitem,amsthm,bbm,framed,graphicx,natbib, tocloft, rotating, threeparttable, subcaption, lscape, sectsty,forloop}
\usepackage{hyperref, caption, subcaption, pdfpages}
\usepackage{appendix}
\usepackage[onehalfspacing]{setspace}
\hypersetup{pdfborder = {0 0 0},colorlinks=true,linkcolor=blue,citecolor=blue}

\usepackage{dutchcal} %lowercase script math letters, used for kernel order k below.
\usepackage[mathscr]{euscript} %use for \N for Normal distribution, closer to the original mathcal

\newtheorem{theorem}{Theorem}
\newtheorem{corollary}{Corollary}

\newtheorem{lemma}{Lemma}
\newtheorem{assumption}{Assumption}
\numberwithin{assumption}{section}

\theoremstyle{definition}
\newtheorem{remark_tmp}{Remark}
\newenvironment{remark}
	{ \begin{remark_tmp} 	}
	{ 
		
		\qed 
		\end{remark_tmp} 
	}

\allowdisplaybreaks[4]

	\newcommand{\sumi}{\sum_{i=1}^n}

	\DeclareMathOperator*{\argmin}{arg\,min}
	
	\DeclareMathOperator*{\tr}{tr}
	\DeclareMathOperator*{\sign}{sgn}
	\DeclareMathOperator*{\supp}{supp}

	\newcommand{\X}{\mathcal{X}}
	
	\newcommand{\E}{\mathbb{E}}
	
	\newcommand{\V}{\mathbb{V}}
	\renewcommand{\P}{\mathbb{P}}

	\newcommand{\Eqref}[1]{Eqn.\ \eqref{#1}}

\newcommand{\shat}{\hat{\sigma}}

	\newcommand{\US}{\mathtt{us}}
	\newcommand{\BC}{\mathtt{bc}}
	\newcommand{\RBC}{\mathtt{rbc}}
	\newcommand{\MSE}{\mathtt{mse}}
	\newcommand{\MV}{\mathtt{mv}}
	\newcommand{\ROT}{\mathtt{rot}}   %rule of thumb
	\newcommand{\PI}{\mathtt{dpi}}   %plug in
	\newcommand{\GJ}{\mathtt{GJ}}   %generalize jackknife

		\newcommand{\tus}{T_\US}
		\newcommand{\tbc}{T_\BC}
		\newcommand{\trbc}{T_\RBC}

		\newcommand{\eus}{\eta_\US}
		\newcommand{\ebc}{\eta_\BC}
		\newcommand{\et}{\tilde{\eta}}
		\newcommand{\etus}{\et_\US}
		\newcommand{\etbc}{\et_\BC}
		
		\newcommand{\ethbc}{\hat{\tilde{\eta}}_\BC}

		\newcommand{\sus}{\sigma_\US}
		\newcommand{\srbc}{\sigma_\RBC}
		\newcommand{\shatus}{\hat{\sigma}_\US}
		\newcommand{\shatrbc}{\hat{\sigma}_\RBC}	

		\newcommand{\ius}{I_\US}
		\newcommand{\ibc}{I_\BC}
		\newcommand{\irbc}{I_\RBC}

	\newcommand{\h}{h}
	\renewcommand{\b}{b}

\DeclareMathOperator*{\diag}{diag}
\DeclareMathOperator*{\vech}{vech}

	\newcommand{\Ttilde}{\tilde{T}}
		
	\newcommand{\st}{\tilde{\sigma}}

	\newcommand{\stus}{\st_\US}
	\newcommand{\strbc}{\st_\RBC}	

		\newcommand{\smvt}{\bar{\sigma}}

		\newcommand{\bb}{\boldsymbol{b}}

	\newcommand{\bbeta}{\boldsymbol{\beta}}

	\newcommand{\bhat}{\boldsymbol{\hat{\beta}}}

	\newcommand{\be}{\mathbf{e}}

	\newcommand{\br}{\mathbf{r}}
	\newcommand{\bR}{\mathbf{R}}
	\newcommand{\bW}{\mathbf{W}}

	\newcommand{\bY}{\mathbf{Y}}

	\newcommand{\G}{\boldsymbol{\Gamma}}
	\newcommand{\Gp}{\G_p}
	\newcommand{\Gq}{\G_q}
	\newcommand{\Gt}{\tilde{\G}}
	\newcommand{\Gpt}{\Gt_p}
	\newcommand{\Gqt}{\Gt_q}
	\renewcommand{\L}{\boldsymbol{\Lambda}}
	\newcommand{\Lp}{\L_p}
	\newcommand{\Lt}{\tilde{\L}}

	\newcommand{\Psit}{\boldsymbol{\tilde{\Psi}}}
	
	\newcommand{\Psic}{\boldsymbol{\check{\Psi}}}
	\newcommand{\Psihat}{\boldsymbol{\hat{\Psi}}}
	
	\newcommand{\bSig}{\boldsymbol{\Sigma}}
	\newcommand{\Shat}{\boldsymbol{\hat{\Sigma}}}
	
	\newcommand{\bQ}{\mathbf{Q}}

	\newcommand{\bXi}{\boldsymbol{\Xi}}

	\renewcommand{\l}{\ell}
	\newcommand{\e}{\varepsilon} % I just hate typing the full word varepsilon

	\newcommand{\N}{\mathscr{N}}

	\renewcommand{\k}{\mathcal{k}}

	\setenumerate[1]{label*={\bf(\alph*)}}

\begin{document}
\setcounter{page}{0}

\title{\vspace{-.5in}On the Effect of Bias Estimation on Coverage Accuracy in Nonparametric Inference\footnotetext{
Sebastian Calonico is Assistant Professor of Economics, Department of Economics, University of Miami,  Coral Gables, FL 33124 (email: scalonico@bus.miami.edu).
Matias D. Cattaneo is Professor of Economics and Statistics, Department of Economics and Department of Statistics, University of Michigan, Ann Arbor, MI 48109 (email: cattaneo@umich.edu).
Max H. Farrell is Assistant Professor of Econometrics and Statistics, Booth School of Business, University of Chicago, Chicago, IL 60637 (email: max.farrell@chicagobooth.edu). 
The second author gratefully acknowledges financial support from the National Science Foundation (SES 1357561 and SES 1459931).
We thank Ivan Canay, Xu Cheng, Joachim Freyberger, Bruce Hansen, Joel Horowitz, Michael Jansson, Francesca Molinari, Ulrich M\"uller, and Andres Santos for thoughtful comments and suggestions, as well as seminar participants at Cornell, Cowles Foundation, CREST Statistics, London School of Economics, Northwestern, Ohio State University, Princeton, Toulouse School of Economics, University of Bristol, and University College London. The Associate Editor and three reviewers also provided very insightful comments that improved this manuscript.}
}
\author{Sebastian Calonico \\ Department of Economics \\ University of Miami \\ Coral Gables, FL 33124 \and Matias D. Cattaneo \\ Department of Economics \\ Department of Statistics \\ University of Michigan \\ Ann Arbor, MI 48109 \and Max H. Farrell \\ Booth School of Business \\ University of Chicago \\Chicago, IL 60637}
\date{December 16, 2017}
%\date{\today}
\maketitle

\thispagestyle{empty}

\newpage
\setcounter{page}{0}
\thispagestyle{empty}

\begin{abstract}
Nonparametric methods play a central role in modern empirical work. While they provide inference procedures that are more robust to parametric misspecification bias, they may be quite sensitive to tuning parameter choices. We study the effects of bias correction on confidence interval coverage in the context of kernel density and local polynomial regression estimation, and prove that bias correction can be preferred to undersmoothing for minimizing coverage error and increasing robustness to tuning parameter choice. This is achieved using a novel, yet simple, Studentization, which leads to a new way of constructing kernel-based bias-corrected confidence intervals. In addition, for practical cases, we derive coverage error optimal bandwidths and discuss easy-to-implement bandwidth selectors. For interior points, we show that the MSE-optimal bandwidth for the original point estimator (before bias correction) delivers the fastest coverage error decay rate after bias correction when second-order (equivalent) kernels are employed, but is otherwise suboptimal because it is too ``large''. Finally, for odd-degree local polynomial regression, we show that, as with point estimation, coverage error adapts to boundary points automatically when appropriate Studentization is used; however, the MSE-optimal bandwidth for the original point estimator is suboptimal. All the results are established using valid Edgeworth expansions and illustrated with simulated data. Our findings have important consequences for empirical work as they indicate that bias-corrected confidence intervals, coupled with appropriate standard errors, have smaller coverage error and are less sensitive to tuning parameter choices in practically relevant cases where additional smoothness is available.
\end{abstract}


\newpage

\doublespacing

%%%%%%%%%%%%%%%%%%%%%%%%%%%%%%%%%%%
%%%%%%%%%%%%%%%%%%%%%%%%%%%%%%%%%%%
%%%%%%%%%%%%%%%%%%%%%%%%%%%%%%%%%%%
\section{Introduction}
	\label{sec:intro}

Nonparametric methods are widely employed in empirical work, as they provide point estimates and inference procedures that are robust to parametric misspecification bias. Kernel-based methods are commonly used to estimate densities, conditional expectations, and related functions nonparametrically in a wide variety of settings. However, these methods require specifying a bandwidth and their performance in applications crucially depends on how this tuning parameter is chosen. In particular, valid inference requires the delicate balancing act of selecting a bandwidth small enough to remove smoothing bias, yet large enough to ensure adequate precision. Tipping the scale in either direction can greatly skew results. This paper studies kernel density and local polynomial regression estimation and inference based on the popular Wald-type statistics and demonstrates (via higher-order expansions) that by coupling explicit bias correction with a novel, yet simple, Studentization, inference can be made substantially more robust to bandwidth choice, greatly easing implementability.

Perhaps the most common bandwidth selection approach is to minimize the asymptotic mean-square error (MSE) of the point estimator, and then use this bandwidth choice even when the goal is inference. So difficult is bandwidth selection perceived to be, that despite the fact that the MSE-optimal bandwidth leads to \emph{invalid} confidence intervals, even asymptotically, this method is still advocated, and is the default in most popular software. Indeed, \citet[p. 1446]{Hall-Kang2001_AoS} write: ``there is a growing belief that the most appropriate approach to constructing confidence regions is to estimate [the density] in a way that is optimal for pointwise accuracy\ldots. [I]t has been argued that such an approach has advantages of clarity, simplicity and easy interpretation.''

The underlying issue, as formalized below, is that bias must be removed for valid inference, and the MSE-optimal bandwidth (in particular) is ``too large'', leaving a bias that is still first order. Two main methods have been proposed to address this: undersmoothing and explicit bias correction. We seek to compare these two, and offer concrete ways to better implement the latter. Undersmoothing amounts to choosing a bandwidth smaller than would be optimal for point estimation, then arguing that the bias is smaller than the variability of the estimator asymptotically, leading to valid distributional approximations and confidence intervals. In practice this method often involves simply shrinking the MSE-optimal bandwidth by an ad-hoc amount. The second approach is to bias correct the estimator with the explicit goal of removing the bias that caused the invalidity of the inference procedure in the first place. 

It has long been believed that undersmoothing is preferable for two reasons. First, theoretical studies showed inferior asymptotic coverage properties of bias-corrected confidence intervals. The pivotal work was done by \citet{Hall1992_AoS_density}, and has been relied upon since. Second, implementation of bias correction is perceived as more complex because a second (usually different) bandwidth is required, deterring practitioners. However, we show theoretically that bias correction is always as good as undersmoothing, and better in many practically relevant cases, if the new standard errors that we derive are used. Further, our findings have important implications for empirical work because the resulting confidence intervals are more robust to bandwidth choice, including to the bandwidth used for bias estimation. Indeed, the two bandwidths may be set equal, a simple and automatic choice that performs well in practice and is optimal in certain objective senses. 

Our proposed robust bias correction method delivers valid confidence intervals (and related inference procedures) even when using the MSE-optimal bandwidth for the original point estimator, the most popular approach in practice. Moreover, we show that at interior points, when using second-order kernels or local linear regressions, the coverage error of such intervals vanishes at the best possible rate. (Throughout, the notion of ``optimal'' or ``best'' rate is defined as the fastest achievable coverage error decay for a \emph{fixed} kernel order or polynomial degree; and is also different from optimizing point estimation.) When higher-order kernels are used, or boundary points are considered, we find that the corresponding MSE-optimal bandwidth leads to asymptotically valid intervals, but with suboptimal coverage error decay rates, and must be shrunk (sometimes considerably) for better inference.

Heuristically, employing the MSE-optimal bandwidth for the original point estimator, prior to bias correction, is like undersmoothing the bias-corrected point estimator, though the latter estimator employs a possibly random, $n$-varying kernel, and requires a different Studentization scheme. It follows that the conventional MSE-optimal bandwidth commonly used in practice need not be optimal, even after robust bias correction, when the goal is inference. Thus, we present new coverage error optimal bandwidths and a fully data-driven direct plug-in implementation thereof, for use in applications. In addition, we study the important related issue of asymptotic length of the new confidence intervals.

Our comparisons of undersmoothing and bias correction are based on Edgeworth expansions for density estimation and local polynomial regression, allowing for different levels of smoothness of the unknown functions. We prove that explicit bias correction, coupled with our proposed standard errors, yields confidence intervals with coverage that is as accurate, or better, than undersmoothing (or, equivalently, yields dual hypothesis tests with lower error in rejection probability). Loosely speaking, this improvement is possible because explicit bias correction can remove more bias than undersmoothing, while our proposed standard errors capture not only the variability of the original estimator but also the additional variability from bias correction. To be more specific, our robust bias correction approach yields higher-order refinements whenever additional smoothness is available, and is asymptotically equivalent to the best undersmoothing procedure when no additional smoothness is available.

Our findings contrast with well established recommendations: \citet{Hall1992_AoS_density} used Edgeworth expansions to show that undersmoothing produces more accurate intervals than explicit bias correction in the density case and \citet{Neumann1997_Statistics} repeated this finding for kernel regression. The key distinction is that their expansions, while imposing the same levels of smoothness as we do, crucially relied on the assumption that the bias correction was first-order negligible, essentially forcing bias correction to remove less bias than undersmoothing. In contrast, we allow the bias estimator to potentially have a first order impact, an alternative asymptotic experiment designed to more closely mimic the finite-sample behavior of bias correction. Therefore, our results formally show that whenever additional smoothness is available to characterize leading bias terms, as is usually the case in practice where MSE-optimal bandwidth are employed, our robust bias correction approach yields higher-order improvements relative to standard undersmoothing.

Our standard error formulas are based on fixed-$n$ calculations, as opposed to asymptotics, which also turns out to be important. We show that using asymptotic variance formulas can introduce further errors in coverage probability, with particularly negative consequences at boundary points. This turns out to be at the heart of the ``quite unexpected'' conclusion found by \citet[Abstract]{Chen-Qin2002_SJS} that local polynomial based confidence intervals are not boundary-adaptive in coverage error: we prove that this is not the case with proper Studentization. Thus, as a by-product of our main theoretical work, we establish higher-order boundary carpentry of local polynomial based confidence intervals that use a fixed-$n$ standard error formula, a result that is of independent (but related) interest.

This paper is connected to the well-established literature on nonparametric smoothing, see \citet{Wand-Jones1995_book}, \citet{Fan-Gijbels1996_book}, \citet{Horowitz2009_book}, and \citet{Ruppert-Wand-Carroll2009_book} for reviews. For more recent work on bias and related issues in nonparametric inference, see \citet{Hall-Horowitz2013_AoS}, \citet{Calonico-Cattaneo-Titiunik2014_Ecma}, \citet{Armstrong-Kolesar2015_minimax}, \citet{Schennach2015_bias}, and references therein. We also contribute to the literature on Edgeworth expansions, which have been used both in parametric and, less frequently, nonparametric contexts; see, e.g., \citet{Bhattacharya-Rao1976_book} and \citet{Hall1992_book}. Fixed-$n$ versus asymptotic-based Studentization has also captured some recent interest in other contexts, e.g., \citet{Mykland-Zhang2015_WP}. Finally, see \citet{Calonico-Cattaneo-Farrell2016_EE-RD} for uniformly valid Edgeworth expansions and optimal inference.

The paper proceeds as follows. Section \ref{sec:density} studies density estimation at interior points and states the main results on error in coverage probability and its relationship to bias reduction and underlying smoothness, as well as discussing bandwidth choice and interval length. Section \ref{sec:locpoly} then studies local polynomial estimation at interior and boundary points. Practical guidance is explicitly discussed in Sections \ref{sec:practical density} and \ref{sec:practical locpoly}, respectively; all methods are available in {\sf R} and {\tt STATA} via the {\tt nprobust} package, see \citet{Calonico-Cattaneo-Farrell2017_nprobust}. Section \ref{sec:simuls} summarizes the results of a Monte Carlo study, and Section \ref{sec:conclusion} concludes. Some technical details, all proofs, and additional simulation evidence are collected in a lengthy online supplement.

%%%%%%%%%%%%%%%%%%%%%%%%%%%%%%%%%%%
%%%%%%%%%%%%%%%%%%%%%%%%%%%%%%%%%%%
%%%%%%%%%%%%%%%%%%%%%%%%%%%%%%%%%%%
\section{Density Estimation and Inference}
	\label{sec:density}

We first present our main ideas and conclusions for inference on the density at an interior point, as this requires relatively little notation. The data are assumed to obey the following.
\begin{assumption}[Data-generating process] 
	\label{dgp density} \,
	$\{X_1, \ldots, X_n\}$ is a random sample with an absolutely continuous distribution with Lebesgue density $f$. In a neighborhood of $x$, $f>0$, $f$ is $S$-times continuously differentiable with bounded derivatives $f^{(s)}$, $s=1,2,\cdots,S$, and $f^{(S)}$ is H\"older continuous with exponent $\varsigma$.
\end{assumption}
The parameter of interest is $f(x)$ for a fixed scalar point $x$ in the interior of the support. (In the supplemental appendix we discuss how our results extend naturally to multivariate $X_i$ and derivative estimation.) The classical kernel-based estimator of $f(x)$ is
\begin{equation}
	\label{eqn:f hat}
	\hat{f}(x) = \frac{1}{n \h} \sumi K\left( \frac{x - X_i}{\h} \right),
\end{equation}
for a kernel function $K$ that integrates to $1$ and positive bandwidth $\h \to 0$ as $n \to \infty$. The choice of $\h$ can be delicate, and our work is motivated in part by the standard empirical practice of employing the MSE-optimal bandwidth choice for $\hat{f}(x)$ when conducting inference.

In this vein, let us suppose for the moment that $K$ is a kernel of order $\k$, where $\k \leq S$ so that the MSE-optimal bandwidth can be characterized. The bias is then given by
\begin{equation}
	\label{eqn:bias density}
%	\E[\hat{f}(x)] - f(x) = \h^\k f^{(\k)} \mu_{K,\k} + \h^{\k+2} f^{(\k+2)} \mu_{K,\k+2} + o(\h^{\k+2}),
%	\E[\hat{f}(x)] - f(x) = \h^\k f^{(\k)}(x) \mu_{K,\k} + o(\h^\k),
	\E[\hat{f}(x)] - f(x) = \h^\k f^{(\k)}(x) \mu_{K,\k} + o(\h^\k),
\end{equation}
where $f^{(\k)}(x):=\partial^\k f(x)/\partial x^\k$ and $\mu_{K,\k} =  \int u^\k K(u) du / \k!$. Computing the variance gives
\begin{equation}
	\label{eqn:variance density}
	(n\h) \V[\hat{f}(x)] = \frac{1}{\h} \left\{ \E \left[ K \left( \frac{x - X_i}{\h} \right)^2 \right] - \E \left[ K \left( \frac{x - X_i}{\h} \right) \right]^2   \right\},
\end{equation}
which is \emph{non-asymptotic}: $n$ and $\h$ are fixed in this calculation. Using other, first-order valid approximations, e.g.\ $(n\h) \V[\hat{f}(x)] \approx f(x) \int K(u)^2 du$, will have finite sample consequences that manifest as additional terms in the Edgeworth expansions. In fact, Section \ref{sec:locpoly} shows that using an asymptotic variance for local polynomial regression  removes automatic coverage-error boundary adaptivity.

Together, the prior two displays are used to characterize the MSE-optimal bandwidth, $\h^*_\MSE\propto n^{-1/(1 + 2 \k)}$. However, using this bandwidth leaves a bias that is too large, relative to the variance, to conduct valid inference for $f(x)$. To address this important practical problem, researchers must either undersmooth the point estimator (i.e., construct $\hat{f}(x)$ with a bandwidth smaller than  $\h^*_\MSE$) or bias-correct the point estimator (i.e., subtract an estimate of the leading bias). Thus, the question we seek to answer is this: if the bias is given by \eqref{eqn:bias density}, is one better off estimating the leading bias (explicit bias correction) or choosing $\h$ small enough to render the bias negligible (undersmoothing) when forming nonparametric confidence intervals?

To answer this question, and to motivate our new robust approach, we first detail the bias correction and variance estimators. Explicit bias correction estimates the leading term of \Eqref{eqn:bias density}, denoted by $B_f$, using a kernel estimator of $f^{(\k)}(x)$, defined as:
\begin{equation*}
	\label{eqn:f r hat}
	\hat{B}_f = \h^\k \hat{f}^{(\k)}(x) \mu_{K,\k},   		\qquad \quad\text{ where }\qquad \quad		   \hat{f}^{(\k)}(x) = \frac{1}{n \b^{1 + \k} } \sumi L^{(\k)}\left( \frac{x - X_i}{\b} \right),
\end{equation*}
for a kernel $L(\cdot)$ of order $\ell$ and a bandwidth $\b \to 0$ as $n \to \infty$. Importantly, $\hat{B}_f$ takes this form for any $\k$ and $S$, even if \eqref{eqn:bias density} fails; see Sections \ref{sec:corollaries} and \ref{sec:compare} for discussion. Conventional Studentized statistics based on undersmoothing and explicit bias correction are, respectively, 
\[\tus(x) = \frac{\sqrt{n\h}\bigl(\hat{f}(x) - f(x)\bigr)}{\shatus}  	\qquad \text{ and } \qquad 	 \tbc(x) =  \frac{\sqrt{n\h}\bigl(\hat{f}(x) - \hat{B}_f -  f(x)\bigr)}{\shatus},\]
where $\shatus^2 := \hat{\V}[ \hat{f}(x)]$ is the natural estimator of the variance of $\hat{f}(x)$ which only replaces the two expectations in \eqref{eqn:variance density} with sample averages, thus maintaining the nonasymptotic spirit. These are the two statistics compared in the influential paper of \cite{Hall1992_AoS_density}, under the same assumption imposed herein.

From the form of these statistics, two points are already clear. First, the numerator of $\tus$ relies on choosing $\h$ vanishing fast enough so that the bias is asymptotically negligible after scaling, whereas $\tbc$ allows for slower decay by virtue of the manual estimation of the leading bias. Second, $\tbc$ requires that the variance of $\h^\k \hat{f}^{(\k)}(x) \mu_{K,\k}$ be first-order asymptotically negligible: $\shatus$ in the denominator only accounts for the variance of the main estimate, but $\hat{f}^{(\k)}(x)$, being a kernel-based estimator, naturally has a variance controlled by its bandwidth. That is, even though $\shatus^2$ is based on a fixed-$n$ calculation, the variance of the numerator of $\tbc$ only coincides with the denominator asymptotically. Under this regime, \cite{Hall1992_AoS_density} showed that the bias reduction achieved in $\tbc$ is too expensive in terms of noise and that undersmoothing dominates explicit bias correction for coverage error.

We argue that there need not be such a ``mismatch'' between the numerator of the bias-corrected statistic and the Studentization, and thus consider a third option corresponding to the idea of capturing the finite sample variability of $\hat{f}^{(\k)}(x)$ directly. To do so, note that we may write, after setting $\rho=h/b$,
\begin{equation}
	\label{eqn:kernel M}
	\hat{f}(x) - \h^\k \hat{f}^{(\k)}(x) \mu_{K,\k} = \frac{1}{n \h} \sumi M\left( \frac{x - X_i}{\h} \right),		\quad		 M(u) = K(u) - \rho^{1 + \k} L^{(\k)}(\rho u) \mu_{K,\k}.
\end{equation}
We then define the collective variance of the density estimate and the bias correction as $\srbc^2 = (n\h) \V[\hat{f}(x) - \hat{B}_f]$, exactly as in \Eqref{eqn:variance density}, but with $M(\cdot)$ in place of $K(\cdot)$, and its estimator $\shatrbc^2$ exactly as $\shatus^2$. Therefore, our proposed robust bias corrected inference approach is based on 
\[\trbc =   \frac{\sqrt{n\h}\bigl(\hat{f}(x) - \h^\k \hat{f}^{(\k)}(x) \mu_{K,\k} -  f(x)\bigr)}{ \shatrbc  }.\]
That is, our proposed standard errors are based on a fixed-$n$ calculation that captures the variability of both $\hat{f}(x)$ and $\hat{f}^{(\k)}(x)$, and their covariance. As shown in Section \ref{sec:locpoly}, the case of local polynomial regression is analogous, but notationally more complicated.% because there is no obvious simple corresponding change of kernel in general.

The quantity $\rho=h/b$ is key. If $\rho \to 0$, then the second term of $M$ is dominated by the first, i.e.\ the bias correction is first-order negligible. In this case, $\sus^2$ and $\srbc^2$ (and their estimators) will be first-order, but not higher-order, equivalent. This is exactly the sense in which traditional bias correction relies on an asymptotic variance, instead of a fixed-$n$ one, and pays the price in coverage error. To more accurately capture finite sample behavior of bias correction we allow $\rho$ to converge to any (nonnegative) finite limit, allowing (but not requiring) the bias correction to be first-order important, unlike prior work. We show that doing so yields more accurate confidence intervals (i.e., higher-order corrections).

%%%%%%%%%%%%%%%%%%%%%%%%%%%%%%%%%%%
%%%%%%%%%%%%%%%%%%%%%%%%%%%%%%%%%%%
\subsection{Generic Higher Order Expansions of Coverage Error}
	\label{sec:coverage}

We first present generic Edgeworth expansions for all three procedures (undersmoothing, traditional bias correction, and robust bias correction), which are agnostic regarding the level of available smoothness (controlled by $S$ in Assumption \ref{dgp density}). To be specific, we give higher-order expansions of the error in coverage probability of the following $(1-\alpha)\%$ confidence intervals based on Normal approximations for the statistics $\tus$, $\tbc$, and $\trbc$:
\begin{equation}
	\begin{split}
		\label{eqn:intervals}
		\ius & = \left[\hat{f} - z_{1 - \frac{\alpha}{2}} \frac{\shatus}{\sqrt{n\h}},	\ \ 	\hat{f} - z_{\frac{\alpha}{2}} \frac{\shatus}{\sqrt{n\h}}\right],   		\\
		\ibc &  = \left[\hat{f} - \hat{B}_f - z_{1 - \frac{\alpha}{2}}\frac{\shatus}{\sqrt{n\h}},	\ \ 	\hat{f} - \hat{B}_f - z_{\frac{\alpha}{2}} \frac{\shatus}{\sqrt{n\h}}\right],  \qquad\qquad \text{ and }   		\\
		\irbc & = \left[\hat{f} - \hat{B}_f - z_{1 - \frac{\alpha}{2}} \frac{\shatrbc}{\sqrt{n\h}},	 \ \  	\hat{f} - \hat{B}_f - z_{\frac{\alpha}{2}} \frac{\shatrbc}{\sqrt{n\h}} \right],
	\end{split}
\end{equation}
where $z_{\alpha}$ is the upper $\alpha$-percentile of the Gaussian distribution. Here and in the sequel we omit the point of evaluation $x$ for simplicity. Equivalently, our results can characterize the error in rejection probability of the corresponding hypothesis tests. In subsequent sections, we give specific results under different smoothness assumptions and make direct comparisons of the methods. 

We require the following standard conditions on the kernels $K$ and $L$. 
\begin{assumption}[Kernels] 
	\label{kernel density} \, 
	The kernels $K$ and $L$ are bounded, even functions with support $[-1,1]$, and are of order $\k \geq 2$ and $\ell \geq 2$, respectively, where $\k$ and $\ell$ are even integers. That is, $\mu_{K,0} = 1$, $\mu_{K,k} = 0$ for $1 \leq k < \k$, and $\mu_{K,\k} \neq 0$ and bounded, and similarly for $\mu_{L,k}$ with $\ell$ in place of $\k$. Further, $L$ is $\k$-times continuously differentiable. For all integers $k$ and $l$ such that $k + l = \k-1$, $f^{(k)}(x_0) L^{(l)}((x_0 - x)/\b) = 0$ for $x_0$ in the boundary of the support.
\end{assumption}
The boundary conditions are needed for the derivative estimation inherent in bias correction, even if $x$ is an interior point, and are satisfied if the support of $f$ is the whole real line. Higher order results also require a standard $n$-varying Cram\'er's condition, given in the supplement to conserve space (see Section S.I.3). Altogether, our assumptions are identical to those of \citet{Hall1991_Statistics,Hall1992_AoS_density}.

To state the results some notation is required. First, let the (scaled) biases of the density estimator and the bias-corrected estimator be $\eus = \sqrt{n\h}(\E[\hat{f}] - f)$ and $\ebc = \sqrt{n\h}(\E[\hat{f} - \hat{B}_f] - f)$. Next, let $\phi(z)$ be the standard Normal density, and for any kernel $K$ define
\begin{align*}
    q_1(K) & = \vartheta_{K,2}^{-2} \vartheta_{K,4}(z_{\frac{\alpha}{2}}^3 - 3z_{\frac{\alpha}{2}})/6 - \vartheta_{K,2}^{-3} \vartheta_{K,3}^2 [2z^3/3 + (z_{\frac{\alpha}{2}}^5 - 10z_{\frac{\alpha}{2}}^3 + 15z_{\frac{\alpha}{2}})/9],\\
    q_2(K) & = - \vartheta_{K,2}^{-1}\, z_{\frac{\alpha}{2}},		\qquad \text{ and } \qquad 		q_3(K)  = \vartheta_{K,2}^{-2} \vartheta_{K,3}( 2 z_{\frac{\alpha}{2}}^3/3),
\end{align*}
where $\vartheta_{K,k} = \int K(u)^k du$. All that is conceptually important is that these functions are known, odd polynomials in $z$ with coefficients that depend only on the kernel, and not on the sample or data generating process. Our main theoretical result for density estimation is the following.
\begin{theorem}
	\onehalfspacing
	\label{thm:coverage density}
	Let Assumptions \ref{dgp density}, \ref{kernel density}, and Cram\'er's condition hold and $n \h/ \log(n\h) \to \infty$.
	\begin{enumerate}

		\item If $\eus \to 0$, then 
			\begin{align*}
				\P[f \in \ius] = 1 - \alpha    &    +   \left\{ \frac{1}{n \h} q_1(K)  +  \eus^2 q_2(K)   +   \frac{\eus}{\sqrt{n \h}} q_3(K) \right\}  \frac{\phi(z_{\frac{\alpha}{2}})}{f}  \; \{1+o(1)\}.
			\end{align*}

		\item If $\ebc \to 0$ and $\rho \to 0$, then 
			\begin{align*}
				\P[f \in \ibc]   = 1 - \alpha   &  +   \left\{\frac{1}{n \h} q_1(K)  +  \ebc^2 q_2(K)   +   \frac{\ebc}{\sqrt{n \h}} q_3(K) \right\}  \frac{\phi(z_{\frac{\alpha}{2}})}{f}  \; \{1+o(1)\}    			\\
				& + \rho^{1+\k} (\Omega_1 + \rho^{\k} \Omega_2) \phi(z_{\frac{\alpha}{2}}) z_{\frac{\alpha}{2}} \; \{1+o(1)\},
			\end{align*}
			for constants $\Omega_1$ and $\Omega_2$ given precisely in the supplement.

		\item If $\ebc \to 0$ and $\rho \to \bar{\rho}<\infty$, then 
			\begin{align*}
				\P[f \in \irbc]   = 1 - \alpha  & +    \left\{ \frac{1}{n \h} q_1(M)  +  \ebc^2 q_2(M)   +   \frac{\ebc}{\sqrt{n \h}} q_3(M)  \right\}  \frac{\phi(z_{\frac{\alpha}{2}})}{f}  \; \{1+o(1)\}.
			\end{align*}

	\end{enumerate}

\end{theorem}

This result leaves the scaled biases $\eus$ and $\ebc$ generic, which is useful when considering different levels of smoothness $S$, the choices of $\k$ and $\ell$, and in comparing to local polynomial results. In the next subsection, we make these quantities more precise and compare them, paying particular attention to the role of the underlying smoothness assumed.

At present, the most visually obvious feature of this result is that all the error terms are of the same form, except for the notable presence of $\rho^{1+\k}(\Omega_1 + \rho^\k \Omega_2)$ in part (b). These are the leading terms of $\srbc^2 / \sus^2 - 1$, consisting of the covariance of $\hat{f}$ and $\hat{B}_f$ (denoted by $\Omega_1$) and the variance of $\hat{B}_f$ (denoted by $\Omega_2$), and are entirely due to the ``mismatch'' in the Studentization of $\tbc$. \citet{Hall1992_AoS_density} showed how these terms prevent bias correction from performing as well as undersmoothing in terms of coverage. In essence, the potential for improved bias properties do not translate into improved inference because the variance is not well-controlled: in any finite sample, $\hat{B}_f$ would inject variability (i.e., $\rho=\h/\b>0$ for each $n$) and thus $\rho \to 0$ may not be a good approximation. Our new Studentization does not simply remove these leading $\rho$ terms; the entire sequence is absent. As explained below, allowing for $\bar{\rho}=\infty$ can not reduce bias, but will inflate variance; hence restricting to $\bar{\rho}<\infty$ capitalizes fully on the improvements from bias correction.

%%%%%%%%%%%%%%%%%%%%%%%%%%%%%%%%%%%
%%%%%%%%%%%%%%%%%%%%%%%%%%%%%%%%%%%
\subsection{Coverage Error and the Role of Smoothness}
	\label{sec:corollaries}

Theorem \ref{thm:coverage density} makes no explicit assumption about smoothness beyond the requirement that the scaled biases vanish asymptotically. The fact that the error terms in parts (a) and (c) of Theorem \ref{thm:coverage density} take the same form implies that comparing coverage error amounts to comparing bias, for which the smoothness $S$ and the kernel orders $\k$ and $\ell$ are crucial. We now make the biases $\eus$ and $\ebc$ concrete and show how coverage is affected.

For $\ius$, two cases emerge: (a) enough derivatives exist to allow characterization of the MSE-optimal bandwidth ($\k \leq S$); and (b) no such smoothness is available ($\k > S$), in which case the leading term of \Eqref{eqn:bias density} is exactly zero and the bias depends on the unknown H\"older constant. These two cases lead to the following results.
\begin{corollary}
	\onehalfspacing
	\label{thm:US coverage}
	Let Assumptions \ref{dgp density}, \ref{kernel density}, and Cram\'er's condition hold and $n \h/ \log(n\h) \to \infty$.
	\begin{enumerate}
	
		\item If $\k \leq S$ and $\sqrt{n \h} \h^\k \to 0$,
			\[\P[f \in \ius] = 1 - \alpha  + \biggl\{\frac{1}{n \h} q_1(K)
	           + n\h^{1+2\k}(f^{(\k)})^2 \mu_{K,\k}^2 q_2(K) + \h^\k f^{(\k)} \mu_{K,\k} q_3(K) \biggr\} \frac{\phi(z_{\frac{\alpha}{2}})}{f} \; \{1+o(1)\}.    \]
%		Thus, if $\h^*_\US = H^*_\US n^{-1/(1+r)}$, then $\P[f \in \ius] = 1 - \alpha + O(n^{-r/(1+r)})$, where 
%			\[H^*_\US = \argmin_{H>0} \left\vert H^{-1}  q_1(K) + H^{1+2\k} (f^{(\k)})^2 \mu_{K,\k}^2  q_2(K) + H^\k  f^{(\k)} \mu_{K,\k} q_3(K) \right\vert.\]

		\item If $\k > S$ and $\sqrt{n \h} \h^{S + \varsigma} \to 0$,
			\[\P[f \in \ius] = 1 - \alpha  + \frac{1}{n \h} \frac{\phi(z_{\frac{\alpha}{2}})}{f} q_1(K)    \; \{1+o(1)\}    +   O\left(n \h^{1 + 2(S+\varsigma)}   +    \h^{S+\varsigma}   \right).   \]
	\end{enumerate}
\end{corollary}
The first result is most directly comparable to \citet[\S 3.4]{Hall1992_AoS_density}, and many other past papers, which typically take as a starting point that the MSE-optimal bandwidth can be characterized. This shows that $\tus$ must be undersmoothed, in the sense that the MSE-optimal bandwidth is ``too large'' for valid inference. In fact, we know that $\ius(\h^*_\MSE)$ will asymptotically \emph{under}cover because $\tus(\h^*_\MSE) \to_d \N((2\k)^{-1/2},1)$ (see the supplement). Instead, the optimal $\h$ for coverage error, which can be characterized and estimated, is equivalent in rates to balancing variance against bias, not squared bias as in MSE. Part (b) shows that a faster rate of coverage error decay can be obtained by taking a sufficiently high order kernel, relative to the level of smoothness $S$, at the expense of feasible bandwidth selection.

Turning to robust bias correction, characterization of $\ebc$ is more complex as it has two pieces: the second-order bias of the original point estimator, and the bias of the bias estimator itself. The former is the $o(\h^\k)$ term of \Eqref{eqn:bias density} and is not the target of explicit bias correction; it depends either on higher derivatives, if they are available, or on the H\"older condition otherwise. To be precise, if $\k \leq S-2$, this term is $[\h^{\k + 2} + o(1)] f^{(\k + 2)} \mu_{K_\BC, \k + 2}$, while otherwise is known only to be $O(\h^{S + \varsigma})$. Importantly, the bandwidth $b$ and order $\ell$ do not matter here, and bias reduction beyond $O(\min\{\h^{\k + 2},\h^{S + \varsigma}\})$ is not possible; there is thus little or no loss in fixing $\ell=2$, which we assume from now on to simplify notation.

The bias of the bias estimator also depends on the smoothness available: if enough smoothness is available the corresponding bias term can be characterized, otherwise only its order will be known. To be specific, when smoothness is not binding ($\k \leq S-2$), arguably the most practically-relevant case, the leading term of $\E[\hat{B}_f] - B_f$ will be $\h^\k \b^2 f^{(\k + 2)} \mu_{K,\k} \mu_{L,2}$. Smoothness can be exhausted in two ways, either by the point estimate itself ($\k > S$) or by the bias estimation ($S-1 \leq \k \leq S$), and these two cases yield $O(\h^\k \b^{S - \k})$ and $O(\h^\k \b^{S + \varsigma - \k})$, respectively, which are slightly different in how they depend on the total H\"older smoothness assumed. (Complete details are in the supplement.) Note that regardless of the value of $\k$, we set $\hat{B}_f = \h^\k \hat{f}^{(\k)} \mu_{K,\k}$, even if $\k > S$ and $B_f \equiv 0$.

With these calculations for $\ebc$, we have the following result.
\begin{corollary}
	\onehalfspacing
	\label{thm:RBC coverage}
	Let Assumptions \ref{dgp density}, \ref{kernel density}, and Cram\'er's condition hold, $n \h/ \log(n\h) \to \infty$, $\rho \to \bar{\rho} < \infty$, and $\ell=2$.
	
	\begin{enumerate}
	
		\item If $\k \leq S - 2$ and $\sqrt{n \h} \h^\k \b^2 \to 0$,
			\begin{multline*}
				\hspace{-0.5in} \P[f \in \irbc]  =  1 - \alpha   + \biggl\{  \frac{1}{n \h}  q_1(M_{\bar{\rho}})   		
									  +   n\h^{1+2(\k+2)}  (f^{(\k+2)})^2 \left(\mu_{K,\k+2} - \bar{\rho}^{-2}\mu_{K,\k}\mu_{L,2}\right)^2  q_2(M_{\bar{\rho}})		 \\
					                                       +   \h^{\k+2}  f^{(\k+2)} \left( \mu_{K,\k+2} - \bar{\rho}^{-2}  \mu_{K,\k} \mu_{L,2} \right) q_3(M_{\bar{\rho}})   \biggr\}  \frac{\phi(z_{\frac{\alpha}{2}})}{f} \;\{1+o(1)\}.
			\end{multline*}

		\item If $S-1 \leq \k \leq S$ and $\sqrt{n \h} \rho^\k \b^{S + \varsigma} \to 0$,
			%%%  OR PUT THE RATE CONDITION IN TERMS OF (h,b) INSTEAD OF (rho,b)
			%%%   If $S-1 \leq \k \leq S$ and $\sqrt{n \h} \h^\k \b^{S - \k + \varsigma} \to 0$,
			\begin{multline*}
				\P[f \in \irbc]  =  1 - \alpha   + \frac{1}{n \h} \frac{\phi(z_{\frac{\alpha}{2}})}{f}  q_1(M_{\bar{\rho}})  \;\{1+o(1)\}      			        +   O\left(    n  \h \rho^{2 \k} \b^{2(S+\varsigma)}   +  \rho^\k  \b^{S+\varsigma}       \right).  
			\end{multline*}
			%%%  OR PUT THE REMAINDER IN TERMS OF (h,b) INSTEAD OF (rho,b)  
			%%%     O\left(    n  \h^{1 + 2 \k} \b^{2(S+\varsigma - \k)}   +  \h^\k  \b^{S+\varsigma - \k}       \right).  

		\item If $\k > S$ and $\sqrt{n \h} \bigl(\h^{S+\varsigma} \vee \rho^{\k} \b^{S}\bigr) \to 0$,
			%%%  OR PUT THE RATE CONDITION IN TERMS OF (h,b) INSTEAD OF (rho,b)
			%%%   If $\k > S$ and $\sqrt{n \h} \h^\k \bigl(\h^{S+\varsigma - \k} \vee \b^{S - \k}\bigr) \to 0$,
			\[ \hspace{-0.4in} \P[f \in \irbc]  =  1 - \alpha   + \frac{1}{n \h} \frac{\phi(z_{\frac{\alpha}{2}})}{f}  q_1(M_{\bar{\rho}}) \{1+o(1)\}      	      +  	  O\left(    n  \h  (\h^{S+\varsigma} \! \vee \! \rho^\k \b^S)^2   +  (\h^{S+\varsigma} \! \vee \! \rho^{k}\b^{S})       \right).  \]
			%%%  OR PUT THE REMAINDER IN TERMS OF (h,b) INSTEAD OF (rho,b)
			%%%       O\left(    n  \h^{1 + 2 \k} \bigl(\h^{S+\varsigma - \k} \vee \b^{S - \k}\bigr)^2   +  \h^\k  \bigl(\h^{S+\varsigma - \k} \vee \b^{S - \k}\bigr)       \right).  	
	\end{enumerate}
\end{corollary}

Part (a) is the most empirically-relevant setting, which reflects the idea that researchers first select a kernel order, then conduct inference based on that choice, taking the unknown smoothness to be nonbinding. The most notable feature of this result, beyond the formalization of the coverage improvement, is that the coverage error terms share the same structure as those of Corollary \ref{thm:US coverage}, with $\k$ replaced by $\k+2$, and represent the same conceptual objects. By virtue of our new Studentization, the leading variance remains order $(n \h)^{-1}$ and the problematic correlation terms are absent. We explicitly discuss the advantages of robust bias correction relative to undersmoothing in the following section.

Part (a) also argues for a bounded, positive $\rho$. First, because bias reduction beyond $O(\h^{\k+2})$ is not possible, $\rho \to \infty$ will only inflate the variance. On the other hand, $\bar{\rho} = 0$ requires a delicate choice of $\b$ and $\ell > 2$, else the second bias term dominates $\ebc$, and the full power of the variance correction is not exploited; that is, more bias may be removed without inflating the variance rate. \citet[p.\ 682]{Hall1992_AoS_density} remarked that if $\E[\hat{f}] - f - B_f$ is (part of) the leading bias term, then ``explicit bias correction [\ldots] is even less attractive relative to undersmoothing.'' We show, on the contrary, that with our proposed Studentization, it is optimal that $\E[\hat{f}] - f - B_f$ is part of the dominant bias term.

Finally, in both Corollaries above the best possible coverage error decay rate (for a given $S$) is attained by exhausting all available smoothness. This would also yield point estimators attaining the bound of \citet{Stone1982_AoS}; robust bias correction can not evade such bounds, of course. In both Corollaries, coverage is improved relative to part (a), but the constants and optimal bandwidths can not be quantified. For robust bias correction, Corollary \ref{thm:RBC coverage} shows that to obtain the best rate in part (b) the unknown $f^{(\k)}$ must be consistently estimated and $\rho$ must be bounded and positive, while in part (c), bias estimation merely adds noise, but this noise is fully accounted for by our new Studentization, as long as $\rho \to 0$ ($\b \not\to 0$ is allowed).

%%%%%%%%%%%%%%%%%%%%%%%%%%%%%%%%%%%
%%%%%%%%%%%%%%%%%%%%%%%%%%%%%%%%%%%
\subsection{Comparing Undersmoothing and Robust Bias Correction}
	\label{sec:compare}

We now employ Corollaries \ref{thm:US coverage} and \ref{thm:RBC coverage} to directly compare nonparametric inference based on undersmoothing and robust bias correction. To simplify the discussion we focus on three concrete cases, which illustrate how the comparisons depend on the available smoothness and kernel order; the messages generalize to any $S$ and/or $\k$. For this discussion we let $\k_\US$ and $\k_\BC$ be the kernel orders used for point estimation in $\ius$ and $\irbc$, respectively, and restrict attention to sequences $\h\to0$ where both confidence intervals are first-order valid, even though robust bias correction allows for a broader bandwidth range. Finally, we set $\ell=2$ and $\bar{\rho} \in (0,\infty)$ based on the above discussion.

For the first case, assume that $f$ is twice continuously differentiable ($S=2$) and both methods use second order kernels ($\k_\US = \k_\BC = \ell = 2$). In this case, both methods target the \emph{same} bias. The coverage errors for $\ius$ and $\irbc$ then follow directly from Corollaries \ref{thm:US coverage}(a) and \ref{thm:RBC coverage}(b) upon plugging in these kernel orders, yielding
\[ \bigl|  \P[f \in \ius]   -   (1 - \alpha) \bigr| \asymp   \frac{1}{n \h}  +   n\h^5  +  \h^2   		\quad \text{and} \quad 	 		 \bigl|  \P[f \in \irbc]   -   (1 - \alpha) \bigr| \asymp   \frac{1}{n \h}  +       n  \h^{5+2\varsigma}   +  \h^{2 +\varsigma}       .\]
Because $\h \to 0$ and $\bar{\rho} \in (0,\infty)$, the coverage error of $\irbc$ vanishes more rapidly by virtue of the bias correction. A higher order kernel ($\k_\US > 2$) would yield this rate for $\ius$.

Second, suppose that the density is four-times continuously differentiable ($S=4$) but second order kernels are maintained. The relevant results are now Corollaries \ref{thm:US coverage}(a) and \ref{thm:RBC coverage}(a). Both methods continue to target the \emph{same} leading bias, but now the additional smoothness available allows precise characterization of the improvement shown above, and we have
\[ \bigl|  \P[f \in \ius]   -   (1 - \alpha) \bigr| \asymp   \frac{1}{n \h}  +   n\h^5  +  \h^2   		\quad \text{ and } \quad 	 		 \bigl|  \P[f \in \irbc]   -   (1 - \alpha) \bigr| \asymp   \frac{1}{n \h}  +       n  \h^9    +  \h^4     .\]
This case is perhaps the most empirically relevant one, where researchers first choose the order of the kernel (here, second order) and then conduct/optimize inference based on that choice. Indeed, for this case optimal bandwidth choices can be derived (Section \ref{sec:practical density}).

Finally, maintain $S=4$ but suppose that undersmoothing is based on a fourth-order kernel while bias correction continues to use two second-order kernels ($\k_\US = 4$, $\k_\BC = \ell = 2$). This is the exact example given by \citet[][p.\! 676]{Hall1992_AoS_density}. Now the two methods target \emph{different} biases, but utilize the \emph{same} amount of smoothness. In this case, the relevant results are again Corollaries \ref{thm:US coverage}(a) and \ref{thm:RBC coverage}(a), now with $\k=4$ and $\k=2$, respectively. The two methods have the same coverage error decay rate:
\[ \bigl|  \P[f \in \ius]   -   (1 - \alpha) \bigr| \asymp    \bigl|  \P[f \in \irbc]   -   (1 - \alpha) \bigr| \asymp   \frac{1}{n \h}  +       n  \h^9    +  \h^4     .\]
Indeed, more can be said: with the notation of \Eqref{eqn:kernel M}, the difference between $\tus$ and $\trbc$ is the change in ``kernel'' from $K$ to $M$, and since  $\k_\BC + \ell = \k_\US$, the two kernels are the same order. ($M$ acts as a $n$-varying, higher-order kernel for bias, but may not strictly fit the definition, as explored in the supplement.) This tight link between undersmoothing and robust bias correction does not carry over straightforwardly to local polynomial regression, as we discuss in more detail in Section \ref{sec:locpoly}.

In the context of this final example, it is worth revisiting traditional bias correction. The fact that undersmoothing targets a different, and asymptotically smaller, bias than does explicit bias correction, coupled with the requirement that $\rho \to 0$, implicitly constrains bias correction to remove \emph{less} bias than undersmoothing. This is necessary for traditional bias correction, but on the contrary, robust bias correction attains the \emph{same} coverage error decay rate as undersmoothing under the same assumptions.

In sum, these examples show that under identical assumptions, bias correction is not inferior to undersmoothing and if any additional smoothness is available, can yield improved coverage error. These results are confirmed in our simulations.

%%%%%%%%%%%%%%%%%%%%%%%%%%%%%%%%%%%
%%%%%%%%%%%%%%%%%%%%%%%%%%%%%%%%%%%
\subsection{Optimal Bandwidth and Data-Driven Choice}
	\label{sec:practical density}

The prior sections established that robust bias correction can equal, or outperform, undersmoothing for inference. We now show how the method can be implemented to deliver these results in applications. We mimic typical empirical practice where researchers first choose the order of the kernel, then conduct/optimize inference based on that choice. Therefore, we assume the smoothness is unknown but taken to be large and work within Corollary \ref{thm:RBC coverage}(a), that is, viewing $\k \leq S-2$ and $\ell=2$ as fixed and $\rho$ bounded and positive. This setup allows characterization of the coverage error optimal bandwidth for robust bias correction.
\begin{corollary}
	\label{thm:bandwidth density}
	Under the conditions of Corollary \ref{thm:RBC coverage}(a) with $\bar{\rho} \in (0,\infty)$, if $\h = \h^*_\RBC = H^*_\RBC (\rho) n^{-1/(1+(\k+2))}$, then $\P[f \in \irbc] = 1 - \alpha + O(n^{-(\k+2)/(1+(\k+2))})$, where 
	\vspace{-0.15in}
	\begin{align*}
		H^*_\RBC (\bar{\rho}) &= \argmin_{H>0} \bigl\vert  H^{-1} q_1(M_{\bar{\rho}}) + H^{1+2(\k+2)} (f^{(\k+2)})^2 \left( \mu_{K,\k+2} - \bar{\rho}^{-2}  \mu_{K,\k} \mu_{L,2} \right)^2  q_2(M_{\bar{\rho}})\\
                           & \qquad\qquad\qquad + H^{\k+2} f^{(\k+2)}\left(\mu_{K,\k+2} - \bar{\rho}^{-2}\mu_{K,\k}\mu_{L,2}\right) q_3(M_{\bar{\rho}}) \bigr\vert.
	\end{align*}
\end{corollary}

We can use this result to give concrete methodological recommendations. At the end of this section we discuss the important issue of interval length. Construction of the interval $\irbc$ from \Eqref{eqn:intervals} requires bandwidths $\h$ and $\b$ and kernels $K$ and $L$. Given these choices, the point estimate, bias correction, and variance estimators are then readily computable from data using the formulas above. For the kernels $K$ and $L$, we recommend either second order minimum variance (to minimize interval length) or MSE-optimal kernels \citep[see, e.g.,][and the supplemental appendix]{Gasser-Muller-Mammitzsch1985_JRSSB}.

The bandwidth selections are more important in applications. For the bandwidth $\h$, Corollary \ref{thm:RBC coverage}(a) shows that the MSE-optimal choice $\h^*_\MSE$ will deliver valid inference, but will be suboptimal in general (Corollary \ref{thm:bandwidth density}). From a practical point of view, the robust bias corrected interval $\irbc(\h)$ is attractive because it allows for the MSE-optimal bandwidth and kernel, and hence is based on the MSE-optimal point estimate, while using the same effective sample for both point estimation and inference. Interestingly, although $\irbc(\h^*_\MSE)$ is always valid, its coverage error decays as $n^{-\min\{4,\k+2\}/(1+2\k)}$ and is thus rate optimal only for second order kernels ($\k=2$), while otherwise being suboptimal, with a rate that is slower the larger is the order $\k$. 

Corollary \ref{thm:bandwidth density} gives the coverage error optimal bandwidth, $\h^*_\RBC$, which can be implemented using a simple direct plug-in (DPI) rule: $\hat{h}_{\PI} = \hat{H}_{\PI} \; n^{-1/(\k+3)}$, where $\hat{H}_{\PI}$ is a plug-in estimate of $H^*_\RBC$ formed by replacing the unknown $f^{(\k+2)}$ with a pilot estimate (e.g., a consistent nonparametric estimator based on the appropriate MSE-optimal bandwidth). In the supplement we give precise implementation details, as well as an alternative rule-of-thumb bandwidth selector based on rescaling already available data-driven MSE-optimal choices.

For the bandwidth $\b$, a simple choice is $\b = \h$, or, equivalently, $\rho = 1$. We show in the supplement that setting $\rho=1$ has good theoretical properties, minimizing interval length of $\irbc$ or the MSE of $\hat{f} - \hat{B}_f$, depending on the conditions imposed. In our numerical work, we found that $\rho=1$ performed well. As a result, from the practitioner's point of view, the choice of $\b$ (or $\rho$) is completely automatic, leaving only one bandwidth to select.

An extensive simulation study, reported in the supplement, illustrates our findings and explores the numerical performance of these choices. We find that coverage of $\irbc$ is robust to both $\h$ and $\rho$ and that our data-driven bandwidth selectors work well in practice, but we note that estimating bandwidths may have higher-order implications \citep[e.g.][]{Hall-Kang2001_AoS}.

Finally, an important issue in applications is whether the good coverage properties of $\irbc$ come at the expense of increased interval length. When coverage is asymptotically correct, Corollaries \ref{thm:US coverage} and \ref{thm:RBC coverage} show that $\irbc$ can accommodate (and will optimally employ) a larger bandwidth (i.e.\ $\h \to 0$ more slowly), and hence $\irbc$ will have shorter average length in large samples than $\ius$. Our simulation study (see below and the supplement) gives the same conclusion.

%%%%%%%%%%%%%%%%%%%%%%%%%%%%%%%%%%%
%%%%%%%%%%%%%%%%%%%%%%%%%%%%%%%%%%%
\subsection{Other Methods of Bias Correction}
	\label{sec:other}

We study a plug-in bias correction method, but there are alternatives. In particular, as pointed out by a reviewer, a leading alternative is the generalized jackknife method of \citet{Schucany-Sommers1977_JASA} (see \citet{Cattaneo-Crump-Jansson_2013_JASA} for an application to kernel-based semiparametric inference and for related references). We will briefly summarize this approach and show a tight connection to our results, restricting to second-order kernels and $S\geq 2$ only for simplicity.

The generalized jackknife estimator is $\hat{f}_{\GJ,R} := ( \hat{f}_1 - R \hat{f}_2 ) / (1 - R)$, where $\hat{f}_1$ and $\hat{f}_2$ are two initial kernel density estimators, with possibly different bandwidths ($\h_1,\h_2$) and second-order kernels ($K_1,K_2$). From \Eqref{eqn:bias density}, the bias of $\hat{f}_{\GJ,R}$ is $(1-R)^{-1}   f^{(2)}  \left( \h_1^2  \mu_{K_1,2}  -  R  \h_2^2  \mu_{K_2,2} \right) + o(\h_1^2 + \h_2^2)$, whence choosing $R = (\h_1^2  \mu_{K_1,2} )/( \h_2^2  \mu_{K_2,2})$ renders the leading bias term exactly zero. Further, if $S \geq 4$, $\hat{f}_{\GJ,R}$ has bias $O(\h_1^4 + \h_2^4)$; behaving as a point estimator with $\k=4$. To connect this approach to ours, observe that with this choice of $R$ and $\tilde{\rho} = \h_1 / \h_2$, 
\[\hat{f}_{\GJ,R} = \frac{1}{n \h_1} \sumi \tilde{M}\left(\frac{X_i - x}{\h_1} \right) ,  		\quad   		 	 \tilde{M}(u) = K_1(u) - \tilde{\rho}^{1+2}  \left\{    \frac{K_2(\tilde{\rho}u) - \tilde{\rho}^{-1} K_1(u)  }{\mu_{K_2,2}(1-R)} \right\} \mu_{K_1, 2},\]
exactly matching \Eqref{eqn:kernel M}; alternatively, write $\hat{f}_{\GJ,R} = \hat{f}_1  -  \h_1^2  \tilde{f}^{(2)} \mu_{K_1,2}$, where
\[\tilde{f}^{(2)} = \frac{1}{n\h_2^{1+2}}\sumi \tilde{L}\left( \frac{X_i - x}{\h_2} \right),  		\qquad  		\tilde{L}(u) = \frac{ K_2(u) - \tilde{\rho}^{-1} K_1( \tilde{\rho}^{-1} u)  }{\mu_{K_2,2}(1-R)},  \]
is a derivative estimator. Therefore, we can view $\hat{f}_{\GJ,R}$ as a specific kernel $M$ or a specific derivative estimator, and all our results directly apply to $\hat{f}_{\GJ,R}$; hence our paper offers a new way of conducting inference (new Studentization) for this case as well. Though we omit the details to conserve space, this is equally true for local polynomial regression (Section \ref{sec:locpoly}).

%Therefore, we can view $\hat{f}_{\GJ,R}$ as a specific kernel $M$ or a specific derivative estimator, and hence our results directly apply. This is also true for local polynomial regression (Section \ref{sec:locpoly}), though we do not provide more details to conserve space. These observations imply that all the results presented above apply directly to nonparametric inference based on generalized jackknife bias correction, and hence this paper offers a new way of conducting inference (new Studentization) with demonstrably superior properties for this case as well.

More generally, our main ideas and generic results apply to many other bias correction methods. For a second example, \citet{Singh1977_AoS} also proposed a plug-in bias estimator, but without using the derivative of a kernel. Our results cover this approach as well; see the supplement for further details and references. The key, common message in all cases is that to improve inference one must account for the additional variability introduced by any bias correction method (i.e., to avoid the mismatch present in $\tbc$).

%%%%%%%%%%%%%%%%%%%%%%%%%%%%%%%%%%%
%%%%%%%%%%%%%%%%%%%%%%%%%%%%%%%%%%%
%%%%%%%%%%%%%%%%%%%%%%%%%%%%%%%%%%%
\section{Local Polynomial Estimation and Inference}
	\label{sec:locpoly}

This section studies local polynomial regression \citep{Ruppert-Wand1994_AoS,Fan-Gijbels1996_book}, and has two principal aims. First, we show that the conclusions from the density case, and their implications for practice, carry over to odd-degree local polynomials. Second, we show that with proper fixed-$n$ Studentization, coverage error adapts to boundary points. We focus on what is novel relative to the density, chiefly variance estimation and boundary points. For interior points, the implications for coverage error, bandwidth selection, and interval length are all analogous to the density case, and we will not retread those conclusions.

To be specific, throughout this section we focus on the case where the smoothness is large relative to the local polynomial degree $p$, which is arguably the most relevant case in practice. The results and discussion in Sections \ref{sec:corollaries} and \ref{sec:compare} carry over, essentially upon changing $\k$ to $p+1$ and $\ell$ to $q-p$ (or $q-p+1$ for interior points with $q$ even). Similarly, but with increased notational burden, the conclusions of Section \ref{sec:other} also remain true. The present results also extend to multivariate data and derivative estimation.

To begin, we define the regression estimator, its bias, and the bias correction. Given a random sample $\{(Y_i,X_i):1\leq i\leq n\}$, the local polynomial estimator of $m(x)=\E[Y_i|X_i=x]$, temporarily making explicit the evaluation point, is
\[\hat{m}(x) = \be_0' \bhat_p, 		 \qquad\quad 		\bhat_p = \argmin_{\bb \in \mathbb{R}^{p+1}} \sumi ( Y_i - \br_p(X_i - x)'\bb)^2  K \left( \frac{X_i - x}{\h}\right),  \]
where, for an integer $p \geq 1$, $\be_0$ is the $(p+1)$-vector with a one in the first position and zeros in the rest, and $\br_p(u) = (1, u, u^2, \ldots, u^p)'$. We restrict attention to $p$ odd, as is standard, though the qualifier may be omitted. We define $\bY = (Y_1, \cdots, Y_n)'$, $\bR_p = [ \br_p( (X_1 - x) / \h), \cdots, \br_p( (X_n - x) / \h)]'$, $\bW_p = \diag(\h^{-1} K((X_i - x)/\h): i = 1, \ldots, n)$, and $\Gp = \bR_p' \bW_p \bR_p/n$ (here $\diag(a_i:i = 1, \ldots, n)$ denotes the $n\times n$ diagonal matrix constructed using $a_1, a_2, \cdots, a_n$). Then, reverting back to omitting the argument $x$, the local polynomial estimator is $\hat{m} = \be_0'\Gp^{-1} \bR_p' \bW_p \bY / n$.

Under regularity conditions below, the conditional bias satisfies
\begin{equation}
	\label{eqn:bias locpoly}
	\E[\hat{m} \vert X_1, \ldots, X_n] - m =  \h^{p+1} m^{(p+1)} \frac{1}{(p+1)!} \be_0' \Gp^{-1} \Lp + o_P(\h^{p+1}),
\end{equation}
where $\Lp = \bR_p' \bW_p [ ((X_1 - x)/\h)^{p+1}, \cdots, ((X_n - x)/\h)^{p+1}]'/n.$ Here, the quantity $\be_0' \Gp^{-1} \Lp / (p+1)!$ is random, unlike in the density case (c.f. \eqref{eqn:bias density}), but it is known and bounded in probability. Following \citet[][p. 116]{Fan-Gijbels1996_book}, we will estimate $m^{(p+1)}$ in (\ref{eqn:bias locpoly}) using a second local polynomial regression, of degree $q > p$ (even or odd), based on a kernel $L$ and bandwidth $\b$. Thus, $\br_q(u)$, $\bR_q$, $\bW_q$, and $\Gq$ are defined as above, but substituting $q$, $L$, and $\b$ in place of $p$, $K$, and $\h$, respectively. Denote by $\be_{p+1}$ the $(q+1)$-vector with one in the $p+2$ position, and zeros in the rest. Then we estimate the bias with
	\[\hat{B}_m = \h^{p+1} \hat{m}^{(p+1)} \frac{1}{(p+1)!} \be_0' \Gp^{-1} \Lp, 		\qquad 
	  \hat{m}^{(p+1)} = \b^{-p-1} (p+1)!\be_{p+1}'\Gq^{-1} \bR_q' \bW_q \bY / n.\]
Exactly as in the density case, $\hat{B}_m$ introduces variance that is controlled by $\rho$ and will be captured by robust bias correction.

%%%%%%%%%%%%%%%%%%%%%%%%%%%%%%%%%%%
%%%%%%%%%%%%%%%%%%%%%%%%%%%%%%%%%%%
\subsection{Variance Estimation}
	\label{sec:variance}

The Studentizations in the density case were based on fixed-$n$ expectations, and we will show that retaining this is crucial for local polynomials. The fixed-$n$ versus asymptotic distinction is separate from, and more fundamental than, whether we employ feasible versus infeasible quantities. The advantage of fixed-$n$ Studentization also goes beyond bias correction.

To begin, we condition on the covariates so that $\Gp^{-1}$ is fixed. Define $v(\cdot) = \V[Y \vert X = \cdot]$ and $\bSig = \diag( v(X_i): i = 1,\ldots, n)$. Straightforward calculation gives
\begin{equation}
	\label{eqn:US variance locpoly}
	\sus^2 = (n \h) \V[\hat{m} \vert X_1, \cdots, X_n] = \frac{\h}{n} \be_0' \Gp^{-1} \left( \bR_p' \bW_p \bSig \bW_p \bR_p \right) \Gp^{-1} \be_0.
\end{equation}
One can then show that $\sus^2 \to_P v(x) f(x)^{-1} \mathscr{V}(K,p)$, with $\mathscr{V}(K,p)$ a known, constant function of the kernel and polynomial degree. Importantly, both the nonasymptotic  form and the convergence hold in the interior or on the boundary, though $\mathscr{V}(K,p)$ changes.

To first order, one could use $\sus^2$ or the leading asymptotic term; all that remains is to make each feasible, requiring estimators of the variance function, and for the asymptotic form, also the density. These may be difficult to estimate when $x$ is a boundary point. Concerned by this, \citet[][p.\ 93]{Chen-Qin2002_SJS} consider feasible and infeasible versions but conclude that ``an increased coverage error near the boundary is still the case even when we know the values of $f(x)$ and $v(x)$.'' Our results show that this is not true in general: using fixed-$n$ Studentization, feasible or infeasible, leads to confidence intervals with the same coverage error decay rates at interior and boundary points, thereby retaining the celebrated boundary carpentry property.

For robust bias correction, $\srbc^2  =  \ (n \h) V[\hat{m} - \hat{B}_m \vert X_1, \ldots, X_n]$ captures the variances of $\hat{m}$ and $\hat{m}^{(p+1)}$ as well as their covariance. A fixed-$n$ calculation gives
\begin{equation}
	\label{eqn:RBC variance locpoly}
	\srbc^2  = \frac{\h}{n} \be_0' \Gp^{-1} \left( \bXi_{p,q} \bSig \ \bXi_{p,q}'  \right) \Gp^{-1} \be_0,   			   \ \qquad 			\bXi_{p,q} = \bR_p' \bW_p  -  \rho^{p+1} \Lp \be_{p+1}' \Gq^{-1} \bR_q' \bW_q
\end{equation}

To make the fixed-$n$ scalings feasible, $\shatus^2$ and $\shatrbc^2$ take the forms \eqref{eqn:US variance locpoly} and \eqref{eqn:RBC variance locpoly} and replace $\bSig$ with an appropriate estimator. First, we form $\hat{v}(X_i) = ( Y_i - \br_p(X_i - x)'\bhat_p )^2$ for $\shatus^2$ or $\hat{v}(X_i) = ( Y_i - \br_q(X_i - x)'\bhat_q )^2$ for $\shatrbc^2$. The latter is bias-reduced because $\br_p(X_i - x)'\bbeta_p$ is a $p$-term Taylor expansion of $m(X_i)$ around $x$, and $\bhat_p$ estimates $\bbeta_p$ (similarly with $q$ in place of $p$), and we have $q>p$. Next, motivated by the fact that least-squares residuals are on average too small, we appeal to the HC$k$ class of estimators (see \citet{MacKinnon2013_BookChap} for a review), which are defined as follows. First, $\shatus^2$-HC0 uses $\Shat_\US = \diag(\hat{v}(X_i): i = 1,\ldots, n)$. Then, $\shatus^2$-HC$k$, $k=1, 2, 3$, is obtained by dividing $\hat{v}(X_i)$ by, respectively, $(n-2\tr(\bQ_p)+\tr(\bQ_p'\bQ_p))/n$, $(1-\bQ_{p,ii})$, or $(1-\bQ_{p,ii})^2$, where $\bQ_p:=\bR_p '\Gp^{-1} \bR_p' \bW_p/n$ is the projection matrix and $\bQ_{p,ii}$ its $i$-th diagonal element. The corresponding estimators $\shatrbc^2$-HC$k$ are the same, but with $q$ in place of $p$. For theoretical results, we use HC0 for concreteness and simplicity, though inspection of the proof shows that simple modifications allow for the other HC$k$ estimators and rates do not change. These estimators may perform better for small sample sizes. Another option is to use a nearest-neighbor-based variance estimators with a fixed number of neighbors, following the ideas of \citet{Muller-Stadtmuller1987_AoS} and \citet{Abadie-Imbens2008_AdES}. Note that none of these estimators assume local or global homoskedasticity nor rely on new tuning parameters. Details and simulation results for all these estimators are given in the supplement, see \S S.II.2.3 and Table S.II.9.

%%%%%%%%%%%%%%%%%%%%%%%%%%%%%%%%%%%
%%%%%%%%%%%%%%%%%%%%%%%%%%%%%%%%%%%
\subsection{Higher Order Expansions of Coverage Error}
	\label{sec:coverage locpoly}

Recycling notation to emphasize the parallel, we study the following three statistics:
	\[\tus = \frac{\sqrt{n\h}( \hat{m} - m)}{\shatus}, 			\qquad		 \tbc =  \frac{\sqrt{n\h}( \hat{m} - \hat{B}_m -  m)}{\shatus},		\qquad		 \trbc = \frac{\sqrt{n\h}( \hat{m} - \hat{B}_m -  m)}{\shatrbc},\]
and their associated confidence intervals $\ius$, $\ibc$, and $\irbc$, exactly as in \Eqref{eqn:intervals}. Importantly, all present definitions and results are valid for an evaluation point in the interior and at the boundary of the support of $X_i$. The following standard conditions will suffice, augmented with the appropriate Cram\'er's condition given in the supplement to conserve space.

\begin{assumption}[Data-generating process]
	\label{dgp locpoly} 
	$\{(Y_1, X_1), \ldots, (Y_n, X_n)\}$ is a random sample, where $X_i$ has the absolutely continuous distribution with Lebesgue density $f$, $\E[Y^{8+\delta} \vert X] < \infty$ for some $\delta>0$, and in a neighborhood of $x$, $f$ and $v$ are continuous and bounded away from zero, $m$ is $S > q+2$ times continuously differentiable with bounded derivatives, and $m^{(S)}$ is H\"older continuous with exponent $\varsigma$.
\end{assumption}

\begin{assumption}[Kernels]
	\label{kernel locpoly}
	The kernels $K$ and $L$ are positive, bounded, even functions, and have compact support.
\end{assumption}

We now give our main, generic result for local polynomials, analogous to Theorem \ref{thm:coverage density}. For notation, the polynomials $q_1$, $q_2$, and $q_3$ and the biases $\eus$ and $\ebc$, are cumbersome and exact forms are deferred to the supplement. All that matters is that the polynomials are known, odd, bounded, and bounded away from zero and that the biases have the usual convergence rates, as detailed below.
\begin{theorem}
	\onehalfspacing
	\label{thm:coverage locpoly}
	Let Assumptions \ref{dgp locpoly}, \ref{kernel locpoly}, and Cram\'er's condition hold and $n \h/ \log(n\h) \to \infty$.
	\begin{enumerate}[label=(\alph*)]

		\item If $\eus \log(n\h) \to 0$, then
			\begin{align*}
				\P[m \in \ius]   = 1 - \alpha  &  + \left\{ \frac{1}{n \h} q_{1,\US} + \eus^2 q_{2,\US}
				                                  + \frac{\eus}{\sqrt{n \h}} q_{3,\US}\right\}   \phi(z_{\frac{\alpha}{2}}) \;\{1+o(1)\}.
			\end{align*}

		\item If $\ebc \log(n\h) \to 0$ and $\rho \to 0$, then
			\begin{align*}
				\P[m \in \ibc]   = 1 - \alpha  &  +   \left\{  \frac{1}{n \h} q_{1,\US} + \ebc^2 q_{2,\US}  +   \frac{\ebc}{\sqrt{n \h}} q_{3,\US} \right\}  \phi(z_{\frac{\alpha}{2}})\;\{1+o(1)\} 		\\
				& +  \rho^{p+2} (\Omega_{1,\BC} + \rho^{p+1} \Omega_{2,\BC}) \phi(z_{\frac{\alpha}{2}}) z_{\frac{\alpha}{2}}\;\{1+o(1)\}.
			\end{align*}

		\item If $\ebc \log(n\h) \to 0$ and $\rho \to \bar{\rho} < \infty$, then
			\begin{align*}
				\P[m \in \irbc]   = 1 - \alpha  &  +    \left\{  \frac{1}{n \h}q_{1,\RBC}  +  \ebc^2 q_{2,\RBC}   +   \frac{\ebc}{\sqrt{n \h}} q_{3,\RBC}  \right\}  \phi(z_{\frac{\alpha}{2}})\;\{1+o(1)\}.
			\end{align*}

	\end{enumerate}

\end{theorem}

This theorem, which covers both interior and boundary points, establishes that the conclusions found in the density case carry over to odd-degree local polynomial regression. (Although we focus on $p$ odd, part (a) is valid in general and (b) and (c) are valid at the boundary for $p$ even.) In particular, this shows that robust bias correction is as good as, or better than, undersmoothing in terms of coverage error. Traditional bias correction is again inferior due to the variance and covariance terms $\rho^{p+2} (\Omega_{1,\BC} + \rho^{p+1} \Omega_{2,\BC})$. Coverage error optimal bandwidths can be derived as well, and similar conclusions are found. Best possible rates are defined for fixed $p$ here, the analogue of $\k$ above; see Section \ref{sec:corollaries} for further discussion on smoothness.

Before discussing bias correction, one aspect of the undersmoothing result is worth mentioning. The fact that Theorem \ref{thm:coverage locpoly} covers both interior and boundary points, without requiring additional assumptions, is in some sense, expected: one of the strengths of local polynomial estimation is its adaptability to boundary points. In particular, from \Eqref{eqn:bias locpoly} and $p$ odd it follows that $\eus \asymp \sqrt{nh} \h^{p+1}$ at the interior and the boundary. Therefore, part (a) shows that the decay rate in coverage error does not change at the boundary for the standard confidence interval (but the leading constants will change). This finding contrasts with the result of \citet{Chen-Qin2002_SJS} who studied the special case $p=1$ without bias correction (part (a) of Theorem \ref{thm:coverage locpoly}), and is due entirely to our fixed-$n$ Studentization.

Turning to robust bias correction, we will, in contrast, find rate differences between the interior and the boundary, no matter the parity of $q$. As before, $\ebc$ has two terms, representing the higher-order bias of the point estimator and the bias of the bias estimator. The former can be viewed as the bias if $m^{(p+1)}$ were zero, and since $p+1$ is even, we find that it is of order $\sqrt{nh} \h^{p+3}$ in the interior but $\sqrt{n\h} \h^{p+2}$ at the boundary. The bias of the bias correction depends on both bandwidths $\h$ and $\b$, as well as $p$ and $q$, in exact analogy to the density case. For $q$ odd, it is of order $\h^{p+1} \b^{q-p}$ at all points, whereas for $q$ even this rate is attained at the boundary, but in the interior the order increases to $\h^{p+1} \b^{q+1-p}$. Collecting these facts: in the interior, $\ebc \asymp \sqrt{n\h} \h^{p+3} ( 1 + \rho^{-2} \b^{q-p-2} ) $ for odd $q$ or with $\b^{q-p-1}$ for $q$ even; at the boundary, $\ebc \asymp \sqrt{n\h} \h^{p+2} ( 1 + \rho^{-1} \b^{q-p-1})$. Further details are in the supplement.

In light of these rates, the same logic of Section \ref{sec:corollaries} leads us to restrict attention to bounded, positive $\rho$ and $q=p+1$, and thus even. \citet[][Remark 7]{Calonico-Cattaneo-Titiunik2014_Ecma} point out that in the special case of $q=p+1$, $K=L$, and $\rho=1$, $\hat{m} - \hat{B}_m$ is identical to a local polynomial estimator of order $q$; this is the closest analogue to $M$ being a higher-order kernel. If the point of interest is in the interior, then $q = p+2$ yields the same rates.

For notational ease, let $\etbc^{\tt int}$ and $\etbc^{\tt bnd}$ be the leading constants for the interior and boundary, respectively, so that e.g. $\ebc = \sqrt{n \h}  \h^{p+3} [\etbc^{\tt int} + o(1)]$ in the interior (exact expressions are in the supplement). We then have the following, precise result; the analogue of Corollary \ref{thm:RBC coverage}(a).
\begin{corollary}%[Robust bias correction: bounded, positive $\rho$]
	\onehalfspacing
	\label{thm:RBC locpoly}
	Let the conditions of Theorem \ref{thm:coverage locpoly}(c) hold, with $\bar{\rho} \in (0,\infty)$ and $q=p+1$.
	\begin{enumerate}[label=(\alph*)]

		\item For an interior point, 
			\[   \P[m \in \irbc]   = 1 - \alpha    +    \biggl\{  \frac{1}{n \h}q_{1,\RBC}  +  n \h^{1 + 2(p+3)} (\etbc^{\tt int})^2 q_{2,\RBC}   +      \h^{p+3} (\etbc^{\tt int} ) q_{3,\RBC}  \biggr\}  \phi(z_{\frac{\alpha}{2}}) \;\{1+o(1)\}.    \]

		\item For a boundary point, 
			\[   \P[m \in \irbc]   = 1 - \alpha   +    \biggl\{  \frac{1}{n \h}q_{1,\RBC}  +  n \h^{1 + 2(p+2)} (\etbc^{\tt bnd})^2 q_{2,\RBC}    +      \h^{p+2} (\etbc^{\tt bnd} ) q_{3,\RBC}  \biggr\}  \phi(z_{\frac{\alpha}{2}}) \;\{1+o(1)\}.    \]

	\end{enumerate}
	
\end{corollary}
There are differences in both the rates and constants between parts (a) and (b) of this result, though most of the changes to constants are ``hidden'' notationally by the definitions of $\etbc^{\tt bnd}$ and the polynomials $q_{k,\RBC}$. Part (a) most closely resembles Corollary \ref{thm:RBC coverage} due to the symmetry yielding the corresponding rate improvement (recall that $\k$ in the density case is replaced with $p+1$ here), and hence all the corresponding conclusions hold qualitatively for local polynomials.

%%%%%%%%%%%%%%%%%%%%%%%%%%%%%%%%%%%
%%%%%%%%%%%%%%%%%%%%%%%%%%%%%%%%%%%
\subsection{Practical Choices and Empirical Consequences}
	\label{sec:practical locpoly}
	
As we did for the density, we now derive bandwidth choices, and data-driven implementations, to optimize coverage error in applications. 
\begin{corollary}%[Robust bias correction: bounded, positive $\rho$]
	\onehalfspacing
	\label{thm:bandwidth locpoly}
	Let the conditions of Corollary \ref{thm:RBC locpoly} hold.
	\begin{enumerate}[label=(\alph*)]

		\item For an interior point, if $\h=\h^*_\RBC = H^*_\RBC n^{-1/(p + 4)}$, then $\P[m \in \irbc] = 1 - \alpha + O(n^{-(p+3)/(p+4)})$, where
			\[H^*_\RBC  (\bar{\rho}) = \argmin_{H>0} \bigl\vert  H^{-1} q_{1,\RBC}   +    H^{1+2(p+3)} (\etbc^{\tt int})^2 q_{2,\RBC}   +    H^{p+3} (\etbc^{\tt int} ) q_{3,\RBC}  \bigr\vert .\]

		\item For a boundary point, if $\h=\h^*_\RBC = H^*_\RBC (\rho) n^{-1/(p+3)}$, then $\P[m \in \irbc] = 1 - \alpha + O(n^{-(p+2)/(p+3)})$, where 
			\[ H^*_\RBC (\bar{\rho}) = \argmin_{H>0} \bigl\vert  H^{-1} q_{1,\RBC}   +    H^{1+2(p+2)} (\etbc^{\tt bnd})^2 q_{2,\RBC}   +    H^{p+2} (\etbc^{\tt bnd} ) q_{3,\RBC} \bigr\vert \]

	\end{enumerate}
	
\end{corollary}

To implement these results, we first set $\rho=1$ and the kernels $K$ and $L$ equal to any desired second order kernel, typical choices being triangular, Epanechnikov, and uniform. The variance estimator $\shatrbc^2$ is defined in Section \ref{sec:variance}, and is fully implementable, and thus so is $\irbc$, once the bandwidth $\h$ is chosen.

For selecting $\h$ at an interior point, the same conclusions from density estimation apply: (i) coverage of $\irbc$ is quite robust with respect to $\h$ and $\rho$, (ii) feasible choices for $\h$ are easy to construct, and (iii) an MSE-optimal bandwidth only delivers the best coverage error for $p=1$ (that is, $\k=2$ in the density case). On the other hand, for a boundary point, an interesting consequence of Corollary \ref{thm:bandwidth locpoly} is that an MSE-optimal bandwidth \emph{never} delivers optimal coverage error decay rates, even for local linear regression: $\h^*_\MSE \propto n^{-1/(2p+3)} \gg \h^*_\RBC \propto n^{-1/(p + 3)}$.

Keeping this in mind, we give a fully data-driven direct plug-in (DPI) bandwidth selector for both interior and boundary points: $\hat{h}^{\tt int}_\PI = \hat{H}^{\tt int}_{\PI} \; n^{-1/(p+4)}$ and $\hat{h}^{\tt bnd}_\PI = \hat{H}^{\tt bnd}_{\PI} \; n^{-1/(p+3)}$, where $\hat{H}^{\tt int}_{\PI}$ and $\hat{H}^{\tt bnd}_{\PI}$ are estimates of (the appropriate) $H^*_\RBC$ of Corollary \ref{thm:bandwidth locpoly}, obtained by estimating unknowns by pilot estimators employing a readily-available pilot bandwidth. The complete steps to form $\hat{H}^{\tt int}_{\PI}$ and $\hat{H}^{\tt bnd}_{\PI}$ are in the supplement, as is a second data-driven bandwidth choice, based on rescaling already-available MSE-optimal bandwidths. All our methods are available in the {\tt nprobust} package: see \url{http://sites.google.com/site/nppackages/nprobust}.

%%%%%%%%%%%%%%%%%%%%%%%%%%%%%%%%%%%
%%%%%%%%%%%%%%%%%%%%%%%%%%%%%%%%%%%
%%%%%%%%%%%%%%%%%%%%%%%%%%%%%%%%%%%
\section{Simulation Results}
	\label{sec:simuls}

We now report a representative sample of results from a simulation study to illustrate our findings. We drew 5,000 replicated data sets, each being $n=500$ i.i.d.\ draws from the model $Y_i = m(X_i) + \e_i$, with $m(x) =  \sin(3\pi x/2) (1+18x^2[\sign(x)+1])^{-1}$, $X_i \sim \mathcal{U}[0,1]$, and $\e_i \sim \N(0,1)$. We consider inference at the five points $x \in \{-2/3, -1/3, 0, 1/3, 2/3\}$. The function $m(x)$ and the five evaluation points are plotted in Figure \ref{fig:dgp}; this function was previously used by \citet{Berry-Carroll-Ruppert2002_JASA} and \citet{Hall-Horowitz2013_AoS}. The supplement gives results for other models, bandwidth selectors and their simulation distributions, alternative variance estimators, and more detailed studies of coverage and length.

We compared robust bias correction to undersmoothing, traditional bias correction, the off-the-shelf {\sf R} package {\tt locfit} \citep{locfit}, and the procedure of \citet{Hall-Horowitz2013_AoS}. In all cases the point estimator is based on local linear regression with the data-driven bandwidth $\hat{h}^{\tt int}_\PI$, which shares the rate of $\hat{h}_\MSE$ in this case, and $\rho=1$. The {\tt locfit} package has a bandwidth selector, but it was ill-behaved and often gave zero empirical coverage. \citet{Hall-Horowitz2013_AoS} do not give an explicit optimal bandwidth, but do advocate a feasible $\hat{h}_\MSE$, following \citet{Ruppert-Sheather-Wand1995_JASA}. To implement their method, we used 500 bootstrap replications and we set $1-\xi=0.9$ over a sequence $\{x_1,...,x_N\}=\{-0.9,-0.8,\ldots,0,\ldots,0.8,0.9\}$ to obtain the final quantile $\hat{\alpha}_\xi(\alpha_0)$, and used their proposed standard errors $\hat{\sigma}_{\tt HH}^2=\kappa\hat{\sigma}^2 / \hat{f}_X$, where $\hat{\sigma}^2=\sum_{i=1}^n\hat{\varepsilon}_i^2/n$ for $\hat{\varepsilon}_i=\tilde{\varepsilon}_i-\bar{\varepsilon}$, with $\tilde{\varepsilon}_i=Y_i-\hat{m}(X_i)$ and $\bar{\varepsilon}=\sum_{i=1}^n \tilde{\varepsilon}_i / n$.

Table \ref{table:locpoly simuls} shows empirical coverage and average length at all five points for all five methods. Robust bias correction yields accurate coverage throughout the support; performance of the other methods varies. For $x=-2/3$, the regression function is nearly linear, leaving almost no bias, and the other methods work quite well. In contrast, at $x=-1/3$ and $x=0$, all methods except robust bias correction suffer from coverage distortions due to bias. Indeed, \citet[][p.\ 1893]{Hall-Horowitz2013_AoS} report that ``[t]he `exceptional' 100$\xi$\% of points that are not covered are typically close to the locations of peaks and troughs, [which] cause difficulties because of bias.'' Finally, bias is still present, though less of a problem, for $x=1/3$ and $x=2/3$, and coverage of the competing procedures improves somewhat. Motivated by the fact that the data-driven bandwidth selectors may be ``too large'' for proper undersmoothing, we studied the common practice of ad-hoc undersmoothing of the MSE-optimal bandwidth choice $\hat{h}_\MSE$: the results in Table S.II.8 of the supplement show this to be no panacea.

%% Another quote from Hall+Horowitz, page 1893: Our approach accommodates bias by increasing the width of confidence bands. However, the amount by which we increase width is no greater than a constant factor, rather than the polynomial amount (as a function of n) associated with most suggestions for undersmoothing

To illustrate our findings further, Figures \ref{fig:results}(a) and \ref{fig:results}(b) compare coverage and length of different inference methods over a range of bandwidths. Robust bias correction delivers accurate coverage for a wide range of bandwidths, including larger choices, and thus can yield shorter intervals. For undersmoothing, coverage accuracy requires a delicate choice of bandwidth, and for correct coverage, a longer interval. Figure \ref{fig:results}(c), in color online, reinforces this point by showing the ``average position'' of $\ius(\h)$ and $\irbc(\h)$ for a range of bandwidths: each bar is centered at the average bias and is of average length, and then color-coded by coverage (green indicates good coverage, fading to red as coverage deteriorates). These results show that when $\ius$ is short, bias is large and coverage is poor. In contrast, $\irbc$ has good coverage at larger bandwidths and thus shorter length.

%%%%%%%%%%%%%%%%%%%%%%%%%%%%%%%%%%%
%%%%%%%%%%%%%%%%%%%%%%%%%%%%%%%%%%%
%%%%%%%%%%%%%%%%%%%%%%%%%%%%%%%%%%%
\section{Conclusion}
	\label{sec:conclusion}

This paper has made three distinct, but related points regarding nonparametric inference. First, we showed that bias correction, when coupled with a new standard error formula, performs as well or better than undersmoothing for confidence interval coverage and length. Further, such intervals are more robust to bandwidth choice in applications. Second, we showed theoretically when the popular empirical practice of using MSE-optimal bandwidths is justified, and more importantly, when it is not, and we gave concrete implementation recommendations for applications. Third, we proved that confidence intervals based on local polynomials do have automatic boundary carpentry, provided proper Studentization is used. These results are tied together through the themes of higher order expansions and the importance of finite sample variance calculations and the key, common message that inference procedures must account for additional variability introduced by bias correction.

%%%%%%%%%%%%%%%%%%%%%%%%%%%%%%%%%%%
%%%%%%%%%%%%%%%%%%%%%%%%%%%%%%%%%%%
%%%%%%%%%%%%%%%%%%%%%%%%%%%%%%%%%%%
%\newpage
\singlespacing

% Alternatively, locally redefine the command '\section' to get something that appears in the table of contents.
	\section{References}
	\begingroup
	\renewcommand{\section}[2]{}	%this removes the standard 'References' header that BibTeX puts in.
	\bibliography{Edgeworth-Nonparametrics--Bibliography}{}
	\bibliographystyle{jasa}
	\endgroup

\newpage

%%%%%%%%%%%%%%%%%%%%%%%%%%%%%%%
%%									%%
%%		tables and figures go here			%%
%%									%%
%%%%%%%%%%%%%%%%%%%%%%%%%%%%%%%

\begin{table}[h]
	\renewcommand{\arraystretch}{1.12}
	\begin{center}
		\caption{Empirical Coverage and Average Interval Length of 95$\%$ Confidence Intervals\label{table:locpoly simuls}}
		\resizebox{\columnwidth}{!}{%latex.default(table.reg, file = paste("lp_main_m", m, "_", vce,     "_n", n, ".txt", sep = ""), landscape = FALSE, n.cgroup = c(1,     1, 5, 4), cgroup = c("Evaluation Point", "Average Bandwidth",     "Empirical Coverage", "Interval Length"), outer.size = "scriptsize",     col.just = rep("c", 11), center = "none", title = "", table.env = FALSE)%
\begin{tabular}{ccccccccccccc}
\hline\hline
\multicolumn{1}{c}{\bfseries Evaluation}& \multicolumn{1}{c}{\bfseries Average }&&\multicolumn{5}{c}{\bfseries Empirical Coverage}&\multicolumn{1}{c}{\bfseries }&\multicolumn{4}{c}{\bfseries Interval Length}\tabularnewline
\cline{4-8} \cline{10-13}
\multicolumn{1}{c}{{\bf Point} }&\multicolumn{1}{c}{ {\bf Bandwidth} }&\multicolumn{1}{c}{}&\multicolumn{1}{c}{US}&\multicolumn{1}{c}{Locfit}&\multicolumn{1}{c}{BC}&\multicolumn{1}{c}{HH}&\multicolumn{1}{c}{RBC}&\multicolumn{1}{c}{}&\multicolumn{1}{c}{US}&\multicolumn{1}{c}{Locfit}&\multicolumn{1}{c}{HH}&\multicolumn{1}{c}{RBC}\tabularnewline
\hline
-2/3&0.203&&95.4&95.3&83.4&93.7&95.0&&0.437&0.472&0.410&0.627\tabularnewline
-1/3&0.307&&44.1&66.3&82.1&31.2&94.1&&0.357&0.381&0.275&0.507\tabularnewline
0&0.320&&73.5&83.1&81.0&58.8&93.6&&0.348&0.376&0.267&0.498\tabularnewline
1/3&0.343&&93.3&93.7&82.0&83.1&94.5&&0.332&0.361&0.245&0.477\tabularnewline
2/3&0.262&&94.2&94.5&82.0&89.2&94.3&&0.386&0.418&0.329&0.554\tabularnewline
\hline
\end{tabular}
}\bigskip
	\end{center}
	\renewcommand{\arraystretch}{1}
	\vspace{-.2in}\footnotesize\textbf{Notes}: {\bf (i)} Column ``Average Bandwidth'' reports simulation average of estimated bandwidths $\h = \hat{h}_\PI \equiv \hat{h}^{\tt int}_\PI$. Simulation distributions for estimated bandwidths are reported in the supplement. {\bf (ii)} US = Undersmoothing, Locfit = \texttt{R} package \texttt{locfit} by \cite{locfit}, BC = Bias Corrected, HH = \cite{Hall-Horowitz2013_AoS}, RBC = Robust Bias Corrected. 
\end{table}

\vspace{1in}
\begin{figure}[h]
	\caption{True Regression Model and Evaluation Points}	\label{fig:dgp}
  \centering\includegraphics[scale=.8]{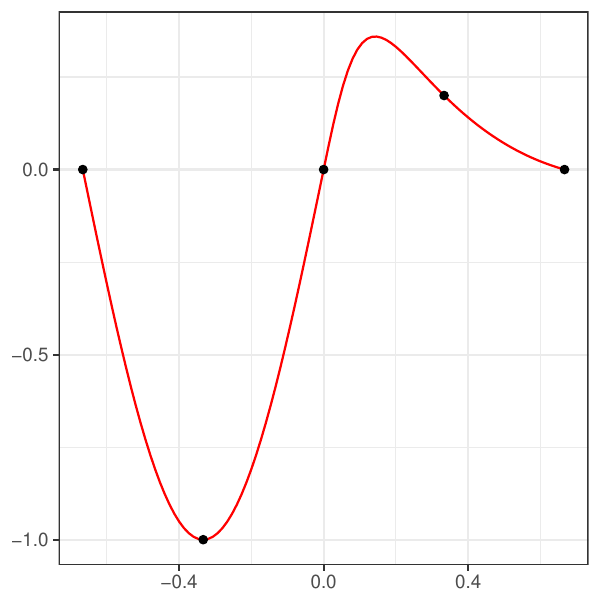}
\end{figure}

%\begin{landscape}
\begin{figure}[h]\centering
	%\captionsetup[subfigure]{justification=centering}
	\caption{Local Polynomial Simulation Results for $x=0$}\label{fig:results}
        \begin{subfigure}[t]{0.5\textwidth}
				  \includegraphics[scale=0.7]{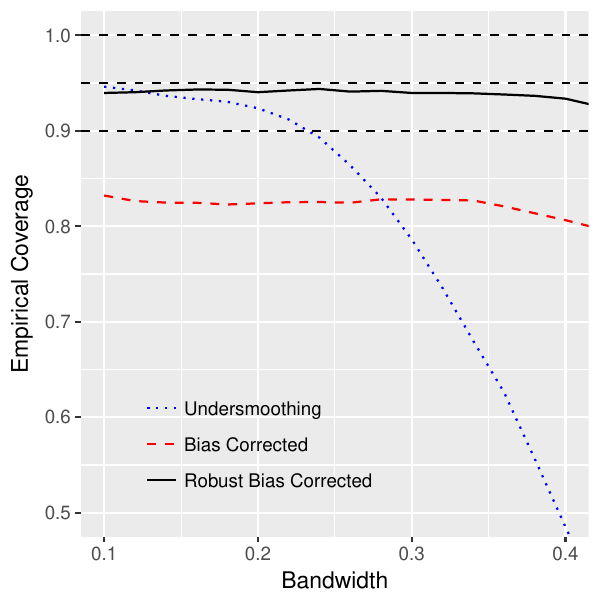}
	        \caption{Empirical Coverage}
        \end{subfigure}%\hspace{.2in}
        \begin{subfigure}[t]{0.5\textwidth}
						\includegraphics[scale=0.7]{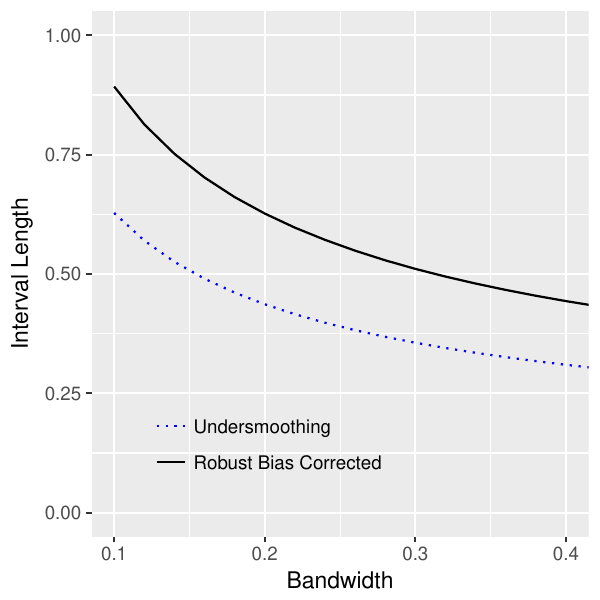}
	          \caption{Interval Length}
        \end{subfigure}\\
        \begin{subfigure}[t]{0.5\textwidth}
					\begin{center}	\includegraphics[scale=0.7]{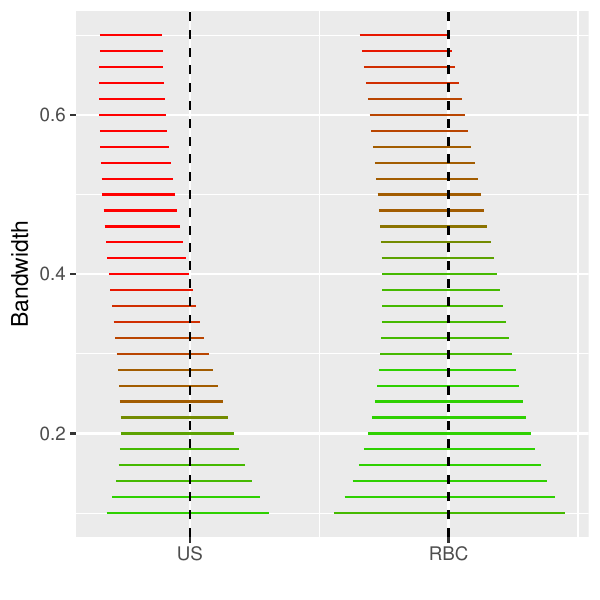}
		       \caption{Average Confidence Intervals and Coverage for Undersmoothing (US) and Robust Bias Correction (RBC).}
		       \end{center}
        \end{subfigure}%
       \hspace*{-9em}
    \begin{subfigure}[b]{0.5\textwidth}
    	\includegraphics[width=1.0\textwidth]{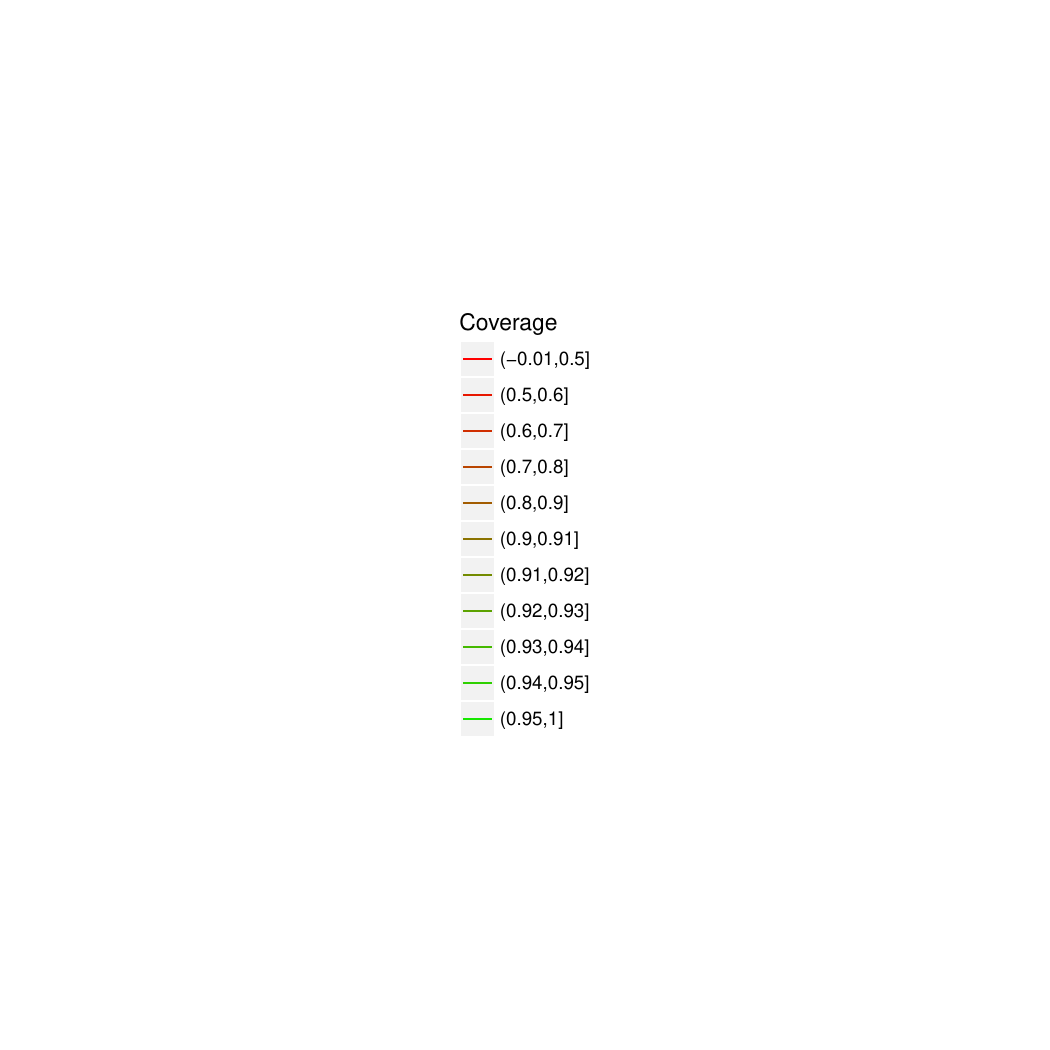}
    \end{subfigure}   
\end{figure}
%\end{landscape}

%%%%%%%%%%%%%%%%%%%%%%%%%%%%%%%%%%%
%%%%%%%%%%%%%%%%%%%%%%%%%%%%%%%%%%%
%%%%%%%%%%%%%%%%%%%%%%%%%%%%%%%%%%%
%%%%										%%%%
%%%%										%%%%
%%%%										%%%%
%%%%			Supplement starts here			%%%%
%%%%										%%%%
%%%%										%%%%
%%%%										%%%%
%%%%%%%%%%%%%%%%%%%%%%%%%%%%%%%%%%%
%%%%%%%%%%%%%%%%%%%%%%%%%%%%%%%%%%%
%%%%%%%%%%%%%%%%%%%%%%%%%%%%%%%%%%%

\clearpage

	\newcommand{\sumj}{\sum_{j=1}^n}
	\newcommand{\sumk}{\sum_{k=1}^n}

	\newcommand{\defsym}{:=}

%Equal up to remainder terms
	\newcommand{\oeq}{\stackrel{o}{=}}

\numberwithin{section}{part}
\numberwithin{table}{part}
\numberwithin{figure}{part}

\renewcommand\thepart{S.\Roman{part}}
\renewcommand\thesection{\thepart.\arabic{section}}
\renewcommand\thetable{\thepart.\arabic{table}}
\renewcommand\thefigure{\thepart.\arabic{figure}}

\setenumerate[1]{label=\bf(\alph*)}
\setenumerate[2]{label=\bf(\roman*)}

\setlength{\cftbeforesecskip}{2pt}
\setlength{\cftbeforepartskip}{2em}

\setlength{\cftsecnumwidth}{3.5em}
\setlength{\cftsubsecnumwidth}{3.75em}
\setlength{\cftsubsubsecnumwidth}{4em}

\setlength{\cftsecindent}{2em}
\setlength{\cftsubsecindent}{3em}
\setlength{\cftsubsubsecindent}{4em}

\partfont{\centering}

\setcounter{page}{0}
\begin{center}
{\bf\Large
\vspace{-0.75in} Supplement to ``On the Effect of Bias Estimation on Coverage Accuracy in Nonparametric Inference''
}
\end{center}
\maketitle
\thispagestyle{empty}

\singlespacing

\bigskip

This supplement contains technical and notational details omitted from the main text, proofs of all results, further technical details and derivations, and additional simulations results and numerical analyses. The main results are Edgeworth expansions of the distribution functions of the $t$-statistics $\tus$, $\tbc$, and $\trbc$, for density estimation and local polynomial regression. Stating and proving these results is the central purpose of this supplement. The higher-order expansions of confidence interval coverage probabilities in the main paper follow immediately by evaluating the Edgeworth expansions at the interval endpoints.

Part \ref{supp:density} contains all material for density estimation at interior points, while Part \ref{supp:locpoly} treats local polynomial regression at both interior and boundary points, as in the main text. Roughly, these have the same generic outline:
\begin{itemize}
	\singlespacing
	\itemsep=1pt
	\item We first present all notation, both for the estimators themselves and the Edgeworth expansions, regardless of when the notation is used, as a collective reference;
	\item We then discuss optimal bandwidths and other practical matters, expanding on details of the main text;
	\item Assumptions for validity of the Edgeworth expansions are restated from the main text, and Cram\'er's condition is discussed;
	\item Bias properties are discussed in more detail than in the main text, and some things mentioned there are made precise;
	\item The main Edgeworth expansions are stated, some corollaries are given, and the proofs are given;
	\item Complete simulation results are presented.
\end{itemize}

\bigskip

All our methods are implemented in {\sf R} and {\tt STATA} via the {\tt nprobust} package, available from \url{http://sites.google.com/site/nppackages/nprobust} (see also \url{http://cran.r-project.org/package=nprobust}). See \citet{Calonico-Cattaneo-Farrell2017_nprobust} for a complete description.

\newpage
\setcounter{tocdepth}{2}
\singlespacing
\tableofcontents

\onehalfspacing

%%%%%%%%%%%%%%%%%%%%%%%%%%%%%%%%%%%
%%%%%%%%%%%%%%%%%%%%%%%%%%%%%%%%%%%
\clearpage
\part{Kernel Density Estimation and Inference}
	\label{supp:density}
%%%%%%%%%%%%%%%%%%%%%%%%%%%%%%%%%%%
%%%%%%%%%%%%%%%%%%%%%%%%%%%%%%%%%%%

%%%%%%%%%%%%%%%%%%%%%%%%%%%%%%%%%%%
%%%%%%%%%%%%%%%%%%%%%%%%%%%%%%%%%%%
\section{Notation}
	\label{supp:notation density}

Here we collect notation to be used throughout this section, even if it is restated later. Throughout this supplement, let $X_{\h,i} = (x - X_i)/\h$ and similarly for $X_{\b,i}$. The evaluation point is implicit here. In the course of proofs we will frequently write $s=\sqrt{n\h}$.

%%%%%%%%%%%%%%%%%%%%%%%%%%%%%%%%%%%
\subsection{Estimators, Variances, and Studentized Statistics}

To begin, recall that the original and bias-corrected density estimators are
	\[\hat{f}(x) = \frac{1}{n \h} \sumi K\left( X_{\h,i} \right)\] 
and 
\begin{equation}
	\label{suppeqn:kernel M}
	\hat{f} - \hat{B}_f = \frac{1}{n \h} \sumi M\left( X_{\h,i} \right),		\qquad\qquad		 M(u) \defsym K(u) - \rho^{1 + \k} L^{(\k)}(\rho u) \mu_{K,\k},
\end{equation}
for symmetric kernel functions $K(\cdot)$ and $L(\cdot)$ that integrate to one on their compact support, $\h$ and $\b$ are bandwidth sequences that vanish as $n \to \infty$, and where 
	\[\rho = \h /\b, 		\qquad  \qquad		\hat{B}_f = \h^\k \hat{f}^{(\k)}(x) \mu_{K,\k}, 		\qquad  \qquad		\hat{f}^{(\k)}(x) = \frac{1}{n \b^{1 + \k} } \sumi L^{(\k)}\left( X_{\b,i} \right),\]
and integrals of the kernel are denoted
\[\mu_{K,k} = \frac{(-1)^k}{k!}\int u^k K(u) du, 	 \quad\qquad\text{ and }\quad\qquad	  \vartheta_{K,k} = \int K(u)^k du.\]

The three statistics $\tus$, $\tbc$, and $\trbc$ share a common structure that is exploited to give a unified theorem statement and proof. For $v \in \{1,2\}$, define
\[\hat{f}_v = \frac{1}{n \h} \sumi N_v \left( X_{\h,i} \right),	 \quad\qquad \text{where} \quad\qquad	 N_1(u) = K(u) \text{ and } N_2(u) = M(u),\]
and $M$ is given in \Eqref{suppeqn:kernel M}. Thus, $\hat{f}_1 = \hat{f}$ and $\hat{f}_2 = \hat{f} - \hat{B}_f$. In exactly the same way, define 
\[\sigma^2_v \defsym n\h \V[\hat{f}_v] = \frac{1}{\h} \left\{ \E \left[ N_v \left( X_{\h,i} \right)^2 \right] - \E \left[ N_v \left( X_{\h,i} \right) \right]^2   \right\}\]
and the estimator 
\[\shat^2_v  = \frac{1}{\h} \left\{ \frac{1}{n}\sumi \left[ N_v \left( X_{\h,i} \right)^2 \right] - \left[ \frac{1}{n}\sumi N_v \left( X_{\h,i} \right) \right]^2 \right\}.\]

The statistic of interest for the generic Edgeworth expansion is, for $1 \leq w \leq v\leq 2$, 
\[T_{v,w} \defsym \frac{ \sqrt{n \h} (\hat{f}_v - f) }{ \shat_w }.\]
In this notation, 
\[\tus = T_{1,1}, 	\qquad	\tbc = T_{2,1}, 	\qquad \text{ and } \qquad 	\trbc = T_{2,2}.\]

%%%%%%%%%%%%%%%%%%%%%%%%%%%%%%%%%%%
\subsection{Edgeworth Expansion Terms}
	\label{supp:terms density}

The scaled bias is $\eta_v = \sqrt{n \h} (\E[\hat{f}_v] - f)$. The Standard Normal distribution and density functions are $\Phi(z)$ and $\phi(z)$, respectively.

The Edgeworth expansion for the distribution of $T_{v,w}$ will consist of polynomials with coefficients that depend on moments of the kernel(s). To this end, continuing with the generic notation, for nonnegative integers $j, k, p$, define 
	\[\gamma_{v,p} = \h^{-1} \E\left[N_v\left(X_{\h,i}\right)^p\right],		\qquad \quad  \qquad 	       \Delta_{v,j} = \frac{1}{s} \sumi \left\{ N_v\left(X_{\h,i}\right)^j - \E\left[N_v\left(X_{\h,i}\right)^j\right]\right\},\]
and
	\[\nu_{v,w}(j,k,p) = \frac{1}{\h}\E\left[ \left(N_v\left(X_{\h,i}\right) - \E\left[N_v\left(X_{\h,i}\right)\right] \right)^j \left(N_w\left(X_{\h,i}\right)^p - \E\left[N_w\left(X_{\h,i}\right)^p\right] \right)^k     \right].\]
We abbreviate $\nu_{v,w}(j,0,p) = \nu_v(j)$.

To expand the distribution function, additional polynomials are needed beyond those used in the main text for coverage error. These are
\begin{align*}
	p_{v,w}^{(1)}(z) & = \phi(z) \sigma_w^{-3} [\nu_{v,w}(1,1,2) z^2/2  - \nu_v(3) (z^2 - 1)/6], 			\\
	p_{v,w}^{(2)}(z) & =  - \phi(z) \sigma_w^{-3} \E[\hat{f}_w] \nu_{v,w}(1,1,1) z^2,		\qquad \text{ and } \qquad 	p_{v,w}^{(3)}(z)  = \phi(z)  \sigma_w^{-1}.
\end{align*}
Next, recall from the main text the polynomials used in \emph{coverage error} expansions, here with an explicit argument for a generic quantile $z$ rather than the specific $z_{\alpha/2}$:
\begin{align*}
    q_1(z;K) & = \vartheta_{K,2}^{-2} \vartheta_{K,4}(z^3 - 3z)/6 - \vartheta_{K,2}^{-3} \vartheta_{K,3}^2 [2z^3/3 + (z^5 - 10z^3 + 15z)/9],\\
    q_2(z;K) & = - \vartheta_{K,2}^{-1}(z),		\quad\qquad \text{ and } \quad\qquad 		q_3(z;K)  = \vartheta_{K,2}^{-2} \vartheta_{K,3}( 2 z^3/3).
\end{align*}
The corresponding polynomials for expansions of the \emph{distribution function} are
	\[q_{v,w}^{(k)}(z) = \frac{1}{2} \frac{\phi(z)}{f} q_k (z; N_w), \qquad k = 1,2,3.\]

Finally, the precise forms of $\Omega_1$ and $\Omega_2$ are:
	\[\Omega_1 = - 2 \frac{\mu_{K,\k}}{\nu_1(2)}  \left\{  \int f(x - u\h )K(u) L^{(\k)}(u \rho) du  - \b \int f(x - u\h) K(u) du \int f(x - u\b) L^{(\k)}(u) du  \right\}\]
and $\Omega_2 = \mu_{K,\k}^2 \vartheta_{K,2}^{-2} \vartheta_{L^{(\k)},2}$.	These only appear for $\tbc$, and so are not indexed by $\{v,w\}$.

All these are discussed in Section \ref{supp:Edgeworth density}.

%%%%%%%%%%%%%%%%%%%%%%%%%%%%%%%%%%%
%%%%%%%%%%%%%%%%%%%%%%%%%%%%%%%%%%%
\section{Details of practical implementation}
	\label{supp:practical density}

We maintain $\ell=2$ and recommend $\k=2$. For the kernels $K$ and $L$, we recommend either the second order minimum variance (to minimize interval length) or the MSE-optimal kernels; see Sections \ref{sec:rho} and \ref{supp:kernel order}. In the next two subsections we discuss choice of $\h$ and $\rho$.

As argued below in Section \ref{sec:rho}, we shall maintain $\rho=1$. In the main text we give a direct plug-in (DPI) rule to implement the coverage-error optimal bandwidth. Here we we give complete details for this procedure as well as document a second practical choice, based on a rule-of-thumb (ROT) strategy. Both choices yield the optimal coverage error decay rate of $n^{-(\k+2)/(1+(\k+2))}$.

All our methods are implemented in {\sf R} and {\tt STATA} via the {\tt nprobust} package, available from \url{http://sites.google.com/site/nppackages/nprobust} (see also \url{http://cran.r-project.org/package=nprobust}). See \citet{Calonico-Cattaneo-Farrell2017_nprobust} for a complete description.

\begin{remark}[Undercoverage of $\ius(\h^*_\MSE)$]
	\label{rem:undercover}
	It is possible not only to show that $\ius(\h^*_\MSE)$ asymptotically undercovers (see \citet{Hall-Horowitz2013_AoS} for discussion in the regression context) but also to quantify precisely the coverage. To do so, write $\tus = \sqrt{n\h} (\hat{f} - \E[\hat{f}]) / \shatus + \eus / \shatus$, where the first term will be asymptotically standard Normal and the second will be a nonrandom, nonvanishing bias when $\h^*_\MSE$ is used.

	To characterize this second term, first we define $\h^*_\MSE$ in in our notation. Recall from \Eqref{suppeqn:bias density} and Section \ref{supp:notation density} that the mean-square error of $\hat{f}$ can be written as $(n\h)^{-1} \sus^2  + (n\h)^{-1}\eus^2$. Define $\etus$ to be the leading constant of the bias, so that $\eus = \sqrt{n\h} \h^\k [ \etus + o(1)]$ and the MSE becomes $(n\h)^{-1} \sus^2  + \h^{2\k}  \etus^2$. Then optimizing the MSE yields, in this notation, 
	\[\h^*_\MSE = n^{-\tfrac{1}{2\k+1}} \left( \frac{\sus^2}{2\k \etus^2} \right)^{-\tfrac{1}{2\k+1}}. \]
Therefore, the second term of $\tus(\h^*_\MSE)$ will be
	\[\frac{\eus}{\shatus} = \frac{\sqrt{n\h^*_\MSE} (\h^*_\MSE)^\k [ \etus + o(1)]}{\shatus} = \left( n (\h^*_\MSE)^{2\k+1}\right)^{1/2}  \frac{\etus}{\shatus}  =  \left( \frac{\sus^2}{2\k \etus^2}\right)^{1/2}  \frac{\etus}{\shatus}  = \left(\frac{1}{2\k}\right)^{1/2}[1+o_p(1)], \]
	using consistency of $\shatus$ (or if a feasible $\h^*_\MSE$ is used, it is the bias estimate that must be consistent). Hence $\tus(\h^*_\MSE) \to_d \N\left( (2\k)^{-1/2},1\right)$.
	
	The most common empirical case would be $\k=2$ and $\alpha = 0.05$, and so $\tus \to_d \N(1/2,1)$ and $\P[f\in\ius(\h^*_\MSE)] \approx 0.92$.
\end{remark}

%%%%%%%%%%%%%%%%%%%%%%%%%%%%%%%%%%%
\subsection{Bandwidth Choice: Rule-of-Thumb (ROT)}

Motivated by the fact that estimating $\hat{H}_{\PI}$ might be difficult in practice, while data-driven MSE-optimal bandwidth selectors are readily-available, the ROT bandwidth choice is to simply rescale any feasible MSE-optimal bandwidth $\hat{\h}_\MSE$ to yield optimal coverage error decay rates (but sub-optimal constants): 
	\[\hat{h}_\ROT = \hat{h}_\MSE \; n^{-(\k-2)/((1+2\k)(\k+3))}.\]
When $\k=2$, $\hat{h}_\ROT=\hat{h}_\MSE$, which is optimal (in rates) as discussed previously.

\begin{remark}[Integrated Coverage Error]
	\label{rem:IMSE}
	A closer analogue of the \citet{Silverman1986_book} rule of thumb, which uses the integrated MSE, would be to integrate the coverage error over the point of evaluation $x$. For point estimation, this approach has some practical benefits. However, in the present setting note that $\int f^{(\k)}(x) dx = 0 $, removing the third term (of order $\h^\k$) entirely and thus, for any given point $x$, yields a lower quality approximation.
\end{remark}

%%%%%%%%%%%%%%%%%%%%%%%%%%%%%%%%%%%
\subsection{Bandwidth Choice: Direct Plug-In (DPI)}
	\label{supp:bandwidth density}

To detail the direct plug-in (DPI) rule from the main text, it is useful to first simplify the problem. Recall from the main text that the optimal choice is $\h^*_\RBC = H^*_\RBC (\rho) n^{-1/(\k + 3)}$, where 
	\begin{align*}
		H^*_\RBC (K, L, \bar{\rho}) &= \argmin_H \bigl\vert  H^{-1} q_1(M_{\bar{\rho}}) + H^{1+2(\k+2)} (f^{(\k+2)})^2 \left( \mu_{K,\k+2} - \bar{\rho}^{-2}  \mu_{K,\k} \mu_{L,2} \right)^2  q_2(M_{\bar{\rho}})     		\\
                           & \qquad\qquad\qquad + H^{\k+2} f^{(\k+2)} \left(\mu_{K,\k+2} - \bar{\rho}^{-2}\mu_{K,\k}\mu_{L,2}\right) q_3(M_{\bar{\rho}}) \bigr\vert  .
	\end{align*}
With $\ell=2$ and $\rho = 1$, and using the definitions of $q_k(M_1)$, $k=1,2,3$, from the main text or Section \ref{supp:terms density}, this simplifies to:
	\begin{align*}
		H^*_\RBC (K, L, 1)  = \argmin_H \; \Biggl\vert  &  H^{-1} \left\{ \vartheta_{M,4}\frac{ z^2 - 3}{6}  -  \vartheta_{M,3}^2 \frac{z^4 - 4z^2 + 15}{9} \right\}                		\\
				&    -      H^{1+2(\k+2)}  \left\{ (f^{(\k+2)})^2  \left( \mu_{K,\k+2} -  \mu_{K,\k} \mu_{L,2} \right)^2 \vartheta_{M,2}\right\}                		\\
				&    +     H^{\k+2}   \left\{   f^{(\k+2)} \left(\mu_{K,\k+2} - \mu_{K,\k}\mu_{L,2}\right)  \vartheta_{M,3} \frac{ 2 z^2}{3}  \right\} \Biggr\vert,
	\end{align*}
where $z = z_{\alpha/2}$ the appropriate upper quantile of the Normal distribution. However, $H^*_\RBC (\rho)$ still depends on the unknown density through $f^{(\k+2)}$.

Our recommendation is a DPI rule of order one, which uses a pilot bandwidth to estimate $f^{(\k+2)}$ consistently. A simple and easy to implement choice is the MSE-optimal bandwidth appropriate to estimating $f^{(\k+2)}$, say $\h^*_{\k+2,\MSE}$, which is different from $\h^*_\MSE$ for the level of the function; see e.g., \citet{Wand-Jones1995_book}. Let us denote a feasible MSE-optimal pilot bandwidth by $\hat{\h}_{\k+2,\MSE}$. Then we have:
	\begin{align*}
		\hat{H}_{\PI} (K, L, 1)  = \argmin_H \; \Biggl\vert  &  H^{-1} \left\{ \vartheta_{M,4}\frac{ z^2 - 3}{6}  -  \vartheta_{M,3}^2 \frac{z^4 - 4z^2 + 15}{9} \right\}                		\\
				&    -      H^{1+2(\k+2)}  \left\{ \hat{f}^{(\k+2)}(x; \hat{\h}_{\k+2,\MSE})^2  \left( \mu_{K,\k+2} -  \mu_{K,\k} \mu_{L,2} \right)^2 \vartheta_{M,2}\right\}                 		\\
				&    +     H^{\k+2}   \left\{   \hat{f}^{(\k+2)}(x; \hat{\h}_{\k+2,\MSE})  \left(\mu_{K,\k+2} - \mu_{K,\k}\mu_{L,2}\right)  \vartheta_{M,3} \frac{ 2 z^2}{3}  \right\} \Biggr\vert.
	\end{align*}
This is now easily solved numerically (see note below). Further, if $\k=2$, the most common case in practice, and $K$ and $L$ are either the respective second order minimum variance or MSE-optimal kernels (Sections \ref{sec:rho} and \ref{supp:kernel order}), then the above may be simplified to:
	\begin{align*}
		\hat{H}_{\PI} (M, 1)  = \argmin_H \; \Biggl\vert  &  H^{-1} \left\{ \vartheta_{M,4}\frac{ z^2 - 3}{6}  -  \vartheta_{M,3}^2 \frac{z^4 - 4z^2 + 15}{9} \right\}                 		\\
				&    -      H^9  \left\{ \hat{f}^{(4)}(x; \hat{\h}_{\k+2,\MSE})^2  \mu_{M,4}^2 \vartheta_{M,2}\right\}                 		\\
				&    +     H^4   \left\{   \hat{f}^{(4)}(x; \hat{\h}_{\k+2,\MSE})  \mu_{M,4}  \vartheta_{M,3} \frac{ 2 z^2}{3}  \right\} \Biggr\vert.
	\end{align*}

Continuing with $\k=2$, a second option is a DPI rule of order zero, which uses a reference model to build the rule of thumb, more akin to \citet{Silverman1986_book}. Using the Normal distribution, so that $f(x) = \phi(x)$ and derivatives have known form, we obtain:
	\begin{align*}
		\hat{H}_{\PI}(M, 1)  = \argmin_H \; \biggl\vert  &  H^{-1} \left\{ \vartheta_{M,4}\frac{ z^2 - 3}{6}  -  \vartheta_{M,3}^2 \frac{z^4 - 4z^2 + 15}{9} \right\}                 		\\
				&    -      H^9  \left\{ \left[ \left(\tilde{x}^4 - 6\tilde{x}^2 + 3 \right) \phi(\tilde{x}) \right]^2  \mu_{M,4}^2 \vartheta_{M,2}\right\}                 		\\
				&    +     H^4     \left\{  \left(\tilde{x}^4 - 6\tilde{x}^2 + 3 \right) \phi(\tilde{x}) \mu_{M,4}  \vartheta_{M,3} \frac{ 2 z^2}{3}  \right\} \biggr\vert
	\end{align*}
where $\tilde{x} = (x - \hat{\mu}) / \hat{\sigma}_X$ is the point of interest centered and scaled.

\begin{remark}[Notes on computation]
	When numerically solving the above minimization problems, computation will be greatly sped up by squaring the objective function.
\end{remark}

%%%%%%%%%%%%%%%%%%%%%%%%%%%%%%%%%%%
\subsection{Choice of $\rho$}
	\label{sec:rho}

First, we expand on the argument that $\rho$ should be bounded and positive. Intuitively, the standard errors $\shatrbc^2$ control variance up to order $(n \h)^{-1}$, while letting $\b \to 0$ faster removes more bias. If $\b$ vanishes too fast, the variance is no longer controlled. Setting $\bar{\rho} \in (0,\infty)$ balances these two. Let us simplify the discussion by taking $\ell=2$, reflecting the widespread use of symmetric kernels. This does not affect the conclusions in any conceptual way, but considerably simplifies the notation. With this choice, \Eqref{suppeqn:kernel M} yields the tidy expression
	\[ \ebc  =  \sqrt{n\h}  \h^{\k+2}  f^{(\k+2)} \left( \mu_{K,\k+2} - \rho^{-2} \mu_{K,\k} \mu_{L,2} \right) \; \{ 1 + o(1)\}. \]
Choice of $\ell$ and $\b$ (or $\rho$) cannot reduce the first term, which represents $\E[\hat{f}] - f - B_f$, and further, if $\bar{\rho} = \infty$, the bias rate is not improved, but the variance is inflated beyond order $(n \h)^{-1}$. On the other hand, if $\bar{\rho} = 0$, then not only is a delicate choice of $\b$ needed, but $\ell > 2$ is required, else the second term above dominates $\ebc$, and the full power of the variance correction is not exploited; that is, more bias may be removed without inflating the variance rate. \citet[p.\ 682]{Hall1992_AoS_density} remarked that if $\E[\hat{f}] - f - B_f$ is (part of) the leading bias term, then ``explicit bias correction [\ldots] is even less attractive relative to undersmoothing.'' We show that, on the contrary, when using our proposed Studentization, it is optimal that $\E[\hat{f}] - f - B_f$ is (part of) the dominant bias term. This reasoning is not an artifact of choosing $\k$ even and $\ell=2$, but in other cases $\rho \to 0$ can be optimal if the convergence is sufficiently slow to equalize the two bias terms.

The following result which makes the above intuition precise.
\begin{corollary}[Robust bias correction: $\rho \to 0$]
	\label{thm:rho to zero}
	Let the conditions of Theorem \ref{thm:Edgeworth density}(c) hold, with $\bar{\rho} = 0$, and fix $\ell=2$ and $\k \leq S-2$. Then 
	\begin{align*}
		\P[f \in \irbc]  = 1 - \alpha & + \biggl\{ \frac{1}{n \h}   q_1(K) 			
		                             + n\h^{1 + 2(\k+2)}   (f^{(\k+2)})^2  \left(\mu_{K,\k+2}^2 + \rho^{-4}\mu_{K,\k}^2 \mu_{L,2}^2 \right)  q_2(K)			\\
		                          & \qquad  + \h^{\k+2}   f^{(\k+2)}\left(\mu_{K,\k+2} - \rho^{-2}\mu_{K,\k}\mu_{L,2}\right) q_3(K)  \biggr\} \frac{\phi(z_{\frac{\alpha}{2}})}{f}	 \; \{1+o(1)\}
	\end{align*}
\end{corollary}
By virtue of our new studentization, the leading variance remains order $(n \h)^{-1}$ and the problematic correlation terms are absent, however by forcing $\rho \to 0$, the $\rho^{-2}$ terms of $\ebc$ are dominant (the bias of $\hat{B}_f$), and in light of our results, unnecessarily inflated. This verifies that $\bar{\rho} = 0$ or $\infty$ will be suboptimal.

We thus restrict to bounded and positive, $\rho$. Therefore, $\rho$ impacts only the shape of the ``kernel'' $M_\rho(u) = K(u) - \rho^{1 + \k} L^{(\k)}(\rho u) \mu_{K,\k}$, and hence the choice of $\rho$ depends on what properties the user desires for the kernel. It happens that $\rho=1$ has good theoretical properties and performs very well numerically (see Section \ref{supp:simuls density}). As a result, from the practitioner's point of view, choice of $\rho$ (or $\b$) is completely automatic.

To see the optimality of $\rho=1$, consider two cogent and well-studied possibilities: finding the kernel shape to minimize (i) interval length and (ii) MSE. The following optimal shapes are derived by \citet{Gasser-Muller-Mammitzsch1985_JRSSB} and references therein. Given the above results, we set $\k=2$. Indeed, the optimality properties here do not extend to higher order kernels.

Minimizing interval length is (asymptotically) equivalent to finding the minimum variance fourth-order kernel, as $\srbc^2 \to f \vartheta_{M,2}$. Perhaps surprisingly, choosing $K$ and $L^{(2)}$ to be the second-order minimum variance kernels for estimating $f$ and $f^{(2)}$ respectively, yields an $M_1(u)$ that is exactly the minimum variance kernel. The fourth order minimum variance kernel for estimating $f$ is $K_{\MV}(u) = (3/8)(-5u^2 + 3)$, which is identical to $M_1(u)$ when $K$ is the uniform kernel and $L^{(2)} = (15/4) (3u^2 -1)$, the minimum variance kernels for $f$ and $f^{(2)}$ respectively.

The result is similar for minimizing MSE: choosing $K$ and $L^{(2)}$ to be the MSE-optimal kernels for their respective point estimation problems yields an MSE-optimal $M_1(u)$. The optimal fourth order kernel is $K_{\MSE}(u) = (15/32)(7u^4 - 10u^2 + 3)$, and the respective second-order MSE optimal kernels are $K(u) = (3/4)(1-u^2)$ and $L^{(2)}(u)  = (105/16)(6u^2 - 5u^4 - 1)$. A practitioner might use the MSE-optimal kernels (along with $\h^*_\MSE$) to obtain the best possible point estimate. Our results then give an accompanying measure of uncertainty that both has correct coverage and the attractive feature of using the same effective sample.

In Section \ref{supp:kernel order} we numerically compare several kernel shapes, focusing on: (i) interval length, measured by $\vartheta_{M,2}$, (ii) bias, given by $\tilde{\mu}_{M,4}$, and (iii) the associated MSE, given by $(\vartheta_{M,2}^8 \tilde{\mu}_{M,4}^2 )^{1/9}$. These results, and the discussion above, give the foundations for our recommendation of $\rho=1$, which delivers an easy-to-implement, fully automatic choice for implementing robust bias-correction that performs well numerically, as in Section \ref{supp:simuls density}.

\begin{remark}[Coverage Error Optimal Kernels]
	\label{rem:coverage kernel}
	Our results hint at a third notion of optimal kernel shape: minimizing coverage error. This kernel, for a fixed order $\k$, would minimize the constants in Corollary 1 of the main text. In that result, $\h$ is chosen to optimize the rate and the constant $H^*_\US$ gives the minimum for a fixed kernel $K$. A step further would be to view $H^*_\US$ as a function of $K$, and optimizing. To our knowledge, such a derivation has not been done and may be of interest.
\end{remark}

%%%%%%%%%%%%%%%%%%%%%%%%%%%%%%%%%%%
%%%%%%%%%%%%%%%%%%%%%%%%%%%%%%%%%%%
\section{Assumptions}
	\label{sec:assumptions density}

The following assumptions are sufficient for our results. The first two are copied directly from the main text (see discussion there) and the third is the appropriate Cram\'er's condition.

\begin{assumption}[Data-generating process] 
	\label{supp:dgp density} \,
	$\{X_1, \ldots, X_n\}$ is a random sample with an absolutely continuous distribution with Lebesgue density $f$. In a neighborhood of $x$, $f>0$, $f$ is $S$-times continuously differentiable with bounded derivatives $f^{(s)}$, $s=1,2,\cdots,S$, and $f^{(S)}$ is H\"older continuous with exponent $\varsigma$.
\end{assumption}

\begin{assumption}[Kernels] 
	\label{supp:kernel density} \, 
	The kernels $K$ and $L$ are bounded, even functions with support $[-1,1]$, and are of order $\k \geq 2$ and $\ell \geq 2$, respectively, where $\k$ and $\ell$ are even integers. That is, $\mu_{K,0} = 1$, $\mu_{K,k} = 0$ for $1 \leq k < \k$, and $\mu_{K,\k} \neq 0$ and bounded, and similarly for $\mu_{L,k}$ with $\ell$ in place of $\k$. Further, $L$ is $\k$-times continuously differentiable. For all integers $k$ and $l$ such that $k + l = \k-1$, $f^{(k)}(x_0) L^{(l)}((x_0 - x)/\b) = 0$ for $x_0$ in the boundary of the support.
\end{assumption}

It will cause no confusion (as the notations never occur in the same place), but in the course of proofs we will frequently write $s=\sqrt{n\h}$.

\begin{assumption}[Cram\'er's Condition] 
	\label{supp:Cramer density} 
	For each $\xi>0$ and all sufficiently small $\h$
	\[  \sup_{t \in \mathbb{R}^2,\ t_1^2 + t_2^2 > \xi} \left\vert \int \exp \{ i (t_1 M(u) + t_2 M(u)^2) \} f(x - u \h) du \right\vert \leq 1 - C(x, \xi) \h,\]
	where $C(x,\xi)>0$ is a fixed constant and $i=\sqrt{-1}$.
\end{assumption}

\begin{remark}[Sufficient Conditions for Cram\'er's Condition]
Assumption \ref{supp:Cramer density} is a high level condition, but one that is fairly mild. \citet{Hall1991_Statistics} provides a primitive condition for Assumption \ref{supp:Cramer density} and Lemma 4.1 in that paper verifies that Assumption \ref{supp:Cramer density} is implied. \citet{Hall1992_book} and \citet{Hall1992_AoS_density}  assume the same primitive condition. This condition is as follows. On their compact support, assumed here to be $[-1,1]$, there exists a partition $-1=a_0 < a_1 < \cdots < a_m = 1$, such that on each $(a_{j-1},a_j)$, $K$ and $M$ are differentiable, with bounded, strictly monotone derivatives.

	This condition is met for many kernels, with perhaps the only exception of practical importance being the uniform kernel. As \citet{Hall1991_Statistics} describes, it is possible to prove the Edgeworth expansion for the uniform kernel using different methods than we use in below. The uniform kernel is also ruled out for local polynomial regression, see Section \ref{sec:assumptions locpoly}.
\end{remark}

%%%%%%%%%%%%%%%%%%%%%%%%%%%%%%%%%%%
%%%%%%%%%%%%%%%%%%%%%%%%%%%%%%%%%%%
\section{Bias}

This section accomplishes three things. First, we first carefully derive the bias of the initial estimator and the bias correction. Second, we explicate the properties of the induced kernel $M_\rho$ in terms of bias reduction and how exactly this kernel is ``higher-order''. Finally, we examine two other methods of bias reduction: (i) estimating the derivatives without using derivatives of kernels \citep{Singh1977_AoS}, and (ii) the generalized jackknife approach \citep{Schucany-Sommers1977_JASA}. Further methods are discussed and compared by \citet{Jones-Signorini1997_JASA}. The message from both alternative methods echoes our main message: it is important to account for any bias correction when doing inference, i.e., to avoid the mismatch present in $\tbc$.

%%%%%%%%%%%%%%%%%%%%%%%%%%%%%%%%%%%
\subsection{Precise Bias Calculations}
	\label{supp:bias density}

Recall that the biases of the two estimators are as follows:
\begin{equation}
	\label{suppeqn:bias density}
	\E[\hat{f}] - f = \begin{cases}
					\h^\k f^{(\k)} \mu_{K,\k} + \h^{\k+2} f^{(\k+2)} \mu_{K,\k+2} + o(\h^{\k+2}) & \text{if } \k \leq S - 2			\\
					\h^\k f^{(\k)} \mu_{K,\k} + O(\h^{S + \varsigma})  & \text{if } \k \in \{S-1,S\}			\\
					0 + O(\h^{S + \varsigma}) & \text{if } \k > S
				\end{cases}
\end{equation}
and
\begin{equation}
	\label{suppeqn:bias corrected}
	\E[\hat{f} - \hat{B}_f] - f = \begin{cases}
					\h^{\k+2} f^{(\k+2)} \mu_{K,\k+2}  - \h^\k \b^\ell f^{(\k + \ell)} \mu_{K,\k} \mu_{L,\ell} + o(\h^{\k+2} + \h^\k \b^\ell) & \text{if } \k + \ell \leq S				\\
					\h^{\k+2} f^{(\k+2)} \mu_{K,\k+2}  + O(\h^\k \b^{S - \k + \varsigma})  + o(\h^{\k+2})  & \text{if } 2 \leq S - \k < \ell			\\
					O(\h^{S + \varsigma})  + O(\h^\k \b^{S - \k+\varsigma}) & \text{if } \k \in \{S-1,S\}			\\
					O(\h^{S + \varsigma})  + O(\h^\k \b^{S - \k})  & \text{if } \k > S.
				\end{cases}
\end{equation}

The following Lemma gives a rigorous proof of these statements.
\begin{lemma}
	\label{lem:bias density}
	Under Assumptions \ref{supp:dgp density} and \ref{supp:kernel density}, Equations \eqref{suppeqn:bias density} and \eqref{suppeqn:bias corrected} hold.
\end{lemma}
\begin{proof}
To show \Eqref{suppeqn:bias density}, begin with the change of variables and the Taylor expansion
\begin{align*}
	\E[\hat{f}] & = \h^{-1} \int K\left(X_{\h,i}\right) f(X_i) dX_i = \int K(u) f(x - u \h) du		\\
	& = \sum_{k=0}^{S} \left\{ (-\h)^k f^{(k)}(x) \int u^k K(u) du / k!  \right\}   +  (-\h)^S \int u^S K(u) \left( f^{(S)}(\bar{x}) - f^{(S)}(x) \right)du.
\end{align*}
where $\bar{x} \in [x,x-u\h]$. By the H\"older condition of Assumption \ref{supp:dgp density}, the final term is $O(\h^{S + \varsigma})$. If $\k > S$, then all $\int u^k K(u) du = 0$, and only this remainder is left. In all other cases, $\h^\k f^{(\k)}(x) \mu_{K,\k}$ is the first nonzero term of the summation, and hence the leading bias term. Further, by virtue of $\k$ being even and $K$ symmetric, $\int u^{\k+1} K(u) du =0$, leaving only $O(\h^{S + \varsigma})$ when $\k = S-1$, and otherwise, when $\k \leq S-2$, leaving $\h^{\k+2} f^{(\k+2)}(x) \mu_{K,\k+2} + o(\h^{\k+2})$. This completes the proof of \Eqref{suppeqn:bias density}.

To establish \Eqref{suppeqn:bias corrected}, first write 
\[\E[\hat{f} - \hat{B}_f] - f = \E[\hat{f} - f - B_f] + \E[B_f - \hat{B}_f],\]
where $B_f$ follows the convention of being identically zero if $\k > S$. The first portion is characterized by rearranging \Eqref{suppeqn:bias density}, so it remains to examine the second term. Let $\tilde{\k} = \k \vee S$. By repeated integration by parts, using the boundary conditions of Assumption \ref{supp:kernel density}:
\begin{align*}
	\E[\hat{f}^{(\k)}] & = \frac{1}{\b^{1+\k} } \int L^{(\k)}\left(X_{\b,i}\right) f(X_i) dX_i 		\\
	& =  \left.  - \frac{1}{\b^{1 + (\k-1)}} L^{(\k-1)}\left(X_{\b,i}\right) f(X_i)  \right|_\mathcal{X}  +  \frac{1}{\b^{1 + (\k-1)}} \int L^{(\k-1)}\left(X_{\b,i}\right) f^{(1)}(X_i) dX_i 		\\
	& =  0  +  \frac{1}{\b^{1 + (\k-1)}} \int L^{(\k-1)}\left(X_{\b,i}\right) f^{(1)}(X_i) dX_i 		\\
	& = - \frac{1}{\b^{1 + (\k-2)}} L^{(\k-2)}\left(X_{\b,i}\right) f^{(1)}(X_i)   +  \frac{1}{\b^{1 + (\k-2)}} \int L^{(\k-2)}\left(X_{\b,i}\right) f^{(2)}(X_i) dX_i 		\\
	& \quad  \vdots  			\\
	& =  \frac{1}{\b^{1 + (\k-\tilde{\k})}} \int L^{(\k-\tilde{\k})}\left(X_{\b,i}\right) f^{(\tilde{\k})}(X_i) dX_i 		\\
	& =  \frac{1}{\b^{\k-\tilde{\k}}} \int L^{(\k-\tilde{\k})}(u) f^{(\tilde{\k})}(x - u \b) du,
\end{align*}
where the last line follows by a change of variables. We now proceed separately for each case delineated in \eqref{suppeqn:bias corrected}, from top to bottom. For $\k > S$, no reduction is possible, and the final line above is $O(\b^{S-\k})$, and with $B_f = 0$, we have $\E[B_f - \hat{B}_f] = 0 - \h^\k \mu_{K,\k} \E[\hat{f}^{(\k)}]  = O(\h^\k \b^{S-\k})$, as shown. For $\k \leq S$, by a Taylor expansion, the final line displayed above becomes
\begin{align*}
	\sum_{k=\k}^S \left\{ \b^{k-\k}  f^{(k)}(x) \mu_{L,k-\k} \right\} + \b^{S-\k} \int u^{S-\k} L(u) \left( f^{(S)}(\bar{x}) - f^{(S)}(x) \right)du. 
\end{align*}
The second term above is $O(\b^{S - \k + \varsigma})$ in all cases, and $\mu_{L,0} = 1$, which yields $\E[\hat{f}^{(\k)}] = f^{(\k)} + O(\b^{S - \k + \varsigma})$ for $\k \in \{S-1,S\}$, using $\mu_{L,1} = 0$ in the former case. Next, if $\k + \ell \leq S$, the above becomes $\E[\hat{f}^{(\k)}] =  f^{(\k)} + \b^\ell f^{(\k + \ell)} \mu_{L,\ell} + o(\b^\ell)$, as $\mu_{L,k} = 0$ for $1< k < \ell$, whereas if $\k + \ell > S$, the remainder terms can not be characterized, leaving $\E[\hat{f}^{(\k)}] =  f^{(\k)} + O(\b^{S - \k + \varsigma})$. Plugging any of these results into $\E[B_f - \hat{B}_f] = \h^\k \mu_{K,\k} (f^{(\k)} - \E[\hat{f}^{(\k)}])$ completes the demonstration of \Eqref{suppeqn:bias corrected}.
\end{proof}

%%%%%%%%%%%%%%%%%%%%%%%%%%%%%%%%%%%
\subsection{Properties of the kernel $M_\rho(\cdot)$}
	\label{supp:kernel order}

As made precise below, $M_\rho$ is a higher-order kernel. The choices of $K$, $L$, and $\rho$ determine the shape of $M_\rho$, which in turn effects the variance and bias constants. In standard kernel analyses, these constants are used to determine optimal kernel shapes for certain problems (see \citet{Gasser-Muller-Mammitzsch1985_JRSSB} and references therein). For several choices of $K$, $L$, and $\rho$, Table \ref{table:kernel shapes} shows numerical results for the various constants of the induced kernel $M_\rho$. The table includes (i) the variance, given by $\vartheta_{M,2}$ and relevant for interval length, (ii) a measure of bias given by $\tilde{\mu}_{M,4}$, and finally (iii) the resulting mean square error constant, $ [ \vartheta_{M,2}^8 \tilde{\mu}_{M,4}^2 ]^{1/9}$ ($\tilde{\mu}_{M,4}=(k!)(-1)^{k}\mu_{M,4}$). These specific constants are due to $M_\rho$ being a fourth order kernel, as discussed next, and would otherwise remain conceptually the same but rely on different moments. A more general, but more cumbersome procedure would be to choose $\rho$ numerically to minimize some notation of distance (e.g., $L_2$) between the resulting kernel $M_\rho$ and the optimal kernel shape already available in the literature. However, using $\rho=1$ as a simple rule-of-thumb exhibits very little lost performance, as shown in the Table and discussed in the paper.

\begin{sidewaystable}
	\begin{center}
		\begin{threeparttable}
			\caption{Numerical results for bias and variance constants of the induced higher-order kernel $M$ for several choices of $K$, $L$, and $\rho$}
			\label{table:kernel shapes}
			%latex.default(table1, file = paste("output/Table_Kernels.txt",     sep = ""), col.just = c("l", "l", "c", "c", "c", "c", "c",     "c", "c", "c", "c"), center = "none", title = "", n.cgroup = c(2,     3, 3, 3), cgroup = c("", "$\\rho==0.5$", "$\\rho=1$", "$\\rho=1.5$"),     table.env = FALSE)%
\begin{tabular}{llcccccccccccc}
\hline\hline
\multicolumn{2}{c}{\bfseries }&\multicolumn{1}{c}{\bfseries }&\multicolumn{3}{c}{\bfseries $\rho=0.5$}&\multicolumn{1}{c}{\bfseries }&\multicolumn{3}{c}{\bfseries $\rho=1$}&\multicolumn{1}{c}{\bfseries }&\multicolumn{3}{c}{\bfseries $\rho=1.5$}\tabularnewline
\cline{1-14}
\multicolumn{1}{c}{Kernel $K$}&\multicolumn{1}{c}{Kernel $L^{(2)}$}&\multicolumn{1}{c}{}&\multicolumn{1}{c}{$\tilde{\mu}_{M,4}$}&\multicolumn{1}{c}{$\vartheta_{M,2}$}&\multicolumn{1}{c}{MSE}&\multicolumn{1}{c}{}&\multicolumn{1}{c}{$\tilde{\mu}_{M,4}$}&\multicolumn{1}{c}{$\vartheta_{M,2}$}&\multicolumn{1}{c}{MSE}&\multicolumn{1}{c}{}&\multicolumn{1}{c}{$\tilde{\mu}_{M,4}$}&\multicolumn{1}{c}{$\vartheta_{M,2}$}&\multicolumn{1}{c}{MSE}\tabularnewline
\hline
Epanechnikov&$(105/16)(6u^2-5u^4-1)$&&  0.0690&  0.6430&  0.3729&& -0.0476&  1.2500&  0.6199&& -0.3643&  5.5992&  3.6944\tabularnewline
Uniform&$(105/16)(6u^2-5u^4-1)$&&  0.1722&  0.5152&  0.3752&& -0.0222&  1.4722&  0.6052&& -0.5500& 11.5742&  7.7202\tabularnewline
Biweight&$(105/16)(6u^2-5u^4-1)$&&  0.0357&  0.7617&  0.3744&& -0.0476&  1.2500&  0.6199&& -0.2738&  3.9537&  2.5448\tabularnewline
Triweight&$(105/16)(6u^2-5u^4-1)$&&  0.0210&  0.8617&  0.3715&& -0.0438&  1.2774&  0.6202&& -0.2197&  3.2395&  2.0300\tabularnewline
Tricube&$(105/16)(6u^2-5u^4-1)$&&  0.0335&  0.7542&  0.3658&& -0.0506&  1.2332&  0.6207&& -0.2786&  3.9344&  2.5436\tabularnewline
Cosine&$(105/16)(6u^2-5u^4-1)$&&  0.0629&  0.6617&  0.3747&& -0.0476&  1.2503&  0.6199&& -0.3475&  5.2717&  3.4651\tabularnewline
Epanechnikov&$(15/4)(3u^2-1)$&&  0.0643&  0.6410&  0.3660&& -0.0857&  1.1250&  0.6432&& -0.4929&  4.1754&  3.0440\tabularnewline
Uniform&$(15/4)(3u^2-1)$&&  0.1643&  0.5098&  0.3678&& -0.0857&  1.1250&  0.6432&& -0.7643&  7.6191&  5.7276\tabularnewline
Biweight&$(15/4)(3u^2-1)$&&  0.0323&  0.7543&  0.3630&& -0.0748&  1.1352&  0.6291&& -0.3656&  3.0550&  2.1579\tabularnewline
Triweight&$(15/4)(3u^2-1)$&&  0.0184&  0.8517&  0.3568&& -0.0649&  1.1631&  0.6229&& -0.2911&  2.5444&  1.7435\tabularnewline
Tricube&$(15/4)(3u^2-1)$&&  0.0300&  0.7487&  0.3547&& -0.0780&  1.1319&  0.6333&& -0.3712&  3.0729&  2.1764\tabularnewline
Cosine&$(15/4)(3u^2-1)$&&  0.0584&  0.6583&  0.3669&& -0.0836&  1.1254&  0.6399&& -0.4693&  3.9510&  2.8668\tabularnewline
Biweight&Biweight$^{(2)}$&&  0.0323&  0.7543&  0.3630&& -0.0748&  1.1352&  0.6291&& -0.3656&  3.0550&  2.1579\tabularnewline
Tricube&Tricube$^{(2)}$&&  0.0299&  0.7516&  0.3556&& -0.0790&  1.1993&  0.6687&& -0.3746&  3.7063&  2.5762\tabularnewline
Gaussian&Gaussian$^{(2)}$&&  2.2500&  0.3006&  0.4113&& -3.0000&  0.4760&  0.6599&&-17.2500&  1.3606&  2.4758\tabularnewline
\hline
\end{tabular}

			\begin{tablenotes}
				\item[1] As discussed in Section \ref{supp:kernel order}, $M_\rho$ behaves as a fourth order kernel in terms of bias reduction, but does not strictly fit within the class of kernels used in derivation of optimal kernel shapes. This explains the super-optimal behavior exhibited by some choices of $K$, $L$, and $\rho$.
				\item[2] The constants $\tilde{\mu}_{M,4}$ and $\vartheta_{M,2}$ measure bias and variance, respectively (the latter also being relevant for interval length). The MSE is measured by $ [ \vartheta_{M,2}^8 \tilde{\mu}_{M,4}^2 ]^{1/9}$, owing to $M_\rho$ being a fourth-order kernel.
			\end{tablenotes}
		\end{threeparttable}
	\end{center}
\end{sidewaystable}

It is worthwhile to make precise the sense in which the $n$-varying ``kernel'' $M_\rho(\cdot)$ of \Eqref{suppeqn:kernel M} is a higher-order kernel. Comparing Equations \eqref{suppeqn:bias density} and \eqref{suppeqn:bias corrected} shows exactly what is meant by this statement: the bias rate attained agrees with a standard estimate using a kernel of order $\k+2$ (if $\bar{\rho} > 0$), as $\ell \geq 2$. For example, if $\k = \ell = 2$ and $\bar{\rho} > 0$, then $M_{\bar{\rho}}(\cdot)$ behaves as a fourth-order kernel in terms of bias reduction.

However, it is not true in general that $M(\cdot)$ is a higher-order kernel in the sense that its moments below $\k + 2$ are zero. That is, for any $k < \k$, by the change of variables $w = \rho u$, 
\begin{align*}
	\int_{-1}^1 u^k M(u) du & = \int_{-1}^1 u^k K(u) du - \rho^{1+\k} \mu_{K,\k} \int_{-1}^1 u^k L^{(\k)}(\rho u)du			\\
	& = 0  - \rho^{1+\k} \mu_{K,\k} \rho^{-1 - k} \int_{-\rho}^\rho w^k L^{(\k)}(w) du			\\
	& = 0  - \rho^{\k-k} \mu_{K,\k}\int_{-\rho}^\rho w^k L^{(\k)}(w) du.
\end{align*}
Now, $ L(u) = L(-u)$ implies that $L^{(k)}(u) = (-1)^k L^{(k)}(-u) $. Since $\k$ is even, $L^{(\k)}(w)$ is symmetric, therefore if $k$ is odd $0=\int_{-\rho}^\rho w^k L^{(\k)}(w) du$ for any $\rho$. But this fails for $k$ even, even for $\rho=1$, and hence $\int_{-1}^1 u^k M(u) du \neq 0$. For example, in the leading case of $\k=\ell=2$, $\int_{-1}^1 u^2 M(u) du \neq 0$ in general, and so $M(\cdot)$ is not a fourth-order kernel in the traditional sense. 

Instead, the bias reduction is achieved differently. The proof of Lemma \ref{lem:bias density} makes explicit use of the structure imposed by estimating $f^{(\k)}$ using the \emph{derivative} of the kernel $L(\cdot)$. From a technical standpoint, an integration by parts argument shows how the properties of the kernel $L(\cdot)$ (not the function $L^{(\k)}(\cdot)$) are used to reduce bias. This argument \emph{precedes} the Taylor expansion of $f$, and thus moments of $M$ are never encountered and there is no requirement that they be zero. This approach is simple, intuitive, and leads to natural restrictions on the kernel $L$, and for this reason it is commonly employed in the literature and in practice \citep{Hall1992_AoS_density}.

%%%%%%%%%%%%%%%%%%%%%%%%%%%%%%%%%%%
\subsection{Other Bias Reduction Methods}
	\label{supp:other}

We now examine two other methods of bias reduction: (i) estimating the derivatives without using derivatives of kernels \citep{Singh1977_AoS}, and (ii) the generalized jackknife approach \citep{Schucany-Sommers1977_JASA}. Further methods are discussed and compared by \citet{Jones-Signorini1997_JASA}. Both methods are shown to be tightly connected to our results. Further, a more general message is that it is important to account for any bias correction when doing inference, i.e., to avoid the mismatch present in $\tbc$.

The first method, which dates at least to \citet{Singh1977_AoS}, is to introduce a class of kernel functions directly for derivative estimation, more closely following the standard notion of a higher-order kernel rather than using the derivative of a kernel to estimate the density derivative and proving bias reduction via integration by parts. \citet{Jones1994_CSTM} expands on this method and gives further references. This class of kernels is used in the derivation of optimal kernel shapes (for derivative estimation) by \citet{Gasser-Muller-Mammitzsch1985_JRSSB}. It is worthwhile to show how this class of kernel achieves bias correction and how this approach fits into our Edgeworth expansions.

Consider estimating $f^{(\k)}$ with 
	\[\tilde{f}^{(\k)}(x) = \frac{1}{n \b^{1 + \k} } \sumi J \left( X_{\b,i} \right),\]
for some kernel function $J(\cdot)$. Note well that $J$ is generic, it need not itself be a derivative, but this is the only difference here. A direct Taylor expansion (i.e. without first integrating by parts) then gives
\[\E[\tilde{f}^{(\k)}] = \b^{-\k} \sum_{k=0}^S \b^k \mu_{J,k} f^{(k)} + O(\b^{S + \varsigma}).\]
Thus, if $J$ satisfies $\mu_{J,k} = 0$ for $k=0, 1, \ldots, \k-1, \k+1, \k+2, \ldots, \k+(\ell-1)$, $\mu_{J,\k} = 1$, and $\mu_{J,\k + \ell} \neq 0$, and $S$ is large enough then 
\[\E[\tilde{f}^{(\k)}] =  f^{(\k)} + \b^\ell f^{(\k + \ell)} \mu_{J,\k + \ell} + o(\b^\ell),\]
just as achieved by $\hat{f}^{(\k)}$ and exactly matching \Eqref{suppeqn:bias density}. Note that $\mu_{J,0} = 0$, that is, the kernel $J$ does not integrate to one. In the language of \citet{Gasser-Muller-Mammitzsch1985_JRSSB}, $J$ is a kernel of order $(\k,\k + \ell)$.

Given this result, bias correction can of course be performed using $\tilde{f}^{(\k)}(x)$ (based on $J$) rather than $\hat{f}^{(\k)}$ (based on $L^{(\k)}$). Much will be the same: the structure of \Eqref{suppeqn:kernel M} will hold with $J$ in place of $L^{(\k)}$ and the results in \Eqref{suppeqn:bias corrected} are achieved with modifications to the constants (e.g., in the first line, $\mu_{J,\k + \ell}$ appears in place of $\mu_{L,\ell}$). In either case, the same bias rates are attained. Our Edgeworth expansions will hold for this class under the obvious modifications to the notation and assumptions, and all the same conclusions are obtained.

When studying optimal kernel shapes, \citet{Gasser-Muller-Mammitzsch1985_JRSSB} actually further restrict the class, by placing a limit on the number of sign changes over the support of the kernel, which ensures that the MSE and variance minimization problems have well-defined solutions. Collectively, these differences in the kernel classes explain why it is possible to demonstrate ``super-optimal'' MSE and variance performance for certain choices of $K$, $L^{(\k)}$, and $\rho$, as in Table \ref{table:kernel shapes}.

A second alternative is the generalized jackknife method of \citet{Schucany-Sommers1977_JASA}, and expanded upon by \citet{Jones-Foster1993_JNPS}. To simplify the notation and ease exposition, we describe this approach for second order kernels ($\k=2$), but the method, and all the conclusions below, generalize fully. We thank an anonymous reviewer for encouraging us to include these details.

Begin with two estimators $\hat{f}_1$ and $\hat{f}_2$, with (possibly different) bandwidths and second-order kernels $\h_j$ and $K_j$, $j=1, 2$; thus \Eqref{suppeqn:bias density} gives 
	\[\E[\hat{f}_j] - f(x) = \h_j^2 f^{(2)} \mu_{K_j,2} + o(\h_j^2),  		\qquad \qquad  	  j=1, 2.\]
\citet{Schucany-Sommers1977_JASA} propose to estimate $f$ with $\hat{f}_{\GJ,R} := ( \hat{f}_1 - R \hat{f}_2 ) / (1 - R)$, the bias of which is
\[\E[ \hat{f}_{\GJ,R} - f ] = \frac{f^{(2)}}{1-R} \left( \h_1^2  \mu_{K_1,2}  -  R  \h_2^2  \mu_{K_2,2} \right) + o(\h_1^2 + \h_2^2).\]
Hence, setting $R = (\h_1^2  \mu_{K_1,2} )/(   \h_2^2  \mu_{K_2,2})$ renders the leading bias exactly zero. Moreover, if $S \geq 4$, $\hat{f}_{\GJ,R}$ has bias $O(\h_1^4 + \h_2^4)$; behaving as a single estimator with $\k=4$. To put this in context of our results, observe that with this choice of $R$, if we let $\tilde{\rho} = \h_1 / \h_2$, then
\[\hat{f}_{\GJ,R} = \frac{1}{n \h_1} \sumi \tilde{M}\left(\frac{X_i - x}{\h_1} \right) ,  		\quad   		 	 M(u) = K_1(u) - \tilde{\rho}^{1+2}  \left\{    \frac{K_2(\tilde{\rho}u) - \tilde{\rho}^{-1} K_1(u)  }{\mu_{K_2,2}(1-R)} \right\} \mu_{K_1, 2},\]
exactly matching \Eqref{suppeqn:kernel M}. Or equivalently, $\hat{f}_{\GJ,R} = \hat{f}_1  -  \h_1^2  \tilde{f}^{(2)} \mu_{K_1,2}$, for the derivative estimator 
\[\tilde{f}^{(2)} = \frac{1}{n\h_2^{1+2}}\sumi \tilde{L}\left( \frac{X_i - x}{\h_2} \right),  		\quad  		\tilde{L}(u) = \frac{ K_2(u) - \tilde{\rho}^{-1} K_1( \tilde{\rho}^{-1} u)  }{\mu_{K_2,2}(1-R)}.  \]
Therefore, we can view $\hat{f}_{\GJ,R}$ as a change in the kernel $M(\cdot)$ or an explicit bias estimation described directly above with a specific choice of $J(\cdot)$ (depending on $\tilde{\rho}$ in either case). Again, \Eqref{suppeqn:kernel M} holds exactly. Thus, our results cover the generalized jackknife method as well, and the same lessons apply. 

Finally, we note that these bias correction methods can be applied to nonparametric regression as well, and local polynomial regression in particular, and that the same conclusions are found. We will not repeat this discussion however.

%%%%%%%%%%%%%%%%%%%%%%%%%%%%%%%%%%%
%%%%%%%%%%%%%%%%%%%%%%%%%%%%%%%%%%%
\section{First Order Properties}
	\label{supp:clt density}

Here we briefly state the first-order properties of $\tus$, $\tbc$, and $\trbc$, using the common notation $T_{v,w}$ defined in Section \ref{supp:notation density}. Recall that $\eta_v = \sqrt{n \h} (\E[\hat{f}_v] - f)$ is the scaled bias in either case. With this notation, we have the following result.

\begin{lemma}%[First-order properties]
	\label{thm:clt density}
	Let Assumptions \ref{supp:dgp density} and \ref{supp:kernel density} hold. Then if $n \h \to \infty$, $\eta_v \to 0$, and if $v=2$, $\rho \to 0 + \bar{\rho} \mathbbm{1}\{v=w\} < \infty$, it holds that $T_{v,w} \to_d \N(0,1)$.
\end{lemma}

The conditions on $\h$ and $\b$ behind the generic assumption that the scaled bias vanishes can be read off of \eqref{suppeqn:bias density} and \eqref{suppeqn:bias corrected}: $\tus$ requires $\sqrt{n\h} \h^\k \to 0$ whereas $\tbc$ and $\trbc$ require only $\sqrt{n \h} \h^\k (\h^2 \vee \b^\ell) \to 0$, and thus accommodate $\sqrt{n \h} \h^\k \not\to 0$ or $\b \not\to 0$ (but not both). However, bias correction requires a choice of $\rho=\h/\b$. One easily finds that $\V[\sqrt{n\h} \hat{B}_f] = O(\rho^{1 + 2\k})$, whence $\rho \to 0$ is required for $\tbc$. But $\trbc$ does not suffer from this requirement because of our proposed, new Studentization. From a first-order point of view, traditional bias correction allows for a larger class of sequences $\h$, but requires a delicate choice of $\rho$ (or $\b$), and \citet{Hall1992_AoS_density} shows that this constraint prevents $\tbc$ from improving inference. Our novel standard errors remove these constraints, allowing for improvements in bias to carry over to improvements in inference. The fact that a wider range of bandwidths is allowed hints at the robustness to tuning parameter choice discussed above and formalized by our Edgeworth expansions.

\begin{remark}[$\rho \to \infty$]
	\label{rem:rho}
	$\trbc \to_d \N(0,1)$ will hold even for $\bar{\rho} =  \infty$, under the even weaker bias rate restriction that $\ebc = o(\rho^{1/2+\k})$, provided $n \b \to \infty$. In this case $\hat{B}_f$ dominates the first-order approximation, but $\srbc^2$ still accounts for the total variability. However there is no gain for inference: the bias properties can not be improved due to the second bias term ($\E[\hat{f}] - f - B_f$), while variance can only be inflated. Thus, we restrict to bounded $\bar{\rho}$. Section \ref{sec:rho} has more discussion on the choice of $\rho$.
\end{remark}

%%%%%%%%%%%%%%%%%%%%%%%%%%%%%%%%%%%
%%%%%%%%%%%%%%%%%%%%%%%%%%%%%%%%%%%
\section{Main Result: Edgeworth Expansion}
	\label{supp:Edgeworth density}

Recall the generic notation:
\[T_{v,w} \defsym \frac{\sqrt{n\h} (\hat{f}_v - f) }{ \shat_w },\]
for $1 \leq w \leq v\leq 2$. The Edgeworth expansion for the distribution of $T_{v,w}$ will consist of polynomials with coefficients that depend on moments of the kernel(s). Additional polynomials are needed beyond those used in the main text for coverage error. These are:
\begin{align*}
	p_{v,w}^{(1)}(z) & = \phi(z) \sigma_w^{-3} [\nu_{v,w}(1,1,2) z^2/2  - \nu_v(3) (z^2 - 1)/6], 			\\
	p_{v,w}^{(2)}(z) & =  - \phi(z) \sigma_w^{-3} \E[\hat{f}_w] \nu_{v,w}(1,1,1) z^2,		\qquad \text{ and } \qquad 	p_{v,w}^{(3)}(z)  = \phi(z)  \sigma_w^{-1}.
\end{align*}
The polynomials $p_{v,w}^{(k)}$ are even, and hence cancel out of coverage probability expansions, but are used in the expansion of the distribution function itself (or equivalently, the coverage of a one-sided confidence interval).

Next, recall from the main text the polynomials used in \emph{coverage error} expansions:
\begin{align*}
    q_1(z;K) & = \vartheta_{K,2}^{-2} \vartheta_{K,4}(z^3 - 3z)/6 - \vartheta_{K,2}^{-3} \vartheta_{K,3}^2 [2z^3/3 + (z^5 - 10z^3 + 15z)/9],\\
    q_2(z;K) & = - \vartheta_{K,2}^{-1}(z),		\qquad \text{ and } \qquad 		q_3(z;K)  = \vartheta_{K,2}^{-2} \vartheta_{K,3}( 2 z^3/3).
\end{align*}
The corresponding polynomials for expansions of the \emph{distribution function} are
	\[q_{v,w}^{(k)}(z) = \frac{1}{2} \frac{\phi(z)}{f} q_k (z; N_w), \qquad k = 1,2,3.\]
As before, the $q_{v,w}^{(k)}$ are odd and hence do not cancel when computing coverage: the $q_k (z; N_w)$ in the main text are doubled for just this reason.

Note that, despite the notation, $q_{v,w}^{(k)}(z)$ depends only on the ``denominator'' kernel $N_w$. The notation comes from the fact that when first computed, the terms which enter into the $q_{v,w}^{(k)}(z)$ depend on both kernels, but the simplifications in \Eqref{eqn:simplifying nu} reduce the dependence to $N_w$. This is because for undersmoothing and robust bias correction, $v=w$, and for traditional bias correction $N_2 = M = K + o(1) = N_1 + o(1)$, as $\rho \to 0$ is assumed. Thus, when computing $\vartheta_{M,q}$ the terms with the lowest powers of $\rho$ will be retained. These can be found by expanding
	\[\vartheta_{M,q} = \int \left(K(u) - \rho^{1 + \k}\mu_{K,\k} L^{(\k)}(u)\right)^q du	
					= \sum_{j=0}^q {q \choose j}\left(-\mu_{K,\k}\rho^{1+\k}\right)^{q-j}\int K(u)^j L^{(\k)}(\rho u)^{q-j}du,\]
and hence we can write $\vartheta_{M,q} = \vartheta_{K,q} - \rho^{1+\k}  q \mu_{K,\k} L^{(\k)}(0) \vartheta_{K,q-1} + O(\h + \rho^{2 + \k})$. We can thus write $q_j(z ; M) = q_j(z ; K) + o(1)$ in this case. If the expansions were carried out beyond terms of order $(n\h)^{-1} + (n\h)^{-1/2}\eta_v + \eta_v^2  + \mathbbm{1}\{v \!\neq\! w\} \rho^{1+2\k}$ this would not be the case.

Finally, for traditional bias correction, there are additional terms in the expansion (see discussion in the main text) representing the covariance of $\hat{f}$ and $\hat{B}_f$ (denoted by $\Omega_1$) and the variance of $\hat{B}_f$ ($\Omega_2$). We now state their precise forms. These arise from the mismatch between the variance of the numerator of $\tbc$ and the standardization used, $\sus^2$, that is $\srbc^2/\sus^2$ is given by
\[  \frac{ n\h \V[\hat{f} - \hat{B}_f] }{ n\h \V[\hat{f}] } = \frac{ n\h \V[\hat{f}]  - 2 n\h \mathbb{C}[\hat{f},\hat{B}_f]  + n\h \V[\hat{B}_f] }{ n\h \V[\hat{f}] }  = 1   -  2 \frac{ n\h \mathbb{C}[\hat{f},\hat{B}_f] }{ n\h \V[\hat{f}] }  + \frac{ n\h \V[\hat{B}_f] }{ n\h \V[\hat{f}] } .\]
This makes clear that $\Omega_1$ and $\Omega_2$ are the constant portions of the last two terms. We have
	\[ -  2 \frac{ n\h \mathbb{C}[\hat{f},\hat{B}_f] }{ n\h \V[\hat{f}] } = \rho^{1+\k} \Omega_1, \]
where
	\[\Omega_1 = - 2 \frac{\mu_{K,\k}}{\nu_1(2)}  \left\{  \int f(x - u\h )K(u) L^{(\k)}(u \rho) du  - \b \int f(x - u\h) K(u) du \int f(x - u\b) L^{(\k)}(u) du  \right\}.\]
Note $\nu_1(2) = \sus^2$. Turning to $\Omega_2$, using the calculations in Section \ref{supp:bias density} (recall $\tilde{\k} = \k \vee S$), we find that 
\[\frac{ n\h \V[\hat{B}_f] }{ n\h \V[\hat{f}] } = \rho^{1 + 2\k} \Omega_2 \quad \text{ where }\quad \Omega_2 = \frac{  \mu_{K,\k}^2}{ \nu_1(2) } \left\{ \int f(x - u\b) L^{(\k)}(u)^2 du   -   \b^{1 + 2\tilde{\k}} \left( \int L^{(\k-\tilde{\k})}(u) f^{(\tilde{\k})}(x - u \b) du \right)^2 \right\}.\]
Fully simplifying would yield
\[\Omega_2 = \mu_{K,\k}^2 \vartheta_{K,2}^{-2} \vartheta_{L^{(\k)},2},\]
which can be used in Theorem \ref{thm:Edgeworth density}.

As a last piece of notation, define the scaled bias as $\eta_v = \sqrt{n \h}(\E[\hat{f}_v] - f)$.

We can now state our generic Edgeworth expansion, from whence the coverage probability expansion results follow immediately.
\begin{theorem}
	\label{thm:Edgeworth density}
	Suppose Assumptions \ref{supp:dgp density}, \ref{supp:kernel density}, and \ref{supp:Cramer density} hold, $n \h/ \log(n) \to \infty$, $\eta_v \to 0$, and if $v=2$, $\rho \to 0 + \bar{\rho} \mathbbm{1}\{v=w\}$. Then for 
	\begin{align*}
		F_{v,w}(z) = \Phi(z)  & +  \frac{1}{\sqrt{n \h}} p^{(1)}_{v,w}(z)  +  \sqrt{\frac{\h}{n}} p^{(2)}_{v,w}(z)   +   \eta_v p^{(3)}_{v,w}(z)   +   \frac{1}{n \h} q^{(1)}_{v,w}(z)  +  \eta_v^2 q^{(2)}_{v,w}(z)  +  \frac{\eta_v}{\sqrt{n \h}} q^{(3)}_{v,w}(z)		\\
			&  -  \mathbbm{1}\{v \!\neq\! w\} \rho^{1+\k}(\Omega_1 + \rho^\k \Omega_2)\frac{\phi(z)}{2}z,
	\end{align*}
	we have
	\[\sup_{z\in\mathbb{R}}\left|\P[T_{v,w}<z]-F_{v,w}(z)\right| = o\left((n\h)^{-1} + (n\h)^{-1/2}\eta_v + \eta_v^2  + \mathbbm{1}\{v \!\neq\! w\} \rho^{1+2\k} \right).\]
\end{theorem}

To use this result to find the expansion of the error in coverage probability of the Normal-based confidence interval, the function $F_{v,w}(z)$ is simply evaluated at the two endpoints of the interval. (Note: if the confidence interval were instead constructed with the bootstrap, a few additional steps are needed, but these do not alter any conclusions or results outside of constant terms.)

%%%%%%%%%%%%%%%%%%%%%%%%%%%%%%%%%%%
\subsection{Undersmoothing vs. Bias-Correction Exhausting all Smoothness}
	\label{supp:small S}

In general, we have assumed that the level of smoothness was large enough to be inconsequential in the analysis, and in particular this allowed for characterization of optimal bandwidth choices. In this section, in contrast, we take the level of smoothness to be binding, so that we can fully utilize the $S$ derivatives \emph{and} the H\"older condition to obtain the best possible rates of decay in coverage error for both undersmoothing and robust bias correction, but at the price of implementability: the leading bias constants can not be characterized, and hence feasible ``optimal'' bandwidths are not available.

For undersmoothing, the lowest bias is attained by setting $\k>S$ (see \Eqref{suppeqn:bias density}), in which case the bias is only known to satisfy $\E[\hat{f}] - f = O(\h^{S+\varsigma})$ (i.e., $B_f$ is identically zero) and bandwidth selection is not feasible. Note that this approach allows for $\sqrt{n \h} \h^S \not\to 0$, as $\eus = O(\sqrt{n \h} \h^{S+\varsigma})$. 

Robust bias correction has several interesting features here. If $\k \leq S-2$ (the top two cases in \Eqref{suppeqn:bias corrected}), then the bias from approximating $\E[\hat{f}] - f$ by $B_f$, that is not targeted by bias correction, dominates $\ebc$ and prevents robust bias correction from performing as well as the best possible infeasible (i.e., oracle) undersmoothing approach. That is, even bias correction requires a sufficiently large choice of $\k$ in order to ensure the fastest possible rate of decay in coverage error: if $\k \geq S-1$, robust bias correction can attain error decay rate as the best undersmoothing approach, and allow $\sqrt{n \h} \h^S \not\to 0$. 

Within $\k \geq S-1$, two cases emerge. On the one hand, if $\k =S-1$ or $S$, then $B_f$ is nonzero and $f^{(\k)}$ must be consistently estimated to attain the best rate. Indeed, more is required. From \Eqref{suppeqn:bias corrected}, we will need a bounded, positive $\rho$ to equalize the bias terms. This (again) highlights the advantage of robust bias correction, as the classical procedure would enforce $\rho \to 0$, and thus underperform. On the other hand, $\rho \to 0$ will be required if $\k>S$ because (from the final case of \eqref{suppeqn:bias corrected}) we require $\rho^{\k-S} = O(\h^\varsigma)$ to attain the same rate as undersmoothing. Note that we can accommodate $\b \not\to 0$ (but bounded). Interestingly, $B_f$ is identically zero and $\hat{B}_f$ merely adds noise to the problem, but this noise is fully accounted for by the robust standard errors, and hence does not affect the rates of coverage error (though the constants of course change). The $\hat{f}^{(\k)}$ in $\hat{B}_f$ is \emph{inconsistent} ($f^{(\k)}$ does not exist), but the nonvanishing bias of $\hat{f}^{(\k)}$ is dominated by $\h^\k$.

This discussion is summarized by the following result:
\begin{corollary}
	\label{thm:known S}
	Let the conditions of Theorem \ref{thm:Edgeworth density} hold.
	\begin{enumerate}[label=(\alph*)]

		\item If $\k > S$, then 
			\begin{align*}
				\P[f \in \ius] = 1 - \alpha   &   +   \frac{1}{n \h}  \frac{\phi(z_{\frac{\alpha}{2}})}{f} q_1(K) \;\{1+o(1)\}   +   O\left( n \h^{1+2S+2\varsigma} +  \h^{S+\varsigma} \right).
			\end{align*}
		\item If $\k \geq S-1$, then
			\begin{align*}
				\P[f \in \irbc]  = 1 - \alpha   &  +   \frac{1}{n \h}  \frac{\phi(z_{\frac{\alpha}{2}})}{f} q_1(M) \;\{1+o(1)\}    		\\
								&  +   O\left( n\h ( \h^{S + \varsigma} \vee \h^\k \b^{S - \k+ \varsigma \mathbbm{1}\{\k\leq S\} })^2 + ( \h^{S + \varsigma} \vee \h^\k \b^{S - \k+ \varsigma \mathbbm{1}\{\k\leq S\} }) \right).
			\end{align*}

	\end{enumerate}
\end{corollary}

%%%%%%%%%%%%%%%%%%%%%%%%%%%%%%%%%%%
\subsection{Multivariate Densities and Derivative Estimation}
	\label{supp:multivariate}

We now briefly present state analogues of our results, both for distributional convergence and Edgeworth expansions, that cover multivariate data and derivative estimation. The conceptual discussion and implications are similar to those in the main text, once adjusted notationally to the present setting, and are hence omitted.

For a nonnegative integral $d$-vector $q$ we adopt the notation that: (i) $[q] = q_1 + \cdots + q_d$, (ii) $g^{(q)}(x) = \partial^{[q]} g(x)/(\partial^{q_1} x_1 \cdots \partial^{q_d} x_d)$, (iii) $k! = q_1!\cdots q_d!$, and (iv) $\sum_{[q] = Q}$ for some integer $Q \geq 0$ denotes the sum over all indexes in the set $\{q : [q] = Q\}$.

The parameter of interest is $f^{(q)}(x)$, for $x \in \mathbb{R}^d$ and $[q] \leq S$. The estimator is
\[\hat{f}^{(q)}(x) = \frac{1}{n \h^{d + [q]} } \sumi K^{(q)}\left( X_{\h,i} \right).\]
Note that here, and below for bias correction, we use a constant, diagonal bandwidth matrix, e.g. $h \times I_d$. This is for simplicity and comparability, and could be relaxed at notational expense.

The bias, for a given kernel of order $\k \leq S - [q]$ (we restrict attention to the case where $S$ is large enough), is
\[\h^\k  \sum_{k: [k + q] = \k} \mu_{K,k}  f^{(q+k)}(x) + o(\h^\k ),\]
exactly mirroring \Eqref{suppeqn:bias density}, where now $\mu_{K,k}$ represents a $d$-dimensional integral. Bias estimation is straightforward, relying on estimates $\hat{f}^{(q+k)}(x)$, for all $[k] = \k - [q]$. The form of $\hat{f}^{(q)}_2(x) = \hat{f}^{(q)}(x) - \hat{B}_{f^{(q)}}(x)$ is now given by
\[ \hat{f}^{(q)}_2(x)  = \frac{1}{n \h^{d + [q]}} \sumi M_{(q)}\left(  X_{\h,i}\right) \quad \text{where} \quad M_{(q)}(u)=  K^{(q)}(u)   -  \left(\rho \right)^{d + [q] + \k}   \sum_{[k] = \k } \mu_{K,k}  L^{(q+k)}(u), \]
exactly analogous to \Eqref{suppeqn:kernel M}.

With these changes in notation out of the way, we can (re-)define the generic framework for both estimators exactly as above. Dropping the point of evaluation $x$, for $v \in \{1,2\}$, define the estimator as
\[\hat{f}^{(q)}_v = \frac{1}{n \h^{d + [q]} } \sumi N_v \left( X_{\h,i} \right),	 \quad\qquad \text{where} \quad\qquad	 N_1(u) =  K^{(q)}(u) \text{ and } N_2(u) = M_{(q)}(u);\]
the variance
\[\sigma^2_v \defsym n \h^{d + [q]} \V[\hat{f}^{(q)}_v] = \frac{1}{\h^d} \left\{ \E \left[ N_v \left( X_{\h,i} \right)^2 \right] - \E \left[ N_v \left( X_{\h,i} \right) \right]^2   \right\}\]
and its estimator as
\[\shat^2_v  = \frac{1}{\h^d} \left\{ \frac{1}{n}\sumi \left[ N_v \left( X_{\h,i} \right)^2 \right] - \left[ \frac{1}{n}\sumi N_v \left( X_{\h,i} \right) \right]^2 \right\};\]
and the $t$-statistics, for $1 \leq w \leq v\leq 2$, as, 
\[T_{v,w} \defsym \frac{ \sqrt{n \h^{d + 2[q]} } \left(\hat{f}^{(q)}_v - f^{(q)} \right) }{ \shat_w }.\]
As before, $\tus = T_{1,1}$, $\tbc = T_{2,1}$, and $\trbc = T_{2,2}$.

The scaled bias $\eta_v$ has the same general definition as well: the bias of the numerator of the $T_{v,w}$. In this case, given by
\[\eta_v = \sqrt{n \h^{d + 2[q]} } \left( \E \left[\hat{f}^{(q)}_v\right] - f^{(q)}(x) \right). \]
The asymptotic order of $\eta_v$ for different settings can be obtained straightforwardly via the obvious multivariate extensions of Equation \eqref{suppeqn:bias corrected} and the corresponding conclusion of Lemma \ref{lem:bias density}.

First-order convergence is now given by the following result. the proof of which is standard.
\begin{lemma}
	Suppose appropriate multivariate versions of Assumptions \ref{supp:dgp density} and \ref{supp:kernel density} hold, $n \h^{d + 2[q]} \to \infty$, $\eta_v \to 0$, and if $v=2$, $\rho \to 0 + \bar{\rho} \mathbbm{1}\{v=w\}$. Then $T_{v,w}  \to_d \N(0,1)$.
\end{lemma}

For the Edgeworth expansion, redefine
	\[\nu_{v,w}(j,k,p) = \frac{1}{\h^{d + [q]\mathbbm{1}\{j + pk=1\}} } \E\left[ \left(N_v(u_i) - \E[N_v(u_i)] \right)^j \left(N_w(u_i)^p - \E[N_v(u_i)^p] \right)^k     \right],\]
where $u_i = (x - X_i ) / \h$. The polynomials $p_{v,w}^{(k)}(z)$ and $q_{v,w}^{(k)}(z)$ are as given above, but using multivariate moments. The analogue of Theorem \ref{thm:Edgeworth density} is given by the following result, which can be proven following the same steps as in Section \ref{supp:proof density}. 
\begin{theorem}
	\label{thm:multivariate}
	Suppose appropriate multivariate versions of Assumptions \ref{supp:dgp density}, \ref{supp:kernel density}, and \ref{supp:Cramer density} hold, $n \h^{d + 2[q]}/\log(n) \to \infty$, $\eta_v \to 0$, and if $v=2$, $\rho \to 0 + \bar{\rho} \mathbbm{1}\{v=w\}$. Then for
	\begin{align*}
		\hspace{-2em}  F_{v,w}(z) & = \Phi(z)  +   \frac{1}{\sqrt{n \h^d}} p^{(1)}_{v,w}(z)   +   \sqrt{\frac{\h^{d + 2[q]}}{n}} p^{(2)}_{v,w}(z)   +   \eta_v p^{(3)}_{v,w}(z)  +   \frac{1}{n \h^d} q^{(1)}_{v,w}(z)  + \eta_v^2 q^{(2)}_{v,w}(z) + \frac{\eta_v}{\sqrt{n \h^d}} q^{(3)}_{v,w}(z)		\\
			& \qquad +  \mathbbm{1}\{v \!\neq\! w\} \rho^{d + \k + [q]}(\Omega_1 + \rho^{\k + [q]} \Omega_2)\frac{\phi(z)}{2}z ,
	\end{align*}
	we have
	\[ \sup_{z \in \mathbb{R}} \left| \P[T_{v,w} < z ] - F_{v,w}(z)  \right|  = o\left( ( (n \h^d)^{-1/2} + \eta_v)^2  +   \mathbbm{1}\{v \!\neq\! w\} \rho^{d + 2(\k+[q])} \right).\]
\end{theorem}

The same conclusions reached in the main text continue to hold for multivariate and/or derivative estimation, both in terms of comparing undersmoothing, bias correction, and robust bias correction, as well as for inference-optimal bandwidth choices. In particular, it is straightforward that the MSE optimal bandwidth in general has the rate $n^{- 1 / (d + 2 \k + 2[q])}$, whereas the coverage error optimal choice is of order $n^{- 1 / (d +  \k + [q])}$. Note that these two fit the same patter as in the univariate, level case, with $\k + [q]$ in place of $\k$ and $d$ in place of one. One intuitive reason for the similarity is that the number of derivatives in question does not impact that variance or higher order moment terms of the expansion, \emph{once the scaling is accounted for}. That is, for all averages beyond the first, for example of the kernel squared, $\sqrt{n\h^d}$ can be thought of as the effective sample size, since that is the multiplier which stabilizes averages.

%%%%%%%%%%%%%%%%%%%%%%%%%%%%%%%%%%%
%%%%%%%%%%%%%%%%%%%%%%%%%%%%%%%%%%%
\section{Proof of Main Result}
	\label{supp:proof density}

Throughout $C$ shall be a generic constant that may take different values in different uses. If more than one constant is needed, $C_1$, $C_2$, \ldots, will be used. It will cause no confusion (as the notations never occur in the same place), but in the course of proofs we will frequently write $s=\sqrt{n\h}$, which overlaps with the order of the kernel $L$.

The first step is to write $T_{v,w}$ as a smooth function of sums of i.i.d.\! random variables plus a remainder term that is shown to be of higher order. In addition to the notation above, define
\[\gamma_{v,p} = \h^{-1} \E\left[N_v\left(X_{\h,i}\right)^p\right]		\qquad \text{ and } \qquad 	       \Delta_{v,j} = \frac{1}{s} \sumi \left\{ N_v\left(X_{\h,i}\right)^j - \E\left[N_v\left(X_{\h,i}\right)^j\right]\right\}.\]
With this notation $\hat{f}_v - \E[\hat{f}_v] = s^{-1} \Delta_{v,1}$, $\sigma^2_w = \E[\Delta_{w,1}^2] = \gamma_{w,2} - \h \gamma_{w,1}^2$ and 
\begin{equation}
	\label{eqn:variance difference}
	\shat^2_w - \sigma^2_w = s^{-1} \Delta_{w,2}  - \h 2 \gamma_{w,1} s^{-1} \Delta_{w,1} - \h s^{-2} \Delta_{w,1}^2.
\end{equation}

By a change of variables
\[\gamma_{v,p}  =   \h^{-1} \int   N_v\left(X_{\h,i}\right)^p f(X_i) d X_i  = \int   N_v(u)^{p} f(x - u \h) d u 	 =  O(1).\]
Further, by construction $\E[\Delta_{w,j}] = 0$ and 
\begin{align*}
	\V \left[\Delta_{w,j} \right]  & =\h^{-1} \E \left[N_v\left(X_{\h,i}\right)^{2j} \right]   - \h^{-1} \E\left[N_v\left(X_{\h,i}\right)^j\right]^2 		\\
	& \leq  \h^{-1} \E \left[N_v\left(X_{\h,i}\right)^{2j} \right]			\\
	& = \gamma_{v,2j}	 =  O(1).
\end{align*}

Returning to \Eqref{eqn:variance difference} and applying Markov's inequality, we find that $\h s^{-2} \Delta_{w,1}^2 = n^{-1} \Delta_{w,1}^2 = O_p(n^{-1})$ and $\shat^2_w - \sigma^2_w = s^{-1} O_p(1)  - \h  O(1) s^{-1} O_p(1) - \h s^{-2} O_p(1) = O_p(s^{-1})$, whence $\left| \shat^2_w - \sigma^2_w \right|^2 = O_p(s^{-2})$. Using these results preceded by a Taylor expansion, we have
\begin{align*}
	\left( \frac{\shat^2_w}{\sigma^2_w}  \right)^{-1/2} = \left( 1 + \frac{ \shat^2_w - \sigma^2_w}{\sigma^2_w} \right)^{-1/2} &  = 1 - \frac{1}{2} \frac{\shat^2_w - \sigma^2_w}{\sigma^2_w}  + \frac{3}{8} \frac{(\shat^2_w - \sigma^2_w)^2}{\sigma^4_w}  + o_p((\shat^2_w - \sigma^2_w)^2)			\\
	& = 1 - \frac{1}{2 \sigma^2_w} \left( s^{-1} \Delta_{w,2}  - \h 2 \gamma_{w,1} s^{-1} \Delta_{w,1} \right)  +  O_p(n^{-1} + s^{-2}).
\end{align*}

Combining this result with the fact that
\[T_{v,w} = \frac{\Delta_{v,1} + \eta_v}{\shat_w} = \frac{\Delta_{v,1}}{\shat_w}  +  \frac{\eta_v}{\sigma_w} \left( \frac{\shat^2_w}{\sigma^2_w}  \right)^{-1/2},\]
we have
\begin{equation}
	\label{supp:linearization}
	\P[T_{v,w} < z]  = \P\left[ \tilde{T}_{v,w}   -   R_{v,w}   < z  -  \frac{\eta_v}{\sigma_w}  \right],
\end{equation}
where 
\[\tilde{T}_{v,w}  =  \frac{\Delta_{v,1}}{\shat_w}   -    \frac{\eta_v}{2 \sigma^3_w} \left( s^{-1} \Delta_{w,2}  - \h 2 \gamma_{w,1} s^{-1} \Delta_{w,1} \right) \]
and is a smooth function of sums of i.i.d.\! random variables and the remainder term is
\[R_{v,w} = \frac{\eta_v}{\sigma_w} \left( \h s^{-2} \frac{\Delta_{w,1}^2}{2 \sigma^2_w}  +  \frac{3}{8}  \frac{(\shat^2_w - \sigma^2_w)^2}{\sigma^4_w}  + o_p((\shat^2_w - \sigma^2_w)^2)  \right). \]

Next we apply the delta method, see \citet[Chapter 2.7]{Hall1992_book} or \citet[Lemma 5(a)]{Andrews2002_Ecma}. It will be true that
\begin{equation}
	\label{supp:delta method}
	\P[T_{v,w} < z]   = \P\left[\tilde{T}_{v,w}  < z  -  \frac{\eta_v}{\sigma_w} \right]  + o(s^{-2})
\end{equation}
if it can be shown that $s^2 \P[|R_{v,w}| > \e^2  s^{-2} \log(s)^{-1} ] = o(1)$.\footnote{Here, $s^{-2} \log(s)^{-1}$ may be replaced with any sequence that is $o(s^{-2} + \eta_v^2 + s^{-1} \eta_v)$.} This can be demonstrated by applying Bernstein's inequality to each piece of $R_{v,w}$, as the kernels $K$ and $L$, and their derivatives, are bounded. 

To apply this inequality to the first term of $R_{v,w}$, note that $| N_w ((x - X_i)/\h) | \leq C_1$ and that $\V[N_w ((x - X_i)/\h)] \leq C_2 \h$, for different constants, and so for $\e > 0$ we have
\begin{align*}
	s^2 \P & \left[  \frac{\eta_v}{\sigma_w}  \h s^{-2} \frac{\Delta_{w,1}^2}{2 \sigma^2_w} > \e^2 s^{-2} \log(s)^{-1} \right] 		\\
	& = s^2 \P \left[   \left| \sumi \left\{ N_w\left(X_{\h,i}\right) - \E\left[N_w\left(X_{\h,i}\right) \right]\right\} \right| > \e s^{-1} \log(s)^{-1/2} \left( \frac{2 \sigma_w^3 n s^2}{\eta_v} \right)^{1/2} \right]  		\\
	& = s^2 \P \left[   \left| \sumi \left\{ N_w\left(X_{\h,i}\right) - \E\left[N_w\left(X_{\h,i}\right) \right]\right\} \right| > \e  \left( \frac{2 \sigma_w^3 n}{\eta_v \log(s)} \right)^{1/2} \right]  		\\
	& \leq 2 s^2  \exp \left\{ - \frac{1}{2}   \frac{ \e^2 2 \sigma_w^3 n \eta_v^{-1} \log(s)^{-1} }{ C_2 n \h + \frac{1}{3} \e C_1 \sqrt{2 \sigma_w^3 n / [ \eta_v \log(s) ] } }  \right\}   		\\
	& \leq s^2  \exp \left\{ -   C  \frac{ \e^2   \log(s)^{-1} }{ \eta \h +  \e \sqrt{ \eta_v  / [ n \log(s) ] } }  \right\}   		\\
	& \leq   \exp \left\{C_1 \log(s) \left[ 1  -   C_2  \frac{ \e^2  }{ \eta \h \log(s)^2 +  \e \sqrt{ \eta_v \log(s)^3  / n  ] } }  \right] \right\},
\end{align*}
which tends to zero because $\eta_v \to 0$ as $n \to \infty$ is assumed. To see why, note first that the second term of the denominator automatically vanishes, as $\eta_v \to 0$ and $\log(s)^3/n \to 0$. Second, suppose $\eta_v^2 \asymp n \h^\omega$ (for example, if $\eus \asymp s \h^\k$, then $\omega = 1 + 2\k$) and the first term diverges, it must be that $\h$ is at least as large (in order) as 
\[\left( \frac{1}{n \log(s)^4} \right)^{1/(2 + \omega)},\]
which makes the requirement that $\eta_v \to 0$ equivalent to 
\[\eta_v^2 \asymp n \h^\omega = n^{1 - \omega/(2 + \omega)} \log(s)^{ - 4 \omega/(2 + \omega)} \to 0,\]
which is impossible. The remaining terms of $R_{v,w}$, characterized using \Eqref{eqn:variance difference}, are handled in exactly the same way. This establishes \Eqref{supp:delta method}.

Next, the proofs of \cite[Chapters 4.4 and 5.5]{Hall1992_book} show that $\tilde{T}_{v,w}$ has an Edgeworth expansion valid through $o(s^{-2} + s^{-1}\eta_v + \eta_v^2)$. Thus, for a smooth function $G(z)$ we can write $\P[\tilde{T}_{v,w}  < z ] = G(z) + o(s^{-2} + s^{-1}\eta_v + \eta_v^2)$. Therefore 
\begin{equation}
	\label{supp:taylor}
	\P\left[\tilde{T}_{v,w}  < z  -  \frac{\eta_v}{\sigma_w} \right] =  \P\left[\tilde{T}_{v,w}  < z  \right]   -  \frac{\eta_v}{\sigma_w} G^{(1)}(z) +  o(s^{-2} + s^{-1}\eta_v + \eta_v^2).
\end{equation}

The final result now follows by combining Equations \eqref{supp:linearization}, \eqref{supp:delta method}, and \eqref{supp:taylor} with the terms of the expansion computed below.\qed

%%%%%%%%%%%%%%%%%%%%%%%%%%%%%%%%%%%
%%%%%%%%%%%%%%%%%%%%%%%%%%%%%%%%%%%
\subsection{Computing the Terms of the Expansion}

Identifying the terms of the expansion is a matter of straightforward, if tedious, calculation. The first four cumulants of $T_{v,w}$ must be calculated, which are functions of the first four moments. In what follows, we give a short summary. Note well that we always discard higher-order terms for brevity, and to save notation we will write $\oeq$ to stand in for ``equal up to $o((n\h)^{-1} + (n\h)^{-1/2}\eta_v + \eta_v^2  + \mathbbm{1}\{v \!\neq\! w\} \rho^{1+2\k} )$''.

Referring to the Taylor expansion above, for the purpose of computing moments and cumulants, we can use
\[T_{v,w} \approx \left( \frac{\Delta_{v,1}}{\sigma_w} +   \frac{\eta_v}{\sigma_w} \right) \left(1 - \frac{ s^{-1} \Delta_{w,2}}{2 \sigma_w}  + \frac{\h  \gamma_{w,1} s^{-1} \Delta_{w,1}}{\sigma_w} + \frac{3}{8} \frac{s^{-2} \Delta_{w,2}^2}{\sigma^2_w} \right).\]
Moments of the two sides agree up to the requisite order. Straightforward moment calculations then give
\begin{align*}
	\E[T_{v,w}] & \oeq \frac{s^{-1} \E[\Delta_{v,1} \Delta_{w,2}]}{2 \sigma_w^3}  +  \frac{\h s^{-1} \gamma_{w,1} \E[\Delta_{v,1} \Delta_{w,1}]}{ \sigma_w^3}  +  \frac{3s^{-2} \E[\Delta_{v,1} \Delta_{w,2}^2]}{8\sigma_w^5}   +  \frac{\eta_v}{\sigma_w}  +  \frac{3 s^{-2} \eta_v \E[\Delta_{w,2}^2]}{8 \sigma_w^5}		\\
	& \oeq - s^{-1} \frac{ \nu_{v,w}(1,1,2)}{2 \sigma_w^3}  +  \frac{\h s^{-1} \gamma_{w,1} \nu_{v,w}(1,1,1)}{ \sigma_w^3}   +  \frac{\eta_v}{\sigma_w},
\end{align*}
\begin{align*}
	\E[T_{v,w}^2] & \oeq  \frac{\E[\Delta_{v,1}^2]}{\sigma_w^2}  +  s^{-2} \frac{\E[\Delta_{v,1}^2\Delta_{w,2}^2]}{\sigma_w^6}  +  s^{-1} \frac{\E[\Delta_{v,1}^2\Delta_{w,2}]}{\sigma_w^4}    +  2 \h s^{-1}  \frac{\gamma_{w,1} \E[\Delta_{v,1}^2\Delta_{w,1}]}{\sigma_w^2} 		\\
	& \qquad  -  \eta_v s^{-1} \frac{2 \E[\Delta_{v,1} \Delta_{w,2}]}{\sigma_w^4}   +  \eta_v \h s^{-1}  \frac{ 4 \gamma_{w,1} \E[\Delta_{v,1}\Delta_{w,1}]}{\sigma_w^2}  +  \frac{\eta_v^2}{\sigma_w^2}		\\
	& \oeq \frac{\sigma_v^2}{\sigma_w^2}  +  s^{-2} \frac{\sigma_v^2 \nu_{v,w}(0,2,2)}{\sigma_w^6}   +  s^{-2} \frac{2 \nu_{v,w}(1,1,2)^2}{\sigma_w^6}  -  s^{-2} \frac{ \nu_{v,w}(2,1,2)^2}{\sigma_w^2}   -  \eta_v s^{-1} \frac{ 2 \nu_{v,w}(1,1,2)}{\sigma_w^2}     +  \frac{\eta_v^2}{\sigma_w^2},
\end{align*}
\begin{align*}
	\E[T_{v,w}^3] & \oeq  \frac{\E[\Delta_{v,1}^3]}{\sigma_w^3 }  -  3 s^{-1} \frac{\E[\Delta_{v,1}^3\Delta_{w,2}]}{2 \sigma_w^5}  +  3 \h s^{-1}  \frac{\gamma_{w,1} \E[\Delta_{v,1}^3\Delta_{w,1}]}{\sigma_w^5}  +  \eta_v \frac{3 \E[\Delta_{v,1}^2]}{\sigma_w^3}  -  \eta_v s^{-1} \frac{9 \E[\Delta_{v,1}^2 \Delta_{w,2}]}{2\sigma_w^5}		\\
	& \oeq s^{-1} \frac{\nu_{v}(3)}{\sigma_w^3}  -  s^{-1} \frac{ 9 \nu_{v,w}(1,1,2) \sigma_v^2}{2 \sigma_w^5}  +  \h s^{-1} \frac{ 9 \gamma_{w,1} \nu_{v,w}(1,1,1)}{\sigma_w^5}  +  \eta_v \frac{3 \sigma_v^2}{\sigma_w^3},
\end{align*}
and,
\begin{align*}
	\E[T_{v,w}^4] & \oeq \frac{\E[\Delta_{v,1}^4]}{\sigma_w^4}  -  s^{-1} \frac{2 \E[\Delta_{v,1}^4\Delta_{w,2}]}{\sigma_w^6}    +  4 \h s^{-1}  \frac{\gamma_{w,1} \E[\Delta_{v,1}^4 \Delta_{w,1}]}{\sigma_w^6}    +  s^{-2} \frac{3 \E[\Delta_{v,1}^4\Delta_{w,1}^2]}{\sigma_w^8}		\\
	& \qquad +    \eta_v \frac{4 \E[\Delta_{v,1}^3]}{\sigma_w^4}  -  \eta_v s^{-1} \frac{8 \E[\Delta_{v,1}^3 \Delta_{w,2}]}{\sigma_w^6}  +  \eta_v^2 \frac{6 \E[ \Delta_{v,1}^2]}{\sigma_w^4}		\\
	& \oeq s^{-2} \frac{\nu_v(4)}{\sigma_w^4}  +  3\frac{\sigma_v^4}{\sigma_w^4}  -  s^{-2} \frac{ 8 \nu_v(3) \nu_{v,w}(1,1,2) + 12 \sigma_v^2 \nu_{v,w}(2,1,2)}{\sigma_w^6}  +  s^{-2} \frac{ 9 \sigma_v^4 \nu_{v,w}(0,2,2) }{\sigma_w^8}		\\
	& \qquad     +  s^{-2} \frac{ 36 \sigma_v^2 \nu_{v,w}(1,1,2)^2 }{\sigma_w^8}  +  \eta_v s^{-1} \frac{ 4 \nu_v(3)}{\sigma_w^4}  -  \eta_v s^{-1} \frac{ 24 \sigma_v^2 \nu_{v,w}(1,1,2)}{\sigma_w^6}  +  \eta_v^2 \frac{ 6 \sigma_v^2}{\sigma_w^2}.
\end{align*}

The expansion now follows, formally, from the following steps. First, combining the above moments into cumulants. Second, these cumulants may be simplified using that
\[\frac{\sigma_v^2}{\sigma_w^2} = 1 + \mathbbm{1}(w\!\neq\! v) \left( \rho^{1+ \k} \Omega_1 + \rho^{1 + 2\k} \Omega_2\right) \]
and in all cases present
\begin{equation}
	\label{eqn:simplifying nu}
	\nu_{v,w}(i,j,p) = f \vartheta_{N_v,i + jp} + o(1).
\end{equation}
The second relation is readily proven for $v=w$, as $\nu_{v,v}(i,j,p) = \E[N_v(X_{\h,i})^{i + jp}] + O(\h)$, where the remainder represents products of expectations. In the case for $v \neq w$, we find $\nu_{2,1}(i,j,p) = f \vartheta_{N_1,i + jp}  +  O( \rho^{1+\k} + \h)$, and in this case $\rho \to 0$ is assumed. For any term of a cumulant with a rate of $(n\h)^{-1}$, $(n\h)^{-1/2}\eta_v$, $\eta_v^2$, or $\rho^{1+2\k}$ (i.e., the extent of the expansion), these simplifications may be inserted as the remainder will be negligible. Note that this is exactly why the polynomials $p_{v,w}^{(k)}$ do not simplify, while the $q_{v,w}^{(k)}$ do. Third, with the cumulants in hand, the terms of the expansion are determined as described by e.g., \citet[Chapter 2]{Hall1992_book}.

Finally, for traditional bias correction, there are additional terms in the expansion (see discussion in the main text) representing the covariance of $\hat{f}$ and $\hat{B}_f$ (denoted by $\Omega_1$) and the variance of $\hat{B}_f$ ($\Omega_2$). We now state their precise forms. These arise from the mismatch between the variance of the numerator of $\tbc$ and the standardization used, $\sus^2$, that is $\srbc^2/\sus^2$ is given by
\[  \frac{ n\h \V[\hat{f} - \hat{B}_f] }{ n\h \V[\hat{f}] } = \frac{ n\h \V[\hat{f}]  - 2 n\h \mathbb{C}[\hat{f},\hat{B}_f]  + n\h \V[\hat{B}_f] }{ n\h \V[\hat{f}] }  = 1   -  2 \frac{ n\h \mathbb{C}[\hat{f},\hat{B}_f] }{ n\h \V[\hat{f}] }  + \frac{ n\h \V[\hat{B}_f] }{ n\h \V[\hat{f}] } .\]
This makes clear that $\Omega_1$ and $\Omega_2$ are the constant portions of the last two terms. First, for $\Omega_1$,
\begin{align*}
	\mathbb{C}[\hat{f},\hat{B}_f] & = \E\left[ \left( \frac{1}{n\h} \sumi K\left( X_{\h,i} \right) \right) \left( \h^\k  \mu_{K,\k} \frac{1}{n \b^{1 + \k} } \sumi L^{(\k)}\left( X_{\b,i} \right) \right)\right]  		\\
	& =  \h^\k  \mu_{K,\k} \frac{1}{n \b^{1 + \k} }  \biggl\{  \E\left[ \h^{-1} K\left( X_{\h,i} \right) L^{(\k)}\left( X_{\b,i} \right)  \right]    		\\
	& \qquad \qquad \qquad\qquad - \b \E\left[ \h^{-1} K\left( X_{\h,i} \right) \right] \E \left[ \b^{-1} L^{(\k)}\left( X_{\b,i} \right) \right]  \biggr\}  		\\
	& =   \frac{\rho^\k  \mu_{K,\k}}{n \b }  \left\{  \int f(x - u\h )K(u) L^{(\k)}(u \rho) du  - \b \int f(x - u\h) K(u) du \int f(x - u\b) L^{(\k)}(u) du  \right\}.
\end{align*}
Therefore
	\[ -  2 \frac{ n\h \mathbb{C}[\hat{f},\hat{B}_f] }{ n\h \V[\hat{f}] } = \rho^{1+\k} \Omega_1, \]
where
	\[\Omega_1 = - 2 \frac{\mu_{K,\k}}{\nu_1(2)}  \left\{  \int f(x - u\h )K(u) L^{(\k)}(u \rho) du  - \b \int f(x - u\h) K(u) du \int f(x - u\b) L^{(\k)}(u) du  \right\}.\]
Note $\nu_1(2) = \sus^2$. If we did not include $\Omega_2$ in the Edgeworth expansion, i.e. we stopped at order $\rho^{1+\k}$, then we could capture only the leading terms of $\Omega_1$, as follows, using that kernel integrates to 1 and $\rho \to 0$,
\begin{align*}
	\Omega_1 & = - 2 \frac{\mu_{K,\k}}{\nu_1(2)}  \left\{  \int f(x - u\h )K(u) L^{(\k)}(u \rho) du  - \b \int f(x - u\h) K(u) du \int f(x - u\b) L^{(\k)}(u) du  \right\}  		\\
	& = - 2 \frac{\mu_{K,\k}}{f(x) \vartheta_{K,2}^2 + O(\h)}  \left\{  f(x) L^{(\k)}(0)[1 + O(\h + h\rho)] - \b f(x)^2 \int L^{(\k)}(u) du[1+O(\b + \h)]  \right\}  		\\
	& \to - 2 \mu_{K,\k} \vartheta_{K,2}^{-2}   L^{(\k)}(0).
\end{align*}
Note that this matches the term \citet{Hall1992_AoS_density} calls $w_2$. We do not do this, for completeness. There are no other terms of up to order $\rho^{1+2\k}$, so capturing the full contribution of $\sigma_2^2 / \sigma_1^2 - 1 = \srbc^2 / \sus^2 - 1$ is natural and informative.

Turning to $\Omega_2$, using the calculations in Section \ref{supp:bias density} (recall $\tilde{\k} = \k \vee S$), we find that 
\begin{align*}
	\V[\hat{B}_f] & = \frac{\h^{2\k}}{n } \mu_{K,\k}^2 \left\{ \frac{1}{\b^{1 + 2\k}} \E\left[\b^{-1} L^{(\k)} \left( X_{\b,i} \right)^2 \right]   -  \left( \frac{1}{\b^{1+\k}}  \E\left[ L^{(\k)} \left( X_{\b,i} \right) \right] \right)^2\right\}     		\\
	& = \frac{\rho^{2\k}  \mu_{K,\k}^2}{n \b } \left\{ \int f(x - u\b) L^{(\k)}(u)^2 du   -   \b^{1 + 2\tilde{\k}} \left( \int L^{(\k-\tilde{\k})}(u) f^{(\tilde{\k})}(x - u \b) du \right)^2 \right\},
\end{align*}
and hence
\[\frac{ n\h \V[\hat{B}_f] }{ n\h \V[\hat{f}] } = \rho^{1 + 2\k} \Omega_2 \quad \text{ where }\quad \Omega_2 = \frac{  \mu_{K,\k}^2}{ \nu_1(2) } \left\{ \int f(x - u\b) L^{(\k)}(u)^2 du   -   \b^{1 + 2\tilde{\k}} \left( \int L^{(\k-\tilde{\k})}(u) f^{(\tilde{\k})}(x - u \b) du \right)^2 \right\}.\]
The final piece will be $\b^{1 + 2S} f^{(\k)}(x)^2 [ 1+o(1)]$ if $\k \leq S$. Substituting this is permitted because $\rho^{1+2\k}$ is the limit of the expansion, though it is not necessary to do, because this term is always higher order. Fully simplifying would yield
\[\Omega_2 = \mu_{K,\k}^2 \vartheta_{K,2}^{-2} \vartheta_{L^{(\k)},2},\]
which can be used in Theorem \ref{thm:Edgeworth density}.

%%%%%%%%%%%%%%%%%%%%%%%%%%%%%%%%%%%
%%%%%%%%%%%%%%%%%%%%%%%%%%%%%%%%%%%
\section{Complete Simulation Results}
	\label{supp:simuls density}

To illustrate the gains from robust bias correction we conduct a Monte Carlo study to compare undersmoothing, traditional bias correction, and robust bias correction in terms coverage accuracy and interval length using several data-driven procedures to select the bandwidth.  We generate $n=500$ observations from a density $f$ given by:

\begin{itemize}
\item[] Model 1 (Gaussian Density): $x\backsim\N(0,1)$
\item[] Model 2 (Skewed Unimodal Density): $x\backsim \frac{1}{5} \N(0,1) + \frac{1}{5} \N\left(\frac{1}{2},\left(\frac{2}{3}\right)^2\right) + \frac{3}{5} \N\left(\frac{13}{12},\left(\frac{5}{9}\right)^2\right)$
\item[] Model 3 (Bimodal Density): $x\backsim \frac{1}{2} \N\left(-1,\left(\frac{2}{3}\right)^2\right) + \frac{1}{2} \N\left(1,\left(\frac{2}{3}\right)^2\right)$
\item[] Model 4 (Asymmetric Bimodal Density): $x\backsim \frac{3}{4} \N\left(0,1\right) + \frac{1}{4} \N\left(\frac{3}{2},\left(\frac{1}{3}\right)^2\right)$

\end{itemize}

We evaluate the density at $x=\{-2,-1,0,1,2\}$. These models were previously analyzed in \citet{Marron-Wand1992_AoS} and they are plotted in Figure \ref{fig:dgp_den}.
In this simulation study we compare the performance of the confidence intervals defined by $\tus$, $\tbc$, and $\trbc$.
For $\tus$, we take $K$ to be the Epanechnikov kernel, while bias correction uses the Epanechnikov and MSE-optimal kernels for $K$ and $L^{(2)}$, respectively. 
The bandwidth $\h$ is chosen in three different ways:
\begin{enumerate}[label=\it(\roman*)]
	\item population MSE-optimal choice $h_\MSE$;
	\item estimated ROT optimal coverage error rate $\hat{h}_\ROT$. 
	\item estimated DPI optimal coverage error rate $\hat{h}_\PI$. 
\end{enumerate}
%We consider two main data-drive bandwidth selectors. First, a Silverman rule-of-thumb alternative $\hat{h}_\ROT=\hat{\sigma}2.34n^{-1/(2r+1)}$. 
%Second, the direct plug-in (DPI) for coverage error optimal bandwidth $\hat{h}_\PI = \hat{H}_\PI n^{-1/(r + 3)}$, where $\hat{H}_{\PI}$ uses $\hat{h}_\ROT$ as a pilot bandwidth to estimate $f^{(r+2)}$ consistently.
%We also include the unfeasible, population value for $h_\MSE$.

Empirical coverage and length are reported in Tables \ref{table:dentable1}--\ref{table:dentable4} (Panel A) using our two proposed data-driven bandwidth selectors, as well as the infeasible $h_\MSE$. The most obvious finding is that robust bias correction has accurate coverage for all bandwidth choices in all models. The intervals are generally longer than for undersmoothing, but neither undersmoothing nor traditional bias correction yield correct coverage outside of a few special cases (e.g., undersmoothing at the infeasible MSE-optimal bandwidth in Model 4). The DPI bandwidth selector generally results in slightly smaller bandwidths (on average). Summary statistics for the two fully data-driven bandwidths are shown in Panel B. The fact that the DPI bandwidth is slightly smaller is born out. It is also, in general, more variable.

To illustrate the robustness to tuning parameter selection, Figures \ref{fig:dengrid_ec1}--\ref{fig:dengrid_il4} show coverage and length for all four models. The dotted vertical line shows the population MSE-optimal bandwidth for reference. These figures demonstrate the delicate balance required for undersmoothing to provide correct coverage, whereas for a wide range of bandwidths robust bias correction provides correct coverage. Further, interval length is not unduly inflated for bandwidths that provide correct coverage. Recall that robust bias correction can accommodate, and will optimally employ, a larger bandwidth, yielding higher precision. Further emphasizing the point of robustness, we depart from $\rho = 1$ in Figures \ref{fig:ec3d} and \ref{fig:il3d} to show coverage and length over a grid of $h$ and $\rho$.

The simulation results for local polynomial regression reported in Section \ref{supp:simuls locpoly} below bear out these same conclusions and study these issues in more detail, in particular interval length. 

All our methods are implemented in {\sf R} and {\tt STATA} via the {\tt nprobust} package, available from \url{http://sites.google.com/site/nppackages/nprobust} (see also \url{http://cran.r-project.org/package=nprobust}). See \citet{Calonico-Cattaneo-Farrell2017_nprobust} for a complete description.

\begin{figure}
        \centering
        \caption{Density Functions}\label{fig:dgp_den}
        \begin{subfigure}[b]{0.5\textwidth}
                \includegraphics[width=1.05\textwidth]{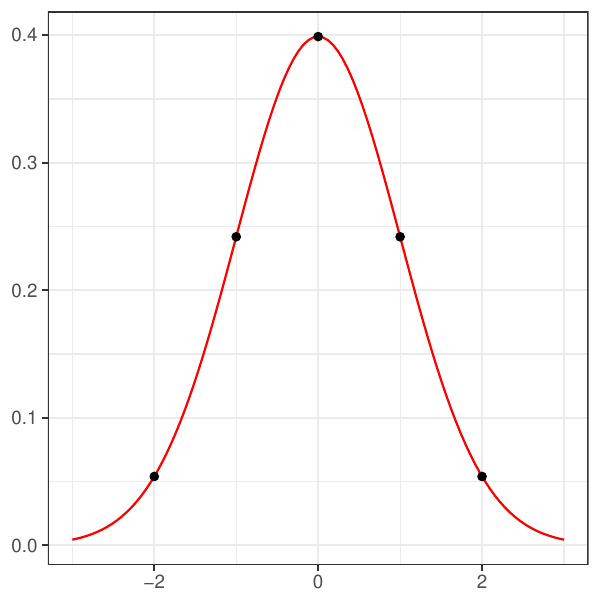}
                \caption{Model 1}
        \end{subfigure}%
        ~ %add desired spacing between images, e. g. ~, \quad, \qquad, \hfill etc.
        \begin{subfigure}[b]{0.5\textwidth}
                \includegraphics[width=1.05\textwidth]{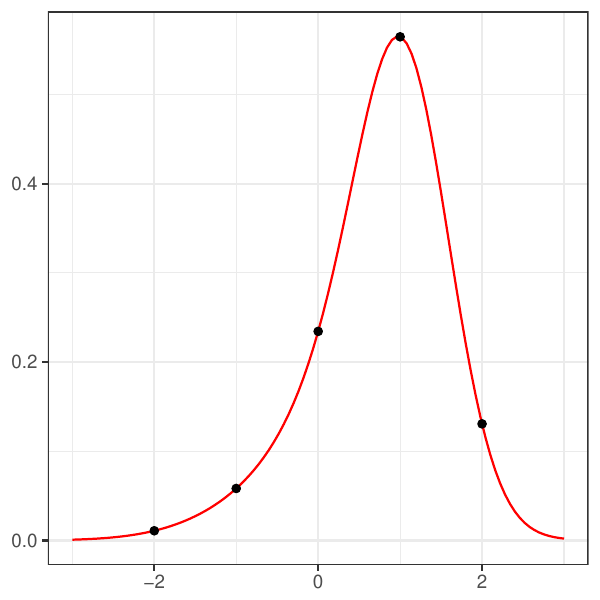}
                \caption{Model 2}
        \end{subfigure}
        
        \begin{subfigure}[b]{0.5\textwidth}
                        \includegraphics[width=1.05\textwidth]{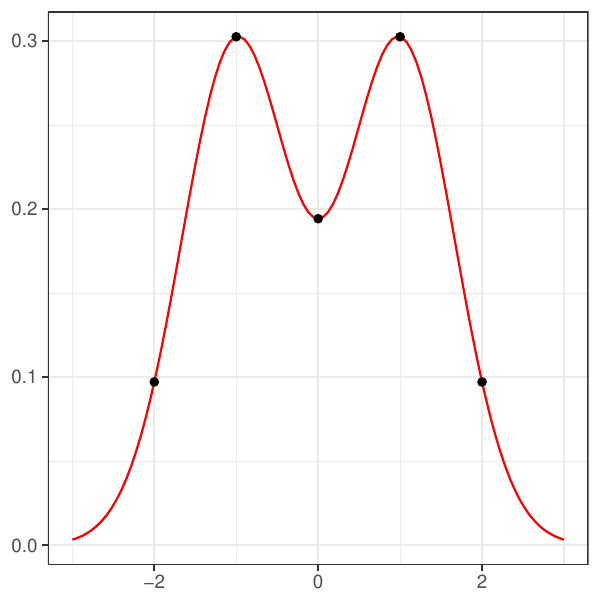}
                        \caption{Model 3}
        \end{subfigure}%
        \begin{subfigure}[b]{0.5\textwidth}
                             \includegraphics[width=1.05\textwidth]{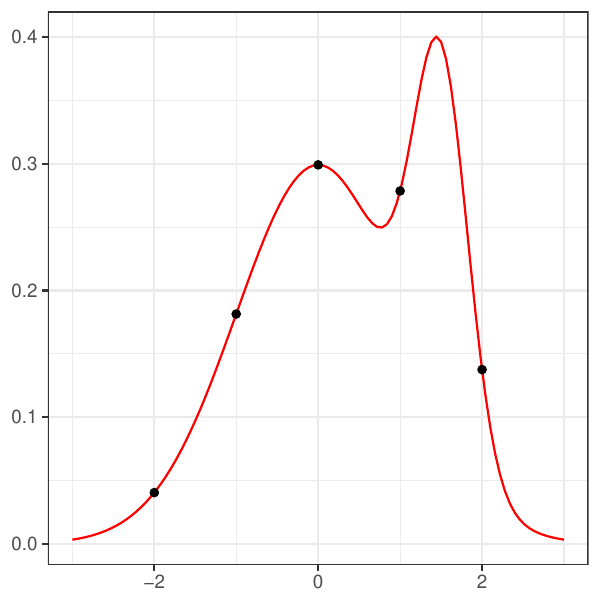}
                             \caption{Model 4}
        \end{subfigure}
\end{figure}

\newpage
\newcounter{model}
\forloop{model}{1}{\value{model} < 5}{
	\begin{table}[h]%\renewcommand{\arraystretch}{1.1}
		{\begin{center}
				\caption{Simulations Results for Model \themodel}\label{table:dentable\themodel}
				Panel A:  Empirical Coverage and Average Interval Length of 95$\%$ Confidence Intervals 
				{\input{kd_bws_m\themodel_n500.txt}}\bigskip
				
				Panel B: Summary Statistics for the Estimated Bandwidths 
				{\input{kd_bws_stats_m\themodel_n500.txt}}\bigskip
		\end{center}}
		\vspace{-.2in}\footnotesize\textbf{Notes}:\\
		\smallskip
		(i) US = Undersmoothing, BC = Bias Corrected, RBC = Robust Bias Corrected.\\
		\smallskip
		(ii) Columns under ``Bandwidth'' report the average estimated bandwidths choices, as appropriate, for bandwidth $h_n$.\\
	\end{table}
}

\newpage
\begin{landscape}

\newpage
\forloop{model}{1}{\value{model} < 5}{
\begin{figure}[h]
      \centering
       \caption{Empirical Coverage of 95$\%$ Confidence Intervals - Model \themodel }\label{fig:dengrid_ec\themodel}

             \begin{subfigure}[b]{0.5\textwidth}
          	\includegraphics[width=1.0\textwidth]{kd_grid_ec_m\themodel_c1_n500.pdf}
          	\caption{$x=-2$}
          \end{subfigure}%
          ~ %add desired spacing between images, e. g. ~, \quad, \qquad, \hfill etc.
          \begin{subfigure}[b]{0.5\textwidth}
          	\includegraphics[width=1.0\textwidth]{kd_grid_ec_m\themodel_c2_n500.pdf}
          	\caption{$x=-1$}
          \end{subfigure}%
           \begin{subfigure}[b]{0.5\textwidth}
          	\includegraphics[width=1.0\textwidth]{kd_grid_ec_m\themodel_c3_n500.pdf}
          	\caption{$x=0$}
          \end{subfigure}%
      
          \begin{subfigure}[b]{0.5\textwidth}
          	\includegraphics[width=1.0\textwidth]{kd_grid_ec_m\themodel_c4_n500.pdf}
          	\caption{$x=1$}
          \end{subfigure}%
          \begin{subfigure}[b]{0.5\textwidth}
          	\includegraphics[width=1.0\textwidth]{kd_grid_ec_m\themodel_c5_n500.pdf}
          	\caption{$x=2$}
          \end{subfigure}

\end{figure}
}

\newpage
\forloop{model}{1}{\value{model} < 5}{
\begin{figure}[h]
      \centering
       \caption{Average Interval Length of 95$\%$ Confidence Intervals - Model \themodel }\label{fig:dengrid_il\themodel}

           \begin{subfigure}[b]{0.5\textwidth}
          	\includegraphics[width=1.0\textwidth]{kd_grid_il_m\themodel_c1_n500.pdf}
          	\caption{$x=-2$}
          \end{subfigure}%
          ~ %add desired spacing between images, e. g. ~, \quad, \qquad, \hfill etc.
          \begin{subfigure}[b]{0.5\textwidth}
          	\includegraphics[width=1.0\textwidth]{kd_grid_il_m\themodel_c2_n500.pdf}
          	\caption{$x=-1$}
          \end{subfigure}%
          \begin{subfigure}[b]{0.5\textwidth}
          	\includegraphics[width=1.0\textwidth]{kd_grid_il_m\themodel_c3_n500.pdf}
          	\caption{$x=0$}
          \end{subfigure}%
          
          \begin{subfigure}[b]{0.5\textwidth}
          	\includegraphics[width=1.0\textwidth]{kd_grid_il_m\themodel_c4_n500.pdf}
          	\caption{$x=1$}
          \end{subfigure}%
          \begin{subfigure}[b]{0.5\textwidth}
          	\includegraphics[width=1.0\textwidth]{kd_grid_il_m\themodel_c5_n500.pdf}
          	\caption{$x=2$}
          \end{subfigure}

\end{figure}
}	

\end{landscape}

\newpage
\begin{figure}[h]
        \centering
        \caption{Empirical Coverage of 95$\%$ Confidence Intervals ($x=0$)}\label{fig:ec3d}
        \begin{subfigure}[b]{0.5\textwidth}
                \includegraphics[width=1.05\textwidth]{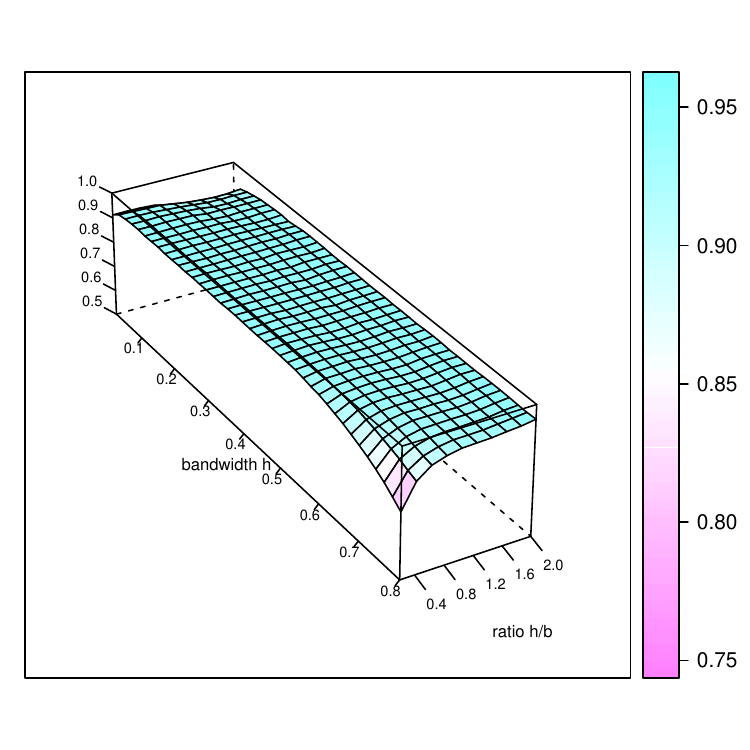}
                \caption{Model 1}
        \end{subfigure}%
        ~ %add desired spacing between images, e. g. ~, \quad, \qquad, \hfill etc.
        \begin{subfigure}[b]{0.5\textwidth}
                \includegraphics[width=1.05\textwidth]{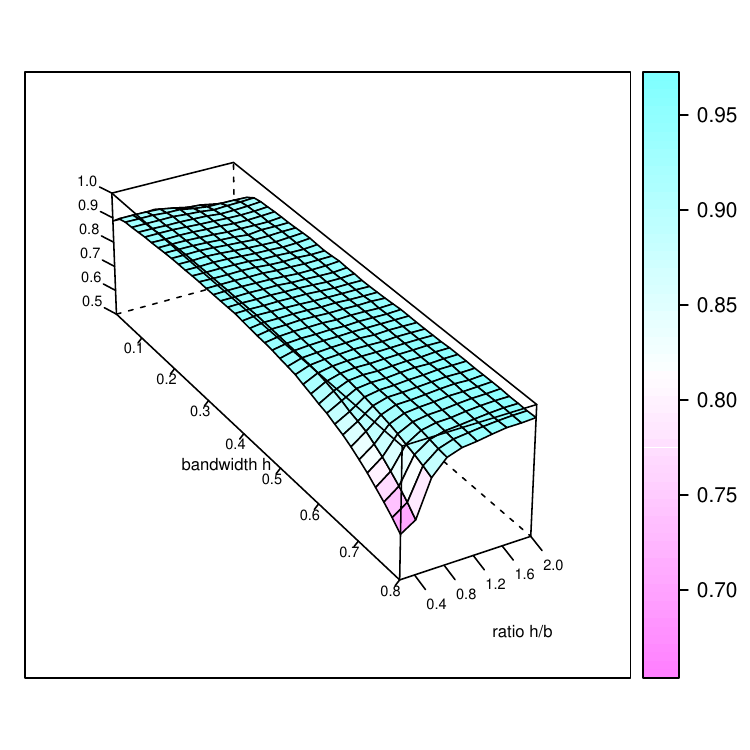}
                \caption{Model 2}
        \end{subfigure}
        
        \begin{subfigure}[b]{0.5\textwidth}
                        \includegraphics[width=1.05\textwidth]{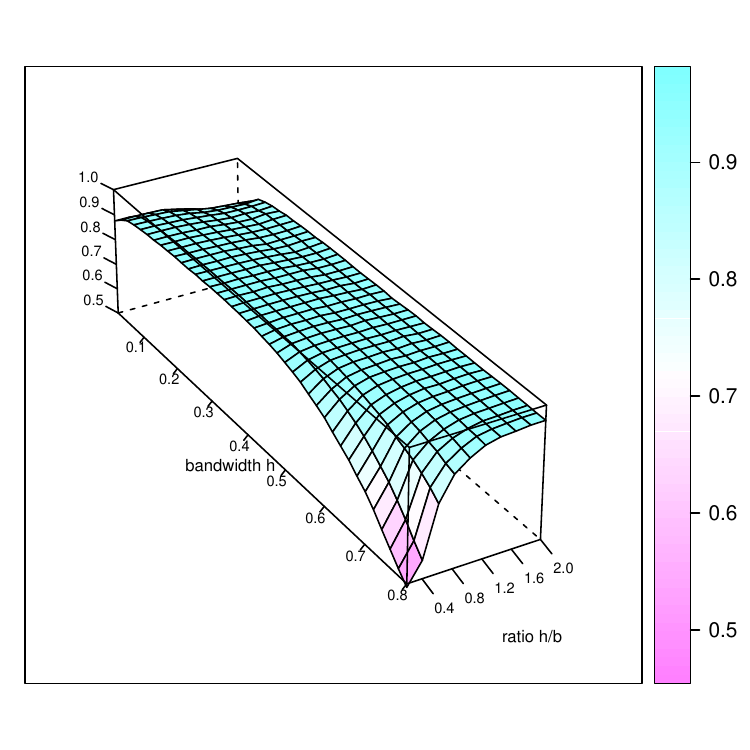}
                        \caption{Model 3}
        \end{subfigure}%
        \begin{subfigure}[b]{0.5\textwidth}
                             \includegraphics[width=1.05\textwidth]{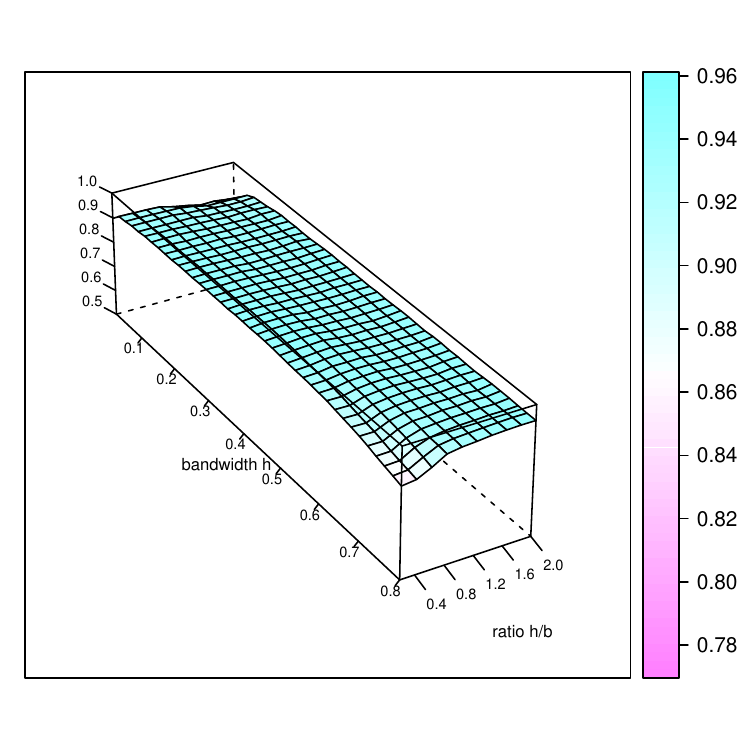}
                             \caption{Model 4}
        \end{subfigure}
\end{figure}

\newpage
\begin{figure}[h]
        \centering
        \caption{Average Interval Length of 95$\%$ Confidence Intervals ($x=0$)}\label{fig:il3d}
        \begin{subfigure}[b]{0.5\textwidth}
                \includegraphics[width=1.05\textwidth]{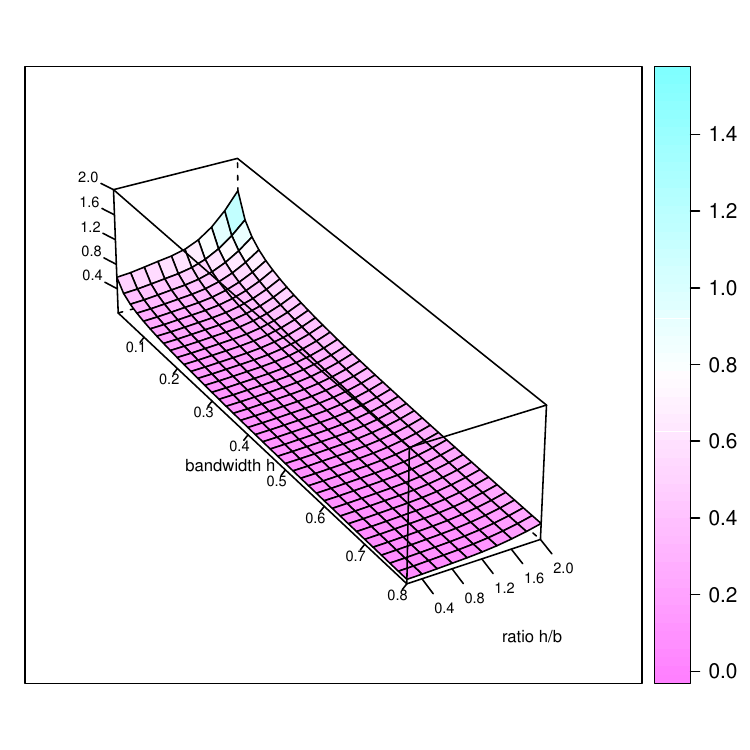}
                \caption{Model 1}
        \end{subfigure}%
        ~ %add desired spacing between images, e. g. ~, \quad, \qquad, \hfill etc.
        \begin{subfigure}[b]{0.5\textwidth}
                \includegraphics[width=1.05\textwidth]{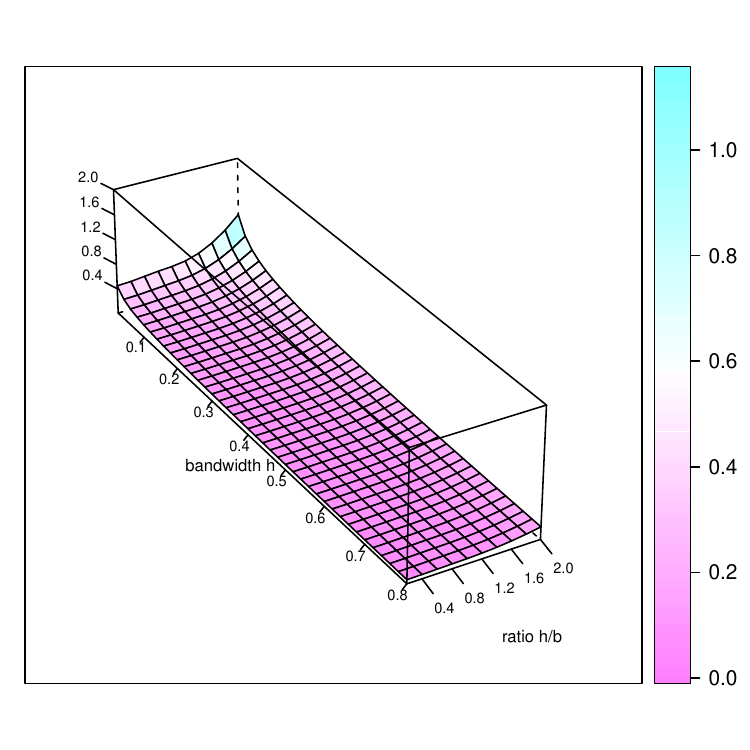}
                \caption{Model 2}
        \end{subfigure}
        
        \begin{subfigure}[b]{0.5\textwidth}
                        \includegraphics[width=1.05\textwidth]{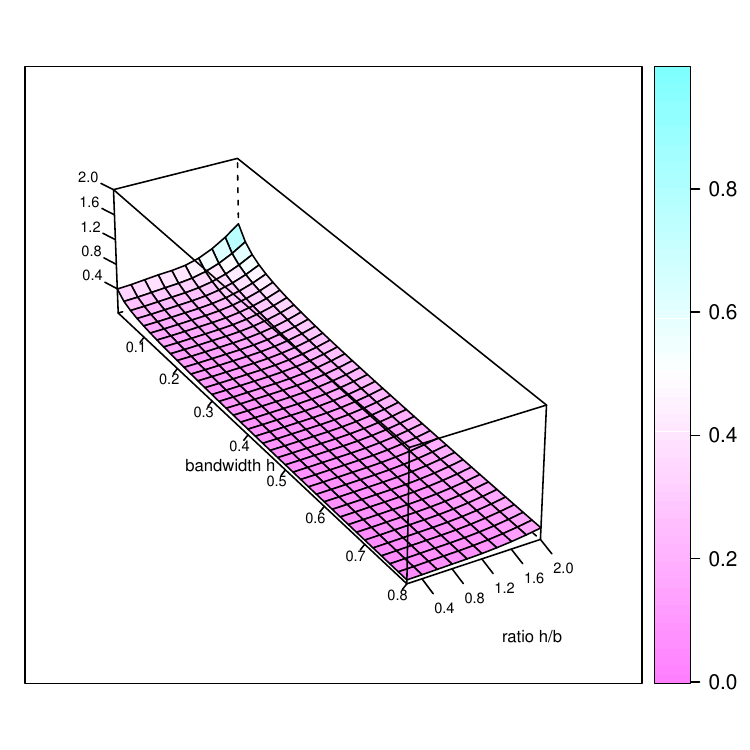}
                        \caption{Model 3}
        \end{subfigure}%
        \begin{subfigure}[b]{0.5\textwidth}
                             \includegraphics[width=1.05\textwidth]{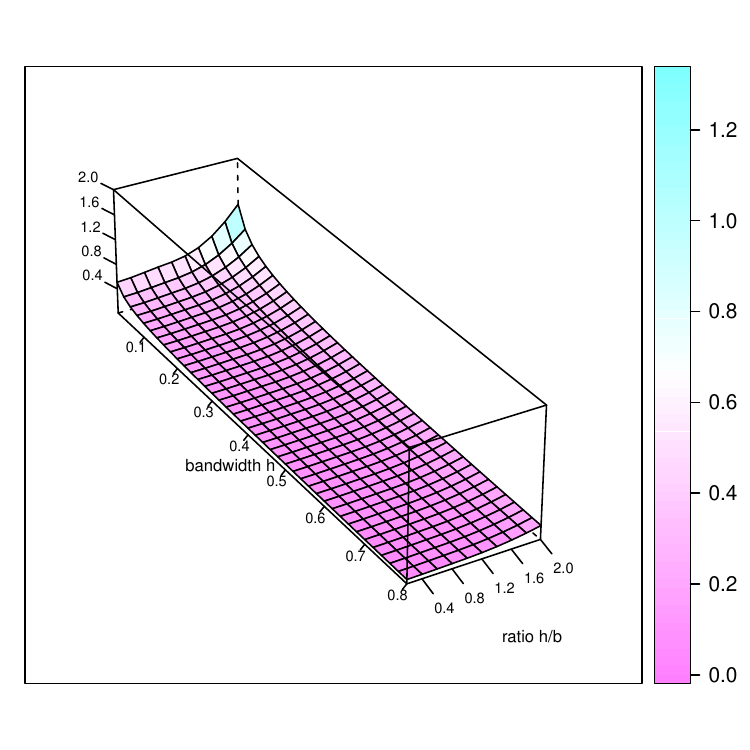}
                             \caption{Model 4}
        \end{subfigure}
\end{figure}

%%%%%%%%%%%%%%%%%%%%%%%%%%%%%%%%%%%
%%%%%%%%%%%%%%%%%%%%%%%%%%%%%%%%%%%
\clearpage
\part{Local Polynomial Estimation and Inference}
	\label{supp:locpoly}
%%%%%%%%%%%%%%%%%%%%%%%%%%%%%%%%%%%
%%%%%%%%%%%%%%%%%%%%%%%%%%%%%%%%%%%

%%%%%%%%%%%%%%%%%%%%%%%%%%%%%%%%%%%
%%%%%%%%%%%%%%%%%%%%%%%%%%%%%%%%%%%
\section{Notation}
	\label{supp:notation locpoly}

Local polynomial regression is notationally demanding, and the Edgeworth expansions will be substantially more so. For ease of reference, we collect all notation here regardless of where it is introduced and used. Much of the notation is fully restated later, when needed. As such, this subsection is designed more for reference, and is not easily readable.

Throughout, a subscript $p$ will generally refer to a quantity used to estimate $m(x)=\E[Y_i|X_i=x]$, while a subscript $q$ will refer to the bias correction portion (the vectors $e_0$ and $e_{p+1}$ below are notable exceptions to this rule). Recall that $p \geq 1$ is odd and $q>p$ may be even or odd.

Throughout this section let $X_{\h,i} = (X_i - x)/\h$ and similarly for $X_{\b,i}$. The evaluation point is implicit here.

To save notation, products of functions will be written together, with only one argument. For example
\[ (K r_p r_p')(X_{\h,i}) \defsym K(X_{\h,i}) r_p(X_{\h,i}) r_p(X_{\h,i})' = K\left(\frac{X_i - x}{\h}\right) r_p \left(\frac{X_i - x}{\h}\right) r_p \left(\frac{X_i - x}{\h}\right)' ,  \]
and similarly for $(K r_p)(X_{\h,i})$, $(L r_q) (X_{\b,i})$, etc.

All expectations are fixed-$n$ calculations. To give concrete examples of this notation ($\L_{p,k}$, $R_p$, and $W_p$ are redefined below):
\[\L_{p,k} = R_p' W_p [ ((X_1 - x)/\h)^{p+k}, \cdots, ((X_n - x)/\h)^{p+k}]'/n = \frac{1}{n\h} \sumi  (K r_p)(X_{\h,i}) X_{\h,i}^{p+k}\]
and 
\[\Lt_{p,k} = \E[\L_{p,k}] = \h^{-1} \E[(K r_p)(X_{\h,j}) X_{\h,i}^{p+k}] = \h^{-1} \int_{\supp\{X\}} K\left(\frac{X_i - x}{\h}\right) r_p \left(\frac{X_i - x}{\h}\right) \left(\frac{X_i - x}{\h}\right)^{p+k} f(X_i) dX_i.\]
Here the range of integration is explicit, but in general it will not be. This is important for boundary issues, where the notation is generally unchanged, and it is to be understood that moments and moments of the kernel be replaced by the appropriate truncated version. Continuing this example, if $\supp\{X\} = [0,\infty)$ and $x = 0$, then by a change of variables
\[\Lt_{p,k} = \h^{-1} \int_{\supp\{X\}}(K r_p)(X_{\h,j}) X_{\h,i}^{p+k} f(X_i) dX_i = \int_0^\infty (K r_p)(u) u^{p+k} f(-u\h)du,\]
whereas if $\supp\{X\} = (-\infty,0]$ and $x = 0$, then 
\[\Lt_{p,k} = \int_{-\infty}^0 (K r_p)(u) u^{p+k} f(-u\h)du.\]
For the remainder of this section, the notation is left generic.

For the proofs (Section \ref{supp:proof locpoly}) we will frequently abbreviate $s = \sqrt{n \h}$.

%%%%%%%%%%%%%%%%%%%%%%%%%%%%%%%%%%%
\subsection{Estimators, Variances, and Studentized Statistics}

To define the estimator $\hat{m}$ of $m$ and the bias correction, begin by defining:
\begin{align}
	\begin{split}
		\label{eqn:notation m hat}
		& r_p(u) = \left(1, u, u^2, \ldots, u^p\right)', 		\quad \qquad 		R_p = \left[ r_p( X_{\h,1}), \cdots, r_p( X_{\h,n} ) \right]', 		\\
		& W_p = \diag\left(\h^{-1} K(X_{\h,i}): i = 1, \ldots, n\right), 		 \qquad  		H_p = \diag\left(1, \h^{-1}, \h^{-2}, \ldots, \h^{-p} \right), 		\\	
		& \Gp = R_p' W_p R_p/n,		\qquad \text{ and } \qquad 		\L_{p,k} = R_p' W_p \left[ X_{\h,1}^{p+k}, \cdots, X_{\h,n}^{p+k} \right]'/n,
	\end{split}
\end{align}
where $\diag(a_i:i = 1, \ldots, n)$ denote the $n\times n$ diagonal matrix constructed using the elements $a_1, a_2, \cdots, a_n$. Note that in the main text $\L_{p,1}$ is denoted by $\L_p$.

Similarly, define
\begin{align}
	\begin{split}
		\label{eqn:notation bias hat}
		& r_q(u) = \left(1, u, u^2, \ldots, u^q \right)', 		\quad \qquad 		R_q = \left[ r_q( X_{\b,1}), \cdots, r_q( X_{\b,n}) \right]', 		\\
		& W_q = \diag\left(\b^{-1} L(X_{\b,i}): i = 1, \ldots, n \right), 		\quad \qquad  		H_q = \diag\left(1, \b^{-1}, \b^{-2}, \ldots, \b^{-q} \right), 		\\	
		& \Gq = R_q' W_q R_q/n,		\qquad \text{ and } \qquad 		\L_{q,k} = R_q' W_q \left[ X_{\b,1}^{q+k}, \cdots, X_{\b,n}^{w+k} \right]'/n,
	\end{split}
\end{align}
These are identical, but substituting $q$, $L$, and $\b$ in place of $p$, $K$, and $\h$, respectively. Note that some dimensions change but other do not: for example, $W_p$ and $W_q$ are both $n \times n$, but $\Gp$ is $(p+1)$ square whereas $\Gq$ is $(q+1)$.

Denote by $e_0$ the $(p+1)$-vector with a one in the first position and zeros in the remaining and $Y = (Y_1, \cdots, Y_n)'$. The local polynomial estimator of $m(x)=\E[Y_i|X_i=x]$ is
\[\hat{m} = e_0' \bhat_p = e_0' H_p \Gp^{-1} R_p' W_p Y / n,\]
where
\[\bhat_p = \argmin_{b \in \mathbb{R}^{p+1}} \frac{1}{n \h} \sumi ( Y_i - r_p(X_i - x)'b)^2  K \left( X_{\h,i} \right) =  H_p \Gp^{-1} R_p' W_p Y / n.  \]
If we define  $\check{R} = \left[ r_p( X_1 - x), \cdots, r_p( X_n - x ) \right]'$ and $M = [m(X_1), \ldots, m(X_n)]'$, then we can split $\hat{m} - m$ into the variance and bias terms
	\[\hat{m} - m = e_0' \Gp^{-1} R_p' W_p (Y-M) / n +  e_0' \Gp^{-1} R_p' W_p (M - \check{R}\beta_p) / n.\]
This will be useful in the course of the proofs.

The conditional bias is given by 
\begin{equation}
	\label{suppeqn:bias locpoly}
	\E[\hat{m} \vert X_1, \ldots, X_n] - m =  \h^{p+1} m^{(p+1)} \frac{1}{(p+1)!} e_0' \Gp^{-1} \L_{p,1} + o_p(\h^{p+1}).
\end{equation}
(Recall that in the main paper, $\L_{p,1}$ is denoted $\L_p$.) This result is valid for $p$ odd, our main focus, but also for $p$ even at boundary points.

Denote by $e_{p+1}$ the $(q+1)$-vector with one in the $p+2$ position, and zeros in the rest. Then we estimate the bias as
	\[\hat{B}_m = \h^{p+1} \hat{m}^{(p+1)} \frac{1}{(p+1)!} e_0' \Gp^{-1} \L_{p,1}, 		\qquad \text{ where } \qquad 		\hat{m}^{(p+1)} = [(p+1)!] e_{p+1}'H_q \Gq^{-1} R_q' W_q Y / n.\]
The bias corrected estimator can then be written
\begin{align*}
	\hat{m} - \hat{B}_m & =  e_0' H_p \Gp^{-1} R_p' W_p Y / n - \h^{p+1}  e_0' \Gp^{-1} \L_{p,1} e_{p+1}'H_q \Gq^{-1} R_q' W_q Y / n  		\\
	& = e_0' \Gp^{-1} \left(  R_p' W_p  - \rho^{p+1}   \L_{p,1} e_{p+1}' \Gq^{-1} R_q' W_q \right) Y / n,
\end{align*}
using the fact that $e_{p+1}'H_q = \b^{p+1} e_{p+1}'$.

The fixed-$n$ variances are
\begin{equation}
	\label{suppeqn:variance locpoly}
	\sus^2 \defsym (n \h) \V[\hat{m} \vert X_1, \cdots, X_n] = e_0' \Gp^{-1} \left( \h R_p' W_p \Sigma W_p R_p / n \right) \Gp^{-1} e_0
\end{equation}
and 
\begin{equation}
	\label{suppeqn:RBC variance locpoly}
	\begin{split}
		\srbc^2  \defsym & \ (n \h) V[\hat{m} - \hat{B}_m \vert X_1, \ldots, X_n]			\\
		 = & \ e_0' \Gp^{-1} \bigl( \h/n \bigr) \left( R_p' W_p  -  \rho^{p+1} \L_{p,1} e_{p+1}' \Gq^{-1} R_q' W_q\right) \Sigma \left( R_p' W_p   -  \rho^{p+1} \L_{p,1} e_{p+1}' \Gq^{-1} R_q' W_q  \right)' \Gp^{-1} e_0,
	\end{split}
\end{equation}
where
\[\Sigma = \diag(v(X_i): i = 1,\ldots, n),  		\qquad \text{ with } \qquad  		 v(x) = \V[Y \vert X = x].\]

These are the closest analogue to the density case, but are still random due to the conditioning on the covariates. Their respective estimators are 
\begin{equation*}
	\label{eqn:variance hat}
	\shatus^2 = e_0' \Gp^{-1} \left( \h R_p' W_p \Shat_p W_p R_p \Gp^{-1} / n \right)  e_0
\end{equation*}
and 
\begin{equation*}
	\label{eqn:RBC variance hat}
	\shatrbc^2 =  \ e_0' \Gp^{-1}  \bigl( \h/n \bigr)  \left( R_p' W_p  -  \rho^{p+1} \L_{p,1} e_{p+1}' \Gq^{-1} R_q' W_q\right) \Shat_q \left( R_p' W_p  -  \rho^{p+1} \L_{p,1} e_{p+1}' \Gq^{-1} R_q' W_q\right)' \Gp^{-1} e_0.
\end{equation*}
The conditional variance matrixes are estimated as 
\[\Shat_p = \diag(\hat{v}(X_i): i = 1,\ldots, n),  		\qquad \text{ with } \qquad  		 \hat{v}(X_i) = ( Y_i - r_p(X_i - x)'\bhat_p )^2,\]
and 
\[\Shat_q = \diag(\hat{v}(X_i): i = 1,\ldots, n),  		\qquad \text{ with } \qquad  		 \hat{v}(X_i) = ( Y_i - r_q(X_i - x)'\bhat_q )^2.\]

The Studentized statistics of interest are then:
	\[\tus = \frac{\sqrt{n\h}( \hat{m} - m)}{\shatus}, 			\qquad		 \tbc =  \frac{\sqrt{n\h}( \hat{m} - \hat{B}_m -  m)}{\shatus},		\qquad		 \trbc = \frac{\sqrt{n\h}( \hat{m} - \hat{B}_m -  m)}{\shatrbc}.\]
The main result of this section is an Edgeworth expansion of the distribution function of these statistics.

%%%%%%%%%%%%%%%%%%%%%%%%%%%%%%%%%%%
\subsection{Edgeworth Expansion Terms}
	\label{supp:terms}

The terms of the Edgeworth expansion require further notation and discussion. The expressions are not nearly as compact as in the density case (cf. Section \ref{supp:Edgeworth density}). 

Define the expectations of $\Gp$, $\Gq$, $\L_{p,k}$, and $\L_{q,k}$ as $\Gpt$, $\Gqt$, $\Lt_{p,k}$, and $\Lt_{q,k}$, such as
	\[\Gpt = \E \left[ \Gp \right] = \E \left[ \h^{-1} (K r_p r_p')(X_{\h,i}) \right]. \]
These will be used to define nonrandom biases and variances that appear in the expansions.

The biases are defined in \Eqref{suppeqn:eta locpoly}, and are given by
\begin{align*}
	\begin{split}
		\eus & = \sqrt{nh} \int e_0' \Gpt^{-1} K(u) r_p(u) \left( m(x - u \h) - r_p(u\h)'\beta_p \right) f(x - u\h)du,		\\
		\ebc & = \sqrt{nh} \int e_0' \Gpt^{-1} K(u) r_p(u) \left( m(x - u \h) - r_{p+1}(u\h)'\beta_{p+1} \right) f(x - u\h)du 		\\		
			& \quad 	-  \sqrt{nh} \rho^{p+1} \int e_0' \Gpt^{-1} \Lt_{p,1} e_{p+1}' \Gqt^{-1} L(u) r_q(u) \left( m(x - u \b) - r_q(u\b)'\beta_q \right) f(x - u\b)du.
	\end{split}
\end{align*}
Further discussion and leading terms are found in Section \ref{supp:bias locpoly}.

The fixed-$n$ variances are computed conditionally, and we must replace them with their nonrandom analogues (just as $\eus$ and $\ebc$ must be nonrandom). Recalling Equations \eqref{suppeqn:variance locpoly} and \eqref{suppeqn:RBC variance locpoly}, define 
\begin{equation*}
	\label{eqn:variance tilde}
	\stus^2 \defsym e_0' \Gpt^{-1} \Psit_p \Gpt^{-1} e_0,
\end{equation*}
where
	\[ \Psit_p = \E \left[  \Psic_p \right]   	\qquad \text{ and } \qquad 		  \Psic_p \defsym \h R_p' W_p \Sigma W_p R_p / n,  \]
and 
\begin{equation*}
	\label{eqn:RBC variance tilde}
		\strbc^2  \defsym  e_0' \Gpt^{-1} \Psit_q \Gpt^{-1} e_0
\end{equation*}
where
	\[ \Psit_q =  \E \left[ \Psic_q  \right]   	\quad \text{and} \quad 		  \Psic_q \defsym \h \left( R_p' W_p  -  \rho^{p+1} \Lt_{p,1} \Gqt^{-1} R_q' W_q\right) \Sigma \left( R_p' W_p  / n  -  \rho^{p+1} \Lt_{p,1} \Gqt^{-1} R_q' W_q  / n \right)' . \]
In the course of the proofs, we will also use $\Psihat_p =  \h R_p' W_p \Shat_p W_p R_p / n$ and the analogously-defined $\Psihat_q$.

We now give the precise forms of the polynomials in the Edgeworth expansion. As with the density, there will be both even and odd polynomials. These are not as compact or simple as the density case. Further, we will not attempt to simplify these functions by making use of limiting versions of moments. For example, we will \emph{not} replace $\Lt_{p,1}$ by $f(x) \int (K r_p)(u) u^{p+1}du$, and similarly for other pieces. The only simplification made will be the use of $q_{k,\US}(z)$ in the expansion for $\tbc$, which otherwise would require further notation than what is below (along the lines of $p_{1,\US}(z)$ below).

First, define the following functions, which depend on $n$, $p$, $q$, $\h$, $\b$, $K$ and $L$, but this is generally suppressed:
\begin{align*}
	\l^0_\US(X_i) & = e_0' \Gpt^{-1} (K r_p)(X_{\h,i}) ; 		\\
	\l^0_\BC(X_i) & = \l^0_\US(X_i)  -   \rho^{p+1} e_0' \Gpt^{-1} \Lt_{p,1} e_{p+1}' \Gqt^{-1}  (L r_q)(X_{\b,i}) ; 		\\
	\l^1_\US(X_i, X_j) & = e_0' \Gpt^{-1} \left( \E[(K r_p r_p')(X_{\h,j})] - (K r_p r_p')(X_{\h,j}) \right) \Gpt^{-1} (K r_p)(X_{\h,i}) ; 		\\
	\l^1_\BC(X_i, X_j) & = \l^1_\US(X_i, X_j)     -    \rho^{p+1}   e_0' \Gpt^{-1} \Bigl\{ \left( \E[(K r_p r_p')(X_{\h,j})] - (K r_p r_p')(X_{\h,j}) \right)  \Gpt^{-1}  \Lt_{p,1} e_{p+1}'   		\\
		& \qquad \qquad \qquad  \qquad    +  \left( (K r_p)(X_{\h,j}) X_{\h,i}^{p+1}   -  \E[(K r_p)(X_{\h,j}) X_{\h,i}^{p+1}]  \right) e_{p+1}'       		\\
		& \qquad \qquad \qquad  \qquad    +  \Lt_{p,1} e_{p+1}' \Gqt^{-1}  \left( \E[(L r_q r_q')(X_{\b,j})] - (L r_q r_q')(X_{\b,j}) \right)   \Bigr\} \Gqt^{-1}  (L r_q)(X_{\b,i}).
\end{align*}

With this notation, we can write
\begin{align*}
	\stus^2 & = \E[ \h^{-1} \l^0_\US(X)^2 v(X) ]  ,    		\\
	\strbc^2 & = \E[ \h^{-1} \l^0_\BC(X)^2 v(X) ]  ,  		\\
	\eus & = s \E \left[ \h^{-1} \l^0_\US(X_i) [m(X_i) - r_p(X_i - x)' \beta_p] \right],
\end{align*}
and
\begin{multline*}
	\ebc  = s \E \Bigl[ \h^{-1} \l^0_\US(X_i) [m(X_i) - r_{p+1}(X_i - x)' \beta_{p+1}]   		\\		  +   \h^{-1} \left( \l^0_\BC(X_i)  -   \l^0_\US(X_i) \right)  [m(X_i) - r_q(X_i - x)' \beta_q]   \Bigr]. 
\end{multline*} 
We will define the Edgeworth expansion polynomials first for the undersmoothing case. The standard Normal density is $\phi(z)$. First, the even polynomials are
	\[p_{1,\US}(z) =   \phi(z) \stus^{-3} \E \left[ \h^{-1} \l^0_\US(X_i)^3 \e_i^3 \right] \left\{  (2z^2 - 1)/6 \right\} \]
and
	\[p_{3,\US}(z) = -  \phi(z)  \stus^{-1}.\]
The absence of $p^{(2)}(z)$ is noteworthy: there is no version of this term for local polynomial estimation, because $\e_i$ is conditionally mean zero.

Next, the odd polynomials for undersmoothing are defined as follows:
\begin{align*}
	\hspace{-0.25in} q_{1,\US}(z) & =  \phi(z)  \stus^{-6} \E \left[ \h^{-1} \l^0_\US(X_i)^3 \e_i^3 \right]^2 \left\{  z^3/3 + 7 z /4 + \stus^2 z (z^2-3)/4 \right\}   		\\
	& \quad +  \phi(z)  \stus^{-2} \E \left[ \h^{-1} \l^0_\US(X_i) \l^1_\US(X_i, X_i) \e_i^2 \right] \left\{ - z (z^2 - 3) /2 \right\}   		\\
	& \quad +  \phi(z)  \stus^{-4} \E \left[ \h^{-1} \l^0_\US(X_i)^4 (\e_i^4 - v(X_i)^2) \right] \left\{ z(z^2-3)/8 \right\}   		\\
	& \quad -  \phi(z)  \stus^{-2} \E \left[ \h^{-1}\l^0_\US(X_i)^2 r_p(X_{\h,i})'\Gpt^{-1} (K r_p)(X_{\h,i}) \e_i^2 \right] \left\{ z(z^2 - 1)/2 \right\}   		\\
	& \quad -  \phi(z)  \stus^{-4} \E \left[ \h^{-1} \l^0_\US (X_i)^3  r_p(X_{\h,i})'\Gpt^{-1} \e_i^2 \right] \E \left[ \h^{-1} (K r_p)(X_{\h,i}) \l^0_\US (X_i) \e_i^2 \right] \left\{ z(z^2 - 1)  \right\}   		\\
	& \quad +  \phi(z)  \stus^{-2} \E \left[ \h^{-2} \l^0_\US(X_i)^2 (r_p(X_{\h,i})'\Gpt^{-1}(K r_p)(X_{\h,j}) )^2 \e_j^2 \right] \left\{ z(z^2 - 1)/4 \right\}   		\\
%	& \quad +   \phi(z)   \stus^{-4} \E \left[ \h^{-1} \l^0_\US (X_j)^2 \left( \E \left[ \h^{-1} r_p(X_{\h,j})'\Gpt^{-1}(K r_p)(X_{\h,i}) \l^0_\US(X_i)  \e_i^2 \vert X_j \right] \right)^2 \right] \left\{  z(z^2 - 1) /2 \right\}   		\\ %%this is the old version of the line right below, with the conditional expectation squared
	& \quad +   \phi(z)   \stus^{-4} \E \left[ \h^{-3} \l^0_\US (X_j)^2 r_p(X_{\h,j})'\Gpt^{-1}(K r_p)(X_{\h,i}) \l^0_\US(X_i) r_p(X_{\h,j})'\Gpt^{-1}(K r_p)(X_{\h,k}) \l^0_\US(X_k)  \e_i^2 \e_k^2 \right]   		\\
	& \quad \qquad \qquad \qquad \qquad \qquad   \times \;   \left\{  z(z^2 - 1) /2 \right\}   		\\
	& \quad +  \phi(z)  \stus^{-4} \E \left[ \h^{-1} \l^0_\US(X_i)^4 \e_i^4 \right] \left\{ - z (z^2 - 3)/24 \right\}   		\\
	& \quad +  \phi(z)  \stus^{-4} \E \left[ \h^{-1} \left( \l^0_\US(X_i)^2 v(X_i) - \E[\l^0_\US(X_i)^2 v(X_i)] \right) \l^0_\US(X_i)^2 \e_i^2 \right] \left\{ z(z^2 - 1)/4 \right\}   		\\
	& \quad +  \phi(z)  \stus^{-4} \E \left[ \h^{-2} \l^1_\US(X_i, X_j) \l^0_\US(X_i)\l^0_\US(X_j)^2 \e_j^2 v(X_i) \right] \left\{  z (z^2 - 3) \right\}   		\\
	& \quad +  \phi(z)  \stus^{-4} \E \left[ \h^{-2} \l^1_\US(X_i, X_j)  \l^0_\US(X_i) \left( \l^0_\US(X_j)^2 v(X_j) - \E[\l^0_\US(X_j)^2 v(X_j)] \right) \e_i^2 \right] \left\{ - z \right\}   		\\
	& \quad +  \phi(z)  \stus^{-4} \E \left[ \h^{-1}  \left( \l^0_\US(X_i)^2 v(X_i) - \E[\l^0_\US(X_i)^2 v(X_i)] \right)^2 \right] \left\{ - z(z^2 + 1) /8 \right\} ;
\end{align*}
\[q_{2,\US}(z) = -  \phi(z)  \stus^{-2} z / 2 ; \]
\[q_{3,\US}(z) =  \phi(z)  \stus^{-4} \E [ \h^{-1} \l^0_\US(X_i)^3 \e_i^3 ]  ( z^3 / 3 ). \]
For robust bias correction, both the even polynomials, $p_{1,\RBC}(z)$ and $p_{3,\RBC}(z)$, and the odd polynomials, $q_{1,\RBC}(z)$, $q_{2,\RBC}(z)$, and $q_{3,\RBC}(z)$ are defined in the exact same way, but changing the $\stus$ to $\strbc$, $\l^k_\US(\cdot)$ to $\l^k_\BC(\cdot)$, $K$ to $L$, and $p$ to $q$, and so forth. For $q_{1,\US}(z)$ and $q_{1,\RBC}(z)$, the seventh term can be rewritten by rearranging the terms and factoring the expectation, as follows:
\begin{equation}
	\begin{split}
		\label{suppeqn:triple term}
		\E &  \left[ \h^{-3} \l^0_\US (X_j)^2 r_p(X_{\h,j})'\Gpt^{-1}(K r_p)(X_{\h,i}) \l^0_\US(X_i) r_p(X_{\h,j})'\Gpt^{-1}(K r_p)(X_{\h,k}) \l^0_\US(X_k)  \e_i^2 \e_k^2 \right]    			\\
		& \quad    = \E \left[ \h^{-1} \l^0_\US(X_i)  \e_i^2 (K r_p')(X_{\h,i}) \Gpt^{-1}  \right]     			\E \left[ \h^{-1}  \l^0_\US (X_j)^2 r_p(X_{\h,j}) r_p(X_{\h,j})'\Gpt^{-1} \right]      			\\
		& \qquad\qquad\qquad\qquad       \times \;     \E \left[ \h^{-1} (K r_p)(X_{\h,k}) \l^0_\US(X_k)  \e_k^2 \right] 
	\end{split}
\end{equation}

The polynomials defined here are for \emph{distribution function} expansions, and are different from those used for \emph{coverage error}. The polynomials $q_{1,\US}$, $q_{2,\US}$, and $q_{3,\US}$ and $q_{1,\RBC}$, $q_{2,\RBC}$, and $q_{3,\RBC}$, which do \emph{not} have an argument, used for \emph{coverage error} in the main text and in Corollary \ref{appx:US locpoly} below, are defined in terms of those given above, which \emph{do} have an argument. Specifically, the polynomials above should be doubled, divided by the standard Normal density, and evaluated at the Normal quantile $z_{\alpha/2}$, that is,
\[ q_{k,\boldsymbol{\bullet}} \defsym \left. \frac{2}{\phi(z)} q_{k,\boldsymbol{\bullet}}(z)  \right|_{z = z_{\alpha/2} },   		  \qquad\qquad k=1,2,3, \quad \boldsymbol{\bullet} = \US, \RBC  \]

For traditional bias correction, $q_{1,\US}(z)$, $q_{2,\US}(z)$, and $q_{3,\US}(z)$ are used, but such simplification can not be done for $p_{1,\BC}(z)$ and $p_{3,\BC}(z)$, which must be defined as
\begin{align*}
	p_{1,\BC}(z) & =   \phi(z)   \stus^{-3} \Bigl( \E \left[ \h^{-1} \l^0_\US(X_i)^3 \e_i^3 \right] \left\{  - (z^2 - 1)/6 \right\}   	 +     \E \left[ \h^{-1} \l^0_\US(X_i)^2 \l^0_\BC(X_i)  \e_i^3 \right] \left\{  - (z^2 - 3)/4 \right\}  \Bigr)   		\\
	& \quad +  \phi(z)  \stus^2 \strbc^{-5}  \E \left[ \h^{-1} \l^0_\US(X_i)^2 \l^0_\BC(X_i)  \e_i^3 \right] \left\{ 3 (z^2 - 1)/4 \right\}   	
\end{align*}
and
	\[p_{3,\BC}(z) = -  \phi(z)   \stus^{-1}.\]

Lastly, traditional bias correction also exhibits additional terms in the expansion (see discussion in the main text) representing the covariance of $\hat{m}$ and $\hat{B}_m$ (denoted by $\Omega_{1,\BC}$) and the variance of $\hat{B}_m$ ($\Omega_{2,\BC}$). We now state their precise forms. These arise from the mismatch between the variance of the numerator of $\tbc$ and the standardization used, $\sus^2$, but these are random, and so $\Omega_{1\BC}$ and $\Omega_{2,\BC}$ must be derived from the nonrandom versions, $\strbc^2$ and $\stus^2$ (cf. Section \ref{supp:Edgeworth density}; for the same reason $\eus$ and $\ebc$ must be nonrandom). Recalling the definitions above,
\begin{align*}
	\frac{\strbc^2}{\stus^2} & = \frac{\E[ \h^{-1} \l^0_\BC(X)^2 v(X) ]}{\E[ \h^{-1} \l^0_\US(X)^2 v(X) ]}  		\\
	& = \frac{\E[ \h^{-1} \{\l^0_\US(X) + (\l^0_\BC(X) - \l^0_\US(X))\}^2 v(X) ]}{\E[ \h^{-1} \l^0_\US(X)^2 v(X) ]}  		\\
	& = 1 - 2 \stus^{-2} \E[ \h^{-1} \{\l^0_\US(X) (\l^0_\BC(X) - \l^0_\US(X))\} v(X) ]  + \stus^{-2} \E[ \h^{-1} \{(\l^0_\BC(X) - \l^0_\US(X))\}^2 v(X) ]  		\\
	& = 1 - 2 \rho^{1 + (p+1)}\stus^{-2} \E[ \h^{-1} \{\rho^{-p-2} \l^0_\US(X) (\l^0_\BC(X) - \l^0_\US(X))\} v(X) ]   		\\
	& \quad\quad  +   \rho^{1 + 2(p+1)} \stus^{-2} \E[ \b^{-1} \{\rho^{-p-2}(\l^0_\BC(X) - \l^0_\US(X))\}^2 v(X) ]  		\\
\end{align*}
Therefore
\[\Omega_{1,\BC} = -2  \stus^{-2} \E[ \h^{-1} \{\rho^{-p-2} \l^0_\US(X) (\l^0_\BC(X) - \l^0_\US(X))\} v(X) ] \]
and
\[\Omega_{2,\BC} =   \stus^{-2} \E[ \b^{-1} \{\rho^{-p-2}(\l^0_\BC(X) - \l^0_\US(X))\}^2 v(X) ] . \]

\begin{remark}[Simplifications]
It is possible for the above-defined polynomials to simplify in special cases. A leading example is in the homoskedastic Gaussian regression model:
	\[Y_i = m(X_i) + \e_i,   	\qquad \text{ where } \qquad  		\e_i \sim \N(0,v).\]
This model is a common theoretical baseline to study, though over-simplified from an empirical point of view. In this special case, $\E[\e_i^3]=0$ and thus $q_{3,\US}(z) \equiv 0$, entirely removing this term from the Edgeworth expansions. This has little bearing on the conceptual conclusions however, and in particular the comparison of undersmoothing and robust bias correction.
\end{remark}

%%%%%%%%%%%%%%%%%%%%%%%%%%%%%%%%%%%
%%%%%%%%%%%%%%%%%%%%%%%%%%%%%%%%%%%
\section{Details of Practical Implementation}
	\label{supp:practical locpoly}

In the main text we give a direct plug-in (DPI) rule to implement the coverage-error optimal bandwidth. Here we we give complete details for this procedure as well as document a second practical choice, based on a rule-of-thumb (ROT) strategy. Both choices yield the optimal coverage error decay rate at interior and boundary points. 

All our methods are implemented in {\sf R} and {\tt STATA} via the {\tt nprobust} package, available from \url{http://sites.google.com/site/nppackages/nprobust} (see also \url{http://cran.r-project.org/package=nprobust}). See \citet{Calonico-Cattaneo-Farrell2017_nprobust} for a complete description.

As in the density case, the MSE-optimal bandwidth undercovers when used in the undersmoothing confidence interval; that is, Remark \ref{rem:undercover} applies directly. See also \citet{Hall-Horowitz2013_AoS}.

%%%%%%%%%%%%%%%%%%%%%%%%%%%%%%%%%%%
\subsection{Bandwidth Choice: Rule-of-Thumb (ROT)}

As with the density case, a simple rule-of-thumb based on rescaling the MSE-optimal bandwidth is:
\[\hat{h}^{\tt int}_\ROT = \hat{h}^{\tt int}_\MSE \; n^{-(p-1)/((2p+3)(p+4))} \qquad \text{ and } \qquad
  \hat{h}^{\tt bnd}_\ROT = \hat{h}^{\tt bnd}_\MSE \; n^{-p/((2p+3)(p+3))}.\]
where $\hat{h}^{\tt int}_\MSE$ and $\hat{h}^{\tt bnd}_\MSE$ denote readily-available implementations of the MSE-optimal bandwidth for interior and boundary points, respectively. See, e.g., \citet{Fan-Gijbels1996_book}. Again, when $p=1$ in the interior, no scaling is needed ($\hat{h}^{\tt int}_\ROT = \hat{h}^{\tt int}_\MSE$), but for $p>1$ any data-driven MSE-optimal bandwidth should always be shrunk to improve inference at the boundary (i.e., reduce coverage errors of the robust bias-corrected confidence intervals). 

The ROT selector may be especially attractive for simplicity, if estimating the constants described below in the DPI case is prohibitive.

Remark \ref{rem:IMSE} applies to this case as well, though less transparently and without consequences that are as dramatic.

%%%%%%%%%%%%%%%%%%%%%%%%%%%%%%%%%%%
\subsection{Bandwidth Choice: Direct Plug-In (DPI)}
	\label{supp:bandwidth locpoly}

We now detail the required steps to implement the plug-in bandwidth $\hat{\h}_{\PI}$ for interior and boundary points. We always set $K=L$, $\rho=1$, and $q=p+1$. The steps are:
\begin{enumerate}[label=(\arabic*)]

	\item As a pilot bandwidth, use $\hat{h}_\MSE$: any data-driven version of $\h^*_\MSE$.

	\item Using this bandwidth, estimate the regression function $m(X_i)$ as $\hat{m}(X_i ; \hat{h}_\MSE) = r_p(X_i - x )'\bhat_p(\hat{h}_\MSE)$, where $\bhat_p(\hat{h}_\MSE)$ is the local polynomial coefficient estimate of order $p$ exactly as defined in the main text, using the bandwidth $\hat{h}_\MSE$.
	
		Form $\hat{\e}_i = Y_i - \hat{m}(X_i ; \hat{h}_\MSE)$.

	\item Following \citet[\S 4.2]{Fan-Gijbels1996_book} we estimate derivatives $m^{(k)}$ using a global least squares polynomial fit of order $k+2$. That is, estimate $\hat{m}^{(p+3)}(x)$ as
			\[\hat{m}^{(p+3)}(x)  =  \left[\hat{\gamma}\right]_{p+4} (p+3)!   +   \left[\hat{\gamma}\right]_{p+5} (p+4)!  \; x   +    \left[\hat{\gamma}\right]_{p+6} \frac{(p+5)!}{2} \; x^2, \]
		where $\left[\hat{\gamma}\right]_k$ is the $k$-th element of the vector $\hat{\gamma}$ that is estimated as
			\[ \hat{\gamma} = \argmin_{\gamma \in \mathbb{R}^{p+6}}\sumi \left( Y_i -  r_{p+5}(X_i)'\gamma \right)^2. \]
		The estimate for $\hat{m}^{(p+2)}(x)$ is similar, with all indexes incremented down once.

		For interior points, both are needed, while only $\hat{m}^{(p+2)}(x)$ is required for the boundary.

	\item The estimated polynomials $\hat{q}_{k,\RBC}$, $k=1,2,3$ and the bias constants $\ethbc^{\tt int}$ and $\ethbc^{\tt bnd}$ are defined as follows. The polynomials $q_{1,\RBC}$, $q_{2,\RBC}$, and $q_{3,\RBC}$, which do \emph{not} have an argument, are defined in terms of those given in Section \ref{supp:terms}, which \emph{do} have an argument. Specifically, the polynomials in Section \ref{supp:terms} should be doubled, divided by the standard Normal density, and evaluated at the Normal quantile $z_{\alpha/2}$, that is, $q_{k,\RBC} =  \phi(z_{\alpha/2})^{-1} q_{k,\RBC}(z_{\alpha/2})$. For $q_{1,\RBC}$, the form given in \Eqref{suppeqn:triple term} should be used.

		Note that with the recommended choice of $K=L$, $\rho=1$, and $q=p+1$, the polynomials $\hat{q}_{k,\RBC}$, $k=1,2,3$ can be read off the expressions for the undersmoothing versions, $\hat{q}_{k,\US}$, $k=1,2,3$, with $p$ replaced by $p+1$.

		The bias terms, for the interior and boundary, are given as follows (dropping remainder terms). With $q=p+1$, and hence even, and $\rho=1$, the expressions of Section \ref{supp:bias locpoly} simplify. For the interior: $\ebc^{\tt int}   =   \sqrt{n \h}  \h^{p+3} \etbc^{\tt int}$, with
	\begin{align*}
		\etbc^{\tt int}   &  =   \h^{-1}  \frac{ m^{(p+2)} } { (p+2)! } \left\{e_0' \Gpt^{-1}  \left(  \Lt_{p,2}  -  \Lt_{p,1} e_{p+1}' \Gqt^{-1}\Lt_{q,1} \right) \right\}       		\\
			 		&  \qquad +   \frac{ m^{(p+3)} } { (p+3)! }   \left\{  e_0' \Gpt^{-1}  \left(   \Lt_{p,3}  -  \Lt_{p,1} e_{p+1}' \Gqt^{-1}\Lt_{q,2} \right) \right\}  ;
	\end{align*}
At the boundary: $\ebc^{\tt bnd}   =   \sqrt{n \h}  \h^{p+2} \etbc^{\tt bnd}$, with 
	\[ \etbc^{\tt bnd}   =    \frac{ m^{(p+2)} } { (p+2)! } \; \left\{ e_0'  \Gpt^{-1}  \left( \Lt_{p,2}    -      \Lt_{p,1} e_{p+1}' \Gqt^{-1} \Lt_{q,1} \right)\right\}.\]

	The estimates of these, $\hat{q}_{k,\RBC}$, $k=1,2,3$ and $\ethbc^{\tt int}$ and $\ethbc^{\tt bnd}$, are defined by replacing:
		\begin{enumerate}[label=(\roman*)]
			\item $\h$ with $\hat{h}_\MSE$,
			\item population expectations with sample averages (see note below),
			\item residuals $\e_i$ with $\hat{\e}_i$,
			\item derivatives $m^{(p+2)}$ and $m^{(p+3)}$ with their estimators from above,
			\item limiting matrices $\Gpt$, $\Lt_{p,2}$, etc, with the corresponding sample versions using the bandwidth $\hat{h}_\MSE$, e.g., $\Gpt$ is replaced with $\Gp(\hat{h}_\MSE) = R_p' W_p(\hat{h}_\MSE) R_p / n$, where $W_p(\hat{h}_\MSE) = \diag\left(\hat{h}_\MSE^{-1} K\left( (X_i - x)/\hat{h}_\MSE \right) \right)$. 
		\end{enumerate}

	\item Finally $\hat{h}^{\tt int}_\PI =\hat{H}^{\tt int}_{\PI}(\hat{h}_\MSE) n^{-1/(p+4)}$ and $\hat{h}^{\tt bnd}_\PI =\hat{H}^{\tt bnd}_{\PI}(\hat{h}_\MSE) n^{-1/(p+3)}$, where 
	\[\hat{H}^{\tt int}_{\PI}(\hat{h}_\MSE) = \argmin_H \bigl\vert  H^{-1} \hat{q}_{1,\RBC}   +    H^{1+2(p+3)} (\ethbc^{\tt int})^2 \hat{q}_{2,\RBC}   +    H^{p+3} (\ethbc^{\tt int} ) \hat{q}_{3,\RBC}  \bigr\vert,\]
while at (or near) the boundary the optimal bandwidth is $\h^*_\RBC = H^*_\RBC (\rho) n^{-1/(p+3)}$, where 
	\[\hat{H}^{\tt bnd}_{\PI}(\hat{h}_\MSE) = \argmin_H \bigl\vert  H^{-1} \hat{q}_{1,\RBC}   +    H^{1+2(p+2)} (\ethbc^{\tt bnd})^2 \hat{q}_{2,\RBC}   +    H^{p+2} (\ethbc^{\tt bnd} ) \hat{q}_{3,\RBC} \bigr\vert.\]
		
		These numerical minimizations are easily solved; see note below. Code available from the authors' websites performs all the above steps.

\end{enumerate}

\begin{remark}[Notes on computation] \ 

\begin{itemize}
	\item When numerically solving the above minimization problems, computation will be greatly sped up by squaring the objective function.
	
	\item For step 4 above, in estimating $q_{1,\RBC}$, the form given in \Eqref{suppeqn:triple term} should be used. The original form requires evaluating a triple sum, or third order $U$-statistic, which will be far slower than the right hand side of \Eqref{suppeqn:triple term}.

	\item For step 4(ii) above, in estimating $\hat{q}_{1,\RBC}$, and specifically when replacing population expectations with sample averages, we use the appropriate $U$-statistic forms to reduce bias. There are several terms which are expectations over two or three observations, and for these the second or third order $U$-statistic forms are preferred. For example, when estimating terms such as
	\[ \E \left[ \h^{-2} \l^0_\US(X_i)^2 (r_p(X_{\h,i})'\Gpt^{-1}(K r_p)(X_{\h,j}) )^2 \e_j^2 \right] \]
	we use
	\[ \frac{1}{n (n-1)} \sum_{i=1}^n \sum_{j \neq i} \left[ \hat{h}_\MSE^{-2} \hat{\l}^0_\RBC(X_i)^2 (r_p(X_{\hat{h}_\MSE,i})'\Gp^{-1}(K r_p)(X_{\hat{h}_\MSE,j}) )^2 \hat{\e}_j^2 \right], \]
	where $\hat{\l}^0_\RBC(X_i)$ is made feasible as in step 4(v).
\end{itemize}
\vspace{-2em}  %bring the black square up, to end the remark closer to the last line of text.
\end{remark}
\vspace{1em}

%%%%%%%%%%%%%%%%%%%%%%%%%%%%%%%%%%%
\subsection{Alternative Standard Errors}
	\label{supp:standard errors}

As argued in the main text, using variance forms other than \eqref{suppeqn:variance locpoly} and \eqref{suppeqn:RBC variance locpoly} can be detrimental to coverage. Within these forms however, two alternative estimates of $\Sigma$ are natural. First, motivated by the fact that the least-squares residuals are on average too small, the well-known HC$k$ class of heteroskedasticity consistent estimators can be used; see \citet{MacKinnon2013_BookChap} for details and a recent review. In our notation, these are defined as follows. First, $\shatus^2$-HC0 is the estimator above. Then,  for $k=1, 2, 3$, the $\shatus^2$-HC$k$ estimator is obtained by dividing $\hat{\e}_i^2$ by, respectively, $(n-2\tr(Q_p)+\tr(Q_p'Q_p))/n$, $(1-Q_{p,ii})$, and $(1-Q_{p,ii})^2$, where $Q_{p,ii}$ is the $i$-th diagonal element of the projection matrix $Q_p:=R_p '\Gp^{-1} R_p' W_p/n$. The corresponding estimators $\shatrbc^2$-HC$k$ are the same way, with $q$ in place of $p$. As is well-known in the literature, these estimators perform better for small sample sizes, a fact we confirm in our simulation study below.

A second option is to use a nearest-neighbor-based variance estimators with a fixed number of neighbors, following the ideas of \citet{Muller-Stadtmuller1987_AoS,Abadie-Imbens2008_AdES}. To define these, let $J$ be a fixed number and $j(i)$ be the $j$-th closest observation to $X_i$, $j=1, \ldots, J$, and set $ \hat{v}(X_i) = \frac{J}{J+1} ( Y_i  -  \sum_{j=1}^J Y_{j(i)} / J )^2$. This ``estimate'' is unbiased (but inconsistent) for $v(X_i)$.

Both types of residual estimators could be handled in our results. The constants will change, but the rates will not. This is because, in all cases, the errors in estimating $v(X_i)$ are no greater than in the original $\hat{m}(x)$. Inspection of the proof shows that simple modifications allow for the HC$k$ estimators: only the terms of \Eqref{supp:A1 terms} will change, and indeed, we conjecture that the HC$k$ estimators will result in fewer terms and a reduced coverage error. This is consistent with the improved finite-sample behavior of these estimators and the fact that they are asymptotically equivalent. Accommodating the nearest-neighbor estimates require slightly more work and a modified version of Assumption \ref{supp:Cramer locpoly}.

One crucial property of our method, in the context of Edgeworth expansions, is that the bias in estimation of $\Sigma$ is of the same order as the original $\hat{m}(x)$.  Using other methods may result in additional terms, with possibly distinct rates, appearing in the Edgeworth expansions. Some examples that may have this issue are (i) using $\hat{v}(X_i) = ( Y_i - \hat{m}(x) )^2$; (ii) using local or assuming global heteroskedasticity; (iii) using other nonparametric estimators for $v(X_i)$, relying on new tuning parameters.

%%%%%%%%%%%%%%%%%%%%%%%%%%%%%%%%%%%
%%%%%%%%%%%%%%%%%%%%%%%%%%%%%%%%%%%
\section{Assumptions}
	\label{sec:assumptions locpoly}

The following assumptions are sufficient for our results. The first two are copied directly from the main text (see discussion there) and the third is the appropriate Cram\'er's condition.

\begin{assumption}[Data-generating process]
	\label{supp:dgp locpoly} 
	$\{(Y_1, X_1), \ldots, (Y_n, X_n)\}$ is a random sample, where $X_i$ has the absolutely continuous distribution with Lebesgue density $f$, $\E[Y^{8+\delta} \vert X] < \infty$ for some $\delta>0$, and in a neighborhood of $x$, $f$ and $v$ are continuous and bounded away from zero, $m$ is $S > q+2$ times continuously differentiable with bounded derivatives, and $m^{(S)}$ is H\"older continuous with exponent $\varsigma$.
\end{assumption}

\begin{assumption}[Kernels]
	\label{supp:kernel locpoly}
	The kernels $K$ and $L$ are positive, bounded, even functions, and with compact support.
\end{assumption}

\begin{assumption}[Cram\'er's Condition]
	\label{supp:Cramer locpoly} 
	For each $\delta>0$ and all sufficiently small $\h$, the random variables $Z_\US(u)$ and $Z_\RBC(u)$ defined below obey
	\[  \sup_{t \in \mathbb{R}^{\dim\{Z(u)\}}, \| t \| > \delta} \left\vert \int \exp \{ i t'Z(u) \} f(x - u \h) du \right\vert \leq 1 - C(x, \delta) \h,\]
	where $C(x,\delta)>0$ is a fixed constant, $\| t \|^2 = \sum_{d=1}^{\dim\{Z(u)\}} t_d^2$, and $i=\sqrt{-1}$.
\end{assumption}

The random variables of Assumption \ref{supp:Cramer locpoly} are defined follows. For two kernels $K_1$ and $K_2$, two polynomial orders (i.e. positive integers) $p_1$ and $p_2$, a bandwidth $\b$, and a scalar $\rho$, let
\begin{multline*}
	\hspace{-1em} Z_m(u; K_1, p_1, p_2, \b, \rho) : = \Bigl( K_1(u) r_{p_1}(u)'\e , \     K_1(u) r_{p_1}(u)'(m(x - u \b) - r_{p_2}(u\b)'\beta_{p_2}) , \    \vech(K_1(u) r_{p_1}(u) r_{p_1}(u)')' \Bigr)'  . 
\end{multline*}
and
\begin{align*}
	Z_\sigma(u; K_1, K_2, p_1, p_2, \b, \rho) & := \Bigl(  \vech(K_1(u) K_2(u \rho) r_{p_1}(u) r_{p_2}(u \rho)' \e^2)'  , \         		\\
	& \qquad  \vech(K_1(u) K_2(u \rho) r_{p_1}(u) r_{p_2}(u \rho)' v(x - u\b))'  , \         		\\
	& \qquad  \vech(K_1(u) K_2(u \rho) r_{p_1}(u) r_{p_2}(u \rho)' \e(m(x - u \b) - r_{p_2}(u\b)'\beta_{p_2}))'  , \         		\\
	& \qquad  \vech(K_2(u)^2 r_{p_2}(u) r_{p_2}(u)'r_{p_2}(u)' )' , \         		\\
	& \qquad  \vech(K_1(u) K_2(u \rho) r_{p_1}(u) r_{p_2}(u \rho)' r_{p_2}(u)' \e )'  , \         		\\
	& \qquad  \vech(K_1(u) K_2(u \rho) r_{p_1}(u) r_{p_2}(u \rho)' r_{p_2}(u\rho)' \e (m(x - u \b) - r_{p_2}(u\b)'\beta_{p_2}))'  \ \Bigr)'.
\end{align*}
The subscripts are intended to make clear that $Z_m(\cdot)$ collects quantities from the numerator of the Studentized statistic, while $Z_\sigma(\cdot)$ gathers additional variables required for the variance estimation. With this notation, we define 
	\[ Z_\US(u)  = \bigl( Z_m(u; K, p, p, \h, 1)' , \   Z_\sigma(u; K, K, p, p, \h, 1)' \bigr)', \]
	\[ Z_\BC(u)  = \bigl( Z_m(u; K, p, p+1, \h, 1)'   , \  Z_m(u; L, q, q, \b, \rho)'  , \  \vech(K(u) r_p(u) u^{p+1})' , \ Z_\sigma(u; K, K, p, p, \h, 1)'  \bigr)', \]
and
\begin{align*}
	Z_\RBC(u)  &  = \bigl( Z_m(u; K, p, p+1, \h, 1)'   , \  Z_m(u; L, q, q, \b, \rho)'  , \  \vech(K(u) r_p(u) u^{p+1})' , \   		\\
	& \qquad \quad Z_\sigma(u; K, K, p, q, \b, \rho)' , \ Z_\sigma(u; L, L, q, q, \b, 1)' , \ Z_\sigma(u; K, L, p, q, \b, \rho)' \bigr)'.
\end{align*}

\noindent {\bf Discussion.} 
This notation is quite compact, and while it emphasizes the simplicity of Cram\'er's condition and the fact that it puts mild restrictions on the kernels, it does obscure the full notational breadth, particularly for $Z_\RBC$. I is also mostly repetitive: what holds for the kernel $K$ and order $p$ fit must also hold for $L$ and $q$, and for their squares and cross products. To make this clear, we can expand all the $Z_m$ and $Z_\sigma$, to write out the full random variables as
\begin{align*}
	Z_\US(u) & = \Bigl( K(u) r_p(u)'\e , \     K(u) r_p(u)'(m(x - u \h) - r_p(u\h)'\beta_p) , \    \vech(K(u) r_p(u) r_p(u)')'  , \    		\\
	& \qquad  \vech(K(u)^2 r_p(u) r_p(u)'\e^2)'  , \    \vech(K(u)^2 r_p(u) r_p(u)'v(x - u\h))'  , \         		\\
	& \qquad  \vech(K(u)^2 r_p(u) r_p(u)'\e(m(x - u \h) - r_p(u\h)'\beta_p))'  , \    \vech(K(u)^2 r_p(u) r_p(u)'r_p(u)' )' , \         		\\
	& \qquad    \vech(K(u)^2 r_p(u) r_p(u)'r_p(u)' \e )' , \  \vech(K(u)^2 r_p(u) r_p(u)'r_p(u)' \e (m(x - u \h) - r_p(u\h)'\beta_p))'   \Bigr)',
\end{align*}
\begin{align*}
	Z_\BC(u) & = \Bigl( K(u) r_p(u)'\e , \    \vech(K(u) r_p(u) r_p(u)')'  , \    		\\
	& \qquad  \vech(K(u)^2 r_p(u) r_p(u)'\e^2)'  , \    \vech(K(u)^2 r_p(u) r_p(u)'v(x - u\h))'  , \         		\\
	& \qquad  \vech(K(u)^2 r_p(u) r_p(u)'\e(m(x - u \h) - r_p(u\h)'\beta_p))'  , \    \vech(K(u)^2 r_p(u) r_p(u)'r_p(u)' )' , \         		\\
	& \qquad    \vech(K(u)^2 r_p(u) r_p(u)'r_p(u)' \e )' , \  \vech(K(u)^2 r_p(u) r_p(u)'r_p(u)' \e (m(x - u \h) - r_p(u\h)'\beta_p))' , \  		\\
	& \qquad     K(u) r_p(u)'(m(x - u \h) - r_{p+1}(u\h)'\beta_{p+1})  , \    L(u\rho) r_q(u\rho)'\e  , \    \vech(L(u\rho) r_q(u\rho) r_q(u\rho)')'  , \    		\\
	& \qquad  \vech(K(u) r_p(u) u^{p+1})' , \     L(u\rho) r_q(u\rho)'(m(x - u \h) - r_q(u\h)'\beta_q)  \Bigr)',
\end{align*}
and
\begin{align*}
	Z_\RBC(u) & = \Big( Z_\BC(u)'  , \  \vech(K(u)^2 r_p(u) r_p(u)'\e^2)'  , \    \vech(K(u)^2 r_p(u) r_p(u)'v(x - u\b))'  , \         		\\
	& \qquad  \vech(K(u)^2 r_p(u) r_p(u)'\e(m(x - u \b) - r_q(u\b)'\beta_q))'  , \    \vech(K(u)^2 r_p(u) r_p(u)'r_q(u\rho)' )' , \         		\\
	& \qquad    \vech(K(u)^2 r_p(u) r_p(u)'r_q(u\rho)' \e )' , \  \vech(K(u)^2 r_p(u) r_p(u)'r_q(u\rho)' \e (m(x - u \b) - r_q(u\b)'\beta_q))'  , \         		\\
	& \qquad  \vech(L(u)^2 r_q(u) r_q(u)'\e^2)'  , \    \vech(L(u)^2 r_q(u) r_q(u)'v(x - u\b))'  , \         		\\
	& \qquad  \vech(L(u)^2 r_q(u) r_q(u)'\e(m(x - u \b) - r_q(u\b)'\beta_q))'  , \    \vech(L(u)^2 r_q(u) r_q(u)'r_q(u)' )' , \         		\\
	& \qquad    \vech(L(u)^2 r_q(u) r_q(u)'r_q(u)' \e )' , \  \vech(L(u)^2 r_q(u) r_q(u)'r_q(u)' \e (m(x - u \b) - r_q(u\b)'\beta_q))'  , \         		\\
	& \qquad  \vech(K(u) L(u \rho) r_p(u) r_q(u \rho)' \e^2)'  , \    \vech(K(u) L(u \rho) r_p(u) r_q(u \rho)' v(x - u\b))'  , \         		\\
	& \qquad  \vech(K(u) L(u \rho) r_p(u) r_q(u \rho)' \e(m(x - u \b) - r_q(u\b)'\beta_q))'  , \    \vech(L(u)^2 r_q(u) r_q(u)'r_q(u)' )' , \         		\\
	& \qquad    \vech(K(u) L(u \rho) r_p(u) r_q(u \rho)' r_q(u)' \e )' , \    \\
	& \qquad   \vech(K(u) L(u \rho) r_p(u) r_q(u \rho)' r_q(u\rho)' \e (m(x - u \b) - r_q(u\b)'\beta_q))'  \ \Bigr)'.
\end{align*}

Finally, the precise random variables $Z_\US(u)$, $Z_\BC(u)$, and $Z_\RBC(u)$ used can be replaced with slightly different constructions without altering the conclusions of Theorem \ref{thm:Edgeworth locpoly}: there are other potential functions $\Ttilde$ that satisfy \Eqref{eqn:delta method 1} in the proof. Such changes necessarily involve asymptotically negligible terms, and do not materially alter the severity of the restrictions imposed.

\begin{remark}[Sufficient Conditions for Cram\'er's Condition]
Assumption \ref{supp:Cramer locpoly} is a high level condition, but one that is fairly mild. It is essentially a continuity requirement, and is discussed at length by (among others) \citet{Bhattacharya-Rao1976_book}, \citet{Bhattacharya-Ghosh1978_AoS}, and \citet{Hall1992_book}. For a recent work in econometrics, the present condition can be compared to that employed by \cite{Kline-Santos2012_JoE} for parametric regression (the role of the covariates is here played by $r_p(X_{\h,i})$): ours is more complex due to the nonparametric smoothing bias and the fact that the expansion is carried out to higher order.

	It is straightforward to provide sufficient conditions for Assumption \ref{supp:Cramer locpoly}, given that Assumptions \ref{supp:dgp locpoly} and \ref{supp:kernel locpoly} hold. In particular, if we additionally assume that $(1, \ \vech(K(u) r_p(u) r_p(u)')')'$ comprises a linearly independent set of functions on $[-1,1]$, then it holds $Z_\US(u)$ has components that are nondegenerate and absolutely continuous, and this will imply that Assumption \ref{supp:Cramer locpoly} holds for $Z_\US(u)$, by arguing as in \citet[Lemma 2.2]{Bhattacharya-Ghosh1978_AoS} and \citet[p. 65]{Hall1992_book}. This is precisely the approach taken by \citet{Chen-Qin2002_SJS}, when studying undersmoothed local linear regression. If the linear independence continues to hold when the set of functions is augmented with $\vech(L(u) r_q(u) r_q(u)')$, then $Z_\BC(u)$ and $Z_\RBC(u)$ satisfy Assumption \ref{supp:Cramer locpoly} as well. 
%	To obtain the result for $Z_\RBC(u)$ requires that linear independence hold for
%	\begin{align*}
%		& \bigl(1, \vech(K(u) r_p(u) r_p(u)'), \ \vech(K(u)^2 r_p(u) r_p(u)' r_q(u)')' , \ \vech(L(u) r_q(u) r_q(u)'), \  		\\
%		& \qquad \vech(L(u)^2 r_q(u) r_q(u)' r_q(u)')' , \  \vech(K(u) L(u\rho) r_p(u) r_q(u\rho)' r_q(u\rho)')'\bigr).
%	\end{align*}

At heart, these are requirements on the kernel functions, just as in Assumption \ref{supp:Cramer density} in the density case. The uniform kernel is again ruled out. See Section \ref{sec:assumptions density}. Further, note that if these sets of functions are not linearly independent, there will exist a there exists a smaller set of functions which are linearly independent and can replace the original set while leaving the value of the statistic unchanged (see \citet[p. 442]{Bhattacharya-Ghosh1978_AoS}). 
\end{remark}

%%%%%%%%%%%%%%%%%%%%%%%%%%%%%%%%%%%
%%%%%%%%%%%%%%%%%%%%%%%%%%%%%%%%%%%
\section{Bias}
	\label{supp:bias locpoly}

We will not present a detailed discussion of bias issues, along the lines of Section \ref{supp:bias density}, for brevity; we focus only on the case of nonbinding smoothness.

The biases $\eus$ and $\ebc$ are not as conceptually simple as in the density case. The closest parallel to the density case would be (for example) $\eus = \sqrt{nh} (\E[\hat{m}] - m)$, but this can not be used due to the presence of $\Gp^{-1}$ inside the expectation, and the next natural choice, the conditional bias $\sqrt{nh} (\E[\hat{m} \vert X_1, \ldots X_n] - m)$, is still random. Instead, $\eus$ and $\ebc$ are biases computed after replacing $\Gp$, $\Gq$, and $\L_{p,1}$ with their expectations, denoted $\Gpt$, $\Gqt$, and $\Lt_{p,1}$. We thus define
\begin{align}
	\begin{split}
		\label{suppeqn:eta locpoly}
		\eus & = \sqrt{nh} \int e_0' \Gpt^{-1} K(u) r_p(u) \left( m(x - u \h) - r_p(u\h)'\beta_p \right) f(x - u\h)du,		\\
		\ebc & = \sqrt{nh} \int e_0' \Gpt^{-1} K(u) r_p(u) \left( m(x - u \h) - r_{p+1}(u\h)'\beta_{p+1} \right) f(x - u\h)du 		\\		
			& \quad 	-  \sqrt{nh} \rho^{p+1} \int e_0' \Gpt^{-1} \Lt_{p,1} e_{p+1}' \Gqt^{-1} L(u) r_q(u) \left( m(x - u \b) - r_q(u\b)'\beta_q \right) f(x - u\b)du.
	\end{split}
\end{align}

For the generic results of coverage error or the generic Edgeworth expansions of Theorem \ref{thm:Edgeworth locpoly} below, the above definitions of $\eus$ and $\ebc$ are suitable. For the Corollaries detailing specific cases, and to understand the behavior at different points, it is useful to make the leading terms precise, that is, analogues of Equations \eqref{suppeqn:bias density} and \eqref{suppeqn:bias corrected}. We must consider interior and boundary point estimation, and even and odd $q$. We depart slightly from other terms of the expansion in that we do retain only the leading term for some pieces. This is done in order to capture the rate of convergence explicitly and to give practicable results. These results are derived by \citet[Section 3.7]{Fan-Gijbels1996_book} and similar calculations (though our expressions differ slightly as fixed-$n$ expectations are retained as much as possible).

Since $p$ is odd, both at boundary and interior points we have
	\[\eus  = \sqrt{nh} \h^{p+1} \frac{ m^{(p+1)} } { (p+1)! } e_0' \Gpt^{-1} \Lt_{p,1} \left[1 + o(1) \right].\]

Moving to $\ebc$, consider the first term, which in the present notation is: $ \sqrt{n \h} \E[ \h^{-1} \l^0_\US(X) (m(X) - r_{p+1}(X - x)' \beta_{p+1})]$. With $p+1$ even, we find that in the interior the leading terms are
	\[ \sqrt{n \h}  \h^{p+3}  e_0' \Gpt^{-1}\left(  \frac{ m^{(p+2)} } { (p+2)! } \Lt_{p,2}  \h^{-1}  + \frac{ m^{(p+3)} } { (p+3)! } \Lt_{p,3}  \right)  \left[1 + o(1) \right],\]
due to the well-known symmetry properties of local polynomials that result in the cancellation of the leading terms of $\Gpt^{-1}$ and $\Lt_{p,2}$. The rate of $\h^{p+3}$ accounts for this. At the boundary, no such cancellation occurs and we have only 
	\[ \sqrt{n \h}  \h^{p+2}   \frac{ m^{(p+2)} } { (p+2)! } e_0'  \Gpt^{-1}  \Lt_{p,2} \left[1 + o(1) \right].\]
Next, turn to the bias of the bias estimate: 
	\[\sqrt{nh} \rho^{p+1} e_0' \Gpt^{-1} \Lt_{p,1} e_{p+1}' \Gqt^{-1} \int  L(u) r_q(u) \left( m(x - u \b) - r_q(u\b)'\beta_q \right) f(x - u\b)du.\]
If $q$ is odd (so that $q-(p+1)$ is also odd), then at the interior or boundary the leading term will be
	\[\sqrt{nh} \b^{q+1} \rho^{p+1} \frac{ m^{(q+1)} } { (q+1)! } e_0' \Gpt^{-1} \Lt_{p,1} e_{p+1}' \Gqt^{-1} \Lt_{q,1} \left[1 + o(1) \right] \asymp \sqrt{nh}\h^{p+1} \b^{q-p}.\]
The same expression applies at the boundary for $q$ even. However, for the interior, if $q$ is even, which it is in the leading case of $q=p+1$, then we again have cancellation of certain leading terms, resulting in the bias of the bias estimate being
	\[\sqrt{nh} \b^{q+2} \rho^{p+1} e_0' \Gpt^{-1} \Lt_{p,1} e_{p+1}' \Gqt^{-1} \left(  \frac{ m^{(q+1)} } { (q+1)! } \Lt_{q,1} \b^{-1}   + \frac{ m^{(q+2)} } { (q+2)! } \Lt_{q,2}  \right)  \left[1 + o(1) \right]    \asymp \sqrt{nh}\h^{p+1} \b^{q+1-p}.\]
Combining all these results, we find the following. For an interior point $\ebc^{\tt int}   =   \sqrt{n \h}  \h^{p+3} \left[ \etbc^{\tt int}  + o(1)\right]$, where, if $q$ is even 
	\begin{multline*}
		\etbc^{\tt int}   =   e_0' \Gpt^{-1}  \biggl( \frac{ m^{(p+2)} } { (p+2)! } \Lt_{p,2}  \h^{-1}  + \frac{ m^{(p+3)} } { (p+3)! } \Lt_{p,3}  \biggr)      		\\
			 -      \rho^{-2} \b^{q-(p+1)}   e_0' \Gpt^{-1} \Lt_{p,1} e_{p+1}' \Gqt^{-1} \biggl(  \frac{ m^{(q+1)} } { (q+1)! } \Lt_{q,1}  \b^{-1}  + \frac{ m^{(q+2)} } { (q+2)! } \Lt_{q,2}  \biggr) ,
	\end{multline*}
while if $q$ is odd,
	\[\etbc^{\tt int}   =     e_0' \Gpt^{-1}\biggl(   \frac{ m^{(p+2)} } { (p+2)! } \Lt_{p,2}  \h^{-1}  + \frac{ m^{(p+3)} } { (p+3)! } \Lt_{p,3}  \biggr)     -      \rho^{-2} \b^{q-(p+2)}   \frac{ m^{(q+1)} } { (q+1)! } e_0' \Gpt^{-1} \Lt_{p,1} e_{p+1}' \Gqt^{-1} \Lt_{q,1} .\]
At the boundary, for any $q$, $\ebc^{\tt bnd}   =   \sqrt{n \h}  \h^{p+2} \left[ \etbc^{\tt bnd}  + o(1)\right]$, with
	\[ \etbc^{\tt bnd}   =     \frac{ m^{(p+2)} } { (p+2)! } e_0'  \Gpt^{-1}  \Lt_{p,2}    -    \rho^{-1}\b^{q-(p+1)}      \frac{ m^{(q+1)} } { (q+1)! } e_0' \Gpt^{-1} \Lt_{p,1} e_{p+1}' \Gqt^{-1} \Lt_{q,1}  .\]

%%%%%%%%%%%%%%%%%%%%%%%%%%%%%%%%%%%
%%%%%%%%%%%%%%%%%%%%%%%%%%%%%%%%%%%
\section{Main Result: Edgeworth Expansion}
	\label{supp:Edgeworth locpoly}

We now state our generic Edgeworth expansion, from whence the coverage probability expansion results follow immediately. We have opted to state separate results for undersmoothing, bias correction, and robust bias correction, rather than the unified statement of Theorem \ref{thm:Edgeworth density}, for clarity. The unified structure is still present, and will be used in the proof of the result below, but is too cumbersome to use here. The Standard Normal distribution and density functions are $\Phi(z)$ and $\phi(z)$, respectively.

\begin{theorem}
	\label{thm:Edgeworth locpoly}
	Let Assumptions \ref{supp:dgp locpoly}, \ref{supp:kernel locpoly}, and \ref{supp:Cramer locpoly} hold, and assume $n \h/ \log(n) \to \infty$.
	\begin{enumerate}

		\item If $\eus \log(n\h) \to 0$, then for 
			\begin{align*}
				F_\US(z) = \Phi(z)  & +  \frac{1}{\sqrt{n \h}} p_{1,\US}(z)    +    \eus  p_{3,\US}(z)   +   \frac{1}{n \h} q_{1,\US}(z)    +  \eus^2 q_{2,\US}(z)  +    \frac{\eus}{\sqrt{n \h}} q_{3,\US}(z),
			\end{align*}
			we have
			\[\sup_{z\in\mathbb{R}}\left|\P[\tus<z]-F_\US(z)\right| = o\left((n\h)^{-1} + (n\h)^{-1/2}\eus + \eus^2 \right).\]

		\item If $\ebc \log(n\h) \to 0$ and $\rho \to 0$, then for 
			\begin{align*}
				F_\BC(z) = \Phi(z)  & +  \frac{1}{\sqrt{n \h}} p_{1,\BC}(z)    +    \ebc  p_{3,\BC}(z)   +   \frac{1}{n \h} q_{1,\US}(z)    +  \ebc^2 q_{2,\BC}(z)  +    \frac{\ebc}{\sqrt{n \h}} q_{3,\BC}(z)    		\\
				& -  \rho^{p+2} (\Omega_1 + \rho^{p+1} \Omega_2) \frac{  \phi(z) }{2} z ,
			\end{align*}
			we have
			\[\sup_{z\in\mathbb{R}}\left|\P[\tbc<z] - F_\BC(z)\right| = o\left((n\h)^{-1} + (n\h)^{-1/2}\ebc + \ebc^2  +  \rho^{1 + 2(p+1)} \right).\]

		\item If $\ebc \log(n\h) \to 0$ and $\rho \to \bar{\rho} < \infty$, then for 
			\begin{align*}
				F_\RBC(z) = \Phi(z)  & +  \frac{1}{\sqrt{n \h}} p_{1,\RBC}(z)    +    \ebc  p_{3,\RBC}(z)   +   \frac{1}{n \h} q_{1,\RBC}(z)    +  \ebc^2 q_{2,\RBC}(z)  +    \frac{\ebc}{\sqrt{n \h}} q_{3,\RBC}(z),
			\end{align*}
			we have
			\[\sup_{z\in\mathbb{R}}\left|\P[\trbc<z]-F_\RBC(z)\right| = o\left((n\h)^{-1} + (n\h)^{-1/2}\ebc + \ebc^2 \right).\]

	\end{enumerate}

\end{theorem}

%%%%%%%%%%%%%%%%%%%%%%%%%%%%%%%%%%%
\subsection{Coverage Error for Undersmoothing}

For undersmoothing estimators, we have the following result, which is valid for both interior and boundary points, with moments appropriately truncated if necessary. This result is the analogue of the robust bias correction corollary in the main text, and follows directly from the generic theorem there or Theorem \ref{thm:Edgeworth locpoly} above. Exponents such as $1 + 2(p+1)$ are intentionally not simplified to ease comparison to other results, particularly the density case.

The polynomials $q_{1,\US}$, $q_{2,\US}$, and $q_{3,\US}$, which do \emph{not} have an argument, are defined in terms of those given in Section \ref{supp:terms} and used in Theorem \ref{thm:Edgeworth locpoly}, which \emph{do} have an argument. Specifically, the polynomials in Section \ref{supp:terms} and Theorem \ref{thm:Edgeworth locpoly} should be doubled, divided by the standard Normal density, and evaluated at the Normal quantile $z_{\alpha/2}$, that is,
\[ q_{k,\US} \defsym \left. \frac{2}{\phi(z)} q_{k,\US}(z)  \right|_{z = z_{\alpha/2} },   		  \qquad\qquad k=1,2,3. \]

\begin{corollary}[Undersmoothing]
	\label{appx:US locpoly}
	Let the conditions of Theorem \ref{thm:Edgeworth locpoly}(a) hold. Then
			\begin{align*}
				\P[m \in \ius]   = 1 - \alpha  &  +   \biggl\{ \frac{1}{n \h} q_{1,\US}  +  n\h^{1 + 2(p+1)} \left( m^{(p+1)}\right)^2 \left(  e_0' \Gpt^{-1} \Lt_{p,1}  / (p+1)! \right)^2 q_{2,\US}     		\\
									& \qquad  +   \h^{p+1} \left( m^{(p+1)}\right) \left(  e_0' \Gpt^{-1} \Lt_{p,1}  / (p+1)! \right) q_{3,\US} \biggr\}   \phi(z_{\frac{\alpha}{2}})\;\{1+o(1)\}.
			\end{align*}
	In particular, if $\h^*_\US = H^*_\US n^{-1/(1+(p+1))}$, then $\P[m \in \ius] = 1 - \alpha + O(n^{-(p+1)/(1+(p+1))})$, where 
	\begin{multline*}
		H^*_\US = \argmin_H \biggl\vert H^{-1}  q_{1,\US}  +  H^{1+2(p+1)} \left( m^{(p+1)}\right)^2 \left(  e_0' \Gpt^{-1} \Lt_{p,1}  / (p+1)! \right)^2 q_{2,\US}     	\\
		    	+   H^{p+1}   \left( m^{(p+1)}\right) \left(  e_0' \Gpt^{-1} \Lt_{p,1}  / (p+1)! \right) q_{3,\US} \biggr\vert.
	\end{multline*}
\end{corollary}

%%%%%%%%%%%%%%%%%%%%%%%%%%%%%%%%%%%
%%%%%%%%%%%%%%%%%%%%%%%%%%%%%%%%%%%
\section{Proof of Main Result}
	\label{supp:proof locpoly}

We will first prove Theorem \ref{thm:Edgeworth locpoly}(a), as it is notationally simplest. From a technical and conceptual point of view, proving the remainder of Theorem \ref{thm:Edgeworth locpoly} is identical, simply more involved notationally due to the additional complexity of the bias correction. Outlines of these proofs are found below.

%%%%%%%%%%%%%%%%%%%%%%%%%%%%%%%%%%%
\subsection{Proof of Theorem \ref{thm:Edgeworth locpoly}(a)}

Let $s = \sqrt{n \h}$. 

Throughout this proof, we will generally omit the subscripts $\US$ and $p$ when this causes no confusion. This entire proof focuses on the undersmoothing statistic,  $\tus = \shatus^{-1} s(\hat{m} - m)$, and since bias correction is not involved at all, the associated constructions such as $\Gq$, $W_q$, etc, do not appear, and hence there is no need to carry the additional notation to distinguish $W_p$ from $W_q$, or $\shatus$ from $\shatrbc$, for example, and we will simply write $\G$ for $\Gp$, $W$ for $W_p$, $\shat$ for $\shatus$, etc.

Our goal is to expand $\P[\tus<z]$, where $\tus = \shat^{-1} s(\hat{m} - m)$. The proof proceeds by identifying a smooth function $\Ttilde = \Ttilde(z)$ such that, for the random variable $Z_\US := Z_\US(u)$ that obeys Cram\'er's condition (Assumption \ref{supp:Cramer locpoly}), $\Ttilde(\E[Z_\US]) = 0$ and
\begin{equation}
	\label{eqn:delta method 1}
	\P \bigl[ \tus < z \bigr] = \P \bigl[ \Ttilde(\bar{Z}_\US) < \tilde{z} \bigr] + o(s^{-2} + s^{-1}\eta + \eta^2),
\end{equation}
where $\bar{Z} = \sumi Z_i/n$ and $\tilde{z}$ is a known, nonrandom quantity that depends on the original quantile $z$ and the remainder $\tus - \Ttilde$. An Edgeworth expansion for $\Ttilde$ holds under Assumption \ref{supp:Cramer locpoly}, and a Taylor expansion of this function around $\tilde{z}$ yields the final result. As in the density case, $\tilde{z}$ will capture the bias terms of $\tus$: in that case $\tilde{z} = z - \eta/\st$, but here bias is present in both the numerator and the Studentization.

To begin, define the notation $\check{R} = \left[ r_p( X_1 - x), \cdots, r_p( X_n - x ) \right]'$ and $M = [m(X_1), \ldots, m(X_n)]'$, and use this to split $T$ into variance and bias terms, as follows:
	\[T = \shat^{-1} s e_0' \G^{-1} R'W(Y-M)/n  +   \shat^{-1} s e_0' \G^{-1} R'W(M - \check{R}\beta)/n.\]
We use this decomposition to rewrite $\P[ \tus < z ]$ as
\begin{align}
	\P\left[ \tus < z \right] & = \P\left[ \tus - \st^{-1} \eta < z  - \st^{-1} \eta \right]  		\nonumber \\
	  & = \P\left[ \left\{  \shat^{-1} s e_0' \G^{-1} R'W(Y-M)/n  +   \shat^{-1} s e_0' \G^{-1} R'W(M - \check{R}\beta)/n - \st^{-1} \eta \right\} < z  - \st^{-1} \eta \right]  		\nonumber \\
	\begin{split}
		\label{eqn:Ttilde}
	  & = \P \biggl[ \Bigl\{ \st^{-1} s e_0' \G^{-1} R'W(Y-M)/n   		\\
		& \qquad \quad + \st^{-1} s e_0' \Gt^{-1} R'W(M - \check{R}\beta)/n - \st^{-1} \eta   		\\
		& \qquad \quad + \st^{-1} s e_0' \left( \G^{-1} - \Gt^{-1} \right) R'W(M - \check{R}\beta)/n     		\\
		& \qquad \quad + \left( \shat^{-1} - \st^{-1} \right) s e_0' \G^{-1} R'W(Y-M)/n     		\\
		& \qquad \quad + \left( \shat^{-1} - \st^{-1} \right) s e_0' \G^{-1} R'W(M - \check{R}\beta)/n  \Bigr\}      < z  - \st^{-1} \eta \biggr].
	\end{split}
\end{align}
The first three lines in the last equality obey the desired properties of $\Ttilde$ by the orthogonality of $\e_i$, the definition of $\eus$ in \Eqref{suppeqn:eta locpoly} as $\E\left[ s e_0' \Gt^{-1} R'W(M - \check{R}\beta)/n \right]$, and the fact that $\G^{-1} - \Gt^{-1} = \Gt^{-1} \left( \Gt - \G \right) \G^{-1}$. For the final two (which are $\tus - \st^{-1} s (\hat{m} - m) = \shat^{-1} - \st^{-1} s (\hat{m} - m) $), we must expand the difference $\shat^{-1} - \st^{-1}$. Accounting for the resulting terms will constitute the bulk of the remainder of the proof, as well as complete the construction of $\tilde{z}$ and the remainder terms of \Eqref{eqn:delta method 1}.\footnote{Technically, to obtain a $\Ttilde$ with the desired properties, one need not expand $\shat^{-1} - \st^{-1}$ for the variance term: that is, in \Eqref{eqn:Ttilde}, $\st^{-1} s e_0' \G^{-1} R'W(Y-M)/n$ and $\left( \shat^{-1} - \st^{-1} \right) s e_0' \G^{-1} R'W(Y-M)/n$ may be collapsed. This requires strengthening Cram\'er's condition (see Section \ref{sec:assumptions locpoly}), and since $\shat^{-1} - \st^{-1}$ must be accounted for in the final bias term, $\left( \shat^{-1} - \st^{-1} \right) s e_0' \G^{-1} R'W(M - \check{R}\beta)/n$, there is little reason not to do both terms.}

To begin, with $\st^2 = e_0 ' \Gt^{-1} \Psit \Gt^{-1} e_0$ defined in Section \ref{supp:terms},
\begin{equation*}
	\frac{1}{\shat} = \frac{1}{\st} \left( \frac{\shat^2}{\st^2} \right)^{-1/2} = \frac{1}{\st} \left( 1 + \frac{ \shat^2 - \st^2}{\st^2} \right)^{-1/2},
\end{equation*}
and hence a Taylor expansion gives
	\[\frac{1}{\shat} = \frac{1}{\st} \left[ 1 - \frac{1}{2} \frac{ \shat^2 - \st^2}{\st^2}  + \frac{3}{8} \left( \frac{ \shat^2 - \st^2}{\st^2} \right)^2   - \frac{1}{3!} \frac{15}{8} \left( \frac{ \shat^2 - \st^2}{\st^2} \right)^3 \frac{\st^7}{\smvt^7}     \right],\]
for a point $\smvt^2 \in [\st^2, \shat^2]$, and so
\begin{equation}
	\label{eqn:variance taylor}
	\shat^{-1} - \st^{-1} =  - \frac{1}{2} \frac{ \shat^2 - \st^2}{\st^3}  + \frac{3}{8} \frac{  \left( \shat^2 - \st^2\right)^2}{\st^5}    -  \frac{5}{16}  \frac{ \left( \shat^2 - \st^2 \right)^3 }{\smvt^7}    .
\end{equation}
We thus focus on $\shat^2 - \st^2$. Recall the definition of $\Psic = h R' W \Sigma W R/n$. Then define the two terms $A_1$ and $A_2$ through the following:
\begin{equation}
	\label{eqn:variance terms}
	\shat^2 - \st^2 = e_0' \G^{-1} \left( \Psihat - \Psic \right) \G^{-1} e_0  +  \left( e_0' \G^{-1} \Psic \G^{-1} e_0 -   e_0' \Gt^{-1} \Psit \Gt^{-1} e_0 \right)  =: A_1 + A_2.
\end{equation}

For $A_1$, recall that $\hat{\e}_i = y_i - r_p(X_i - x)'\bhat_p$ and so
\begin{align}
	\Psihat - \Psic & = \frac{1}{n \h} \sumi (K^2 r_p r_p ')(X_{\h,i}) \left\{ \hat{\e}_i^2 - v(X_i) \right\}  		\nonumber \\
	& = \frac{1}{n \h} \sumi (K^2 r_p r_p ')(X_{\h,i}) \left\{ \left(  y_i - r_p(X_i - x)'\bhat_p \right)^2 - v(X_i) \right\}  		\nonumber \\
	& = \frac{1}{n \h} \sumi (K^2 r_p r_p ')(X_{\h,i}) \left\{ \left( \e_i  + [m(X_i) -  r_p(X_i - x)'\beta_p]   +  r_p(X_i - x)'\left[ \beta_p - \bhat_p\right] \right)^2 - v(X_i) \right\}  		\nonumber \\
	& =: A_{1,1}  +  A_{1,2}  +  A_{1,3}  +  A_{1,4}  +  A_{1,5}  +  A_{1,6}  +  A_{1,7}  +  A_{1,8},  		\label{supp:A1 terms}
\end{align}
where
\begin{equation*}
	\label{supp:true residuals}
	A_{1,1}  = \frac{1}{n \h} \sumi (K^2 r_p r_p ')(X_{\h,i}) \left\{  \e_i^2 - v(X_i) \right\},  		 %A^2
\end{equation*}
is due to the approximation of the (average over the) conditional variance by the squared residuals (i.e.\ $A_{1,1}$ is the sole remainder that would arise if the true residuals were known and used in place of $\hat{\e}_i^2$), and, using $r_p(X_i - x)' \bhat = r_p(X_i - x)' H_p \G^{-1} R'W Y/n  =  r_p(X_{\h,i})'  \G^{-1} R'W Y/n$, the terms $A_{1,k}$, $k=2, 3, \ldots, 8$ are:
\begin{align*}
	A_{1,2} & = \frac{1}{n \h} \sumi (K^2 r_p r_p ')(X_{\h,i}) \left\{ 2 \e_i [m(X_i) -  r_p(X_i - x)'\beta_p] \right\},  		\\  %2AB
	A_{1,3} & = \frac{1}{n \h} \sumi (K^2 r_p r_p ')(X_{\h,i}) \left\{ -2 \e_i r_p(X_{\h,i})' \right\} \G^{-1} R' W(Y - \check{R}\beta)/n,  		\\  %2AC
	A_{1,4} & = \frac{1}{n \h} \sumi (K^2 r_p r_p ')(X_{\h,i}) \left\{ -2 [m(X_i) -  r_p(X_i - x)'\beta_p] r_p(X_{\h,i})' \right\} \G^{-1} R' W (Y - M)/n,  		\\  %2BC term 1, the epsilon term of C
	A_{1,5} & = \frac{1}{n \h} \sumi (K^2 r_p r_p ' r_p')(X_{\h,i})  \G^{-1} R'W(Y-M)/n \left[ (Y-M)'/n + 2 (M- \check{R}\beta)/n \right] W R \G^{-1} r_p(X_{\h,i}),  		\\  %C^2 terms 1 and 2, the epsilon^2 and 2 epsilon x bias terms
	A_{1,6} & = \frac{1}{n \h} \sumi (K^2 r_p r_p ')(X_{\h,i})  [m(X_i) -  r_p(X_i - x)'\beta_p]^2 ,  		\\  %B^2
	A_{1,7} & = \frac{1}{n \h} \sumi (K^2 r_p r_p ' r_p')(X_{\h,i}) \left\{ -2 [m(X_i) -  r_p(X_i - x)'\beta_p]  \right\} \G^{-1} R' W (M - \check{R} \beta)/n,  		  %2BC term 2, the bias term of C
\intertext{and}
	A_{1,8} & = \frac{1}{n \h} \sumi (K^2 r_p r_p ' r_p')(X_{\h,i})  \G^{-1} [R'W (M - \check{R}\beta)/n] [(M - \check{R}\beta)'/n W R] \G^{-1} r_p(X_{\h,i}).       		%C^2 bias terms
\end{align*}
With this notation, we can write $A_1 = e_0' \G^{-1} \left( \Psihat - \Psic \right) \G^{-1} e_0 = e_0' \G^{-1} \left( \sum_{k=1}^8 A_{1,k} \right) \G^{-1} e_0$. The terms $A_{1,1}$ to $A_{1,5}$ will be incorporated into $\Ttilde$: notice that these terms obey $A_{1,k} = A_{1,k}(\bar{Z}_\US)$ and $A_{1,k}(\E[Z_\US]) = 0$, and hence these properties will be inherited in the final two lines of \Eqref{eqn:Ttilde}. However, $A_{1,6}$, $A_{1,7}$, and $A_{1,8}$ do not have these properties, and will thus be incorporated into $\tilde{z}$ and the remainder. Details are below.

Turning to $A_2$ in \Eqref{eqn:variance terms}, using the identity $\G^{-1} - \Gt^{-1} = \Gt^{-1} \left( \Gt - \G \right) \G^{-1}$ and that $\G$ and $\Psi$ are symmetric, we find that
\begin{align*}
	A_2 & = e_0' \G^{-1} \Psic \G^{-1} e_0 -   e_0' \Gt^{-1} \Psit \Gt^{-1} e_0  		\nonumber \\
	& = e_0' \G^{-1} \left( \Psic - \Psit \right) \G^{-1} e_0  +   e_0' \left( \G^{-1} - \Gt^{-1} \right) \Psit \G^{-1} e_0  +   e_0' \left( \G^{-1} - \Gt^{-1} \right) \Psit \Gt^{-1} e_0  		\nonumber \\
	& = e_0' \G^{-1} \left( \Psic - \Psit \right) \G^{-1} e_0  -   e_0' \Gt^{-1} \left( \G - \Gt \right) \G^{-1} \Psit \left( \G^{-1} + \Gt^{-1} \right) e_0. 		\label{eqn:A2} 
\end{align*}
All of these terms obey the required properties of $\Ttilde$.

We now collect the terms from expanding $\shat^{-1} - \st^{-1}$ and return to \Eqref{eqn:Ttilde}. Plugging the terms $A_{1,1}$--$A_{1,8}$ and $A_2$ into the Taylor expansion in \Eqref{eqn:variance taylor}, by way of \Eqref{eqn:variance terms}, and collecting terms appropriately (i.e. those that belong in $\Ttilde$ as described above), we have the following, which picks up from \Eqref{eqn:Ttilde} and is a precursor to \Eqref{eqn:delta method 1}:
\begin{equation}
	\label{eqn:delta method 2}
	\P[\tus < z ]  = \P \left[ \Ttilde(\bar{Z}_\US) + U < \tilde{z} \right].
\end{equation}
In this statement, we have made the following constructions:
\begin{align*}
	\Ttilde & =  \st^{-1} s e_0' \G^{-1} R'W(Y-M)/n   		\\
		& \quad  + \st^{-1} s e_0' \Gt^{-1} R'W(M - \check{R}\beta)/n - \st^{-1} \eta   		\\
		& \quad  + \st^{-1} s e_0' \left( \G^{-1} - \Gt^{-1} \right) R'W(M - \check{R}\beta)/n     		\\
		& \quad  + \left\{ - \frac{1}{2 \st^3} \left[ e_0' \G^{-1} \left( \sum\nolimits_{k=1}^5 A_{1,k} \right) \G^{-1} e_0 + A_2  \right] + \frac{3}{8 \st^5} \left[ e_0' \G^{-1} A_{1,1} \G^{-1} e_0 + A_2 \right]^2 \right\}     		\\
		& \qquad \qquad \times \Bigl\{  s e_0' \G^{-1} R'W(Y-M)/n  +  s e_0' \G^{-1} R'W(M - \check{R}\beta)/n \Bigr\},
\end{align*}
\begin{align*}
	U & = \left\{ - \frac{1}{2 \st^3}  e_0' \G^{-1} \left( A_{1,6} + A_{1,7} + A_{1,8} \right) \G^{-1} e_0   + \frac{3}{8 \st^5} \left[ e_0' \G^{-1} \left( \sum\nolimits_{k=2}^8 A_{1,k} \right) \G^{-1} e_0  \right]^2  -  \frac{5}{16}  \frac{ \left( \shat^2 - \st^2 \right)^3 }{\smvt^7} \right\}     		\\
		& \qquad \qquad \times \Bigl\{  s e_0' \G^{-1} R'W(Y-M)/n  +  s e_0' \G^{-1} R'W(M - \check{R}\beta)/n \Bigr\}   		\\
	& \quad - \left\{ - \frac{1}{2 \st^3}  e_0' \Gt^{-1} \left( \tilde{A}_{1,6} + \tilde{A}_{1,7} + \tilde{A}_{1,8} \right) \Gt^{-1} e_0     \right\} \,  \eta,
\end{align*}
and
\begin{align*}
	\tilde{z} & = z  - \left\{\st^{-1}  - \frac{1}{2 \st^3}  e_0' \Gt^{-1} \left( \tilde{A}_{1,6} + \tilde{A}_{1,7} + \tilde{A}_{1,8} \right) \Gt^{-1} e_0     \right\} \,  \eta.
\end{align*}
In $U$ and $\tilde{z}$, each $\tilde{A}_{1,k} $ is $A_{1,k}$ where all elements have been replaced by their respective fixed-$n$ expected values, that is,
	\begin{align*}
		\tilde{A}_{1,6} & = \E[A_{1,6}] =  \E\left[ \h^{-1} (K^2 r_p r_p ')(X_{\h,i})  \left[m(X_i) -  r_p(X_i - x)'\beta_p \right]^2  \right],  		\\
		\tilde{A}_{1,7} & = -2 \E \Bigl[ \h^{-1} (K^2 r_p r_p ' r_p')(X_{\h,i})  \left[m(X_i) -  r_p(X_i - x)'\beta_p \right]   \Bigr]    		\\
		   		& \qquad \qquad \qquad \times \Gt^{-1} \E \Bigl[ \h^{-1} (K r_p)(X_{\h,j}) \left[ m(X_j) -  r_p(X_j - x)'\beta_p \right] \Bigr],
	\end{align*}
and
	\[\tilde{A}_{1,8} = \E\left[ \h^{-1}  (K^2 r_p r_p ' )(X_{\h,i})  \E\left[ \left.  \h^{-1} r_p(X_{\h,i}) '  \Gt^{-1} (K r_p)(X_{\h,j}) \left[m(X_j) -  r_p(X_j - x)'\beta_p \right] \right\vert X_i \right]^2 \right].\]

The next step in the proof is to show that, for $r_* = \max\{s^{-2}, \eta^2, \h^{p+1} \}$ (i.e., the slowest decaying), it holds that 
\begin{equation}
	\label{eqn:delta method 3}
	\frac{1}{r_*} \P [|U| > r_n] \to 0 , \qquad \qquad \text{ for some } r_n = o(r_*).
\end{equation}
This result is established by Lemma \ref{lem:undersmoothing} in Section \ref{supp:lemmas locpoly} below. This, together with \Eqref{eqn:delta method 2}, implies \Eqref{eqn:delta method 1}. 
	
Under Assumption \ref{supp:Cramer locpoly}, an Edgeworth expansion holds for $\Ttilde$ up to $o(s^{-2} + s^{-1}\eta + \eta^2)$. Thus, for a smooth function $G(z)$, we have $\P[\Ttilde < z] = G(z) + o(s^{-2} + s^{-1}\eta + \eta^2)$. Therefore, a Taylor expansion gives
	\[ \P[\Ttilde < \tilde{z}] = G(z) -  G^{(1)}(z) \left\{\st^{-1}  - \frac{1}{2 \st^3}  e_0' \G^{-1} \left( \tilde{A}_{1,6} + \tilde{A}_{1,7} + \tilde{A}_{1,8} \right) \G^{-1} e_0     \right\} + o(s^{-2} + s^{-1}\eta + \eta^2),\]
which together with \Eqref{eqn:delta method 1} establishes the validity of the Edgeworth expansion. The terms of the expansion are computed in Section \ref{supp:moments locpoly} below.\qed

%%%%%%%%%%%%%%%%%%%%%%%%%%%%%%%%%%%
\subsection{Proof of Theorem \ref{thm:Edgeworth locpoly}(b) \& (c)}

To prove parts (b) and (c) of Theorem \ref{thm:Edgeworth locpoly} the same steps are required, and so we will not pursue all the details here. Indeed, the same expansions are performed and the same bounds computed on objects which are conceptually similar, only taking into account the bias correction (in the numerator for (b), and also in the denominator for (c)). The bias correction will result in essentially two changes: first, many more terms like $\G - \Gt$ appear, and second, the bias expressions and rates change. To illustrate, we will list several key points where these changes manifest. This list is not exhaustive, but it will show that the same methods used above still apply.

First, for the numerator of $\tbc$ and $\trbc$, recall that the estimator $\hat{m}$ is
	\[\hat{m}  = \Bigl\{ e_0' \Gp^{-1}  R_p' W_p   \Bigr\}Y / n, \]
while the bias corrected estimator is
	\[\hat{m} - \hat{B}_m  = \Bigl\{ e_0' \Gp^{-1} \left(  R_p' W_p  - \rho^{p+1}   \L_{p,1} e_{p+1}' \Gq^{-1} R_q' W_q \right) \Bigr\}Y / n. \]
Comparing these two expressions, it can be seen that the terms in the proof above that involve $\Gp - \Gpt$ will now additionally involve $\Gq - \Gqt$ and $\L_{p,1} - \Lt_{p,1}$, whereas those that with $e_0' \Gpt^{-1} R_p'W_p$ will now have $e_0' \Gpt^{-1} \left(  R_p' W_p  - \rho^{p+1}   \Lt_{p,1} e_{p+1}' \Gqt^{-1} R_q' W_q \right)$ instead. To give a concrete example, consider the third line of \Eqref{eqn:Ttilde},
	\[\stus^{-1} s e_0' \left( \Gp^{-1} - \Gpt^{-1} \right) R_p'W_p(M - \check{R}_p\beta_p)/n,\]
which becomes a piece of the function $\Ttilde$. For part (b) Theorem \ref{thm:Edgeworth locpoly}, treating $\tbc$, this will become
	\begin{multline*}
		\stus^{-1} s e_0'   \left( \Gp^{-1} - \Gpt^{-1} \right)   R_p'W_p(M - \check{R}_{p+1}\beta_{p+1})/n    		\\		  	- s e_0' \rho^{p+1} \left( \Gp^{-1} \L_{p,1} e_{p+1}' \Gq^{-1} - \Gpt^{-1} \Lt_{p,1} e_{p+1}' \Gqt^{-1} \right) R_q' W_q (M - \check{R}_q \beta_q)/n,
	\end{multline*}
and part (c) will have the same but with $\strbc^{-1}$. Then, since
	\begin{multline*}
		\Gp^{-1} \L_{p,1} e_{p+1}' \Gq^{-1} - \Gpt^{-1} \Lt_{p,1} e_{p+1}' \Gqt^{-1} = \left( \Gp^{-1} - \Gpt^{-1} \right) \L_{p,1} e_{p+1}' \Gq^{-1}      		\\		  	 +  \Gpt^{-1} \left(\L_{p,1} - \Lt_{p,1} \right) e_{p+1}' \Gq^{-1}  +  \Gpt^{-1} \Lt_{p,1} e_{p+1}' \left( \Gq^{-1} - \Gqt^{-1} \right),
	\end{multline*}
this term is handled identically, since the appropriate Cram\'er's condition is assumed.

Consider now the denominator of the Studentized statistics. For part (b), there is no change as $\shatus^2$ is still used, and so the terms involving $A_{1,k}$ and $A_2$ will be identical. However, for $\trbc$, we must account for changes of the above form, but also that the residuals are estimated with the degree $q$ fit: $\hat{\e}_ i = y_i - r_q(X_i - x)'\bhat_q$ instead of degree $p$. With these changes in mind, the analogue of \Eqref{eqn:variance terms} will be
\begin{equation}
	\label{eqn:RBC variance terms}
	\shatrbc^2 - \strbc^2 = e_0' \Gp^{-1} \left( \Psihat_q - \Psic_q \right) \Gp^{-1} e_0  +   \left( e_0' \Gp^{-1} \Psic_q \Gp^{-1} e_0 -   e_0' \Gpt^{-1} \Psit_q \Gpt^{-1} e_0 \right).
\end{equation}
The second term will proceed as above, though $\Psic_p - \Psit_p$ will be replaced by
	\[\Psic_q - \Psit_q =  \frac{1}{n \h} \sumi \left\{\tilde{\ell}^0_\BC(X_i) \tilde{\ell}^0_\BC(X_i)' v(X_i)  - \E \left[\tilde{\ell}^0_\BC(X_i) \tilde{\ell}^0_\BC(X_i)' v(X_i)\right]\right\},  \]
where $\tilde{\ell}^0_\BC(X_i) =  (K r_p)(X_{\h,i}) - \rho^{p+1} \Lt_{p,1} \Gqt^{-1} (L r_p)(\rho X_{\h,i})$ (cf. Section \ref{supp:terms}, the function $\ell^0_\BC$ therein is $\ell^0_\BC(X_i) = e_0'\Gpt^{-1}  \tilde{\ell}^0_\BC(X_i)$). To use similar notation, 
	\[\Psic_p - \Psit_p = \frac{1}{n \h} \sumi \left\{\tilde{\ell}^0_\US(X_i) \tilde{\ell}^0_\US(X_i)' v(X_i)  - \E \left[\tilde{\ell}^0_\US(X_i) \tilde{\ell}^0_\US(X_i)' v(X_i)\right]\right\}.\]
Then, expanding $\tilde{\ell}^0_\BC(X_i)$ shows that $\Psic_q - \Psit_q$ is equal to
\begin{align*}
	&  \left(\Psic_p - \Psit_p\right) +  \rho^{2(p+1)+1}  \Lt_{p,1} \Gqt^{-1} \frac{1}{n \b} \sumi \left\{ (L^2 r_q r_q')(X_{\b,i}) v(X_i) - \E\left[ (L^2 r_q r_q')(X_{\b,i}) v(X_i) \right] \right\}  \Gqt^{-1}  \Lt_{p,1}   		\\
	&  \quad  - \rho^{(p+1) + 1} 2  \frac{1}{n \h} \sumi \left\{ (K r_p) (X_{\h,i}) (L r_q')(\rho X_{\h,i}) v(X_i) - \E\left[ (K r_p) (X_{\h,i}) (L r_q')(\rho X_{\h,i}) v(X_i) \right] \right\} \Gqt^{-1}  \Lt_{p,1},
\end{align*}
and since all these terms still obey the appropriate Cram\'er's condition, the same steps apply. (The extra factor of $\rho$ in $\rho^{2(p+1)+1}$ and $\rho^{(p+1) + 1}$ accounts for the fact that $\shatrbc^2$ is scaled by $(n\h)$ instead of $(n\b)$, but the $W_q$ matrixes contribute a $\b^{-1}$.)

The first term of \Eqref{eqn:RBC variance terms} will also follow by the same method as in the prior proof, but more care must be taken as many more terms will be present because $\Psihat_q - \Psic_q$ consists of the following three terms, representing the variance of $\hat{m}$, the variance of $\hat{B}_m$, and their covariance, respectively:
\begin{align*}
	\hspace{-0.35in} \Psihat_q - \Psic_q &= \h R_p' W_p \left( \Shat_q - \Sigma \right) W_p R_p / n  		\\
	& \quad + \h \rho^{2(p+1)} \L_{p,1} \Gq^{-1} \left( R_q' W_q \Shat_q W_q R_q \right) \Gq^{-1} \L_{p,1}' / n -  \h \rho^{2(p+1)} \Lt_{p,1} \Gqt^{-1} \left( R_q' W_q \Sigma W_q R_q \right) \Gqt^{-1} \Lt_{p,1}' / n   		\\
	& \quad - 2\h \rho^{p+1}  R_p' W_p \left( \Shat_q W_q R_q \Gp^{-1} \L_{p,1}' \G - \Sigma W_q R_q \Gpt^{-1} \Lt_{p,1}' \right)  / n. 
\end{align*}
The first of these three is as in the prior proof, and yields the same $A_{1,1}$--$A_{1,8}$, only with the bias of a $q$-degree fit: $m(X_i) - r_q(X_i - x)'\beta_q$. If we define
	\[\check{\Psic}_q :=  \frac{1}{n\b} \sumi (L^2 r_q r_q')(X_{\b,i}) v(X_i)\]
then the second term of $\Psihat_q - \Psic_q$ is equal to
\begin{align*}
	& \rho^{1 + 2(p+1)} \L_{p,1} \Gq^{-1} \left\{  \frac{1}{n\b} \sumi (L^2 r_q r_q')(X_{\b,i}) \left\{ \hat{\e}_i^2 -  v(X_i) \right\} \right\} \Gq^{-1} \L_{p,1}  		\\
	& \quad + \rho^{1 + 2(p+1)} \left( \L_{p,1} - \Lt_{p,1} \right) \Gq^{-1} \check{\Psic}_q  \Gq^{-1} \L_{p,1}  		\\
	& \quad + \rho^{1 + 2(p+1)} \Lt_{p,1} \left( \Gq^{-1} - \Gqt^{-1} \right) \check{\Psic}_q  \Gq^{-1} \L_{p,1}  		\\
	& \quad + \rho^{1 + 2(p+1)} \Lt_{p,1} \Gqt^{-1} \check{\Psic}_q  \left( \Gq^{-1} - \Gqt^{-1} \right) \L_{p,1}  		\\
	& \quad + \rho^{1 + 2(p+1)} \Lt_{p,1} \Gqt^{-1} \check{\Psic}_q  \Gqt^{-1} \left( \L_{p,1} - \Lt_{p,1} \right).
\end{align*}
The first of these terms will also give rise to versions of $A_{1,1}$--$A_{1,8}$, only with the bias of a $q$-degree fit and changing $K$ to $L$, $p$ to $q$, $\h$ to $\b$, etc, and will thus be treated exactly as above. The rest of these are incorporated into $\Ttilde_\RBC$, similar to how $A_2$ is treated, because Cram\'er's condition is satisfied. The third and final piece of $\Psihat_q - \Psic_q$ is equal to
\begin{align*}
	& -2 \rho^{1 + (p+1)} \left\{ \frac{1}{n\h} \sumi (K r_p)(X_{\h,i}) (L r_q')(X_{\h,i} \rho) \left\{ \hat{\e}_i^2 - v(X_i) \right\} \right\} \Gq^{-1} \L_{p,1}'  		\\
	& \quad - 2 \rho^{1 + (p+1)} \check{\Psic}_q  \left( \Gq^{-1} - \Gqt^{-1} \right) \L_{p,1}'  		\\
	& \quad - 2 \rho^{1 + (p+1)} \check{\Psic}_q  \Gqt^{-1} \left( \L_{p,1} - \Lt_{p,1} \right),
\end{align*}
and thus is entirely analogous, with yet another version of $A_{1,1}$--$A_{1,8}$ defined for the remainder in the first line, and the second two easily incorporated into $\Ttilde_\RBC$.

From these arguments, it is clear that the analogue of Lemma \ref{lem:undersmoothing} will hold for these cases as well: the same fundamental pieces are involved, and thus the same arguments will apply, just as above.

%%%%%%%%%%%%%%%%%%%%%%%%%%%%%%%%%%%
\subsection{Lemmas}
	\label{supp:lemmas locpoly}

Our proof of Theorem \ref{thm:Edgeworth locpoly} relies on the following lemmas. The first gives generic results used to derive rate bounds on the probability of deviations of the necessary terms. Some such results are collected in Lemma \ref{lem:results}. Lemma \ref{lem:undersmoothing} shows how to use the previous results to establish negligibility of the remainder terms required for \Eqref{eqn:delta method 3}. 

As above, we will generally omit the details required for Theorem \ref{thm:Edgeworth locpoly} parts (b) and (c), to save space. These are entirely analogous, as can be seen from the steps in Lemma \ref{lem:results}. Indeed, the first results are stated in terms of the kernel $K$ and bandwidth $\h$, but continue to hold for $L$ and $\b$ under the obvious substitutions and appropriate assumptions.

Throughout proofs $C$ shall be a generic conformable constant that may take different values in different places. If more than one constant is needed, $C_1$, $C_2$, \ldots, will be used.

\begin{lemma}
	\label{lem:first}
	Let the conditions of Theorem \ref{thm:Edgeworth locpoly} hold and let $g(\cdot)$ and $t(\cdot)$ be continuous scalar functions.
	\begin{enumerate}[ref=\ref{lem:first}(\alph{*})]
	
		\item   \label{lem:bounded}  For some $\delta > 0$, 
			\[ s^2 \P\left[ \left\vert s^{-2} \sumi \left\{(Kt)(X_{\h,i}) g(X_i) - \E[(Kt)(X_{\h,i}) g(X_i)] \right\}\right\vert > \delta s^{-1} \log(s)^{1/2} \right] \to 0.\]

		\item  \label{lem:truncation}  For some $\delta > 0$, 
			\[  s^2 \P\left[ \left\vert s^{-1} \sumi \left\{(Kt)(X_{\h,i}) g(X_i) \e_i \right\} \right\vert > \delta \log(s)^{1/2} \right] \to 0.\] 
			The same holds with $\e_i^2 - v(X_i)$ in place of $\e_i$, since it is conditionally mean zero and has more than four moments.

		\item \label{lem:bias 1} For any $\delta >0$, an integer $k$, and any $\gamma>0$, 
\[\frac{1}{\h^{p+1}} \P\left[ \left\vert s^{-2} \sumi (Kt)(X_{\h,i}) g(X_i)  \left[m(X_i) - r_p(X_i - x)'\beta_p\right]^k  \right\vert > \delta \h^{(k-1)(p+1)} \log(s)^\gamma \right] \to 0.\]

		\item  \label{lem:bias 2} For any $\delta >0$ and any $\gamma>0$, 
\[s^2 \P\left[ \left\vert s^{-2} \sumi (Kt)(X_{\h,i}) g(X_i)  \e_i \left[m(X_i) - r_p(X_i - x)'\beta_p\right] \right\vert > \delta \h^{p+1} \log(s)^\gamma \right] \to 0.\]

		\item  \label{lem:bias 3} For any $\delta >0$, an integer $k$, and any $\gamma>0$, 
			\begin{align*}
				& s^2 \P\biggl[ \biggl\vert s^{-2} \sumi  \Bigl\{ (Kt)(X_{\h,i}) g(X_i) (m(X_i) - r_p(X_i - x)'\beta_p)^k     		\\
				& \quad\qquad \qquad - \E\left[(Kt)(X_{\h,i}) g(X_i) (m(X_i) - r_p(X_i - x)'\beta_p)^k\right] \Bigr\}\biggr\vert > \delta \h^{k(p+1)} \log(s)^\gamma \biggr] \to 0.
			\end{align*}
		
	\end{enumerate}

\end{lemma}

\begin{proof}[Proof of Lemma \ref{lem:bounded}]
Because the kernel function has compact support and $t$ and $g$ are continuous, we have
	\[\left\vert (Kt)(X_{\h,i}) g(X_i) - \E[(Kt)(X_{\h,i}) g(X_i)] \right\vert < C_1.\]
Further, by a change of variables and using the assumptions on $f$, $g$ and $t$:
	\begin{align*}
		\V[(Kt)(X_{\h,i}) g(X_i)] \leq \E \left[ (Kt)(X_{\h,i})^2 g(X_i)^2 \right]  &  = \int f(X_i) (Kt)(X_{\h,i})^2 g(X_i)^2 dX_i 		\\
			&= \h \int f(x - u \h) g(x - u \h) (Kt)(u)^2 du \leq C_2 \h.
	\end{align*}
Therefore, by Bernstein's inequality
\begin{align*}
	s^2 \P & \left[ \left\vert \frac{1}{s^2} \sumi \left\{(Kt)(X_{\h,i}) g(X_i) - \E[(Kt)(X_{\h,i}) g(X_i)] \right\} \right\vert > \delta s^{-1} \log(s)^{1/2} \right]  		\\
	& \leq 2 s^2 \exp \left\{ - \frac{ (s^4) (\delta s^{-1} \log(s)^{1/2})^2 / 2}{ C_2 s^2 + C_1 s^2 \delta s^{-1} \log(s)^{1/2} / 3} \right\} 		\\
	& = 2 \exp\{2\log(s)\} \exp \left\{ - \frac{ \delta^2 \log(s) / 2}{ C_2 + C_1 \delta s^{-1} \log(s)^{1/2} / 3} \right\} 		\\
	& = 2 \exp \left\{ \log(s) \left[ 2 - \frac{ \delta^2 / 2}{ C_2 + C_1 \delta s^{-1} \log(s)^{1/2} / 3} \right] \right\},
\end{align*}
which vanishes for any $\delta$ large enough, as $s^{-1} \log(s)^{1/2} \to 0$.
\end{proof}

\begin{proof}[Proof of Lemma \ref{lem:truncation}]
For a sequence $r_n \to \infty$ to be given later, define
\[H_i = s^{-1} (Kt)(X_{\h,i}) g(X_i) \left(Y_i \mathbbm{1}\{Y_i \leq r_n\} - \E[Y_i \mathbbm{1}\{Y_i \leq r_n\} \mid X_i]\right)  \]
and
\[T_i = s^{-1}(Kt)(X_{\h,i}) g(X_i) \left(Y_i \mathbbm{1}\{Y_i > r_n\} - \E[Y_i \mathbbm{1}\{Y_i > r_n\} \mid X_i]\right).  \]
By the conditions on $g(\cdot)$ and $t(\cdot)$ and the kernel function, 
	\[\left\vert H_i \right\vert <  C_1 s^{-1} r_n\]
and
\begin{align*}
	\V[H_i] = s^{-2} \V[(Kt)(X_{\h,i}) g(X_i) Y_i \mathbbm{1}\{Y_i \leq r_n\}] & \leq s^{-2} \E\left[(Kt)(X_{\h,i})^2 g(X_i)^2 Y_i^2 \mathbbm{1}\{Y_i \leq r_n\} \right]		\\
	& \leq s^{-2} \E\left[(Kt)(X_{\h,i})^2 g(X_i)^2 Y_i^2 \right]  		\\
	& = s^{-2} \int (Kt)(X_{\h,i})^2 g(X_i)^2 v(X_i) f(X_i) d X_i    		\\
	& = s^{-2} \h \int (Kt)(u)^2 (gvf)(x - u\h) d u   		\\
	& \leq C_2 /n.
\end{align*}
Therefore, by Bernstein's inequality
\begin{align*}
	s^2 \P  \left[ \left\vert  \sumi H_i \right\vert > \delta \log(s)^{1/2} \right]  	& \leq 2 s^2 \exp \left\{ - \frac{ \delta^2 \log(s)/ 2}{ C_2 + C_1 s^{-1} r_n \delta \log(s)^{1/2} / 3} \right\} 		\\
	& \leq 2  \exp\{2\log(s)\}  \exp \left\{ - \frac{ \delta^2 \log(s)/ 2}{ C_2 + C_1 s^{-1} r_n \delta \log(s)^{1/2} / 3} \right\} 		\\
	& \leq 2   \exp \left\{\log(s) \left[ 2  - \frac{ \delta^2 / 2}{ C_2 +  C_1 s^{-1} r_n \delta \log(s)^{1/2} / 3} \right] \right\},
\end{align*}
which vanishes for $\delta$ large enough as long as $s^{-1} r_n \log(s)^{1/2}$ does not diverge.

Next, by Markov's inequality and the moment condition on $Y$ of Assumption \ref{supp:dgp locpoly} 
\begin{align*}
	s^2 \P \left[ \left\vert  \sumi T_i \right\vert > \delta \log(s)^{1/2} \right]  & \leq  s^{2} \frac{1}{\delta^2 \log(s)} \E\left[ \left\vert  \sumi T_i \right\vert^2 \right]		\\
	& \leq  s^{2} \frac{1}{\delta^2 \log(s)} n \E \left[ T_i^2 \right]		\\
	& \leq  s^{2} \frac{1}{\delta^2 \log(s)} n \V \left[ s^{-1} (Kt)(X_{\h,i}) g(X_i) Y_i \mathbbm{1}\{Y_i > r_n\} \right]		\\
	& \leq  s^{2} \frac{1}{\delta^2 \log(s)} n s^{-2} \E \left[  (Kt)(X_{\h,i})^2 g(X_i)^2 Y_i^2 \mathbbm{1}\{Y_i > r_n\} \right]		\\
	& \leq  s^{2} \frac{1}{\delta^2 \log(s)} n s^{-2} \E \left[   (Kt)(X_{\h,i})^2 g(X_i)^2 \vert Y_i \vert^{2+\xi} r_n^{-\eta} \right]		\\
	& \leq  s^{2} \frac{1}{\delta^2 \log(s)} n s^{-2} (C \h r_n^{-\xi}) 		\\
	& \leq  \frac{C}{\delta^2} \frac{s^2}{ \log(s) r_n^{\xi}},
\end{align*}
which vanishes if $s^2 \log(s)^{-1} r_n^{-\xi}\to 0$. 

It thus remains to choose $r_n$ such that $s^{-1} r_n \log(s)^{1/2}$ does not diverge and $s^2 \log(s)^{-1} r_n^{-\xi}\to 0$. This can be accomplished by setting $r_n = s^{\gamma}$ for any $2/\xi \leq \gamma <1$, which is possible as $\xi > 2$. 
\end{proof}

\begin{proof}[Proof of Lemma \ref{lem:bias 1}]
By Markov's inequality
\begin{align*}
	\frac{1}{\h^{p+1}} \P  & \left[ \left\vert s^{-2} \sumi (Kt)(X_{\h,i}) g(X_i)  \left[m(X_i) - r_p(X_i - x)'\beta_p\right]^k  \right\vert > \delta \h^{(k-1)(p+1)} \log(s)^\gamma \right]   		\\
	& \leq \frac{1}{\h^{p+1}} \frac{1}{\delta \h^{(k-1)(p+1)} \log(s)^\gamma} \E\left[ \h^{-1} (Kt)(X_{\h,i}) g(X_i)  \left[m(X_i) - r_p(X_i - x)'\beta_p\right]^k\right]   		\\
	& \leq  \frac{1}{\delta \h^{k(p+1)}\log(s)^\gamma} \h^{k(p+1)}  \E\left[ \h^{-1} (Kt)(X_{\h,i}) g(X_i)  \left[\h^{-p-1}(m(X_i) - r_p(X_i - x)'\beta_p)\right]^k\right]   		\\
	& = O(\log(s)^{-\gamma}) = o(1). 
\end{align*}
This relies on the following calculation, which uses the conditions placed on $m(\cdot)$:
\begin{align*}
	\E & \left[ \h^{-1} \left( (Kt)(X_{\h,i}) g(X_i)  \e_i \right) \left[m(X_i) - r_p(X_i - x)'\beta_p\right]^k \right]   		\\
	& =\h^{-1}  \int (gfv)(X_i) (Kt)(X_{\h,i})   \left[m(X_i) - r_p(X_i - x)'\beta_p\right]^k dX_i 			\\
	& = \h^{-1} \int (gfv)(X_i) (Kt)(X_{\h,i})   \left( \frac{m^{(p+1)}(\bar{x})}{(p+1)!}  (X_i - x)^{p+1}  \right)^k dX_i 		\\
	& = \h^{k(p+1)}\h^{-1}  \int (gfv)(X_i) (Kt)(X_{\h,i})   \left( \frac{m^{(p+1)}(\bar{x})}{(p+1)!}  X_{\h,i}^{p+1}  \right)^k dX_i 		\\
	& = C \h^{k(p+1)} \h^{-1} \int (gfv)(X_i) (Kt)(X_{\h,i})  X_{\h,i}^{k(p+1)} dX_i 		\\
	& = C \h^{k(p+1) } \int (gfv)(x - u\h) (Kt)(u)  u^{k(p+1)} du 		\\
	& \asymp \h^{k(p+1) }. \qedhere
\end{align*}
\end{proof}

\begin{proof}[Proof of Lemma \ref{lem:bias 2}]
By Markov's inequality, since $\e_i$ is conditionally mean zero, we have
\begin{align*}
	s^2 \P & \left[  \left\vert s^{-2} \sumi (Kt)(X_{\h,i}) g(X_i)  \e_i \left[m(X_i) - r_p(X_i - x)'\beta_p\right] \right\vert > \delta \h^{p+1} \log(s)^\gamma \right]   		\\
	& \leq s^2 \frac{1}{\delta \h^{2(p+1)} \log(s)^{2\gamma} } \frac{1}{s^2} \E \left[ \h^{-1} \left( (Kt)(X_{\h,i}) g(X_i)  \e_i \right)^2 \left[m(X_i) - r_p(X_i - x)'\beta_p\right]^2 \right]   		\\
	& \leq \frac{s^2 \h^{2(p+1)} }{\delta s^2 \h^{2(p+1)} \log(s)^\gamma } \E \left[ \h^{-1} \left( (Kt)(X_{\h,i}) g(X_i)  \e_i \right)^2 \left[\h^{-p-1}(m(X_i) - r_p(X_i - x)'\beta_p)\right]^2 \right]  		\\
	& \asymp \log(s)^{-2\gamma} \to 0,
\end{align*}
where we rely on the same argument as above to compute the bias rate.
\end{proof}

\begin{proof}[Proof of Lemma \ref{lem:bias 3}]
Follows from identical steps to \ref{lem:bias 2}.
\end{proof}

To illustrate how the above Lemma is used for the objects under study, we present the following collection of results. This is not meant to be an exhaustive list of all such results needed to prove all parts of Theorem \ref{thm:Edgeworth locpoly}, but any and all omitted terms follow by identical reasoning. 
\begin{lemma}
	\label{lem:results}
	Let the conditions of Theorem \ref{thm:Edgeworth locpoly} hold.
	\begin{enumerate}[ref=\ref{lem:results}(\alph{*})]
		\item \label{lem:gamma} For some $\delta > 0$,  $r_*^{-1} \P[ |\Gp - \Gpt| > s^{-1} \log(s)^{1/2} ] \to 0$. Consequently, there exists a constant $C_\Gamma< \infty$ such that $\P[ \Gp^{-1} > 2C_\Gamma ] = o(s^{-2})$ and so the prior rate result holds for $|\Gp^{-1} - \Gpt^{-1}|$ as well. Finally, these same results hold for $\Gq$ as well.
		\item \label{lem:lambda} For some $\delta > 0$,  $r_*^{-1} \P[ |\L_{p,1} - \Lt_{p,1}| > s^{-1} \log(s)^{1/2} ] \to 0$.
		\item \label{lem:E_1} For some $\delta > 0$, \[s^2\P\left[ \left\vert s^{-1} \sumi \left\{(Kr_p)(X_{\h,i})  \e_i \right\} \right\vert > \delta \log(s)^{1/2} \right] \to 0.\]
		\item \label{lem:B_1} For any $\delta>0$ and $\gamma > 0$, \[\frac{1}{\h^{p+1}} \P \left[ \left\vert s^{-2} \sumi \left\{(Kr_p)(X_{\h,i})  \left[m(X_i) - r_p(X_i - x)'\beta_p\right] \right\} \right\vert > \delta \log(s)^\gamma \right] \to 0.\]
		\item \label{lem:Psi} There is some constant $C_\Psi$ such that $\P[ \Psic_p > 2C_\Psi ] = o(s^{-2})$.
	\end{enumerate}

\end{lemma}
\begin{proof}[Proof of Lemma \ref{lem:gamma}]
A typical element of $\Gp - \Gpt$ is, for some integer $k \leq 2p$, 
	\[\frac{1}{n\h}\sumi \left\{  K(X_{\h,i}) \X_{\h,i}^{k} - \E\left[ K(X_{\h,i}) \X_{\h,i}^{k} \right] \right\}.\]
Therefore, the result follows by applying Lemma \ref{lem:bounded} to each element. Next, note that under the maintained assumptions
\begin{align*}
	\Gpt = \E\left[ \h^{-1} (K r_p r_p')(X_{\h,i}) \right] = \h^{-1} \int (K r_p r_p')(X_{\h,i}) f(X_i) d X_i = \int (K r_p r_p')(u) f(x - u \h) d u
\end{align*}
is bounded away from zero and infinity for $n$ large enough. Therefore, there is a $C_\Gamma< \infty$ such that $| \Gpt^{-1}| < C_\Gamma$ and then
\begin{align*}
	\P\left[ \Gp^{-1} > 2C_\Gamma \right] & = \P\left[ \left( \Gp^{-1} - \Gpt^{-1} \right) + \Gpt^{-1} > 2C_\Gamma \right]   		\\
	& \leq \P\left[ \Gp^{-1} - \Gpt^{-1} >  s^{-1} \log(s)^{1/2} \right] + \P \left[ \Gpt^{-1} > 2 C_\Gamma -  s^{-1} \log(s)^{1/2} \right]  		\\
	& = o(s^{-2}).
\end{align*}
The third result follows from these two and the identity $\Gp^{-1} - \Gpt^{-1} = \Gpt^{-1} (\Gpt - \Gp) \Gp^{-1}$. 

Finally, for $\Gq$, the identical steps apply with $L$, $q$, and $\b$ in place of $K$, $p$, and $\h$.
\end{proof}
\begin{proof}[Proof of Lemma \ref{lem:lambda}]
Follows from identical steps to the previous result.
\end{proof}
\begin{proof}[Proof of Lemma \ref{lem:E_1}]
Follows from identical steps, but using Lemma \ref{lem:truncation} in place of Lemma \ref{lem:bounded}.
\end{proof}
\begin{proof}[Proof of Lemma \ref{lem:B_1}]
Follows from identical steps, but using Lemma \ref{lem:bias 1} in place of Lemma \ref{lem:bounded}.
\end{proof}
\begin{proof}[Proof of Lemma \ref{lem:Psi}]
A typical element of $\Psic_p$ is 
	\[\frac{1}{n\h} \sumi (K^2 r_p r_p') (X_{\h,i}) v(X_i),\]
and hence under the maintained assumptions the result follows just as the comparable result on $\Gp$.
\end{proof}

We next state, without proof, the following fact about the rates appearing in all these Lemmas, which follows from elementary inequalities.
\begin{lemma}
	\label{lem:rates}
	If $r_1 = O(r_1')$ and $r_2 = O(r_2')$, for sequences of positive numbers $r_1$, $r_1'$, $r_2$, and $r_2'$ and if a sequence of nonnegative random variables obeys $(r_1)^{-1} \P [ U_n > r_2 ] \to 0$ it also holds that $ (r_1')^{-1} \P [ U_n > r_2' ] \to 0$.
	
	In particular, since $r_* = \max\{s^{-2}, \eta^2, s^{-1} \eta \}$ is defined as the slowest vanishing of the rates, then $r_1^{-1} \P [|U'| > r_n] = o(1)$ implies $r_*^{-1} \P [|U'| > r_n] = o(1)$, for $r_1$ equal to any of $s^{-2}$, $\eta^2$, or $s^{-1} \eta$. Similarly, $r_n$ may be chosen as any sequence that obeys $r_n = o(r_*)$. Thus, for different pieces of $U$ defined in \Eqref{eqn:delta method 3}, we may make different choices for these two sequences, as convenient.
\end{lemma}

The next Lemma proves \Eqref{eqn:delta method 3}, a crucial step in the proof of Theorem \ref{thm:Edgeworth locpoly}(a). Because this result only involves undersmoothing, we will omit the subscript $p$ as above.
\begin{lemma}
	\label{lem:undersmoothing}
	Let the conditions of Theorem \ref{thm:Edgeworth locpoly}(a) hold. Then \Eqref{eqn:delta method 3} holds, namely, for some $r_n = o(r_*)$
		\[\frac{1}{r_*} \P [|U| > r_n] \to 0.\]
\end{lemma}
\begin{proof}
Recall the definition:
\begin{align*}
	U & = \left\{ - \frac{1}{2 \st^3}  e_0' \G^{-1} \left( A_{1,6} + A_{1,7} + A_{1,8} \right) \G^{-1} e_0   + \frac{3}{8 \st^5} \left[ e_0' \G^{-1} \left( \sum\nolimits_{k=2}^8 A_{1,k} \right) \G^{-1} e_0  \right]^2  -  \frac{5}{16}  \frac{ \left( \shat^2 - \st^2 \right)^3 }{\smvt^7} \right\}     		\\
		& \qquad \qquad \times \Bigl\{  s e_0' \G^{-1} R'W(Y-M)/n  +  s e_0' \G^{-1} R'W(M - \check{R}\beta)/n \Bigr\}   		\\
	& \quad - \left\{ - \frac{1}{2 \st^3}  e_0' \Gt^{-1} \left( \tilde{A}_{1,6} + \tilde{A}_{1,7} + \tilde{A}_{1,8} \right) \Gt^{-1} e_0     \right\} \,  \eta.
\end{align*}
To fully prove the claim of the lemma, we must fully expand $U$ and bound each piece. First, we present complete details on two terms. The remainder are entirely analogous, as discussed below. Consider the pieces involving $A_{1,6}$, namely:
\[ e_0' \G^{-1}  A_{1,6} \G^{-1} e_0 \Bigl\{  s e_0' \G^{-1} R'W(Y-M)/n  +  s e_0' \G^{-1} R'W(M - \check{R}\beta)/n \Bigr\}   -    e_0' \Gt^{-1}  \tilde{A}_{1,6} \Gt^{-1} e_0      \,  \eta. \]
The first of these is
\begin{align*}
	e_0'  \G^{-1}  A_{1,6}\G^{-1} e_0 s e_0' \G^{-1} R'W(Y-M)/n  		
	& = e_0'  \G^{-1} \left(A_{1,6} -  \tilde{A}_{1,6}\right)  \G^{-1} e_0 s e_0' \G^{-1} R'W(Y-M)/n  		\\
	& \quad + e_0' \left( \G^{-1} - \Gt^{-1} \right)\tilde{A}_{1,6}  \G^{-1} e_0 s e_0' \G^{-1} R'W(Y-M)/n  		\\
	& \quad + e_0' \Gt^{-1} \tilde{A}_{1,6}   \left( \G^{-1} - \Gt^{-1} \right) e_0 s e_0' \G^{-1} R'W(Y-M)/n  		\\
	& \quad + e_0' \Gt^{-1} \tilde{A}_{1,6}   \Gt^{-1} e_0 s e_0' \left( \G^{-1} - \Gt^{-1} \right)  R'W(Y-M)/n  		\\
	& \quad + e_0' \Gt^{-1} \tilde{A}_{1,6}   \Gt^{-1} e_0 s e_0' \Gt^{-1} R'W(Y-M)/n.   		\\
	& =: U_{1,1} + U_{1,2} + U_{1,3} + U_{1,4} + U_{1,5}
\end{align*}
We now bound each remainder in turn. First, for $r_n = \h^{p+1} \log(s)^{-1/2}$, we have
\begin{align*}
	s^2 \P\left[ | U_{1,1}| > r_n \right] & = s^2 \P  \left[\left\vert  e_0'  \G^{-1} \left(A_{1,6} -  \tilde{A}_{1,6}\right)  \G^{-1} e_0 s e_0' \G^{-1} R'W(Y-M)/n \right\vert  > r_n \right]    		\\
	& \leq    s^2 \P  \left[ 8 C_\Gamma^3  \left\vert  A_{1,6} -  \tilde{A}_{1,6} \right\vert  > \log(s)^{-1/2} r_n \right]    		\\
		& \qquad \qquad +  s^2 \P \left[  \left\vert s^{-1} \sumi \left\{(Kr_p)(X_{\h,i})  \e_i \right\} \right\vert > \log(s)^{1/2} \right]   +  s^2  3 \P\left[ \Gp^{-1} > 2C_\Gamma \right]  		\\
	& =    s^2 \P  \left[ 8 C_\Gamma^3  \left\vert  A_{1,6} -  \tilde{A}_{1,6}  \right\vert  > \h^{2(p+1)} \log(s)^\gamma \frac{r_n}{\h^{2(p+1)}\log(s)^{1/2+\gamma}}  \right]  +  o(1)  		\\
	& = o(1),
\end{align*}
because $\h^{-2(p+1)} r_n \log(s)^{-1/2-\gamma} = \h^{-(p+1)} \log(s)^{-1-\gamma} \to \infty$.

Next, since $\tilde{A}_{1,6}  \asymp \h^{2(p+1)}$, for $r_n = \h^{p+1} \log(s)^{-1/2}$. 
\begin{align*}
	s^2 \P\left[ | U_{1,2}| > r_n \right] & = s^2 \P  \left[\left\vert  e_0' \left( \G^{-1} - \Gt^{-1} \right)\tilde{A}_{1,6}  \G^{-1} e_0 s e_0' \G^{-1} R'W(Y-M)/n \right\vert  > r_n \right]    		\\
	& \leq    s^2 \P  \left[ 4 C_\Gamma^2  \left| \tilde{A}_{1,6} \right|  \left\vert s^{-1} \sumi \left\{(Kr_p)(X_{\h,i}) \e_i \right\} \right\vert  > s \log(s)^{-1/2} r_n \right]       		\\
		& \qquad  +  s^2   \P\left[ \left| \G^{-1} - \Gt^{-1} \right| > s^{-1} \log(s)^{1/2} \right]    +  s^2  2 \P\left[ \Gp^{-1} > 2C_\Gamma \right]      		\\
	& =    s^2 \P  \left[ 4 C_\Gamma^2   \left\vert s^{-1} \sumi \left\{(Kr_p)(X_{\h,i}) \e_i \right\} \right\vert  > \log(s)^{1/2} \frac{s r_n}{ \h^{2(p+1)} \log(s) } \right]  +  o(1)       		\\  
	& = o(1),
\end{align*}
because $s r_n \h^{-2(p+1)} \log(s)^{-1} = s \h^{-(p+1)} \log(s)^{-3/2} \to \infty$. Terms $U_{1,3}$ and $U_{1,4}$ are nearly identically treated.

Let $r_n = \h^{p+1} \log(s)^{-1/2}$. Then since $\tilde{A}_{1,6} \asymp \h^{2(p+1)}$,
\begin{align*}
	s^2 \P\left[ | U_{1,5}| > r_n \right] & = s^2 \P  \left[\left\vert  e_0' \Gt^{-1} \tilde{A}_{1,6}   \Gt^{-1} e_0 s e_0' \Gt^{-1} R'W(Y-M)/n   \right\vert  > r_n \right]    		\\
	& \leq  s^2 \P \left[  C_\Gamma^3 \left| \tilde{A}_{1,6} \right|  \left\vert s^{-1} \sumi \left\{(Kr_p)(X_{\h,i}) \e_i \right\} \right\vert > r_n \right]    		\\
	& \leq  s^2 \P \left[  C_\Gamma^3  \left\vert s^{-1} \sumi \left\{(Kt)(X_{\h,i}) g(X_i) \e_i \right\} \right\vert > \log(s)^{1/2} \frac{ \log(s)^{-1/2} r_n}{\h^{2(p+1)} } \right]    		\\
	& = o(1),
\end{align*}
because $\h^{-2(p+1)} r_n \log(s)^{-1/2} = \h^{-(p+1)} \log(s)^{-1} \to \infty$.

Thus, since $\st^{-1}$ is bounded away from zero, we find that
\[ s^2 \P\left[ \left\vert  \frac{1}{2 \st^3} e_0'  \G^{-1}  A_{1,6}\G^{-1} e_0 s e_0' \G^{-1} R'W(Y-M)/n \right\vert > r_n \right] \to 0.\]

Turning our attention to the second term, we have 
\begin{align*}
	e_0'  \G^{-1} &  A_{1,6}\G^{-1} e_0 s e_0' \G^{-1} R'W(M-\check{R}\beta)/n  - e_0' \Gt^{-1} \tilde{A}_{1,6} \Gt^{-1} e_0 \eta 	 		\\
	& =	e_0'  \G^{-1} \left(A_{1,6} -  \tilde{A}_{1,6}\right)  \G^{-1} e_0 s e_0' \G^{-1} R'W(M-\check{R}\beta)/n  		\\
	& \quad +	e_0'  \G^{-1}  \tilde{A}_{1,6}  \G^{-1} e_0 s e_0' \G^{-1} \left( R'W(M-\check{R}\beta)/n - \E\left[ R'W(M-\check{R}\beta)/n \right] \right)		\\
	& \quad +	e_0'  \left( \G^{-1} - \Gt^{-1} \right)  \tilde{A}_{1,6}  \G^{-1} e_0 s e_0' \G^{-1}  \E\left[ R'W(M-\check{R}\beta)/n \right] 		\\
	& \quad +	e_0'  \Gt^{-1}  \tilde{A}_{1,6} \left( \G^{-1} - \Gt^{-1} \right)  e_0 s e_0' \G^{-1}  \E\left[ R'W(M-\check{R}\beta)/n \right] 		\\
	& \quad +	e_0'  \Gt^{-1}  \tilde{A}_{1,6} \Gt^{-1}  e_0 s e_0' \left( \G^{-1} - \Gt^{-1} \right)  \E\left[ R'W(M-\check{R}\beta)/n \right]   		\\
	& =: U_{2,1} + U_{2,2} + U_{2,3} + U_{2,4} + U_{2,5}.
\end{align*}

For $r_n = \h^{p+1} \log(s)^{-1}$, we have
\begin{align*}
	r_*^{-1} \P\left[ | U_{2,1}| > r_n \right] & = r_*^{-1} \P\left[ e_0'  \G^{-1} \left(A_{1,6} -  \tilde{A}_{1,6}\right)  \G^{-1} e_0 s e_0' \G^{-1} R'W(M-\check{R}\beta)/n  > r_n \right]  		\\
	& \leq   r_*^{-1} \P \left[ 8 C_\Gamma^3  s \left\vert  A_{1,6} -  \tilde{A}_{1,6} \right\vert  > s \h^{2(p+1)} \log(s)^\gamma \frac{r_n}{s \h^{2(p+1)} \log(s)^{2\gamma} }\right]     		\\
	& \quad  +   r_*^{-1} \P \left[ \left|\frac{1}{n\h} \sumi \left\{(Kr_p)(X_{\h,i})  \left[m(X_i) - r_p(X_i - x)'\beta_p\right] \right\}  \right| > \log(s)^\gamma \right]   		\\
	& \quad  +   r_*^{-1}  3 \P\left[ \Gp^{-1} > 2C_\Gamma \right]  		\\
	& \leq   s^2 \P \left[ 8 C_\Gamma^3  s \left\vert  A_{1,6} -  \tilde{A}_{1,6} \right\vert  > s \h^{2(p+1)} \log(s)^\gamma \frac{r_n}{s \h^{2(p+1)} \log(s)^{2\gamma} }\right]     		\\
	& \quad  +   \h^{-(p+1)} \P \left[ \left|\frac{1}{n\h} \sumi \left\{(Kr_p)(X_{\h,i})  \left[m(X_i) - r_p(X_i - x)'\beta_p\right] \right\}  \right| > \log(s)^\gamma \right]   		\\
	& \quad  +  s^2  3 \P\left[ \Gp^{-1} > 2C_\Gamma \right]  		\\
	& = o(1),
\end{align*}
because $s \h^{2(p+1)} r_n^{-1} \log(s)^{2\gamma} = s \h^{p+1} \log(s)^{1+2\gamma} \to 0$ by the conditions on $\eta$ placed in the theorem.

Next, with $r_n = \h^{p+1}\log(s)^{-1}$ and using $\tilde{A}_{1,6} \asymp \h^{2(p+1)}$, we have
\begin{align*}
	r_*^{-1} \P\left[ | U_{2,2}| > r_n \right] & =  r_*^{-1} \P\left[  \left\vert e_0'  \G^{-1}  \tilde{A}_{1,6}  \G^{-1} e_0 s e_0' \G^{-1} \left( R'W(M-\check{R}\beta)/n - \E\left[ R'W(M-\check{R}\beta)/n \right] \right) \right\vert  > r_n \right]   		\\
	& \leq  r_*^{-1} \P \biggl[ 8 C_\Gamma^3  \left| \tilde{A}_{1,6} \right| \biggl\vert s^{-1} \sumi  \Bigl\{ (Kr_p)(X_{\h,i})  \left[m(X_i) - r_p(X_i - x)'\beta_p\right]   		\\
	& \qquad \qquad \qquad \qquad  - \E\left[(Kr_p)(X_{\h,i})  \left[m(X_i) - r_p(X_i - x)'\beta_p\right] \right] \Bigr\} \biggr\vert > r_n  \biggr]  		\\
	& \quad  +  r_*^{-1}  3 \P\left[ \Gp^{-1} > 2C_\Gamma \right]  		\\
	& \leq s^2 \P \biggl[ 8 C_\Gamma^3  \biggl\vert s^{-2} \sumi  \Bigl\{ (Kr_p)(X_{\h,i})  \left[m(X_i) - r_p(X_i - x)'\beta_p\right]   		\\
	& \qquad \quad  - \E\left[(Kr_p)(X_{\h,i})  \left[m(X_i) - r_p(X_i - x)'\beta_p\right] \right] \Bigr\} \biggr\vert >   \h^{p+1} \log(s)^\gamma \frac{ r_n } {\h^{3(p+1)}  \log(s)^\gamma }  \biggr]  		\\
	& \quad  +   s^2  3 \P\left[ \Gp^{-1} > 2C_\Gamma \right]  		\\
	& = o(1),
\end{align*}
because $r_n \h^{-3(p+1)} \log(s)^{-\gamma} = \h^{-2(p+1)} \log(s)^{-1-\gamma} \to \infty$. 

Third, as $\tilde{A}_{1,6} \asymp \h^{2(p+1)}$ and $\E\left[ R'W(M-\check{R}\beta)/n \right] \asymp \h^{p+1}$, if we choose $r_n = \h^{p+1}\log(s)^{-1}$, 
\begin{align*}
	r_*^{-1} \P\left[ | U_{2,3}| > r_n \right] & \leq  r_*^{-1} \P\left[  4 C_\Gamma^2 s \left| \G^{-1} - \Gt^{-1} \right| > s^{-1} \log(s)^{1/2}  \frac{s r_n }{\h^{3(p+1)}\log(s)^{1/2}}  \right]  		\\
	& \quad  +   r_*^{-1}  2 \P\left[ \Gp^{-1} > 2C_\Gamma \right]    		\\	
	& \leq  s^2 \P\left[  4 C_\Gamma^2 \left| \G^{-1} - \Gt^{-1} \right| > s^{-1} \log(s)^{1/2}  \frac{ r_n  }{\h^{3(p+1)}\log(s)^{1/2}}  \right]  		\\
	& \quad  +  s^2  2 \P\left[ \Gp^{-1} > 2C_\Gamma \right]   		\\
	& = o(1),
\end{align*}
because $r_n \h^{-3(p+1)} \log(s)^{-1/2} = \h^{-2(p+1)} \log(s)^{-1-1/2} \to \infty$. The terms $ U_{2,3}$ and $U_{2,5}$ are handled identically.

Thus, since $\st^{-1}$ is bounded away from zero, we find that
\[ s^2 \P\left[ \left\vert  \frac{1}{2 \st^3} e_0'  \G^{-1}  A_{1,6}\G^{-1} e_0 s e_0' \G^{-1} R'W(M-\check{R}\beta)/n  - e_0' \Gt^{-1} \tilde{A}_{1,6} \Gt^{-1} e_0 \eta \right\vert > r_n \right] \to 0.\]

The same type of arguments, though notationally more challenging, will show that the remainder of $U$ obeys the same bounds. Note that the rest of the terms are even higher order, involving either $A_{1,7}$ and $A_{1,8}$, or the square or cube of the other errors. It is for this reason that only the ``leading'' three terms need be centered, that is, why only 
\[- \left\{ - \frac{1}{2 \st^3}  e_0' \Gt^{-1} \left( \tilde{A}_{1,6} + \tilde{A}_{1,7} + \tilde{A}_{1,8} \right) \Gt^{-1} e_0     \right\} \,  \eta\]
appears in $\tilde{z}$.
\end{proof}

%%%%%%%%%%%%%%%%%%%%%%%%%%%%%%%%%%%
\subsection{Computing the Terms of the Expansion}
	\label{supp:moments locpoly}

Identifying the terms of the expansion is a matter of straightforward, if tedious, calculation. The first four cumulants of the Studentized statistics must be calculated (due to \citet{James-Mayne1962_Sankhya}), which are functions of the first four moments. In what follows, we give a short summary. Note well that we always discard higher-order terms for brevity, and to save notation we will write $\oeq$ to stand in for ``equal up to $o((n\h)^{-1} + (n\h)^{-1/2}\eta + \eta^2)$'', and including $o(\rho^{1+2(p+1)})$ for $\tbc$.

The computations will be aided by putting all three estimators into a common structure. In close parallel to the density case, let us define $\hat{m}_1 := \hat{m}$ and $\hat{m}_2 = \hat{m} - \hat{m}_m$, $\sigma_1^2 := \sus^2$, and $\sigma_2^2 := \srbc^2$, so that subscripts $1$ and $2$ generically stand in for undersmoothing and bias correction, respectively. With this in mind, we write
\[\tus = T_{1,1}, 	\qquad	\tbc = T_{2,1}, 	\qquad  \text{ and } \qquad 	\trbc = T_{2,2},\]
again paralleling the density case, so that the first subscript refers to the numerator and the second to the denominator. In the same vein, with some abuse of notation, we will also use\footnote{Throughout Section \ref{supp:locpoly}, we use only generic polynomial orders $p$ and $q$, and so this notation will not conflict with the local linear or local quadratic fits, which would also be denoted $r_1(u)$ and $r_2(u)$, respectively.} $r_1(u) = r_p(u)$, $r_2(u) = r_q(u)$, $K_1(u) = K(u)$, $K_2(u) = L(u)$, $\h_1 = \h$, and $\h_2 = \b$, as well as 
\begin{align*}
	\l^0_1(X_i) & \equiv  \l^0_\US(X_i)  ,  		\\
	\l^1_1(X_i, X_j) & \equiv  \l^1_\US(X_i, X_j)  ,  		\\
	\l^0_2(X_i) & \equiv  \l^0_\BC(X_i)  ,  		\\
	\l^1_2(X_i, X_j) & \equiv  \l^1_\BC(X_i, X_j).
\end{align*}

For the purpose of computing the expansion terms (i.e. moments of the two sides agree up to the requisite order), recalling the Taylor series expansion above, we will use
\begin{multline*}
	T_{v,w} \approx \left\{ 1 - \frac{1}{2 \st_w^2}\left( W_{w,1}  +  V_{w,1}  + V_{w,2}\right) + \frac{3}{8 \st_w^4} \left( W_{w,1}  +  V_{w,1}  + V_{w,2}\right)^2  \right\}  		\\ 		 \st_w^{-1} \left\{   E_{v,1} + E_{v,2} + E_{v,3} + B_{v,1} \right\},
\end{multline*}
where we define, for $v \in \{1,2\}$,
\begin{align*}
	E_{v,1} & = s  \frac{1}{n\h} \sumi \l^0_v(X_i) \e_i  		\\
	E_{v,2} & = s  \frac{1}{(n\h)^2} \sumi \sumj \l^1_v(X_i, X_j) \e_i,  		\\
	E_{v,3} & =: s  \frac{1}{(n\h)^3} \sumi \sumj \sumk \l^2_v(X_i, X_j, X_k) \e_i,
\end{align*}
where the final line defines $\l^2_\US(X_i, X_j, X_k)$ in the obvious way following $\l^1_\US$. To concretize the notation, for undersmoothing we are defining
\begin{align*}
	E_{1,1} & = s e_0' \Gpt^{-1} R_p'W_p (Y - M)/n,   		\\
	E_{1,2} & = s e_0' \Gpt^{-1}(\Gpt - \Gp) \Gpt^{-1} R_p'W_p (Y - M)/n, 		\\
	E_{1,3} & = s e_0' \Gpt^{-1}(\Gpt - \Gp) \Gpt^{-1} (\Gpt - \Gp) \Gpt^{-1} R_p'W_p (Y - M)/n.
\end{align*}
In a similar way, 
\begin{align*}
	W_{v,1} & =  \frac{1}{n\h} \sumi \left\{\l^0_v(X_i)^2  \left( \e_i^2 - v(X_i) \right) \right\}  -  2 \frac{1}{n^2 \h^2} \sumi \sumj \left\{ \l^0_v(X_i)^2 r_v(X_{\h_v,i})'  \Gt^{-1}_v (K_v r_v)(X_{\h_v,i}) \e_i \e_j \right\}    		\\
		& \qquad \qquad \qquad \qquad +  \frac{1}{n^3 \h^3} \sumi \sumj \sumk \left\{ \l^0_v(X_i)^2 r_v(X_{\h_v,i})'\Gt^{-1}_v (K_v r_v)(X_{\h_v,i}) \e_j \e_k  \right\} ,   		\\
	V_{v,1} & =  \frac{1}{n\h}\sumi \left\{  \l^0_v(X_i)^2  v(X_i)^2  - \E[ \l^0_v(X_i)^2  v(X_i)^2] \right\}  +  2 \frac{1}{n^2 \h^2} \sumi\sumj \l_v^2(X_i, X_j) \l_v^0(X_i) v(X_i),   		\\
	V_{v,2} & = \frac{1}{n^3 \h^3} \sumi \sumj \sumk \l_v^1(X_i, X_j) \l_v^1(X_i,X_k) v(X_i)   +   2 \frac{1}{n^3 \h^3} \sumi \sumj \sumk \l_v^2(X_i, X_j, X_k) \l_v^0(X_i) v(X_i),
\end{align*}
and specifically for undersmoothing and bias correction, let
\[B_{1,1} = s \frac{1}{n\h} \sumi \l_1^0(X_i) [m(X_i) - r_p(X_i - x)'\beta_p]\]
and
\begin{multline*}
	B_{2,1} = s \frac{1}{n\h} \sumi  \Bigl\{\h^{-1} \l^0_\US(X_i) [m(X_i) - r_{p+1}(X_i - x)' \beta_{p+1}]    		\\		  -   \h^{-1} \left( \l^0_\BC(X_i)  -   \l^0_\US(X_i) \right)  [m(X_i) - r_q(X_i - x)' \beta_q]   \Bigr\}.
\end{multline*}
Note that $\eus = \E[B_{1,1}]$ and $\ebc = \E[B_{2,1}]$.

Straightforward moment calculations yield
\begin{align*}
	\E[T_{v,w}] & \oeq \st_w^{-1} \E\left[B_{v,1}\right]  - \frac{1}{2 \st_w^2} \E\left[ W_{w,1} E_{v,1}\right],
\end{align*}
\begin{align*}
	\E[T_{v,w}^2] & \oeq  \frac{1}{\st_w^2} \E\left[ E_{v,1}^2  +  E_{v,2}^2 + 2 E_{v,1} E_{v,2}  + 2 E_{v,1} E_{v,3}  \right]   		\\
		& \quad - \frac{1}{\st_w^4} \E \left[ W_{w,1} E_{v,1}^2  +  V_{w,1} E_{v,1}^2  +   V_{w,2} E_{v,1}^2   + 2 V_{w,1} E_{v,1} E_{v,2}   \right]  		\\
		& \quad + \frac{1}{\st_w^6} \E\left[ W_{w,1}^2 E_{v,1}^2  +  V_{w,1}^2 E_{v,1}^2 \right]  +  \frac{1}{\st_w^2} \E\left[ B_{v,1}^2 \right] - \frac{1}{\st_w^4} \E\left[ W_{w,1} E_{v,1} B_{v,1} \right],
\end{align*}
\begin{align*}
	\E[T_{v,w}^3] & \oeq  \frac{1}{ \st_w^3} \E\left[ E_{v,1}^3 \right]    -   \frac{3}{2\st_w^5} \E \left[ W_{w,1} E_{v,1}^3  \right]  	+   \frac{3}{\st_w^3} \E\left[  E_{v,1}^2 B_{v,1} \right],
\end{align*}
and
\begin{align*}
	\E[T_{v,w}^4] & \oeq  \frac{1}{\st_w^4} \E\left[ E_{v,1}^4  +  4 E_{v,1}^3 E_{v,2}  +  4 E_{v,1}^3 E_{v,3}  + 6 E_{v,1}^2 E_{v,3}^2  \right]   		\\
		& \quad - \frac{2}{\st_w^6} \E \left[ W_{w,1} E_{v,1}^4  +  V_{w,1} E_{v,1}^4  +   4 V_{w,1} E_{v,1}^3 E_{v,2}  + V_{w,2} E_{v,1}  \right]  		\\
		& \quad +  \frac{3}{\st_w^8} \E \left[ W_{w,1}^2 E_{v,1}^4  +  V_{w,1}^2 E_{v,1}^4   \right]  		\\
		& \quad + \frac{4}{\st_w^4} \E\left[ E_{v,1}^3  B_{v,1} \right]   -   \frac{8}{\st_w^6} \E\left[ W_{w,1} E_{v,1}^3 B_{v,1} \right]  +  \frac{6}{\st_w^4} \E\left[ E_{v,1}^2 B_{v,1}^2 \right].
\end{align*}

Computing each term in turn, we have
\begin{align*}
	\E\left[B_{v,1}\right] & = \eta_v,  		\\
	\E\left[ W_{w,1} E_{v,1}\right] & \oeq s^{-1} \E\left[ \h^{-1} \l_w^0(X_i)^2 \l_v^0(X_i) \e_i^3 \right],   		\\
	\E\left[  E_{v,1}^2 \right] & \oeq \st_v^2,   		\\
	\E\left[  E_{v,1}E_{v,2} \right] & \oeq s^{-2}  \E\left[ \h^{-1} \l_v^1(X_i,X_i) \l_v^0(X_i) \e_i^2 \right],   		\\ 
	\E\left[ E_{v,2}^2 \right] & \oeq s^{-1} \E\left[ \h^{-2} \l_v^1(X_i, X_j)^2 \e_i^2 \right],   		\\
	\E\left[  E_{v,2}E_{v,3} \right] & \oeq s^{-2}  \E\left[ \h^{-2} \l_v^2(X_i, X_j, X_j) \l_v^0(X_i) \e_i^2 \right],   		\\ 
	\E\left[  W_{w,1}E_{v,1}^2 \right] & \oeq s^{-2} \Biggl\{  \E\left[ \h^{-1} \l_w^0(X_i)^2 \l_v^0(X_i)^2 \left( \e_i^4 - v(X_i)^2\right) \right]   		\\
		& \quad  -  2 \st_v^2 \E\left[ \h^{-1} \l_w^0(X_i)^2 r_w(X_{\h_w,i})' \Gt_w^{-1} (K_w r_w)(X_{\h_w,i}) \e_i^2 \right]   		\\ 
		& \quad  -  4  \E\left[ \h^{-1} \l_w^0(X_i)^2 \l_v^0(X_i)^2 r_w(X_{\h_w,i})' \Gt_w^{-1} \e_i^2  \right] \E\left[ \h^{-1} (K_w r_w)(X_{\h_w,i}) \l_v^0(X_i) \e_i^2 \right]   		\\ 
		& \quad  + \st_v^2 \E\left[ \h^{-2} \l_w^0(X_i)^2 \left( r_w(X_{\h_w,i})' \Gt_w^{-1} (K_w r_w)(X_{\h_w,j}) \right)^2 \e_j^2 \right]   		\\ 
		& \quad +  \E \left[ \h^{-1} \l^0_\US (X_j)^2 \left( \E \left[ \h^{-1} r_p(X_{\h,j})'\Gpt^{-1}(K r_p)(X_{\h,i}) \l^0_\US(X_i)  \e_i^2 \vert X_j \right] \right)^2 \right]    \Biggr\},   		\\
	\E\left[  V_{w,1}E_{v,1}^2 \right] & \oeq s^{-2} \Bigl\{ \E \left[ \h^{-1} \left( \l^0_w(X_i)^2 v(X_i) - \E[\l^0_w(X_i)^2 v(X_i)] \right) \l^0_v(X_i)^2 \e_i^2 \right]     		\\
		& \quad + 2 \st_v^2 \E \left[ \h^{-1} \l^1_w(X_i, X_i) \l^0_w(X_i) v(X_i) \right]  \Bigr\},   		\\
	\E\left[  V_{w,1}E_{v,1}E_{v,2} \right] & \oeq s^{-2} \Bigl\{ \E \left[ \h^{-2} \left( \l^0_w(X_j)^2 v(X_j) - \E[\l^0_w(X_j)^2 v(X_j)] \right)  \l^1_v(X_i, X_j) \l^0_v(X_i) \e_i^2 \right]     		\\
		& \quad + 2  \E \left[ \h^{-3} \l^1_w(X_i, X_j)  \l^1_v(X_k, X_j) \l^0_w(X_i) \l^0_v(X_k) v(X_i) \e_k^2 \right]  \Bigr\},   		\\
	\E\left[  V_{w,2}E_{v,1}^2 \right] & \oeq s^{-2} \Bigl\{ \st_v^2 \E \left[ \h^{-2} \left( \l^1_w(X_i, X_j)^2  + 2\l^2_w(X_i, X_j, X_j) \right) v(X_i)  \right]  \Bigr\},   		\\
	\E\left[  W_{w,1}^2 E_{v,1}^2 \right] & \oeq s^{-2} \Bigl\{ \st_v^2 \E \left[ \h^{-1}  \l^0_w(X_i)^4 \left( \e_i^4 - v(X_i)^2 \right) \right]  + 2 \E\left[ \h^{-1} \l^0_v(X_i) \l^0_w(X_i)^2   \e_i^3  \right]^2  \Bigr\},   		\\
	\E\left[  V_{w,1}^2 E_{v,1}^2 \right] & \oeq s^{-2} \st_v^2 \Bigl\{ \E \left[ \h^{-1}  \left( \l^0_w(X_i)^2 v(X_i)  -  \E[\l^0_w(X_i)^2 v(X_i)] \right)^2 \right]  		\\
		& \quad  + 4 \E\left[ \h^{-2}  \left( \l^0_w(X_i)^2 v(X_i)  -  \E[\l^0_w(X_i)^2 v(X_i)] \right) \l^1_w(X_j, X_i) \l^0_w(X_j) v(X_j) \right]   		\\
		& \quad  + 4 \E\left[ \h^{-3} \l^1_w(X_i, X_j) \l^0_w(X_i) v(X_i) \l^1_w(X_k, X_j) \l^0_w(X_k) v(X_k) \right] \Bigr\},   		\\
	\E\left[  W_{w,1}E_{v,1}B_{v,1}\right]  & \oeq  \E\left[  W_{w,1}E_{v,1} \right] \E\left[B_{v,1}\right],   		\\
	\E\left[  E_{v,1}^3 \right] & \oeq s^{-1} \E\left[ \h^{-1} \l_v^0(X_i)^3 \e_i^3 \right],   		\\
	\E\left[  W_{w,1} E_{v,1}^3 \right] & \oeq \E\left[  E_{v,1}^2 \right]  \E\left[  W_{w,1} E_{v,1} \right],   		\\
	\E\left[  E_{v,1}^4 \right] & \oeq 3 \st_v^4  +  s^{-2} \E\left[ \h^{-1} \l_v^0(X_i)^4 \e_i^3 \right],   		\\
	\E\left[  E_{v,1}^3 E_{v,2} \right] & \oeq  s^{-2} 6 \st_v^2 \E\left[ \h^{-1} \l^1_v(X_i, X_i) \l_v^0(X_i) \e_i^2 \right],   		\\
	\E\left[  E_{v,1}^3 E_{v,3} \right] & \oeq  s^{-2} 3 \st_v^2 \E\left[ \h^{-2} \l^2_v(X_i, X_j, X_j) \l_v^0(X_i) \e_i^2 \right],   		\\
	\E\left[  E_{v,1}^2 E_{v,2}^2 \right] & \oeq  s^{-2} \Bigl\{ \st_v^2 \E\left[ \h^{-2} \l^1_v(X_i, X_j)^2 \e_i^2 \right]   +   2  \E \left[ \h^{-3} \l^1_v(X_i, X_j)  \l^1_v(X_k, X_j) \l^0_v(X_i) \l^0_v(X_k) \e_i^2 \e_k^2 \right]  \Bigr\},   		\\
	\E\left[  W_{w,1} E_{v,1}^4 \right] & \oeq s^{-2} \Bigl\{ \E\left[ \h^{-1} \l_w^0(X_i)^2 \l_v^0(X_i) \e_i^3 \right] \E\left[ \h^{-1} \l_v^0(X_i)^3 \e_i^3 \right]  +  6 \E\left[  E_{v,1}^2 \right]  \E\left[  W_{w,1} E_{v,1}^2 \right]  \Bigr\},   		\\
	\E\left[  V_{w,1} E_{v,1}^4 \right] & \oeq s^{-2}\st_v^2  6 \Bigl\{ \E \left[ \h^{-1} \left( \l^0_w(X_i)^2 v(X_i) - \E[\l^0_w(X_i)^2 v(X_i)] \right) \l^0_v(X_i)^2 \e_i^2 \right]     		\\
		& \quad + 2  \E \left[ \h^{-2} \l^1_w(X_i, X_j) \l^0_w(X_i) \l^0_v(X_j)^2 \e_j^2 v(X_i) \right]   +  \E\left[ \h^{-1} \l^1_w(X_i, X_i) \l^0_w(X_i) v(X_i) \right] \Bigr\},   		\\
	\E\left[  V_{w,1} E_{v,1}^3 E_{v,2} \right] & \oeq 3 \E\left[   E_{v,1}^2 \right] \E\left[  V_{w,1} E_{v,1} E_{v,2} \right] ,  		\\
	\E\left[  V_{w,2} E_{v,1}^4 \right] & \oeq 3 \E\left[   E_{v,1}^2 \right] \E\left[  V_{w,2} E_{v,1}^2 \right] ,  		\\
	\E\left[  W_{w,1}^2 E_{v,1}^4 \right] & \oeq 3 \E\left[   E_{v,1}^2 \right] \E\left[  W_{w,1}^2 E_{v,1}^2 \right] ,  		\\
	\E\left[  V_{w,1}^2 E_{v,1}^4 \right] & \oeq 3 \E\left[   E_{v,1}^2 \right] \E\left[  V_{w,1}^2 E_{v,1}^2 \right] .
\end{align*}

The expansion now follows, formally, from the following steps. First, combining the above moments into cumulants. Second, these cumulants may be simplified using that
\[\frac{\sigma_v^2}{\sigma_w^2} = 1 + \mathbbm{1}(w\!\neq\! v) \left( \rho^{1+ (p+1)} \Omega_{1,\BC} + \rho^{1 + 2(p+1)} \Omega_{2,\BC} \right) \]
and that in all cases present products such as $\l_w^0(X_i)^{k_1} \l_v^0(X_i)^{k_2}$ and $\l_w^1(X_i, X_j)^{k_1} \l_v^1(X_i, X_j)^{k_2}$ may be replaced with $\l_v^0(X_i)^{k_1 + k_2}$ and $\l_v^1(X_i, X_j)^{k_1 + k_2}$, respectively, provided the arguments match. This is immediate for $v=w$, and for $v \neq w$, follows because $\rho \to 0$ is assumed. This is the analogous step to \Eqref{eqn:simplifying nu} in the density case. For any term of a cumulant with a rate of $(n\h)^{-1}$, $(n\h)^{-1/2}\eta_v$, $\eta_v^2$, or $\rho^{1+2(p+1)}$ (i.e., the extent of the expansion), these simplifications may be inserted as the remainder will be negligible. Third, with the cumulants in hand, the terms of the expansion are determined as described by e.g., \cite[Chapter 2]{Hall1992_book}.

%%%%%%%%%%%%%%%%%%%%%%%%%%%%%%%%%%%
%%%%%%%%%%%%%%%%%%%%%%%%%%%%%%%%%%%
\section{Complete Simulation Results}
	\label{supp:simuls locpoly}

In this section we present the results of a simulation study addressing the finite-sample performance of the methods described in the main paper. As with the density estimator, we report empirical coverage probabilities and average interval length of nominal 95\% confidence interval for different estimators of a regression functions $m(x)$ evaluated at values $x=\{-2/3,-1/3,0,1/3,2/3\}$.
For each replication, the data is generated as i.i.d. draws, $i=1,2,...,n$, $n=500$ as follows:
\begin{equation*}
%Y=m(x)+\varepsilon\text{,\qquad }x\sim\mathcal{U}[0,1] \text{,\qquad }\varepsilon\sim(\chi_4-4)/\sqrt{8}
Y=m(x)+\varepsilon\text{,\qquad }x\sim\mathcal{U}[-1,1] \text{,\qquad }\varepsilon\sim\N(0,1)
\end{equation*}%

\begin{itemize}
\item[] Model 1:  $m(x) = \sin(4x) + 2\exp\{-64x^2\}$
\item[] Model 2:  $m(x) = 2x + 2\exp\{-64x^2\}$
\item[] Model 3:  $m(x) = 0.3\exp\{-4(2x+1)^2\} + 0.7\exp\{-16(2x-1)^2\}$
\item[] Model 4:  $m(x) = x + 5\phi(10x)$
\item[] Model 5:  $\displaystyle m(x) = \frac{\sin(3\pi x/2)}{1+18x^2[sgn(x)+1]}$
\item[] Model 6:  $\displaystyle m(x) = \frac{\sin(\pi x/2)}{1+2x^2[sgn(x)+1]}$
\end{itemize}
Models 1 to 3 were used  by \citet{Fan-Gijbels1996_book} and \citet{Cattaneo-Farrell2013_JoE}, while Models 4 to 6 are from \citet{Hall-Horowitz2013_AoS}, with some originally studied by \citet{Berry-Carroll-Ruppert2002_JASA}. The regression functions are plotted in Figure \ref{fig:regfuns} together with the evaluation points used.

We compute confidence intervals for $m(x)$ using five alternative approaches:
\begin{enumerate}[label=\it(\roman*)]
	\item[\bf US:] local-linear estimator using a conventional approach based on undersmoothing ($I_\US$).
	\item[\bf Locfit:] local lineal estimator computed using default options in the \texttt{R} package \texttt{locfit} (see \cite{locfit} for implementation details).
	\item[\bf BC:] traditional bias corrected estimator using a local-linear estimator with local-quadratic bias-correction, and $\rho=1$ ($I_\BC$).
	\item[\bf HH:] local linear estimator using the bootstrapped confidence bands introduced in \cite{Hall-Horowitz2013_AoS} (see Remark \ref{supprem:Hall-Horowitz} below for additional implementation details).
	\item[\bf RBC:] our proposed local-linear estimator with local-quadratic bias-correction and $\rho=1$ using robust standard errors ($I_\RBC$). 
\end{enumerate}
In all cases the Epanechnikov kernel is used. The bandwidth $\h$ is chosen in three different ways:
\begin{enumerate}[label=\it(\roman*)]
	\item population MSE-optimal choice $h_\MSE$;
	\item estimated ROT optimal coverage error rate $\hat{h}_\ROT$. 
	\item estimated DPI optimal coverage error rate $\hat{h}_\PI$. 
\end{enumerate}
For the construction of the variance estimators $\shatus^2$ and $\shatrbc^2$ we consider HC$3$ plug-in residuals when forming the $\Sigma$ matrix. In Table \ref{table:reg_comp_vce} we report empirical coverage and average interval length of RBC 95\% Confidence Intervals (only for Model 5) using $\hat{h}_\MSE$ for different variance estimators. The results reflect the robustness of the findings to this choice.

The results are presented in detail in the tables and figures below to give a complete picture of the performance of robust bias correction. First, Tables \ref{table:regtables1}-\ref{table:regtables6} show, for each regression model, respectively, the performance of the five methods above, in terms of empirical coverage and interval length, for all evaluation points and bandwidth choices (recall that $\ius$ and $\ibc$ have the same length). Panel A of each shows the coverage and length, while Panel B gives summary statistics for the two fully data-driven bandwidths. Note that in some cases, the population MSE-optimal bandwidth is not defined or is not computable numerically; usually because the bias is too small or other values are too extreme.

The broad conclusion from these tables is that robust bias correction provides excellent coverage and that the data-driven bandwidths perform well and are numerically stable. In almost all cases robust bias correction provides correct coverage, whereas the other methods often, but not always, fail to do so. In cases where there is little to no bias all the methods give good coverage. This can be seen in results for Models 2 and 4, at $|x|=2/3$, far enough away from the ``hump'' in the center of each, where the true regression function is (nearly) linear. But despite the encouraging results away from the center, only robust bias correction yields good coverage closer to the center ($|x| = 1/3$), when there is more bias. Going further, considering $x=0$, the center of the sharp peak in these models, we see that even robust bias correction fails to provide accurate coverage for $\hat{h}_\ROT$, although $\hat{h}_\PI$ performs slightly better. At this point, for these models, the bias is too extreme even for robust bias correction to overcome. The results for the other models yield  similar lessons.

It is somewhat more difficult to compare interval length using these tables. The comparison is invited for a fixed bandwidth, in which case, by construction, undersmoothing will have a shorter length. However, this ignores the fact that robust bias correction can accommodate a larger range of bandwidths, and in particular will optimally use a larger bandwidth. For example, robust bias correction has excellent coverage in many cases for $\hat{h}_\ROT$, which is in this case a data-driven MSE-optimal choice (i.e. they coincide). This bandwidth is generally larger than $\hat{h}_\PI$, and hence undersmoothing generally covers better with the latter. However, if you compare the length of $\ius(\hat{h}_\ROT)$ to the length of $\ius(\hat{h}_\PI)$, we see that robust bias correction compares favorably in terms of length.

Both to better make this point and to illustrate the robustness of $\irbc$ to tuning parameter selection, Figures \ref{table:reggrid_ec1}--\ref{table:reggrid_il6} show empirical coverage and length for all six models, and all evaluation points, across a range of bandwidths. The dotted vertical line shows the population MSE-optimal bandwidth (whenever available) for reference. The coverage figures highlight the delicate balance required for undersmoothing to provide correct coverage, and the generally poor performance of traditional bias correction, but show that for a wide range of bandwidths robust bias correction provides correct coverage. Further, interval length is not unduly inflated for bandwidths that provide correct coverage. Again, by construction, undersmoothing will yield shorter intervals for a fixed bandwidth, and this is clear from Figures \ref{table:reggrid_il1}--\ref{table:reggrid_il6}, but it is also clear that robust bias correction can use much larger bandwidths while still maintaining correct coverage.

To further illustrate this idea, in Tables \ref{table:reg_comp}--\ref{table:reg_comp2} we compare average interval length of US and RBC 95\% confidence intervals but at different bandwidths. First, in Table \ref{table:reg_comp} we compute average interval length at the largest bandwidth that provides close to correct coverage for each method separately. Note that in all cases these bandwidths are not feasible: these are ex-post findings. Next, in Table \ref{table:reg_comp2} we evaluate the performance of US and RBC confidence intervals at certain alternative bandwidths likely to be chosen in practice. First, we evaluate the performance of US confidence intervals at $h=\lambda\hat{h}_\MSE$ for $\lambda=\{0.5;0.7\}$. We then compare the performance with RBC confidence intervals computed using the optimal, fully data-driven choices $\hat{h}_\ROT$ and $\hat{h}_\PI$. Both tables reflect that, once we control for coverage, intervals lengths do not differ systematically between both approaches. 

Figures \ref{fig:CI1}-\ref{fig:CI6} make this same point in a different way. For a range of bandwidths, as in the previous figures, we show the ``average position'' of $\ius$ and $\irbc$, where the center of the bar is placed at the average bias and the length of each bar is the average interval length across the simulations. The bars are then color-coded by coverage (green bars having good coverage, fading to red showing undercoverage). These make visually clear that although undersmoothing provides shorter intervals in general, that this comes at the expense of coverage, while robust bias correction provides good coverage for a range of bandwidths, many of which are ``large'' enough to yield narrow intervals.

All our methods are implemented in {\sf R} and {\tt STATA} via the {\tt nprobust} package, available from \url{http://sites.google.com/site/nppackages/nprobust} (see also \url{http://cran.r-project.org/package=nprobust}). See \citet{Calonico-Cattaneo-Farrell2017_nprobust} for a complete description.

\begin{remark}[Implementation of \cite{Hall-Horowitz2013_AoS}]
	\label{supprem:Hall-Horowitz}
The column $HH$ computes the bootstrapped confidence bands introduced in \cite{Hall-Horowitz2013_AoS}, following as close as possible their implementation choices. First, we estimate $m(x)$ using a local linear estimator using the Epanechnikov kernel for our previously discussed bandwidth choices. Standard errors are calculated using their proposed variance estimator $\hat{\sigma}^2_{HH}=\kappa\hat{\sigma}^2/\hat{f}_X(x)$
where $\kappa=\int K^2$ and $\hat{f}_X(x)$ is a standard kernel density estimator using a data-driven bandwidth choice $h_1$. Then, we use the same estimator for the error variance $\hat{\sigma}^2=\sum_{i=1}^n\hat{\varepsilon}_i^2/n$
and $\hat{\varepsilon}_i=\tilde{\varepsilon}_i-\bar{\varepsilon}$, $\tilde{\varepsilon}_i=Y_i-\hat{m}(X_i)$, $\bar{\varepsilon}=n^{-1}\sum_{i=1}^n\tilde{\varepsilon}_i$.
Next, we take generate $B=500$ bootstrap samples $\mathcal{Z}^*=\{(X_i,Y^*_i)\}, 1\leq i\leq n$, where $Y^*_i=\hat{m}(X_i)+\varepsilon^*_i$, with $\varepsilon^*_i$ obtained by sampling with replacement from the $\{\hat{\varepsilon}_i\}, 1\leq i\leq n$.
With these bootstrap samples we can construct the final confidence bands using the adjusted critical values that approximates the estimated coverage error with the selected one. Following their recommendation, the final critical values are taken to be the $\xi$-level quantile (for $\xi=0.1$) obtained by repeating this exercise over a grid of evaluation points, which we choose to be the sequence $\{x_1,...,x_N\}=\{-0.9,-0.8,...,0,...,0.8,0.9\}$. 
\end{remark}

\newpage
\begin{landscape}
\begin{figure}[h]
        \centering
        \caption{Regression Functions \vspace{-0.5in}}\label{fig:regfuns}
        \begin{subfigure}[b]{0.5\textwidth}
                \includegraphics[width=1.05\textwidth]{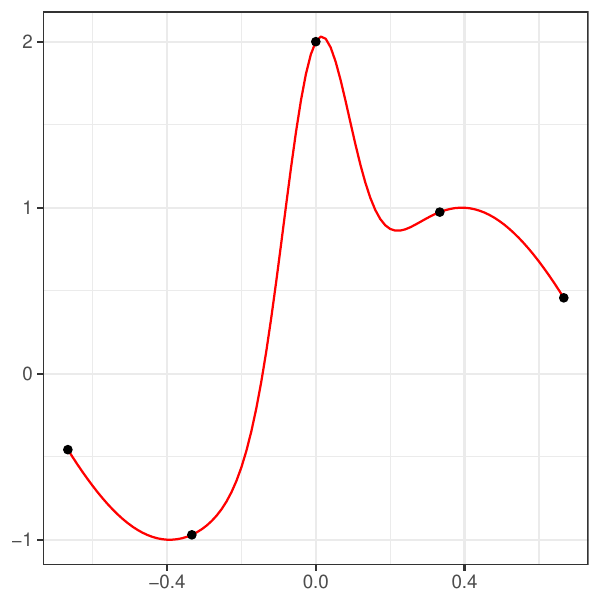}
                \vspace{-0.25in}
                \caption{Model 1}
        \end{subfigure}%
        ~ %add desired spacing between images, e. g. ~, \quad, \qquad, \hfill etc.
        \begin{subfigure}[b]{0.5\textwidth}
                \includegraphics[width=1.05\textwidth]{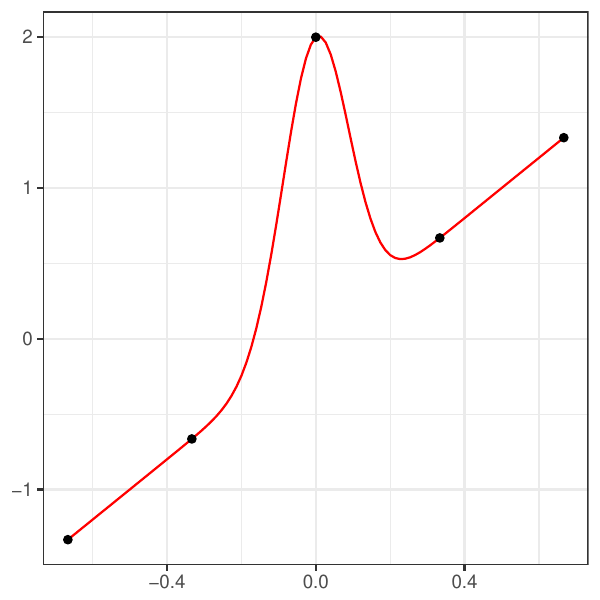}
                \vspace{-0.25in}
                \caption{Model 2}
        \end{subfigure}%
        \begin{subfigure}[b]{0.5\textwidth}
                \includegraphics[width=1.05\textwidth]{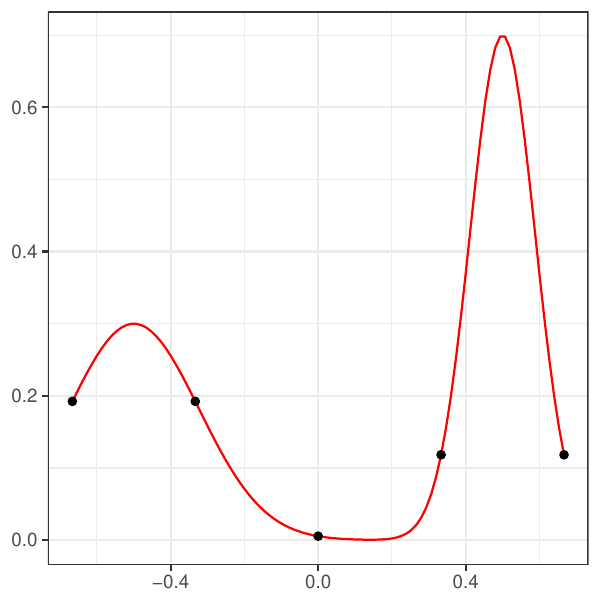}
                \vspace{-0.25in}
                \caption{Model 3}
        \end{subfigure}%
        
        \vspace{-0.25in}
        
        \begin{subfigure}[b]{0.5\textwidth}
                \includegraphics[width=1.05\textwidth]{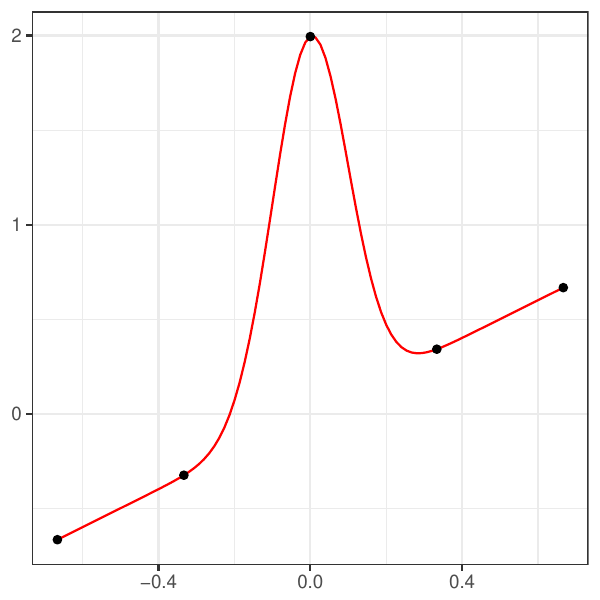}
                \caption{Model 4}
        \end{subfigure}%
        \begin{subfigure}[b]{0.5\textwidth}
        		\includegraphics[width=1.05\textwidth]{model_lp5.pdf}
                \caption{Model 5}
        \end{subfigure}%
        \begin{subfigure}[b]{0.5\textwidth}
                \includegraphics[width=1.05\textwidth]{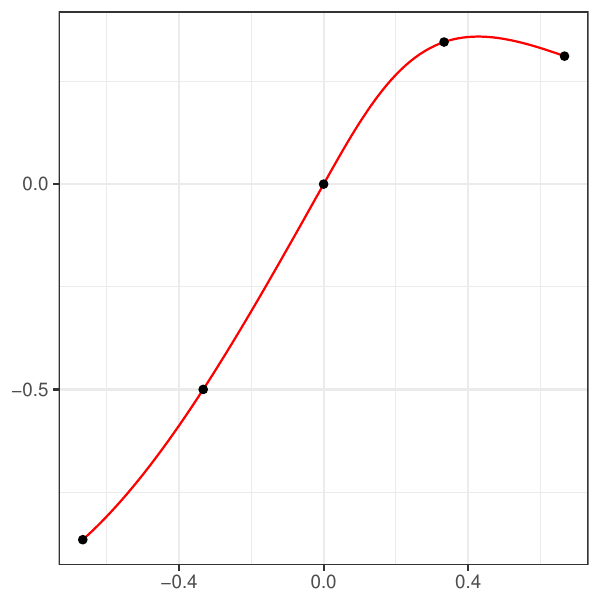}
                \caption{Model 6}
        \end{subfigure}
\end{figure}
\end{landscape}

\newpage
\forloop{model}{1}{\value{model} < 7}{
\begin{table}[h]%\renewcommand{\arraystretch}{1.1}
{\begin{center}
\caption{Simulations Results for Model \themodel}\label{table:regtables\themodel}
Panel A:  Empirical Coverage and Average Interval Length of 95$\%$ Confidence Intervals 
{\input{lp_bws_m\themodel_n500.txt}}\bigskip

Panel B: Summary Statistics for the Estimated Bandwidths 
{\input{lp_bws_stats_m\themodel_n500.txt}}\bigskip
	\end{center}}
\vspace{-.2in}\footnotesize\textbf{Notes}:\\
\smallskip
(i) US = Undersmoothing, Locfit = \texttt{R} package \texttt{locfit} by \cite{locfit}, BC = Bias Corrected, HH = \cite{Hall-Horowitz2013_AoS}, RBC = Robust Bias Corrected.\\
\smallskip
(ii) ``Bandwidth'' column report the population and average estimated bandwidths choices, as appropriate, for bandwidth $h_n$.\\
\smallskip
(iii) The population MSE-optimal choice $h_\MSE$ coincides with the population ROT optimal coverage error rate $h^{\tt rot}_\RBC$. 
\smallskip
(iv) For some evaluation points, $h_\MSE$ is not well defined so it was left missing. 
\end{table}
}

\newpage
\begin{table}[h]%\renewcommand{\arraystretch}{1.1}
{\begin{center}
\caption{Empirical Coverage and Average Interval Length of 95$\%$ Confidence Intervals}\label{table:reg_comp}
{%latex.default(out, file = paste("lp_comp", "_n", n, ".txt", sep = ""),     landscape = FALSE, n.cgroup = c(3, 3), cgroup = c("US", "RBC"),     n.rgroup = rep(5, 6), rgroup = c("Model 1", "Model 2", "Model 3",         "Model 4", "Model 5", "Model 6"), outer.size = "scriptsize",     col.just = rep("c", 6), center = "none", title = "", table.env = FALSE)%
\begin{tabular}{lccccccc}
\hline\hline
\multicolumn{1}{l}{\bfseries }&\multicolumn{3}{c}{\bfseries US}&\multicolumn{1}{c}{\bfseries }&\multicolumn{3}{c}{\bfseries RBC}\tabularnewline
\cline{2-4} \cline{6-8}
\multicolumn{1}{l}{}&\multicolumn{1}{c}{$h$}&\multicolumn{1}{c}{EC}&\multicolumn{1}{c}{IL}&\multicolumn{1}{c}{}&\multicolumn{1}{c}{$h$}&\multicolumn{1}{c}{EC}&\multicolumn{1}{c}{IL}\tabularnewline
\hline
{\bfseries Model 1}&&&&&&&\tabularnewline
~~$x=-2/3$&0.140&94.8&0.523&&0.420&94.8&0.442\tabularnewline
~~$x=-1/3$&0.100&94.7&0.625&&0.420&94.8&0.434\tabularnewline
~~$x=0$&0.100&71.3&0.640&&0.100&93.7&0.893\tabularnewline
~~$x=1/3$&0.300&94.6&0.355&&0.440&94.3&0.425\tabularnewline
~~$x=2/3$&0.100&95.0&0.624&&0.260&94.9&0.546\tabularnewline
\hline
{\bfseries Model 2}&&&&&&&\tabularnewline
~~$x=-2/3$&0.180&94.9&0.459&&0.540&94.9&0.399\tabularnewline
~~$x=-1/3$&0.140&94.8&0.524&&0.440&94.9&0.424\tabularnewline
~~$x=0$&0.100&71.3&0.640&&0.100&93.7&0.893\tabularnewline
~~$x=1/3$&0.140&94.5&0.522&&0.440&94.2&0.424\tabularnewline
~~$x=2/3$&0.260&94.9&0.380&&0.280&94.9&0.525\tabularnewline
\hline
{\bfseries Model 3}&&&&&&&\tabularnewline
~~$x=-2/3$&0.140&94.9&0.523&&0.420&94.9&0.442\tabularnewline
~~$x=-1/3$&0.200&94.9&0.435&&0.400&94.9&0.440\tabularnewline
~~$x=0$&0.100&94.7&0.628&&0.680&94.7&0.337\tabularnewline
~~$x=1/3$&0.100&93.9&0.623&&0.100&94.0&0.887\tabularnewline
~~$x=2/3$&0.100&94.6&0.624&&0.180&94.9&0.658\tabularnewline
\hline
{\bfseries Model 4}&&&&&&&\tabularnewline
~~$x=-2/3$&0.180&94.9&0.459&&0.520&94.8&0.406\tabularnewline
~~$x=-1/3$&0.100&94.8&0.625&&0.400&94.8&0.444\tabularnewline
~~$x=0$&0.100&79.3&0.636&&0.100&93.9&0.893\tabularnewline
~~$x=1/3$&0.100&94.4&0.623&&0.400&94.2&0.443\tabularnewline
~~$x=2/3$&0.320&94.9&0.342&&0.280&94.9&0.525\tabularnewline
\hline
{\bfseries Model 5}&&&&&&&\tabularnewline
~~$x=-2/3$&0.180&94.9&0.459&&0.200&94.8&0.624\tabularnewline
~~$x=-1/3$&0.100&94.7&0.625&&0.180&94.6&0.658\tabularnewline
~~$x=0$&0.100&94.6&0.628&&0.240&94.4&0.572\tabularnewline
~~$x=1/3$&0.140&94.6&0.522&&0.260&94.3&0.545\tabularnewline
~~$x=2/3$&0.200&94.8&0.434&&0.280&94.9&0.525\tabularnewline
\hline
{\bfseries Model 6}&&&&&&&\tabularnewline
~~$x=-2/3$&0.140&94.9&0.523&&0.600&94.9&0.379\tabularnewline
~~$x=-1/3$&0.140&94.8&0.524&&0.420&94.9&0.429\tabularnewline
~~$x=0$&0.100&94.8&0.628&&0.600&94.9&0.359\tabularnewline
~~$x=1/3$&0.140&94.5&0.522&&0.480&94.4&0.401\tabularnewline
~~$x=2/3$&0.260&94.8&0.380&&0.420&94.9&0.442\tabularnewline
\hline
\end{tabular}
}\bigskip
	\end{center}}
\vspace{-.2in}\footnotesize\textbf{Notes}: Bandwidths are selected ex post as the largest bandwidths yielding good coverage, and as can not be made feasible\\
\smallskip
\end{table}

\newpage
\begin{table}[h]%\renewcommand{\arraystretch}{1.1}
{\begin{center}
\caption{Empirical Coverage and Average Interval Length of 95$\%$ Confidence Intervals}\label{table:reg_comp2}
{%latex.default(out, file = paste("lp_comp2", "_n", n, ".txt",     sep = ""), landscape = FALSE, n.cgroup = c(2, 2, 2, 2), cgroup = c("US ($\\lambda=0.5$)",     "US ($\\lambda=0.7$)", "RBC ($\\hat{h}^{\\tt rot}_{\\tt rbc}$)",     "RBC ($\\hat{h}^{\\tt dpi}_{\\tt rbc}$)"), n.rgroup = rep(5,     6), rgroup = c("Model 1", "Model 2", "Model 3", "Model 4",     "Model 5", "Model 6"), outer.size = "scriptsize", col.just = rep("c",     8), center = "none", title = "", table.env = FALSE)%
\begin{tabular}{lccccccccccc}
\hline\hline
\multicolumn{1}{l}{\bfseries }&\multicolumn{2}{c}{\bfseries US ($\lambda=0.5$)}&\multicolumn{1}{c}{\bfseries }&\multicolumn{2}{c}{\bfseries US ($\lambda=0.7$)}&\multicolumn{1}{c}{\bfseries }&\multicolumn{2}{c}{\bfseries RBC ($\hat{h}^{\tt rot}_{\tt rbc}$)}&\multicolumn{1}{c}{\bfseries }&\multicolumn{2}{c}{\bfseries RBC ($\hat{h}^{\tt dpi}_{\tt rbc}$)}\tabularnewline
\cline{2-3} \cline{5-6} \cline{8-9} \cline{11-12}
\multicolumn{1}{l}{}&\multicolumn{1}{c}{EC}&\multicolumn{1}{c}{IL}&\multicolumn{1}{c}{}&\multicolumn{1}{c}{EC}&\multicolumn{1}{c}{IL}&\multicolumn{1}{c}{}&\multicolumn{1}{c}{EC}&\multicolumn{1}{c}{IL}&\multicolumn{1}{c}{}&\multicolumn{1}{c}{EC}&\multicolumn{1}{c}{IL}\tabularnewline
\hline
{\bfseries Model 1}&&&&&&&&&&&\tabularnewline
~~$x=-2/3$&94.4&0.630&&94.7&0.528&&94.3&0.630&&94.7&0.669\tabularnewline
~~$x=-1/3$&56.5&0.410&&21.1&0.362&&63.3&0.417&&91.9&0.504\tabularnewline
~~$x=0$&0.0&0.466&&0.0&0.414&&0.0&0.463&&55.9&0.591\tabularnewline
~~$x=1/3$&93.5&0.479&&94.1&0.404&&92.4&0.486&&91.9&0.504\tabularnewline
~~$x=2/3$&95.0&0.519&&93.3&0.436&&94.9&0.522&&94.4&0.606\tabularnewline
\hline
{\bfseries Model 2}&&&&&&&&&&&\tabularnewline
~~$x=-2/3$&94.9&0.495&&95.2&0.416&&95.1&0.503&&95.3&0.622\tabularnewline
~~$x=-1/3$&92.7&0.408&&57.9&0.350&&94.4&0.417&&84.2&0.432\tabularnewline
~~$x=0$&0.0&0.455&&0.0&0.403&&0.0&0.451&&55.3&0.591\tabularnewline
~~$x=1/3$&92.4&0.407&&58.0&0.350&&93.9&0.417&&83.9&0.430\tabularnewline
~~$x=2/3$&95.3&0.496&&95.0&0.417&&94.9&0.503&&93.9&0.623\tabularnewline
\hline
{\bfseries Model 3}&&&&&&&&&&&\tabularnewline
~~$x=-2/3$&94.4&0.384&&93.9&0.329&&94.9&0.405&&95.5&0.546\tabularnewline
~~$x=-1/3$&93.9&0.328&&91.4&0.277&&94.1&0.336&&94.8&0.400\tabularnewline
~~$x=0$&94.5&0.329&&87.5&0.277&&95.8&0.334&&95.5&0.470\tabularnewline
~~$x=1/3$&71.2&0.331&&77.5&0.281&&73.0&0.343&&71.4&0.406\tabularnewline
~~$x=2/3$&81.4&0.399&&74.7&0.343&&68.9&0.423&&91.7&0.547\tabularnewline
\hline
{\bfseries Model 4}&&&&&&&&&&&\tabularnewline
~~$x=-2/3$&94.9&0.507&&95.1&0.426&&95.0&0.513&&95.3&0.630\tabularnewline
~~$x=-1/3$&90.2&0.418&&51.8&0.358&&93.9&0.425&&82.5&0.435\tabularnewline
~~$x=0$&0.0&0.451&&0.0&0.403&&0.0&0.448&&72.0&0.592\tabularnewline
~~$x=1/3$&90.3&0.417&&52.3&0.357&&93.5&0.424&&82.5&0.433\tabularnewline
~~$x=2/3$&95.4&0.508&&95.0&0.427&&94.9&0.514&&94.1&0.631\tabularnewline
\hline
{\bfseries Model 5}&&&&&&&&&&&\tabularnewline
~~$x=-2/3$&94.6&0.560&&95.0&0.470&&94.4&0.562&&95.0&0.627\tabularnewline
~~$x=-1/3$&85.1&0.437&&55.0&0.370&&93.1&0.440&&94.1&0.507\tabularnewline
~~$x=0$&90.8&0.402&&73.5&0.340&&92.0&0.410&&93.6&0.498\tabularnewline
~~$x=1/3$&94.4&0.378&&94.1&0.319&&92.2&0.385&&94.5&0.477\tabularnewline
~~$x=2/3$&95.2&0.442&&94.7&0.373&&95.0&0.454&&94.3&0.554\tabularnewline
\hline
{\bfseries Model 6}&&&&&&&&&&&\tabularnewline
~~$x=-2/3$&94.3&0.368&&93.2&0.317&&94.9&0.392&&95.2&0.479\tabularnewline
~~$x=-1/3$&94.9&0.362&&94.4&0.305&&94.5&0.366&&94.8&0.408\tabularnewline
~~$x=0$&94.1&0.367&&93.0&0.309&&94.9&0.372&&95.2&0.420\tabularnewline
~~$x=1/3$&92.6&0.372&&86.8&0.313&&93.6&0.377&&93.9&0.414\tabularnewline
~~$x=2/3$&94.8&0.407&&94.5&0.344&&94.7&0.427&&94.4&0.487\tabularnewline
\hline
\end{tabular}
}\bigskip
	\end{center}}
\vspace{-.2in}\footnotesize\textbf{Notes}: Undersmoothing is implemented using bandwidths $h=\lambda\hat{h}_\MSE$ for $\lambda=\{0.5;0.7\}$, in the columns labeled as such. \\
\smallskip
\end{table}

\newpage
\begin{table}[h]%\renewcommand{\arraystretch}{1.1}
{\begin{center}
\caption{Empirical Coverage and Average Interval Length of RBC 95\% Confidence Intervals for Model 5, for Different Variance Estimators}\label{table:reg_comp_vce}
{%latex.default(out, file = paste("lp_vce", "_n", n, ".txt", sep = ""),     landscape = FALSE, n.rgroup = c(5, 5, 5, 5, 5), rgroup = c("$x=-2/3$",         "$x=-1/3$", "$x=0$", "$x=1/3$", "$x=2/3$"), outer.size = "scriptsize",     col.just = rep("c", 5), center = "none", title = "", table.env = FALSE)%
\begin{tabular}{lccc}
\hline\hline
\multicolumn{1}{l}{}&\multicolumn{1}{c}{$h$}&\multicolumn{1}{c}{EC}&\multicolumn{1}{c}{IL}\tabularnewline
\hline
{\bfseries $x=-2/3$}&&&\tabularnewline
~~$HC_0$&0.248&94.2&0.555\tabularnewline
~~$HC_1$&0.249&94.4&0.562\tabularnewline
~~$HC_2$&0.249&94.4&0.559\tabularnewline
~~$HC_3$&0.250&94.4&0.562\tabularnewline
~~$NN$&0.249&93.9&0.560\tabularnewline
\hline
{\bfseries $x=-1/3$}&&&\tabularnewline
~~$HC_0$&0.402&92.9&0.437\tabularnewline
~~$HC_1$&0.403&93.1&0.440\tabularnewline
~~$HC_2$&0.403&92.9&0.439\tabularnewline
~~$HC_3$&0.404&93.1&0.440\tabularnewline
~~$NN$&0.399&92.7&0.441\tabularnewline
\hline
{\bfseries $x=0$}&&&\tabularnewline
~~$HC_0$&0.473&91.9&0.408\tabularnewline
~~$HC_1$&0.474&92.0&0.410\tabularnewline
~~$HC_2$&0.474&91.9&0.409\tabularnewline
~~$HC_3$&0.474&92.0&0.410\tabularnewline
~~$NN$&0.474&91.7&0.404\tabularnewline
\hline
{\bfseries $x=1/3$}&&&\tabularnewline
~~$HC_0$&0.534&92.0&0.383\tabularnewline
~~$HC_1$&0.535&92.2&0.385\tabularnewline
~~$HC_2$&0.535&92.1&0.384\tabularnewline
~~$HC_3$&0.536&92.2&0.385\tabularnewline
~~$NN$&0.541&91.9&0.380\tabularnewline
\hline
{\bfseries $x=2/3$}&&&\tabularnewline
~~$HC_0$&0.396&94.8&0.450\tabularnewline
~~$HC_1$&0.398&95.0&0.454\tabularnewline
~~$HC_2$&0.397&94.9&0.452\tabularnewline
~~$HC_3$&0.398&95.0&0.454\tabularnewline
~~$NN$&0.400&94.7&0.452\tabularnewline
\hline
\end{tabular}
}\bigskip
	\end{center}}
\vspace{-.2in}\footnotesize\textbf{Notes}:\\
\smallskip
(i) The $h$ column reports the average estimated bandwidths $\hat{h}_\ROT$.
\end{table}

\newpage
\begin{landscape}

\newpage
\forloop{model}{1}{\value{model} < 7}{
\begin{figure}[h]
      \centering
       \caption{Empirical Coverage of 95$\%$ Confidence Intervals - Model \themodel }\label{table:reggrid_ec\themodel}

                       \begin{subfigure}[b]{0.5\textwidth}
          	\includegraphics[width=1.0\textwidth]{lp_grid_ec_m\themodel_c1_n500.pdf}
          	\caption{$x=-2/3$}
          \end{subfigure}%
          ~ %add desired spacing between images, e. g. ~, \quad, \qquad, \hfill etc.
          \begin{subfigure}[b]{0.5\textwidth}
          	\includegraphics[width=1.0\textwidth]{lp_grid_ec_m\themodel_c2_n500.pdf}
          	\caption{$x=-1/3$}
          \end{subfigure}%
          \begin{subfigure}[b]{0.5\textwidth}
          	\includegraphics[width=1.0\textwidth]{lp_grid_ec_m\themodel_c3_n500.pdf}
          	\caption{$x=0$}
          \end{subfigure}%
          
          \begin{subfigure}[b]{0.5\textwidth}
          	\includegraphics[width=1.0\textwidth]{lp_grid_ec_m\themodel_c4_n500.pdf}
          	\caption{$x=1/3$}
          \end{subfigure}%
          \begin{subfigure}[b]{0.5\textwidth}
          	\includegraphics[width=1.0\textwidth]{lp_grid_ec_m\themodel_c5_n500.pdf}
          	\caption{$x=2/3$}
          \end{subfigure}

\end{figure}
}

\newpage
\forloop{model}{1}{\value{model} < 7}{
\begin{figure}[h]
      \centering
       \caption{Average Interval Length of 95$\%$ Confidence Intervals - Model \themodel }\label{table:reggrid_il\themodel}

                     \begin{subfigure}[b]{0.5\textwidth}
        	\includegraphics[width=1.0\textwidth]{lp_grid_il_m\themodel_c1_n500.pdf}
        	\caption{$x=-2/3$}
        \end{subfigure}%
        ~ %add desired spacing between images, e. g. ~, \quad, \qquad, \hfill etc.
        \begin{subfigure}[b]{0.5\textwidth}
        	\includegraphics[width=1.0\textwidth]{lp_grid_il_m\themodel_c2_n500.pdf}
        	\caption{$x=-1/3$}
        \end{subfigure}%
        \begin{subfigure}[b]{0.5\textwidth}
        	\includegraphics[width=1.0\textwidth]{lp_grid_il_m\themodel_c3_n500.pdf}
        	\caption{$x=0$}
        \end{subfigure}%
        
        \begin{subfigure}[b]{0.5\textwidth}
        	\includegraphics[width=1.0\textwidth]{lp_grid_il_m\themodel_c4_n500.pdf}
        	\caption{$x=1/3$}
        \end{subfigure}%
        \begin{subfigure}[b]{0.5\textwidth}
        	\includegraphics[width=1.0\textwidth]{lp_grid_il_m\themodel_c5_n500.pdf}
        	\caption{$x=2/3$}
        \end{subfigure}

\end{figure}
}	
	
	\newpage
	\forloop{model}{1}{\value{model} < 7}{
\begin{figure}[h]
        \centering
        \caption{Empirical Coverage and Average Interval Length of 95$\%$ Confidence Intervals - Model \themodel }\label{fig:CI\themodel}

       \begin{subfigure}[b]{0.5\textwidth}
       	\includegraphics[width=1.0\textwidth]{lp_seg_m\themodel_c1_n500.pdf}
       	\caption{$x=-2/3$}
       \end{subfigure}%
       ~ %add desired spacing between images, e. g. ~, \quad, \qquad, \hfill etc.
       \begin{subfigure}[b]{0.5\textwidth}
       	\includegraphics[width=1.0\textwidth]{lp_seg_m\themodel_c2_n500.pdf}
       	\caption{$x=-1/3$}
       \end{subfigure}%
       \begin{subfigure}[b]{0.5\textwidth}
       	\includegraphics[width=1.0\textwidth]{lp_seg_m\themodel_c3_n500.pdf}
       	\caption{$x=0$}
       \end{subfigure}%
       
       \begin{subfigure}[b]{0.5\textwidth}
       	\includegraphics[width=1.0\textwidth]{lp_seg_m\themodel_c4_n500.pdf}
       	\caption{$x=1/3$}
       \end{subfigure}%
       \begin{subfigure}[b]{0.5\textwidth}
       	\includegraphics[width=1.0\textwidth]{lp_seg_m\themodel_c5_n500.pdf}
       	\caption{$x=2/3$}
       \end{subfigure}%
   \hspace*{-8em}
          \begin{subfigure}[b]{0.5\textwidth}
   		\includegraphics[width=1.0\textwidth]{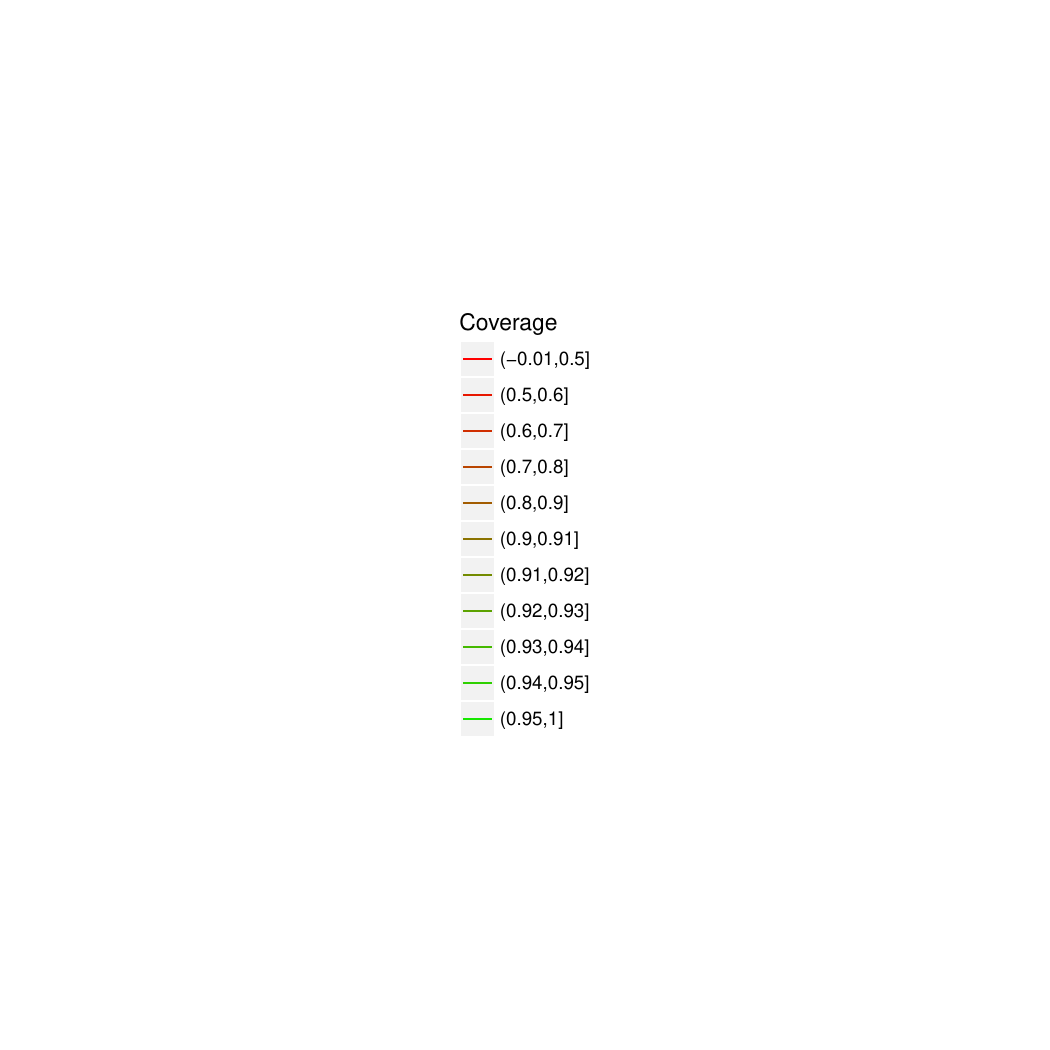}
   		 \end{subfigure}

\end{figure}
}

\end{landscape}

 \clearpage

\end{document}